\documentclass[final,oneside,notitlepage,10pt]{article}

 \pdfoutput=1

\usepackage{amsthm}
\usepackage{thmtools}
\usepackage{graphicx}
\usepackage[margin=2cm,font=normalsize,labelfont=bf]{caption}
\usepackage{verbatim}
\usepackage[shortlabels]{enumitem}
\usepackage{fancyhdr}
\usepackage[]{geometry}
\usepackage{xstring}
\usepackage{needspace}
\usepackage{color}
\usepackage{amstext}
\usepackage{nameref}
\usepackage[hypertexnames=false]{hyperref}
\usepackage{mathdots}
\usepackage{soul}
\usepackage{chngcntr}
\usepackage[section]{placeins}

\makeatletter
\newcounter{denseversion}
\newcounter{comments}
\newcounter{authorcounter}
\newcounter{adresscounter}

\def\title#1{\gdef\@title{#1}}
\def\@title{}
\def\subtitle#1{\gdef\@subtitle{#1}}
\def\@subtitle{}

\def\authortagsused{0}
\def\adresstag#1{\if!#1!\else$^{\;#1\;}$\fi}
\def\@authorsep#1{
  \ifnum\value{authorcounter}=#1 and \else\unskip, \fi
}

\renewcommand{\author}[2][]{
  \stepcounter{authorcounter}
  \if!#1!\else\gdef\authortagsused{1}\fi
  \ifnum\value{authorcounter}=1
    \def\@authorstringa{#2\adresstag{#1}}
    \def\@authorstringb{#2}
    \def\@authorstringc{#2\adresstag{#1}}
  \else
    \ifnum\value{authorcounter}=2
      \g@addto@macro\@authorstringa{\@authorsep{2}#2\adresstag{#1}}
      \g@addto@macro\@authorstringb{\@authorsep{2}#2}
      \g@addto@macro\@authorstringc{\\#2\adresstag{#1}}
    \else
      \g@addto@macro\@authorstringa{\@authorsep{3}#2\adresstag{#1}}
      \g@addto@macro\@authorstringb{\@authorsep{3}#2}
      \g@addto@macro\@authorstringc{\\#2\adresstag{#1}}
    \fi
  \fi}
\def\@author{\ifnum\value{denseversion}=0\@authorstringa\else\@authorstringb\fi}

\def\@adressstringa{}
\def\@adressstringb{}
\newcommand{\adress}[2][]{
  \stepcounter{adresscounter}
  \ifnum\value{adresscounter}=1
    \g@addto@macro\@adressstringa{\ifnum\authortagsused=0\def\br{\\}\else\def\br{, }\fi\adresstag{#1}#2}
    \g@addto@macro\@adressstringb{\def\br{\\}\adresstag{#1}\parbox[t]{14cm}{#2}}
  \else
    \g@addto@macro\@adressstringa{\\[\bigskipamount]\adresstag{#1}#2}
    \g@addto@macro\@adressstringb{\\[\medskipamount]\adresstag{#1}\parbox[t]{14cm}{#2}}
  \fi}

\def\preprint#1{\gdef\@preprint{#1}}
\def\@preprint{}
\def\keywords#1{\gdef\@keywords{#1}}
\def\@keywords{}
\def\msc#1{\gdef\@msc{#1}}
\def\@msc{}
\def\email#1{
   \gdef\@email{#1}
   \g@addto@macro\@authorstringc{ {\it (#1)}}}
\def\@email{}
\def\dedication#1{\gdef\@dedication{#1}}
\def\@dedication{}

\def\mybaselinestretch#1{
  \gdef\@mybaselinestretch{#1}
  \renewcommand{\baselinestretch}{\@mybaselinestretch}}

\def\myparskip#1{
  \gdef\@myparskip{#1}
  \setlength{\parskip}{\@myparskip}}

\mybaselinestretch{1.2}
\myparskip{1.5ex}

\binoppenalty=\maxdimen
\relpenalty=\maxdimen

\normalfont
\normalsize
\newlength{\@listleftmargin}
\def\setenumstandard{
  \setlist{leftmargin=\@listleftmargin,itemsep=0pt,topsep=0pt,partopsep=0pt,parsep=\@myparskip}
  \setlist[enumerate]{align=left,labelsep=*,leftmargin=\@listleftmargin,itemsep=0pt,topsep=0pt,partopsep=0pt,parsep=\@myparskip}
}

\def\denseversion{
  \setcounter{denseversion}{1}
  \newgeometry{left=3cm,right=3cm,top=3cm}
  \mybaselinestretch{1.1}
  \myparskip{0.8ex}
  \normalfont
  \def\possiblelinebreak{}
  \fancyfoot[C]{\itshape{--$\,\,$\thepage$\,\,$--}}}

\def\possiblelinebreak{\\}

\renewcommand{\emph}[1]{\def\reserved@a{it}\ifx\f@shape\reserved@a\ul{#1}\else\textit{#1}\fi}

\def\setcrefnames{}
\newcommand{\mytableofcontents}{
   \ifnum\value{denseversion}=0
     \tableofcontents
     \setcrefnames 
   \else
     \renewcommand{\baselinestretch}{1.1}
     \setlength{\parskip}{0ex}
     \normalfont
     \begingroup
     \def\addvspace##1{\vskip0.4em}
     \tableofcontents
     \setcrefnames 
     \endgroup
     \renewcommand{\baselinestretch}{\@mybaselinestretch}
     \setlength{\parskip}{\@myparskip}
     \normalfont
   \fi}

\newlength{\zeilenlaenge}
\def\putindent#1{
  \settowidth{\zeilenlaenge}{#1}
  \ifnum\zeilenlaenge>\textwidth
    #1
  \else
    \noindent #1
  \fi
}

\hypersetup{
    linktocpage = true,
    colorlinks,
    linkcolor=[RGB]{0,0,120},
    citecolor=[RGB]{0,0,120},
    urlcolor=[RGB]{0,0,120}
}

\def\pdfdaten{
  \hypersetup{
    pdftitle = {\@title},
    pdfauthor = {\@author},
    pdfkeywords = {\@keywords},    
    bookmarksopen = true,
    bookmarksopenlevel = 1
  }}  

\def\showkeywords{\begin{flushleft}\footnotesize\textbf{Keywords}: \@keywords\end{flushleft}}
\def\showmsc{\begin{flushleft}\footnotesize\textbf{MSC 2010}: \@msc\end{flushleft}}

\def\tocsection#1{\section*{#1}\addcontentsline{toc}{section}{#1}}

\def\mytitle{}
\def\zmptitle{
  \begin{tabular}{cc}
    \begin{minipage}[c]{0.4\textwidth}
      \begin{flushleft}
        \includegraphics[width=110pt]{../../tex/zmp}
      \end{flushleft}  
    \end{minipage}&
    \begin{minipage}[c]{0.55\textwidth}
      \begin{flushright}
      {\small\sf\@preprint}
      \end{flushright}
    \end{minipage}
  \end{tabular}
  \vskip 2cm}

\def\maketitle{
  \pdfdaten
  \noindent
  \mytitle
  \begin{center}
    \LARGE\@title\\
    \if!\@subtitle!\else\smallskip\LARGE\@subtitle\\\fi
    \bigskip
    \if!\@author!\else\bigskip\large\@author\\\fi
    \ifnum\value{denseversion}=0
      \if!\@adressstringa!\else\bigskip\normalsize\@adressstringa\\\fi
      \if!\@email!\else\ifnum\value{authorcounter}=1\bigskip\normalsize\textit{\@email}\\\else\fi\fi
    \else
    \fi
    \if!\@dedication!\else\bigskip\normalsize{\@dedication}\\\fi
  \end{center}
  \ifnum\value{denseversion}=0\vskip 1.5cm\else\vskip0.5cm\fi}

\def\kobib#1{
  \begin{raggedright}
  \ifnum\value{denseversion}=0\else\small\fi
  \Oldbibliography{#1/kobib}
  \bibliographystyle{#1/kobib}
  \end{raggedright}
  \ifnum\value{denseversion}=0\else
      \noindent
      \if!\@authorstringc!\else
        \ifnum\authortagsused=0\ifnum\value{authorcounter}>1\normalsize\@authorstringc\\[\medskipamount]\else\fi\else\normalsize\@authorstringc\\[\medskipamount]\fi
      \fi
      \if!\@adressstringb!\else\normalsize\@adressstringb\\{}\fi
      \ifnum\authortagsused=0
        \ifnum\value{authorcounter}=1
          \if!\@email!\else\linebreak\normalsize\textit{\@email}\\{}\fi
        \else
        \fi
      \else
      \fi
  \fi}

\let\Oldbibliography\bibliography
\def\bibliography#1{
  \begin{raggedright}
  \ifnum\value{denseversion}=0\else\small\fi
  \Oldbibliography{#1}
  \end{raggedright}
  \ifnum\value{denseversion}=0\else
      \medskip
      \noindent
      \if!\@authorstringc!\else
        \ifnum\authortagsused=0\ifnum\value{authorcounter}>1\normalsize\@authorstringc\\[\medskipamount]\else\fi\else\normalsize\@authorstringc\\[\medskipamount]\fi
      \fi
      \if!\@adressstringb!\else\normalsize\@adressstringb\\{}\fi
      \ifnum\authortagsused=0
        \ifnum\value{authorcounter}=1
          \if!\@email!\else\linebreak\normalsize\textit{\@email}\\{}\fi
        \else
        \fi
      \else
      \fi
  \fi
}

\newenvironment{commentfigure}{}
\newenvironment{sidewayscommentfigure}{\begin{minipage}}{\end{minipage}}
\newenvironment{displaycomment}{\begin{list}{}{\rightmargin=1cm\leftmargin=1cm}\item\sf\begin{small}}{\end{small}\end{list}}

\def\tocmark#1{}

\def\draftstamp#1{
  \def\tocmark##1{
    \ifnum\c@secnumdepth=0\section{##1}\fi
    \ifnum\c@secnumdepth=1\subsection{##1}\fi
    \ifnum\c@secnumdepth=2\subsubsection{##1}\fi
    \ifnum\c@secnumdepth=3\subsubsection{##1}\fi
  }
  \ifnum\value{comments}=0
    \gdef\@draft{DRAFT - Edited on \today\ by #1 - Comments are not displayed}
  \else
    \gdef\@draft{DRAFT - Edited on \today\ by #1 - Comments are displayed}
  \fi
  \fancyhead[C]{\footnotesize\tt\textcolor{red}{\@draft}}}

\def\skript{
  \renewenvironment{displaycomment}{\begin{color}{blue}}{\end{color}}
  \ifnum\value{comments}=0
    \renewenvironment{skript-exercise}{\comment}{\endcomment}
    \renewenvironment{skript-example}{\comment}{\endcomment}
    \renewenvironment{skript-remark}{\comment}{\endcomment}
  \else
    \renewenvironment{skript-exercise}{\begin{color}{blue}\begin{exercise*}}{\end{exercise*}\end{color}}
    \renewenvironment{skript-example}{\begin{color}{blue}\begin{example*}}{\end{example*}\end{color}}
    \renewenvironment{skript-remark}{\begin{color}{blue}\begin{remark*}}{\end{remark*}\end{color}}
  \fi
  \parindent=0mm	
}

\newcounter{sectioning}
\def\tsection#1{\ifnum\value{sectioning}=0\section{#1}\fi}
\def\lsection#1{
  \ifnum\value{sectioning}=1
    \clearpage
    \def\thesection{\lektionname~\arabic{section}:}
    \section{#1}
    \def\thesection{\arabic{section}}
  \fi}
\def\tsubsection#1{\ifnum\value{sectioning}=0\subsection{#1}\fi}
\def\lsubsection#1{\ifnum\value{sectioning}=1\subsection{#1}\fi}
\def\tsubsubsection#1{\ifnum\value{sectioning}=0\subsubsection{#1}\fi}
\def\lsubsubsection#1{\ifnum\value{sectioning}=1\subsubsection{#1}\fi}

\mybaselinestretch{}
\setenumstandard

\makeatother

\input{kobib}

\usepackage{amssymb}
\usepackage{amsmath}
\usepackage{amstext}
\usepackage{mathrsfs}
\usepackage{euscript}
\usepackage{color}
\usepackage[all,cmtip,color,arrow]{xy}
\usepackage{xstring}
\usepackage{slashed}
\usepackage{tensor}
\usepackage{tikz}

\def\mathscr#1{\EuScript{#1}}

\entrymodifiers={+!!<0pt,\fontdimen22\textfont2>}

\def\Z {\mathbb{Z}}

\def\id{\mathrm{id}}

\def\hc#1{\mathrm{h}_{#1}}
\def\h {\mathrm{H}}

\def\subset{\subseteq}

\def\sep{\;|\;}

\def\maps{\colon}

\def\Ad{\mathrm{Ad}}
\renewcommand{\varepsilon}{\epsilon}

\newcommand{\incl}{\hookrightarrow}
\renewcommand{\iff}{\quad\Leftrightarrow\quad}

\makeatletter
\renewenvironment{proof}[1][\nameProof]
  {\par\pushQED{\qed}%
   \normalfont \topsep6\p@\@plus6\p@\relax
   \trivlist
   \item[\hskip\labelsep
         \itshape
         #1\@addpunct{.}]
  \leavevmode}
  {\popQED\endtrivlist\@endpefalse}
\makeatother

\def\notebox#1#2{\begin{minipage}[b]{#1}\sloppy\renewcommand{\baselinestretch}{0.8}\footnotesize \begin{center}#2\end{center}\end{minipage}}

\def\mquad{\hspace{-2em}}
\def\mqquad{\hspace{-4em}}

\newcommand{\arr}[1][r]{\ar@<0.7ex>[#1]\ar@<-0.7ex>[#1]}
\newcommand{\arrr}[1][r]{\ar@<1.4ex>[#1]\ar[#1]\ar@<-1.4ex>[#1]}
\newcommand{\arrrr}[1][r]{\ar@<2.1ex>[#1]\ar@<-2.1ex>[#1]\ar@<0.7ex>[#1]\ar@<-0.7ex>[#1]}

\def\erf#1{(\ref{#1})}

\def\stackref#1#2{\stackrel{\text{\ref{#1}}}{#2}}
\def\eqref#1{\stackref{#1}{=}}

\newlength{\myeqt} 
\newlength{\myeqs} 
\newlength{\myeqm} 
\newlength{\myeqn} 
\newcommand\symtext[3][\myeqn]{
  \settowidth{\myeqt}{#2}
  \settowidth{\myeqs}{$#3$}
  \addtolength{\myeqs}{\the\myeqm}
  \ifdim\myeqt>\myeqs
    \stackrel{\hspace{-#1}\notebox{#1}{\medskip #2 \\ $\downarrow$\smallskip}\hspace{-#1}}{#3}
  \else
    \stackrel{\text{#2}}{#3}
  \fi}
\newcommand\eqcref[2][\myeqn]{\symcref[#1]{#2}{=}}

\newcommand\symcref[3][\myeqn]{\symtext[#1]{\cref{#2}}{#3}}

\def\brackets#1{\IfStrEq{#1}{-}{}{(#1)}}
\def\subindex#1{\IfStrEq{#1}{-}{}{_{#1}}}

\newcommand{\alxydim}[2]{\begin{aligned}\xymatrix#1{#2}\end{aligned}}

\makeatletter
\def\bigset#1#2{\left\lbrace\;\begin{minipage}[c]{#1}\begin{center}#2\end{center}\end{minipage}\;\right\rbrace}

\makeatother

\newlength{\myl}
\newcommand\sheaf[1]{\unitlength 0.1mm
  \settowidth{\myl}{$#1$}
  \addtolength{\myl}{-0.8mm}
  \begin{picture}(0,0)(0,0)
  \put(2,0){\text{\underline{\hspace{\myl}}}}
  \end{picture}#1\hspace{-0.15mm}}

\def\ddt#1#2#3{\left.\frac{\mathrm{d}^{\IfStrEq{#1}{1}{}{#1}}}{\mathrm{d}#2}\IfStrEq{#2}{#3}{\right.}{\right|_{#3}}}

\def\Aut{\mathrm{Aut}}
\def\AUT{\mathscr{A}\mathrm{ut}}

\newcommand{\ueins}{{\mathrm{U}}(1)}

\newcommand{\spin}[1]{{\mathrm{Spin}}\brackets{#1}}
\def\Spin{\spin}

\newcommand{\str}[1]{{\mathrm{String}}\brackets{#1}}

\def\man{\mathcal{M}\!a\!n}
\def\Man{\mathscr{M}\mathrm{an}}

\def\des{\mathcal{D}\!esc}
\def\fun{\mathcal{F}un}

\def\hom{\mathcal{H}\!om}
\def\Hom{\mathscr{H}\mathrm{om}}

\def\trivtech{\mathcal{T}\hspace{-0.34em}r\hspace{-0.06em}i\hspace{-0.07em}v}
\def\triv#1{\trivtech\brackets{#1}}
\def\trivcon#1{\trivtech^{\nabla\!}\brackets{#1}}

\def\lift#1#2{#1\text{-}\mathcal{L}i\!f\!t(#2)}
\def\liftcon#1#2#3{#1\text{-}\mathcal{L}i\!f\!t^{\nabla\!}_{#3}(#2)}

\def\ev{\mathrm{ev}}

\def\idmorph#1{#1_{dis}}

\def\pr{{\mathrm{pr}}}

\newlength{\widthtmp}
\def\length#1{\settowidth{\widthtmp}{#1}\the\widthtmp}

\def\lli#1{{}_#1}
\def\ttimes#1#2{\hspace{-0.15em}\tensor[_{#1}]{\times}{_{#2}}}

\def\buntech#1#2{\mathcal{B}\hspace{-0.01em}un_{\hspace{0.05em}#1}^{#2}}
\def\trivlin{\mathbf{I}}

\def\bun#1#2{\buntech{#1}{}\brackets{#2}}
\def\buncon#1#2{\buntech{#1}{\nabla}\hspace{-0.05em}\brackets{#2}}

\def\bunconflat#1#2{\buntech#1{\nabla_{\!0}}\hspace{-0.05em}\brackets{#2}}

\def\grbtech#1{\mathcal{G}\hspace{-0.06em}r\hspace{-0.06em}b_{\hspace{0.02em}{#1}}}

\def\grbcon#1#2{\grbtech{#1}^{\nabla\!}\brackets{#2}}

\def\trivgrbcon#1#2{\trivtech\grbtech{{#1}}^{\nabla\!}\brackets{#2}}

\usepackage[latin1]{inputenc}
\usepackage[english]{babel}
\usepackage[english]{cleveref}

\def\refname{References}

\def\quot#1{``#1''}

\def\quand{\quad\text{ and }\quad}
\def\quomma{\quad\text{, }\quad}
\def\quere{\quad\text{ where }\quad}
\def\quith{\quad\text{ with }\quad}

\def\nameTheorem{Theorem}
\def\nameTheorems{Theorems}
\def\nameDefinition{Definition}
\def\nameDefinitions{Definitions}
\def\nameCorollary{Corollary}
\def\nameCorollaries{Corollaries}
\def\nameProposition{Proposition}
\def\namePropositions{Propositions}
\def\nameConstruction{Construction}
\def\nameConstructions{Constructions}
\def\nameLemma{Lemma}
\def\nameLemmas{Lemmas}
\def\nameRemark{Remark}
\def\nameRemarks{Remarks}
\def\nameProblem{Problem}
\def\nameProblems{Problems}
\def\nameExample{Example}
\def\nameExamples{Examples}
\def\nameExercise{Exercise}
\def\nameExercises{Exercises}
\def\nameWarnung{Warning}
\def\nameWarnungs{Warnings}

\def\nameProof{Proof}
\def\nameEquation{Eq.}
\def\nameEquations{Eqs.}
\def\nameSection{Section}
\def\nameSections{Sections}
\def\nameAppendix{Appendix}
\def\nameAppendixs{Appendices}
\def\nameFigure{Figure}
\def\nameFigures{Figures}

\input{theorems}

\def\ff{f\!f}

\def\gen{g\hspace{-0.08em}e\hspace{-0.06em}n}
\def\adj{a\hspace{-0.08em}d\hspace{-0.06em}j}
\def\aptadj{apt\text{-}a\hspace{-0.08em}d\hspace{-0.06em}j}
\def\con#1#2{\mathcal{C}\!on_{#1}\brackets{#2}}
\def\congen#1#2{\mathcal{C}\!on^{{\gen}}_{#1}\brackets{#2}}
\def\conaa#1#2{\mathcal{C}\!on^{{\aptadj}}_{#1}\brackets{#2}}
\def\conadj#1#2{\mathcal{C}\!on^{{\adj}}_{#1}\brackets{#2}}

\def\conff#1#2{\mathcal{C}\!on^{{f\!f}}_{#1}\brackets{#2}}

\def\trivgrbconaa#1#2{\trivtech\grbtech{{#1}}^{\nabla\!}\brackets{#2}^{adpt}}
\def\trivgrbconff#1#2{\trivtech\grbtech{{#1}}^{\nabla\!}\brackets{#2}^{f\!f}}
\def\grbconff#1#2{\grbtech{#1}^{\nabla\!}\brackets{#2}^{f\!f}}
\def\grbconadj#1#2{\grbtech{#1}^{\nabla\!}\brackets{#2}^{adj}}

\def\grbconaa#1#2{\grbtech{#1}^{\nabla\!}\brackets{#2}^{a\hspace{-0.08em}d\hspace{-0.05em}p\hspace{-0.06em}t}}
\def\inv{\mathrm{inv}}
\def\der{\mathfrak{der}}
\def\refine{r\hspace{-0.08em}e\hspace{-0.15em}f}
\def\dis{d\hspace{-0.05em}i\hspace{-0.08em}s}

\def\bibun#1#2{\mathcal{B}i\bun#1#2}
\def\bibuncon#1#2{\mathcal{B}i\buncon#1#2}

\def\pflat{\pi_0\text{-flat}}
\def\pcurved#1{\pi_0\text{-}#1}

\def\checknerve#1{C(#1)}
\def\prestack#1{\overline{#1}}

\def\1mor#1{1\text{-}\mathrm{Mor}(#1)}
\def\2mor#1{2\text{-}\mathrm{Mor}(#1)}

\crefname{equation}{\unskip}{\unskip}

\def\ADJ{\mathcal{A}dj}
\def\BiGrpd{\mathcal{B}i\mathcal{G}\hspace{-0.05em}r\hspace{-0.05em}pd}

\usepackage{cancel}
\usepackage{tikz-cd}

\title{Adjusted connections on non-abelian bundle gerbes}
\author{Konrad Waldorf}

\email{konrad.waldorf@uni-greifswald.de}

\adress{Universität Greifswald\\
Institut für Mathematik und Informatik\\
Walther-Rathenau-Str. 47\\
D-17487 Greifswald}

\msc{}

\denseversion
\allowdisplaybreaks
\usepackage{amstext}
\usepackage{amssymb}
\usepackage{amsmath}
\usepackage{color}

\usepackage[normalem]{ulem}
\usepackage{todonotes}

\begin{document}


\vspace*{-5em}

        \maketitle 

\begin{abstract}
Higher gauge theory for non-abelian structure 2-groups faces significant challenges when extending beyond the fake-flat sector, which suffers from limited applicability in physical models. A promising resolution involves equipping 2-groups with additional structure, known as adjustments. 
We present a comprehensive theory of adjusted connections on non-abelian bundle gerbes, classified by Saemann's adjusted version of non-abelian differential cohomology. This theory enables, in particular, a new coordinate-independent formulation of Tellez-Dominguez' lifting theorem, establishing a correspondence between adjusted connections on non-abelian bundle gerbes and connections on abelian bundle 2-gerbes.
\end{abstract}

\setcounter{tocdepth}{2}
\mytableofcontents


\setsecnumdepth{1}

\section{Introduction}

Bundle gerbes serve as the most common differential-geometric realization of bundles for categorical Lie groups, also known as Lie 2-groups. Their primary advantages are twofold: (1)  the structural simplicity, relying solely on basic differential geometry, and (2) the accessibility of connections through ordinary connections and differential forms. Among bundle gerbes, those associated with the Lie 2-group $B\ueins$ (abelian bundle gerbes) are the most prevalent. These were first developed by Murray, Carey, and collaborators \cite{murray,carey5,murray2,carey2} and were instantly adopted by Gaw\c edzki for applications in conformal field theory \cite{gawedzki5,gawedzki1,gawedzki2}, where connections on abelian bundle gerbes model the B-field in string theory.

Generalizing this framework to Lie 2-groups beyond $BA$ (where $A$ is an abelian Lie group) proved challenging, despite being strongly demanded by
applications in string theory. Aschieri, Cantini, and Jur\v co pioneered a generalization to a non-abelian Lie group $H$ in \cite{aschieri}. While they also investigated connections, they immediately encountered difficulties in defining them consistently beyond the \quot{fake-flat} regime.

A key conceptual hurdle was to recognize that the adequate generalization of abelian Lie groups is not merely non-abelian Lie groups (which led to the term \quot{non-abelian bundle gerbes}), but rather Lie 2-groups. Non-abelian Lie groups $H$ embed into this setting through their automorphism 2-group $\AUT(H)$. In this article, we model general Lie 2-groups through their Lie crossed modules $\Gamma=(H \stackrel{t}{\to} G)$.

The above-mentioned fake-flatness is a condition for categorical connections, discovered by Breen and Messing \cite{breen1}. Locally, such a $\Gamma\!$-connection consists of a pair $(A,B)$ comprising a 1-form $A\in \Omega^1(U,\mathfrak{g})$ and a 2-form $B\in \Omega^2(U,\mathfrak{h})$. Their fake-curvature is defined as
\begin{equation*}
\mathrm{fcurv}(A,B) := \mathrm{d}A + \tfrac{1}{2}[A\wedge A]- t_{*}(B) \in \Omega^2(U,\mathfrak{g})\text{.}
\end{equation*}  
Here, $t_{*}: \mathfrak{h} \to \mathfrak{g}$ denotes the Lie algebra crossed module associated with $\Gamma$. Fake-flatness means the vanishing of the fake curvature. Imposing this condition simplifies the theory of connections on non-abelian gerbes. For instance, in our joint work with Schreiber \cite{schreiber2,schreiber6,schreiber5}, we demonstrated that fake-flat connections admit a well-defined parallel transport along surfaces, and furthermore, that fake-flatness is necessary for this feature. However, in many physical applications, fake-flatness is an undesirable constraint; for example, in string theory, it would restrict the application of bundle gerbes for the string 2-group $\Gamma=\str d$ to sigma models on \emph{flat} target spaces.

{Naively abandoning fake-flatness, surprisingly, jeopardizes the existence of 2-morphisms in the bicategory of $\Gamma\!$-connections (see \cref{rem:twcurvunder2}). These 2-morphisms are essential for the gluing conditions on triple intersections.}
Fortunately, Saemann et al. resolved this issue by modifying a single detail in the definition of the 1-morphisms (the gauge transformations) \cite{Kim2020,Rist2022}. Specifically, a gauge transformation is a pair $(g,\varphi)$ consisting of a function $g:U \to G$ and a 1-form $\varphi \in \Omega^1(U,\mathfrak{h})$. One can then define a curvature notion, $\mathrm{curv}(g,\varphi) \in \Omega^2(U,\mathfrak{h})$. In the fake-flat case, one requires $\mathrm{curv}(g,\varphi)=0$. Insisting on this condition for non fake-flat $\Gamma\!$-connections leads to the before-mentioned constraints on 2-morphisms, while  imposing no condition at all on $\mathrm{curv}(g,\varphi)$ makes the theory arbitrary.  
The crucial insight was, instead, to relate the fake-curvature to the curvature of these gauge transformations. This coupling necessitates a new additional structure on Lie 2-groups, termed an \emph{adjustment}. An adjustment is a map
\begin{equation}
\label{adj-intro}
\kappa: G \times \mathfrak{g} \to \mathfrak{h}
\end{equation}
that is linear in $\mathfrak{g}$ and satisfies several further conditions. The corrected condition on gauge transformations reads
\begin{equation}
\label{adj-intro-2}
\mathrm{curv}(g,\varphi) = \kappa(g,\mathrm{fcurv}(A,B))\text{.}
\end{equation}
Saemann et al. utilized this local formalism to establish an adjusted version of non-abelian differential cohomology, denoted $\hat\h^1(M,\Gamma)^{\adj}$. It represents the most general possible version of non-abelian higher gauge theory. 

In the case of fake-flatness, \cref{adj-intro-2} reduces to $\mathrm{curv}(g,\varphi) = 0$, which becomes independent of the adjustment.
Moreover, it follows immediately from the setup in \cref{adj-intro} that abelian Lie 2-groups $BA$ (where $\mathfrak{g}=0$) possess a unique adjustment $\kappa=0$; similarly, for a categorically discrete Lie 2-group $F_{dis}$ associated with an ordinary Lie group $F$ (which embeds standard gauge theory into higher gauge theory), we have $\mathfrak{h}=0$, again forcing $\kappa=0$. This clarifies why adjustments do not appear in either fake-flat or abelian higher gauge theory, or in ordinary gauge theory.
A comprehensive discussion of the existence and classification of adjustments on Lie 2-groups was recently presented in our joint work with Ludewig \cite{Ludewig2025}.

With the local formalism established via adjustments, the primary aim of this paper is to develop a corresponding formalism for adjusted connections on bundle gerbes. Building on the work of Nikolaus and Schweigert, bundle gerbes arise as the sheafification of the local situation, specifically, of the presheaf $U \mapsto \conadj{\Gamma\!,\kappa} U$ of adjusted $\Gamma\!$-connections. Thus, the first goal of this paper is to provide a detailed procedural elaboration of this sheafification. In \cref{Lie-2-groups-and-connections}, we review and set up the local formalism of adjusted Lie 2-group connections outlined above. In \cref{section-2-group-bundles}, we review principal 2-group bundles, introduce adjusted connections on them, and prove that this construction provides a key ingredient for the sheafification. In \cref{sec:bundlegerbes}, we perform the sheafification and describe the resulting structure -- the bicategory of non-abelian bundle gerbes with adjusted connections -- in formulas:
\begin{equation*}
\grbcon {\Gamma\!,\kappa} M := \big(\overline{\conadj{\Gamma\!,\kappa} -}\big )^{+}(M)
\end{equation*}
We also compare our results for $\Gamma=\AUT(H)$ with the work of Aschieri et al. in \cref{prop:equivbibun,ascj-2,rem:aschieri}. We find full coincidence up to conventional differences, with the notable addition of adjustments on our side. The following theorem summarizes the consequences of our setup, see \cref{th:equivcongrb,classification}.   

\begin{maintheorem}
\label{main:one}
Let $\Gamma$ be a central Lie 2-group equipped with an adjustment, let $M$ be a smooth manifold, and let $\grbcon{\Gamma\!,\kappa} M$ be the bicategory of $\Gamma\!$-bundle gerbes with adjusted connection. The following statements hold:
\begin{enumerate}[(i)]

\item 
\label{main:one:1}
The assignment $M \mapsto \grbcon{\Gamma\!,\kappa} M$ defines a sheaf of bicategories on the site of smooth manifolds. 

\item
\label{main:one:2}
Up to isomorphism, $\Gamma\!$-bundle gerbes with adjusted connection are classified by the adjusted non-abelian differential cohomology of Saemann et al.:
\begin{equation*}
\hc 0 (\grbcon{\Gamma\!,\kappa} M)\cong\hat\h^1(M,(\Gamma,\kappa))^{\adj}\text{.}
\end{equation*}

\end{enumerate}
\end{maintheorem}

The second objective of this paper is to consider the homotopy Lie groups $A=\pi_1(\Gamma)$ and $F=\pi_0(\Gamma)$ of a central and smoothly separable Lie 2-group, and to investigate the relation between $\Gamma\!$-bundle gerbes, abelian $A$-bundle gerbes, and principal $F$-bundles, in a setting with connections.
This investigation reveals the necessity to consider the subcategory 
\begin{equation*}
\grbconaa{\Gamma,u,\kappa} M \subset \grbcon{\Gamma\!,\kappa} M
\end{equation*}
consisting of $\Gamma\!$-bundle gerbes with \emph{adapted} adjusted connections, a concept discovered by T\'ellez-Dom\'inguez \cite{Tellez2023}. Another additional structure for Lie 2-groups, a so-called \emph{splitting} $u$, becomes important. A splitting singles out a subset of \emph{adapted} adjustments and, consequently, a subset of \emph{adapted} and adjusted $\Gamma\!$-connections. A corresponding version of \Cref{main:one} holds for adapted and adjusted connections: (i) $\Gamma\!$-bundle gerbes with adapted and adjusted connections form a sheaf of bicategories, and (ii) they are classified by an adapted and adjusted version of differential cohomology. 
Moreover, we establish the following result, which summarizes \cref{gerbe-morphisms-under-i}. 

\begin{maintheorem}
\label{main:one:3}
If $\Gamma$ is a smoothly separable central Lie 2-group, with homotopy Lie groups $A=\pi_1(\Gamma)$ and $F=\pi_0(\Gamma)$, the sequence
\begin{equation*}
\grbcon A M \stackrel{i_{*}}\to \grbconaa{\Gamma,u,\kappa} M \stackrel{p_{*}}\to \buncon FM_{\dis}
\end{equation*}
of bicategories and functors, expressing the inclusion of abelian bundle gerbes with connection and the projection to an ordinary principal bundle with connection, respectively, is exact. That is, for a $\Gamma\!$-bundle gerbe $\mathcal{G}$ with adapted and adjusted connection, the following are equivalent: 
\begin{enumerate}[(1)]

\item 
Its structure 2-group can be reduced to $A$; i.e., there exists an $A$-bundle gerbe $\mathcal{G}^{ab}$ with connection and an isomorphism $\mathcal{G}\cong i_{*}(\mathcal{G}^{ab}).$ 

\item
$p_{*}(\mathcal{G})$ admits a flat section, i.e., one with vanishing covariant derivative. 

\end{enumerate}
\end{maintheorem}

We also generalize the statement of \cref{main:one:3} to include arbitrary (not necessarily flat) sections (\cref{gerbe-morphisms-under-i-non-flat}), at the cost of varying the functor $i_{*}$ depending on the covariant derivative of the section.

The third objective of this paper is to establish the lifting theory for non-abelian bundle gerbes with adjusted connections, generalizing the existing lifting theory for principal bundles \cite{murray,gomi3}. This part synthesizes our joint work with Nikolaus \cite{Nikolausa}, which treats non-abelian bundle gerbes without connections, and the work of T\'ellez-Dom\'inguez \cite{Tellez2023}, who addressed connections  and already achieved our goal at the cocycle level. The goal of lifting theory is to determine when a given principal $F$-bundle with connection arises as $p_{*}(\mathcal{G})$ for some non-abelian bundle gerbe with connection. In \cite{Nikolausa}, we proved that -- without connections -- the obstruction is represented by an abelian bundle 2-gerbe, the Chern-Simons 2-gerbe $\mathbb{CS}_P(\mathcal{G}_{\Gamma})$ constructed from the given bundle $P$ and the multiplicative $A$-bundle gerbe $\mathcal{G}_\Gamma$ associated to the Lie 2-group $\Gamma$. In \cref{multiplicative-bundle-gerbe} we review the construction of $\mathcal{G}_{\Gamma}$, in particular, T\'ellez-Dom\'inguez' work on how an adjustment $\kappa$ induces a connection on $\mathcal{G}_{\Gamma}$. The following result is proved as \cref{lifting-theorem}.

\begin{maintheorem}
\label{main:two}
Let $\Gamma$ be a smoothly separable, central crossed module equipped with a splitting and an adapted adjustment. Let $P$ be a principal $F$-bundle with connection. Then, there is an equivalence of bicategories: 
\begin{equation*}
\bigset{14em}{Trivializations of the Chern-Simons 2-gerbe $\mathbb{CS}_P(\mathcal{G}_{\Gamma})$ with compatible connections}
\cong 
\bigset{14em}{$\Gamma\!$-bundle gerbes $\mathcal{G}$ with adjusted and adapted connections such that $p_{*}(\mathcal{G})\cong P$}
\end{equation*}
\end{maintheorem}

The proof relies on descent theory, specifically on the fact that bundle gerbes with connection form a sheaf, as established in \cref{main:one}. Its most direct consequence is the identification of two distinct approaches to geometric string structures, obtained by applying \cref{main:two} to the string 2-group $\Gamma= \str d$ and to the spin-oriented frame bundle $P:=\spin M$ of a $d$-dimensional spin manifold $M$, equipped with the Levi-Civita connection:

\begin{maincorollary}
If $M$ is a spin manifold of dimension $d$, then geometric string structures in the sense of \cite{waldorf8} coincide with $\str d$-bundle gerbes $\mathcal{G}$ with adjusted and adapted connections such that $p_{*}(\mathcal{G})\cong \Spin M$. 
\end{maincorollary}

A remarkable aspect of \cref{main:two} is that the left-hand side involves only the standard abelian theory. Since \emph{every} $\Gamma\!$-bundle gerbe $\mathcal{G}$ is a lift of its own underlying principal $F$-bundle $p_{*}(\mathcal{G})$, this result leads to a complete reformulation of non-abelian gauge theory in terms of abelian gauge theory, albeit shifted one level higher.

\begin{maincorollary}
The bicategory $\grbcon\Gamma M$ is equivalent to a bicategory of pairs $(P,\mathbb{T})$, where $P$ is a principal $F$-bundle with connection over $M$, and $\mathbb{T}$ is a trivialization of the Chern-Simons 2-gerbe $\mathbb{CS}_P(\Gamma)$ with compatible connection.
\end{maincorollary}

For the precise formulation of the bicategory of pairs $(P,\mathbb{T})$ and the construction of the equivalence via \quot{unstraightening}, we refer to \cref{equivalence-with-abelian}. A further important consequence of \cref{main:two} and related results is the following.

\begin{maintheorem}
\label{main-existence}
Let $\Gamma$ be a smoothly separable, central crossed module equipped with a splitting and an adapted adjustment. Then, every $\Gamma\!$-bundle gerbe admits an adjusted and adapted connection. 
\end{maintheorem}

The proof strategy involves switching via \cref{main:two} to principal bundles and trivializations of 2-gerbes, where the existence of connections is already established. We refer to \cref{existence-of-connections} for details.
To our best knowledge, \cref{main-existence} is the first general existence result for connections in any formulation of non-abelian higher gauge theory.

As a final application, we transfer one further aspect that is well understood
in abelian gauge theory to the non-abelian setting: the transgression to loop
spaces. Roughly speaking, transgression sends an $A$-bundle $k$-gerbe on $M$
to an $A$-bundle $(k-1)$-gerbe on the free loop space $LM=C^\infty(S^1,M)$,
thereby lowering the categorical degree of the geometric objects involved; in
the setting with connections, it is functorial. Applied to the multiplicative
bundle gerbe $\mathcal{G}_{\Gamma}$, equipped with its connection induced by the splitting $u$ and the adjustment $\kappa$, transgression yields the
\quot{loop group} of $\Gamma$: a (typically) infinite-dimensional Fréchet Lie
group $L\Gamma_{u,\kappa}$ \cite{Nikolausa}. Combining \cref{main:two} with
the existing lifting theory for principal bundles \cite{murray,gomi3}, we
obtain the following transgression theorem, which is proved in
\cref{transgression-to-loop-spaces}.

\begin{maintheorem}
Let $\Gamma$ be a smoothly separable, central crossed module with $F=\pi_0(\Gamma)$ connected, equipped with a splitting and an adapted adjustment. Let $L\Gamma_{u,\kappa}$ be the loop group of $\Gamma$. 
Then, there exists a canonical functor
\begin{equation*}
	\mathscr{T}:\hc 1\grbcon {\Gamma,u,\kappa} M \to \buncon{L\Gamma_{u,\kappa}}{LM}\text{.} 
\end{equation*} 
establishing the transgression of $\Gamma$-bundle gerbes with connections on $M$ to principal $L\Gamma_{u,\kappa}$-bundles on  the loop space.   
\end{maintheorem}

\vspace*{-0.5em}

In summary, the theory of $\Gamma\!$-bundle gerbes with adjusted connections we have developed here, reproduces many features of abelian bundle gerbes --- a sheaf theory, a
classification by differential cohomology, a lifting theory, the existence of connections, and a functorial transgression to loop spaces.

\paragraph{Acknowledgements. }
I would like to thank Matthias Ludewig,  Roberto T\'ellez-Dom\'inguez, and Christian Saemann for many helpful discussions. This work was supported by the German Research Foundation (DFG) under grant no. WA 3300/5-1.

\setsecnumdepth{2}

\section{Lie 2-groups and connections}

\label{Lie-2-groups-and-connections}

In this section, we revisit several foundational concepts of higher gauge theory required for this article. This comprises mostly familiar material, bundled and summarized, with minor extensions and reformulations.

\subsection{Lie 2-groups and Lie 2-algebras}

\label{Lie-2-groups}

For the purposes of this article, a \emph{Lie 2-group} is a crossed module $\Gamma=(H \stackrel{t}\to G \stackrel{\alpha}\to \mathrm{Aut}(H))$ of (possibly Fr\'echet) Lie groups, where $G$ and $H$ are Lie groups, $t$ is a Lie group homomorphism, and $\alpha:G \times H \to H$ is a smooth action of $G$ on $H$ by Lie group homomorphisms, such that
\begin{equation*}
t(\alpha(g,h)) = gt(h)g^{-1}
\quand
\alpha(t(h),h') = hh'h^{-1}
\end{equation*}
hold for all $g\in G$ and $h,h'\in H$.
We denote by $\alpha_g \in \mathrm{Aut}(H)$ the action of a fixed $g\in G$ on $H$. For $h \in H$, we consider the map $\tilde \alpha_h: G \to H$ defined by $\tilde \alpha_h(g) := h^{-1}\alpha(g,h)$, which is \emph{not} a Lie group homomorphism but satisfies $\tilde \alpha_h(1)=1$.
We refer to \cite[\S A]{Waldorf2016} for a useful collection of formulas for calculations in a crossed module.

For an abelian Lie group $A$, we denote by $BA$ the Lie 2-group $(A \to \ast \to \Aut(A))$, and for a Lie group $F$, we denote by $F_{\dis}$ the Lie 2-group $(\ast \to F \to \ast)$. For another Lie group $H$, we denote by $\AUT(H)$ the Lie 2-group $(H \stackrel c\to \Aut(H) \stackrel\id\to \Aut(H))$, where $c$ denotes the assignment of inner automorphisms.  

Given a Lie 2-group $\Gamma$, we define the \emph{homotopy groups} $F := \pi_0\Gamma := G/t(H)$ and $A := \pi_1\Gamma := \mathrm{ker}(t) \subset H$, together with the projection $p: G \to F$ and the inclusion $i:A \to H$. We note that $A$ is abelian and central in $H$, and that the action $\alpha$ induces an action $\alpha_0$ of $F$ on $A$ such that $\alpha_0(p(g),a) =\alpha(g,i(a))$ holds for all $a\in A$ and $g\in G$. The Lie 2-group $\Gamma$ is called \emph{central} if the action $\alpha_0$ is trivial. 
If $\Gamma$ is central, then $\tilde\alpha_a(g)=1$ for all $a\in A$; from this, one can verify that the map $\tilde\alpha_{h}: G \to H$ depends only on $g=t(h)$. 

Moreover, a Lie 2-group $\Gamma$ is called \emph{smoothly separable} if the groups $F$ and $H/A$ admit Lie group structures such that $p:G \to F$ and $H \to H/A$ have smooth local sections, and the map $t: H/A \to t(H)$ is a diffeomorphism \cite[Def. II.1, Def. III.1]{Neeb2005}. In this case, the sequence
\begin{equation*}
1\to A \to H \to G \to F \to 1
\end{equation*}
is an exact sequence of Lie groups. In this article, most Lie 2-groups are assumed to be smoothly separable and central; however, e.g., $\AUT(H)$ is in general neither, depending on $H$.

Again for the purposes of this article, a \emph{Lie 2-group homomorphism} $f: \Gamma \to \Gamma'$ between Lie 2-groups consists of Lie group homomorphisms $f_1: H \to H'$ and $f_0: G \to G'$ that are strictly compatible with all structure, i.e., $t' \circ f_1= f_0 \circ t$ and $\alpha'(f_0(g),f_1(h))=f_1(\alpha(g,h))$ hold for all $g\in G$ and $h\in H$.
If $\Gamma$ is a smoothly separable Lie 2-group, it comes equipped with two Lie 2-group homomorphisms $i: BA \to \Gamma$ and $p: \Gamma \to F_{\dis}$, and 
\begin{equation}
\label{2-group-extension}
BA \stackrel i\to \Gamma \stackrel p\to F_{\dis}
\end{equation} 
is a Lie 2-group extension, as defined, e.g., in \cite{pries2}.
This 2-group extension is central in the sense of \cite{pries2} if and only if $\Gamma$ is central in the sense defined above.

We remark that a class of more general, \quot{weak} Lie 2-groups can been considered \cite{Blohmann2008,pries2}. The smoothly separable ones  also fit into extensions as in \cref{2-group-extension}. At least when $F$ is compact and simple, every weak Lie 2-group is weakly equivalent to a strict one \cite[Rem. 5.5]{Waldorf}; also see \cite[Prop. 5.3.7]{Aldrovandi2009}.     

The Lie 2-algebra $\mathfrak{G}$ of a Lie 2-group $\Gamma$ is the associated crossed module $\mathfrak{G} =(\mathfrak{h} \stackrel{t_{*}}\to \mathfrak{g}\stackrel{\alpha_{*}}\to \der(\mathfrak{h}))$ of Lie algebras, in which $t_{*}$ is the induced Lie algebra homomorphism, and $\alpha_{*}$ denotes the induced Lie algebra action of $\mathfrak{g}$ on $\mathfrak{h}$ by derivations.  For further information about Lie 2-groups and their Lie 2-algebras, we refer to the literature, e.g., \cite{baez5}.
We refer again to \cite[\S A]{Waldorf2016} for formulas showing the interaction between the groups and the Lie algebras of a Lie 2-group.

We define $\mathfrak{a} := \mathrm{ker}(t_{*})\subset \mathfrak{h}$ and $\mathfrak{f} := \mathfrak{g}/t_{*}(\mathfrak{h})$, obtaining an exact sequence of Lie algebras
\begin{equation*}
0 \to \mathfrak{a} \stackrel {i_{*}}\to \mathfrak{h} \stackrel {t_{*}}\to \mathfrak{g} \stackrel{p_{*}}\to \mathfrak{f} \to 0\text{.}
\end{equation*}
If $\Gamma$ is smoothly separable, $\mathfrak{a}$ and $\mathfrak{f}$ are the Lie algebras of $A$ and $F$, respectively.
By a \emph{splitting} of $\mathfrak{G}$, we mean a linear map $u: \mathfrak{g} \to \mathfrak{h}$ such that
\begin{equation}
\label{identities-for-splittings-3}
t_{*}\circ u \circ t_{*}= t_{*}
\quand
u\circ t_{*}\circ u= u\text{.}
\end{equation}
Any splitting $u$ determines a complement $\mathfrak{g}' := \mathrm{ker}(u)\subset \mathfrak{g}$ of the sub-Lie algebra $\mathrm{Im}(t_{*})=\mathrm{Ker}(p_{*})$ and a complement $\mathfrak{h}' := \mathrm{im}(u)\subset \mathfrak{h}$ of the sub-Lie algebra $\mathrm{Im}(i_{*})=\mathrm{Ker}(t_{*})$. Conversely, the choice of such complements determines a splitting $u$. In particular, splittings always exist when $\Gamma$ is smoothly separable.  

A \emph{section} of $\mathfrak{G}$ is a linear map $q: \mathfrak{f} \to \mathfrak{g}$ such that $p_{*}\circ q=\id_{\mathfrak{f}}$, and a \emph{retract} of $\mathfrak{G}$ is a linear map $j: \mathfrak{h} \to \mathfrak{a}$ such that $j \circ i_{*}=\id_{\mathfrak{a}}$. 
Any splitting $u$ uniquely determines a section $q_u$ such that $u \circ q_u=0$ and a retract $j_u$ such that $j_u \circ u=0$. We note the following decomposition of $\mathfrak{g}$ and $\mathfrak{h}$ into sub-Lie algebras and their complements determined by $u$:  
\begin{align*}
i_{*} \circ j _u+ u \circ t_{*} &= \id_{\mathfrak{h}}
\\
t_{*}\circ u + q _u\circ p_{*} &= \id_{\mathfrak{g}}\text{.}
\end{align*}
Conversely, any pair $(q,j)$ of a section and a retract determines a splitting $u$ such that $j=j_u$ and $q = q_u$.

For a Lie 2-group homomorphism $f: \Gamma \to \Gamma'$, we say that it \emph{preserves splittings} $u$ and $u'$ of $\Gamma$ and $\Gamma'$, respectively, if 
\begin{equation}
\label{condition-for-splittings}
(f_1)_{*} \circ u =u' \circ (f_0)_{*}
\end{equation}
holds.

\setsecnumdepth{2}

\subsection{Adjustments}

The theory of adjustments was introduced by Sati, Schreiber, and Stasheff \cite{sati1}, and has undergone several improvements by Fiorenza, Schreiber, and Stasheff \cite{Fiorenza}, Saemann et al. \cite{Saemann2020,Kim2020,Rist2022,Kim2022}, and T\'ellez-Dom\'inguez \cite{Tellez2023}. A classification of adjustments for central Lie 2-groups was established in \cite{Ludewig2025}.

\begin{definition}
\label{adjustment}
Let $\Gamma$ be a central Lie 2-group with Lie 2-algebra $\mathfrak{G}$. 
An \emph{adjustment} for $\Gamma$ is a map
\begin{equation*}
\kappa: G \times \mathfrak{g} \to \mathfrak{h}
\end{equation*}
that is linear and continuous in $\mathfrak{g}$, smooth in $G$, and satisfies the following conditions:
\begin{align}
\label{condition-for-adjustments-1}
\kappa(g_1g_2, X) &= \kappa(g_1, \Ad_{g_2}(X)) + \kappa(g_2, X) 
\\
\label{condition-for-adjustments-2}
\kappa(t(h), X) &= (\tilde{\alpha}_{h^{-1}})_* X
\\
\label{condition-for-adjustments-3}
\kappa(g, t_{*}Y) &= (\alpha_g)_{*}(Y) - Y
\end{align}
for all $g, g_1, g_2\in G$, $h\in H$, $X\in \mathfrak{g}$, and $Y\in \mathfrak{h}$.
We say that $\kappa$ is \emph{adapted} to a splitting $u$ of $\mathfrak{G}$ if it additionally satisfies
\begin{equation}
\label{adapted-to-u}
t_* \kappa(g, X) = t_{*}u(\Ad_g(X) - X) 
\end{equation}
for all $g\in G$ and $X\in \mathfrak{g}$.
\end{definition}

Two straightforward consequences of the axioms are:
\begin{align*}
\kappa(g^{-1},X) &= -\kappa(g,\mathrm{Ad}_{g}^{-1}(X))
\\
\kappa(1,X) &= 0\text{.}
\end{align*}

We note that the adjustments in \cite{Rist2022,Tellez2023} are related to ours by inversion in the first argument, and that \cite{Rist2022} only requires one (weaker) axiom instead of our three above.  

\begin{remark}
\label{groupoid-of-adjustments}
We consider the groupoid $\ADJ(\Gamma)$ whose objects are pairs $(u, \kappa)$ where $u:\mathfrak{g} \to \mathfrak{h}$ is a splitting and $\kappa$ is an adjustment that is adapted to $u$. The morphisms $(u, \kappa) \to (u', \kappa')$ are pairs $(\phi, \psi)$ of linear maps $\phi: \mathfrak{f}\to \mathfrak{h}$ and $\psi: \mathfrak{g} \to \mathfrak{a}$ such that 
\begin{equation*}
\label{transformation-of-adjustments}
u'-u = \phi p_{*} + i_{*} \psi \qquad \text{and} \qquad \kappa'(g, X) - \kappa(g, X) =\phi p_{*}( \mathrm{Ad}_g(X)-X).
\end{equation*}
By \cite[Prop. 5.2]{Ludewig2025} this groupoid is equivalent to similar groupoids that consider, instead of splittings $u$, only the section $q_u$ or the retract $j_u$ induced by $u$.
\end{remark}

If $\kappa$ is an adjustment, we define $\kappa_{*}: \mathfrak{g} \times \mathfrak{g} \to \mathfrak{h}$ to be the bilinear form obtained by differentiating $\kappa$ in the first argument. The following result appears in \cite{Tellez2023} and \cite[Lemma 2.4]{Ludewig2025}.

\begin{lemma}
\label{rule-for-adapted-kappa}
For any adjustment $\kappa$ on $\Gamma$, the bilinear form $\kappa_{*}$ satisfies the following rules:
\begin{align}
\label{properties-of-kappa-stern:1}
\kappa_{*}(X,t_{*}(Y))&=\alpha_{*}(X,Y)
\\
\label{properties-of-kappa-stern:2}
\kappa_{*}(t_{*}(Y),X) &=-\alpha_{*}(X,Y)
\\
\label{properties-of-kappa-stern:4}
\kappa_{*}(\Ad_g(X_1),\Ad_g(X_2)) &= \kappa_{*}(X_1,X_2) + \kappa(g,[X_1,X_2])
\\
\label{properties-of-kappa-stern:5}
\kappa_{*}([X_1,X_2],X_3) &= \kappa_{*}(X_1,[X_2,X_3])+\kappa_{*}(X_2,[X_3,X_1])
\end{align}
Moreover, if $\kappa$ is adapted to a splitting $u$, then we have 
\begin{equation*}
t_{*}\kappa_{*}(X_1,X_2) = t_{*}u([X_1,X_2])\text{.}
\end{equation*}
\end{lemma}

If $\eta: \mathfrak{g} \times \mathfrak{g} \to \mathfrak{h}$ is a bilinear form, we denote by $\eta^{sym}$ and $\eta^{skew}$ its symmetrization and skew-symmetrization, respectively, i.e., 
\begin{equation*}
\eta^{sym/skew}(X_1,X_2):= \tfrac{1}{2}(\eta(X_1,X_2)\pm\eta(X_2,X_1))
\end{equation*}
The following is easy to deduce from  \cref{properties-of-kappa-stern:1,properties-of-kappa-stern:2}, see \cite[Eq. (3.17)]{Tellez2023} or \cite[Lemma 4.2]{Ludewig2025}.

\begin{lemma}
Suppose $\kappa$ is an adapted adjustment on $\Gamma$.
Then, there is a unique $\mathrm{Ad}$-invariant symmetric bilinear form $b_{\kappa}: \mathfrak{f} \times \mathfrak{f} \to \mathfrak{a}$ such that 
\begin{equation*}
b_{\kappa}(p_{*}X_1,p_{*}X_2) =\kappa_{*}^{sym}(X_1,X_2)
\end{equation*}
for all $X_1,X_2\in \mathfrak{g}$. 
\end{lemma}

The symmetric invariant bilinear form $b_{\kappa}$ associated to an adapted adjustment is invariant under the morphisms in the groupoid $\ADJ(\Gamma)$, and it plays an important role in the classification of adjustments \cite{Ludewig2025}.  

Care must be exercised when applying adjustments to the values of differential forms, as one might be accustomed to doing this only for symmetric or skew-symmetric forms.
We note the following results, obtained by straightforward calculations.

\begin{lemma}
\label{kappa-with-differential-forms}
Let $X$ be a smooth manifold, and let $\varphi\in \Omega^k(X,\mathfrak{g})$ and $\psi\in \Omega^l(X,\mathfrak{g})$ be differential forms, and let $g:X \to G$ be a smooth map. The following hold: 
\begin{enumerate}[(a)]

\item 
$\mathrm{d}\kappa(g,\varphi) 
= \kappa(g,\mathrm{d}\varphi )+\kappa_{*}(g^{*}\bar\theta \wedge \Ad_{g}(\varphi))$

\item
$\kappa_{*}(\varphi \wedge \varphi) = \begin{cases}\kappa_{*}^{skew}(\varphi\wedge \varphi) & k\text{ odd}  \\
b_{\kappa}(p_{*}\varphi\wedge p_{*}\varphi) & k\text{ even} \\
\end{cases}$

\item
$\kappa_{*}(\varphi \wedge [\varphi \wedge \varphi])= \begin{cases}\tfrac{2}{3} b_{\kappa}(p_{*}\varphi \wedge [p_{*}\varphi \wedge p_{*}\varphi]) & k\text{ odd}  \\
0 & k\text{ even} \\
\end{cases}$

\end{enumerate}
\end{lemma}

Finally, we recall how adjustments are related under Lie 2-group homomorphisms. The following definition was derived in \cite[Remark 4.10]{Ludewig2025} from the more general theory of weak equivalences.  

\begin{definition}
\label{preserve-adjustment}
Let $\Gamma$ and $\Gamma'$ be Lie 2-groups equipped with adjustments $\kappa$ and $\kappa'$, respectively. A Lie 2-group homomorphism $f: \Gamma \to \Gamma'$ is called \emph{adjustment-preserving}, if
\begin{equation}
(f_1)_{*}(\kappa(g,X))=\kappa'(f_0(g),(f_0)_{*}(X))
\end{equation}
holds for all $g\in G$ and $X\in \mathfrak{g}$. 

\end{definition}

\begin{example}
\label{morphisms-of-central-extension-are-adjustment-preserving}
In the case of the Lie 2-group homomorphisms
\begin{equation*}
BA \stackrel i\to \Gamma \stackrel p\to F_{\dis}
\end{equation*} 
associated to a smoothly separable, central Lie 2-group $\Gamma$, where $BA$ and $F_{\dis}$ are both equipped with the zero adjustments $\kappa=0$, $i$ and $p$ are adjustment-preserving for an arbitrary adjustment $\kappa$ on $\Gamma$. 
\end{example}

\subsection{Lie  2-group connections}

In this section we review the local formalism of $\Gamma\!$-connections  \cite{breen1,aschieri,schreiber2,baez2,Demessie,sati1,Gastel2020}. This formalism is classical in the fake-flat regime; beyond that, we systematically employ adjustments. Detailed verifications of the subsequent statements are left to the reader.

Let $X$ be a smooth manifold, let $\Gamma=(H \stackrel t\to  G \stackrel \alpha\to \mathrm{Aut}(H))$ be a Lie 2-group, and let $\mathfrak{G}=(\mathfrak{h} \stackrel{t_{*}}\to \mathfrak{g} \stackrel {\alpha_{*}}\to \mathrm{Der}(\mathfrak{h}))$ be the corresponding Lie 2-algebra.

\begin{definition}
\label{def:gammacon}
A  \emph{$\Gamma\!$-connection on $X$ }is a pair $(A,B)$ consisting of a 1-form $A\in\Omega^1(X,\mathfrak{g})$
and a 2-form $B \in \Omega^2(X,\mathfrak{h})$. The 2-form   
\begin{equation*}
\mathrm{fcurv}(A,B) := \mathrm{d}A +\tfrac{1}{2} [A
\wedge A] - t_{*} ( B)\in\Omega^2(X,\mathfrak{g})
\end{equation*}
is called the \emph{fake-curvature}, and the 3-form
\begin{equation*}
\mathrm{curv}(A,B) := \mathrm{d}B + \alpha_{*}( A\wedge B)\in\Omega^3
(X,\mathfrak{h})
\end{equation*}
is called the \emph{curvature}. A connection $(A,B)$ is called \emph{fake-flat}, if $\mathrm{fcurv}(A,B)=0$, and it is called \emph{flat}, if it is fake-flat and $\mathrm{curv}(A,B)=0$.
\end{definition}

\begin{remark}
\label{prop:gammaconcurv}
Curvature and fake curvature satisfy the  identities
\begin{align*}
\mathrm{d}(\mathrm{fcurv}(A,B)) &= [\mathrm{d}A\wedge A]  - t_{*}(\mathrm{d}B) 
\\
\mathrm{d} (\mathrm{curv}(A,B)) &= \alpha_{*}(\mathrm{fcurv}(A,B)\wedge B)-\alpha_{*}(A\wedge \mathrm{curv}(A,B))   \text{.}
\end{align*}
In case of fake-flatness, we have
$t_{*}(\mathrm{curv}(A,B))=0$, so that $\mathrm{curv}(A,B)$ is  $\mathfrak{a}$-valued.
If $\Gamma$ is additionally central, we have $\mathrm{d}(\mathrm{curv}(A,B))=0$. 
It is useful to define  
\begin{equation*}
F_A:=p_{*}(\mathrm{d}A+\tfrac{1}{2}[A\wedge A])=p_{*}(\mathrm{fcurv}(A,B)) \in \Omega^2(M,\mathfrak{f})\text{;}
\end{equation*}
then, the first identity implies $\mathrm{d}F_A = [F_A \wedge p_{*}A]$, a Bianchi identity for $p_{*}A$.
\end{remark}

\begin{remark}
\label{adjusted-curvature}
For $\kappa$ an adjustment on $\Gamma$, the 3-form
\begin{equation*}
\mathrm{curv}_{\kappa}(A,B) := \mathrm{curv}(A,B)+\kappa_{*}(A\wedge \mathrm{fcurv}(A,B))
\end{equation*}
is called the \emph{adjusted curvature}.
One can easily verify that
\begin{equation*}
\mathrm{curv}_{\kappa}(A,B)  =\mathrm{d}B+\kappa_{*}(A\wedge \mathrm{d}A )+\tfrac{1}{2}\kappa_{*}(A\wedge  [A \wedge A])\text{,}
\end{equation*}
revealing a term resembling a Chern-Simons form. This will be relevant later; see, e.g. \cref{extension-and-adjusted-shift-lemma,lifting-theorem}, where it turns out that the unusual prefactor $\tfrac{1}{2}$ combines with the factor $\tfrac{2}{3}$ from \cref{kappa-with-differential-forms} to the usual  prefactor of $\tfrac{1}{3}$ one expects in Chern-Simons theory. 
We have
\begin{align*}
\mathrm{d} (\mathrm{curv}_{\kappa}(A,B)) &=b_{\kappa}(F_A \wedge F_A)   \text{.}
\end{align*} 
In the case of fake-flatness, curvature and adjusted curvature agree.
\end{remark}

\begin{remark}
Suppose $u: \mathfrak{g} \to \mathfrak{h}$ is a splitting to which $\kappa$ is adapted. We say that a $\Gamma\!$-connection $(A,B)$ is  \emph{adapted to $u$} if $u(\mathrm{fcurv}(A,B))=0$. This is equivalent to 
\begin{equation*}
\mathrm{fcurv}(A,B)=q_u(F_A)\text{,}
\end{equation*}
where $q_u: \mathfrak{f} \to \mathfrak{g}$ is the section corresponding to $u$.  
In \cite{Rist2022} adapted $\Gamma\!$-connections are called \emph{$u$-flat}. One can check that  $t_{*}(\mathrm{curv}_{\kappa}(A,B))=0$, so that the adjusted curvature of an adapted connection is  $\mathfrak{a}$-valued,   $\mathrm{curv}_{\kappa}(A,B)\in \Omega^3(X,\mathfrak{a})$.
\end{remark}

\begin{definition}
\label{curvature-of-a-gauge-transformation}
Let $(A,B)$ and $(A',B')$ be $\Gamma\!$-connections on $X$.
A \emph{gauge transformation}
\begin{equation*}
(g,\varphi): (A,B) \to (A',B')
\end{equation*}
consists of a smooth map $g:X \to G$ and a 1-form $\varphi\in \Omega^1(X,\mathfrak{h})$
such that
\begin{equation*}
A' + t_{*} (\varphi) = \Ad_g(A) - g^{*}\bar\theta\text{,}
\label{gt1}
\end{equation*}
where $\bar\theta$ denotes the right-invariant Maurer-Cartan form.
The 2-form
\begin{equation*}
\mathrm{curv}(g,\varphi) :=  \mathrm{d}\varphi + \tfrac{1}{2}[\varphi \wedge \varphi]+ \alpha_{*}(A' \wedge \varphi)+B'- (\alpha_g)_{*} (B)\in \Omega^2(X,\mathfrak{h})
\end{equation*}
is called the curvature of $(g,\varphi)$. The gauge transformation is called \emph{flat} if $\mathrm{curv}(g,\varphi)=0$.
\end{definition}

\begin{remark}
\begin{enumerate}[(a)]

\item 
The identity gauge transformation is given by $g=1$ and $\varphi=0$.

\item
The composition of gauge transformations
\begin{equation*}
\alxydim{@C=1.5cm}{(A,B) \ar[r]^-{(g_1,\varphi_1)} & (A',B') \ar[r]^-{g_2,\varphi_2} & (A'',B'')}
\end{equation*}
is given by the map $g_2g_1:X \to G$ and the 1-form $\varphi_2+(\alpha_{g_{2}})_{*} (\varphi_1)$. Composition is strictly associative.

\item
The curvature of gauge transformations satisfies
\begin{align}
\mathrm{curv}(\id_{A,B}) &=0
\\
\label{eq:twcurvinv}
\mathrm{curv}((g,\varphi)^{-1}) &=-(\alpha_{g^{-1}})_{*}(\mathrm{curv}(g,\varphi))
\\
\label{gt2}
\mathrm{curv}((g_2,\varphi_2)\circ (g_1,\varphi_1))&= \mathrm{curv}(g_2,\varphi_2)+(\alpha_{g_2})_{*}(\mathrm{curv}(g_1,\varphi_1))\text{.}
\end{align}
As a consequence, flatness of gauge transformations is preserved under composition and inversion.

\item
\label{rem:gt:twisted}
Gauge transformations   establish relations between curvature and fake curvature of their source and target $\Gamma\!$-connections:
\begin{align}
\label{eq:ttwcurv}
\mathrm{fcurv}(A',B') &= \Ad_{g}(\mathrm{fcurv}(A,B))-t_{*}(\mathrm{curv}(g,\varphi))  
\\
\mathrm{curv}(A',B')&=(\alpha_g)_{*}(\mathrm{curv}(A,B))+\alpha_{*}(A'\wedge \mathrm{curv}(g,\varphi))\nonumber
\\&\qquad-\alpha_{*}(\mathrm{fcurv}(A',B')\wedge\varphi)-[\mathrm{curv}(g,\varphi)\wedge\varphi]+\mathrm{d}(\mathrm{curv}(g,\varphi))    
\text{.} 
\label{eq:transcurv}
\end{align}
In particular, flatness and fake-flatness of $\Gamma\!$-connections are preserved under flat gauge transformations; and in the fake-flat case, we have $\mathrm{curv}(A',B') = (\alpha_g)_{*}(\mathrm{curv}(A,B))$.  
Moreover, \cref{eq:ttwcurv} implies, unconditionally,
\begin{equation}
\label{transformation-of-F}
\mathrm{Ad}_{p \circ g}(F_A) = F_{A'}\text{.}
\end{equation}

\item
Fake-flat $\Gamma\!$-connections are locally gauge equivalent to $BA$-connections \cite{Gastel2020}, a fact that typically does not meet the requirements of physical theories \cite{Saemann2020}. 

\item
There is a  Poincar\'e Lemma stating that flat $\Gamma\!$-connections are locally gauge equivalent to the trivial connection \cite[Theorem 2.7]{Demessie}. 

\end{enumerate}
\end{remark}

\begin{definition}
\label{def8}
Suppose $\Gamma$ is equipped with an adjustment $\kappa$. Then, a gauge transformation
$(g,\varphi): (A,B) \to (A',B')$
is called \emph{adjusted} if 
\begin{equation*}
\mathrm{curv}(g,\varphi) = \kappa(g,\mathrm{fcurv}(A,B))\text{.}
\end{equation*}
\end{definition}

\begin{remark}
\begin{enumerate}[(a)]

\item 
\label{remark-adj-gt-1}
The condition in \cref{def8} is equivalent to
\begin{equation*}
\mathrm{d}\varphi + \tfrac{1}{2}[\varphi \wedge \varphi] + \alpha_{*}(A' \wedge \varphi)+B'-B=\kappa(g,\mathrm{d}A + \tfrac{1}{2}[A \wedge A])\text{.}
\end{equation*}
Due to \cref{eq:ttwcurv}, it is also equivalent to the following condition involving the fake-curvature of the target $\Gamma\!$-connection:
\begin{equation*}
\mathrm{curv}(g,\varphi)=-(\alpha_{g})_{*}(\kappa(g^{-1},\mathrm{fcurv}(A',B')))
\end{equation*}
\item 
If $(A,B)$ is fake-flat, a gauge transformation is adjusted if and only if it is flat.

\item
Identity gauge transformations are adjusted. 
Moreover, the composition of adjusted gauge transformations is adjusted, and the inverse of an adjusted gauge transformation is adjusted. 

\item
\label{adjusted-curvature-relation}
For the adjusted curvatures of $\Gamma\!$-connections related by an adjusted gauge transformation, \cref{eq:transcurv} simplifies to
\begin{equation*}
\mathrm{curv}_{\kappa}(A',B')
=\mathrm{curv}_{\kappa}(A,B)\text{.}
\end{equation*}
Note that, if the $\Gamma\!$-connections are adapted with respect to a splitting $u$, then this is an equality of $\mathfrak{a}$-valued 3-forms. 

\item
\label{adaptedness-fake}
If $\kappa$ is adapted to a splitting $u$, then \cref{eq:ttwcurv} implies, for an adjusted gauge transformation,
\begin{equation*}
u(\mathrm{fcurv}(A',B')) = u(\mathrm{fcurv}(A,B))\text{.}
\end{equation*}
Thus, if one of the $\Gamma\!$-connections is adapted to $u$, then the other is as well. 

\end{enumerate}
\end{remark}

\begin{definition}
\label{gauge-2-transformations}
Let $(g_1,\varphi_1),(g_2,\varphi_2):(A,B) \to (A',B')$ be gauge transformations.
A \emph{gauge 2-transformation} 
\begin{equation*}
h:(g_1,\varphi_1) \Rightarrow (g_2,\varphi_2)
\end{equation*}
is a smooth map $h:X \to H$ such that
\begin{equation*}
g_2 = (t \circ h) \cdot g_1
\quad\text{ and }\quad
\Ad_h^{-1}(\varphi_2) +(\tilde\alpha_{h})_{*}(A') =  \varphi_1 -h^{*}\theta\text{.}
\end{equation*}
\end{definition}

\begin{remark}
The vertical composition 
\begin{equation*}
\alxydim{}{(g,\varphi) \ar@{=>}[r]^-{h_1} & (g',\varphi') \ar@{=>}[r]^{h_2} & (g'',\varphi'')}
\end{equation*}
is given by $h_2h_1$. The horizontal composition is
\begin{equation*}
\alxydim{}{(A,B) \ar@/^2pc/[r]^{(g_1,\varphi_1)}="1" \ar@/_2pc/[r]_{(g_1',\varphi_1')}="2" \ar@{=>}"1";"2"|{h_1} & (A',B') \ar@/^2pc/[r]^{(g_2,\varphi_2)}="1" \ar@/_2pc/[r]_{(g_2',\varphi_2')}="2" \ar@{=>}"1";"2"|{h_2} & (A'',B'')}
=
\alxydim{@C=2.3cm}{(A,B) \ar@/^2pc/[r]^{(g_2g_1, (\alpha_{g_2})_{*}(\varphi_1) + \varphi_2)}="1" \ar@/_2pc/[r]_{(g_2'g_1', (\alpha_{g_2'})_{*}(\varphi'_1) +\varphi'_2)}="2" \ar@{=>}"1";"2"|{h_2\alpha(g_2,h_1)} & (A'',B'')\text{,}}
\end{equation*}
and the identity gauge 2-transformation is given by $h=1$.
\end{remark}

\begin{remark}
\label{rem:twcurvunder2}
The equality
\begin{equation}
\label{eq:twcurvunder2}
\mathrm{curv}(g_1,\varphi_1) =\Ad_h^{-1}(\mathrm{curv}(g_2,\varphi_2))+(\tilde\alpha_{h})_{*}(\mathrm{fcurv}(A',B'))\text{.} 
\end{equation}
holds for a gauge 2-transformation $h$ between gauge transformations  $(g_1,\varphi_1)$ and $(g_2,\varphi_2)$. Thus, a gauge 2-transformation $h$ between two \emph{flat} gauge transformations only exists if $(\tilde\alpha_{h})_{*}(\mathrm{fcurv}(A',B'))=0$. This means, when $\Gamma$ is central, that $h$ is forced to be $A$-valued at all points where the fake-curvature does not vanish.  Since gauge 2-transformations are the gluing structure for bundle gerbes on 3-fold intersections, this leads to a separation into a fake-flat and an abelian theory, which is undesirable. Thus, \cref{eq:twcurvunder2} motivates the use of adjusted rather than flat gauge transformations beyond the fake-flat regime.      
\end{remark}

We consider the following three bigroupoids. 
\begin{definition}
\label{gamma-connections}
\begin{enumerate}[(a)]

\item
Fake-flat $\Gamma\!$-connections,  flat gauge transformations, and gauge 2-transformations  form  the bigroupoid of  \emph{fake-flat $\Gamma\!$-connections}, denoted by $\conff\Gamma X$. This groupoid does not require adjustments, and even exists if $\Gamma$ does not allow any adjustments. 

\item
If $\Gamma$ is equipped with an adjustment $\kappa$, then, 
$\Gamma\!$-connections, adjusted gauge transformations, and gauge 2-transformations form the bigroupoid of  \emph{adjusted $\Gamma\!$-connections}, denoted by $\conadj{\Gamma\!,\kappa} X$. 

\item
If $\Gamma$ is equipped with an adjustment $\kappa$, and $\kappa$ is adapted to a splitting $u$, then, adapted 
$\Gamma\!$-connections, adjusted gauge transformations, and gauge 2-transformations form the bigroupoid of  \emph{adapted and adjusted $\Gamma\!$-connections}, denoted by $\conaa{\Gamma\!,u,\kappa} X$.  

\end{enumerate}
\end{definition}

\begin{remark}
We will sometimes also consider the bigroupoid of all $\Gamma\!$-connections,  all gauge transformations, and all gauge 2-transformations, which we call the bigroupoid of \emph{generalized $\Gamma\!$-connections}, and denote it by $\congen \Gamma X$. This bigroupoid is only relevant for organizational reasons -- due to the lacking constraint on the curvature of gauge transformations, it has no relevance in applications. 
\end{remark}

The data of the bigroupoid $\conff{\Gamma}{X}$ has been anticipated by Breen and Messing \cite{breen1}, and have appeared in numerous other references. 
In joint work with Schreiber \cite{schreiber5}  we have provided a method for computing it, by imposing that  a $\Gamma\!$-connection has a well-defined $\Gamma\!$-valued integral along paths and homotopies between paths: 

\begin{theorem}{\normalfont{\cite[Theorem 2.21]{schreiber5}}}
\label{th:smoothfunctorsanddifferentialforms}
There is an isomorphism of bigroupoids
\begin{equation*}
\conff\Gamma X \cong \fun^{\infty}(\mathcal{P}_2(X),\Gamma)\text{.}
\end{equation*}
\end{theorem}

A generalization of this result to adjusted $\Gamma\!$-connections requires a generalized notion of parallel transport. This is by now an open question; we hope, that the lifting theory and the result of \cref{lifting-theorem} will contribute to a solution.

\begin{remark}
\label{inclusion-of-bigroupoids-of-connections}
We have the following  inclusions of bigroupoids:
\begin{equation*}
\conff \Gamma X \subset \conaa{\Gamma\!,u,\kappa} X\subset \conadj{\Gamma\!,\kappa} X\subset \congen\Gamma X\text{,} 
\end{equation*}
of which the first two are full.
They induce, on the level of isomorphism classes, denoted by $\hc 0$,  a sequence of  maps 
\begin{equation*}
\hc 0\conff \Gamma X \incl \hc 0\conaa{\Gamma\!,u,\kappa} X\incl \hc 0\conadj{\Gamma\!,\kappa} X \twoheadrightarrow \hc 0 \congen\Gamma X\text{,} 
\end{equation*}
of which the first two are injective, and the last is surjective. 
\end{remark}

\begin{remark}
\label{covariance-of-connections}
$\Gamma\!$-connections are covariant in the Lie 2-group $\Gamma$, with respect to Lie 2-group homomorphism $f:\Gamma \to \Gamma'$ as considered in  \cref{Lie-2-groups}:
\begin{enumerate}[(a)]

\item 
If $(A,B)$ is a  $\Gamma\!$-connection, then $f_{*}(A,B):=((f_0)_{*}(A),(f_1)_{*}(B))$ is a  $\Gamma'$-connection. Fake-flatness and flatness are preserved. In the presence of splittings satisfying \cref{condition-for-splittings}, adaptedness is also preserved.  

\item
If $(g,\varphi): (A,B) \to (A',B')$ is a gauge transformation, then $f_{*}(g,\varphi):=(f_0 \circ g,(f_1)_{*}(\varphi))$
is again a gauge transformation between the corresponding $\Gamma'$-connections. If $(g,\varphi)$ is flat, then $f_{*}(g,\varphi)$ is flat, too. If $(g,\varphi)$ is adjusted and $f$ is adjustment-preserving in the sense of \cref{preserve-adjustment}, then $f_{*}(g,\varphi)$ is adjusted, too.  

\item
If $h:(g_1,\varphi_1) \Rightarrow (g_2,\varphi_2)$ is a gauge 2-transformation, then $f_{*}(h):=f_1\circ h$ is a gauge 2-transformation.    

\end{enumerate} 
Summarizing, there is a 2-functor
\begin{equation*}
f_{*}: \congen\Gamma X \to \congen{\Gamma'} X\text{,}
\end{equation*}
which restricts  to each of the sub-bigroupoids in \cref{inclusion-of-bigroupoids-of-connections}.
\end{remark}

\begin{remark}
\label{shift-by-2-form}
For $\lambda\in \Omega^2(X,\mathfrak{h})$, we define a \quot{shift-by-2-form} 2-functor
\begin{equation*}
(..)^{\lambda}:\congen\Gamma X \to \congen\Gamma X
\end{equation*}
by sending $(A,B)$ to $(A,B+\lambda)$ on the level of objects, and being the identity on the level of 1- and 2-morphisms. It restricts to $\conadj{\Gamma\!,\kappa} X$, and it restricts  to $\conff \Gamma X$ and $\conaa{\Gamma\!,u,\kappa} X$ under the condition that $t_{*}\lambda=0$. 
\end{remark}

\begin{remark}
\label{covariance-in-the-adjustment}
Adapted and adjusted $\Gamma\!$-connections are compatible with the groupoid of adjustments described in \cref{groupoid-of-adjustments}. Given a morphism $(\phi,\psi):(u,\kappa) \to (u',\kappa')$ in that groupoid and  a $\Gamma\!$-connection $(A,B)$, 
 we define a new $\Gamma\!$-connection $(A',B')$ with $A':= A$ and $B'=B+\phi(F_A)$. 
If $(A,B)$ is adapted to $u$, then $(A',B')$ is adapted to $u'$. 
If $(g,\varphi):(A_1,B_1) \to (A_2,B_2)$ is a gauge transformation with curvature $\mathrm{curv}(g,\varphi)$, then, the same $(g,\varphi)$ is a gauge transformation $(A_1',B_1') \to (A_2',B_2')$ with curvature
\begin{equation*}
\mathrm{curv}'(g,\varphi) = \mathrm{curv}(g,\varphi) +\phi(F_{A_2})-(\alpha_g)_{*}(\phi(F_{A_1}))\text{.} 
\end{equation*} 
The relation between the adjustments $\kappa$ and $\kappa'$ in the groupoid of adjustments  implies then that 
\begin{equation*}
\kappa'(g,\mathrm{fcurv}(A'_1,B'_1))=\kappa(g,\mathrm{fcurv}(A_1,B_1))-(\alpha_g)_{*}(\phi(F_{A_1}))+\phi(F_{A_2})\text{,}
\end{equation*}
and this shows that, if $(g,\varphi)$ is adjusted as a gauge transformation $(A_1,B_1) \to (A_2,B_2)$, then it is also adjusted as a gauge transformation $(A_1',B_1') \to (A_2',B_2')$.
Above considerations show that $(A,B) \mapsto (A,B + \phi(F_A))$ establishes a strict 2-functor
\begin{equation*}
\phi_{*}:\congen\Gamma X \to \congen\Gamma X
\end{equation*}
that restricts to all subcategories in \cref{inclusion-of-bigroupoids-of-connections},
with the restriction to $\conff\Gamma X$ being the identity functor. The assignment $(\phi,\psi)\mapsto \phi_{*}$ is compatible with composition in the groupoid of adjustments, and hence yields a functor
\begin{equation*}
\ADJ(\Gamma) \to \BiGrpd: (u,\kappa) \mapsto \conaa{\Gamma\!,u,\kappa}X\text{,}
\end{equation*} 
where $\BiGrpd$ denotes the category of bigroupoids with strict 2-functors. Thus, equivalent adjustments induce equivalent bigroupoids of adapted and adjusted $\Gamma\!$-connections. 
\end{remark}

\setsecnumdepth{2}
\subsection{Non-abelian differential cohomology}

\label{sec:pre2stackificationofconGamma}

$\Gamma\!$-connections are contravariant in $X$; the contravariance is established by the pullback of functions and differential forms. Equivalently, these assignments define presheaves\begin{equation}
\label{presehaves-of-gamma-connections}
X \mapsto \conff \Gamma X
\quomma
X \mapsto \conaa {\Gamma\!,u,\kappa} X
\quomma
X \mapsto \conadj{\Gamma\!,\kappa} X
\quomma
X \mapsto \congen\Gamma X
\end{equation}
of bigroupoids over the category of smooth manifolds. We sheafify these presheaves using the plus construction of Nikolaus-Schweigert, recalled in \cref{sec:plus}. We begin with the following preliminary result.
\begin{lemma}
\label{rem:gammaconn:prepre}
The presheaves $\conff \Gamma-$, $\conaa{\Gamma\!,u,\kappa}-$, $\conadj{\Gamma\!,\kappa}-$, and $\congen\Gamma-$ are separated in the sense defined in  \cref{sec:plus}.
\end{lemma}

\begin{proof}
The proof is the same for all four presheaves, since it only concerns the 2-morphisms which are the same in all four cases. Therefore, we focus on $\congen\Gamma-$,  and consider a smooth manifold $X$, two objects $(A,B)$ and $(A',B')$ in $\congen\Gamma X$, and a surjective submersion $\pi:Y \to X$. We have to show that the assignment $r_{\pi}$ of descent data with respect to $\pi$ to gauge transformations,
\begin{equation}
\label{functor-to-show-full-and-faithful} 
r_{\pi}: \hom_{\congen\Gamma X}((A,B),(A',B')) \to \hom_{\des_{\congen\Gamma-}(\pi)}(r_{\pi}(A,B),r_{\pi}(A',B'))\text{,}
\end{equation} 
is fully faithful, see  \cref{sec:plus}. Consider two gauge transformations $(g_1,\varphi_1)$ and $(g_1,\varphi_1)$ between $(A,B)$ and $(A',B')$. A morphism on the right hand side between $r_{\pi}(g_1,\varphi_1)$ and $r_{\pi}(g_2,\varphi_2)$ is a gauge 2-transformation $a: \pi^{*}(g_1 ,\varphi_1) \Rightarrow \pi^{*}(g_2,\varphi_2)$ in $\congen\Gamma{Y}$ such that $\pr_1^{*}a = \pr_2^{*}a$ in $\congen\Gamma{Y^{[2]}}$. Since $a:Y \to H$ is a smooth map, and smooth maps form a sheaf,  the latter condition entails that there exists $\tilde a: X \to H$ such that $\pi^{*}\tilde a=a$. Moreover, since $\pi^{*}$ is injective on functions and differential forms, $\tilde a$ is a gauge 2-transformation $(g_1,\varphi_1) \Rightarrow (g_2,\varphi_2)$. This shows that the functor \cref{functor-to-show-full-and-faithful} is full; faithfulness follows again from the injectivity of $\pi^{*}$.
\end{proof}

In order to perform the sheafification, we will here use the shortcut described in \cref{sheafification-shortcut}, by which we save the application of the Hom-set-closure. Hence, we will consider the sheaves of bicategories
\begin{equation*}
(\conff\Gamma-)^{+}\quomma(\conaa{\Gamma\!,u,\kappa}-)^{+}\quomma(\conadj{\Gamma\!,\kappa}-)^{+}\quomma(\congen\Gamma-)^{+}\text{.}
\end{equation*}
In order to obtain the usual picture of \quot{local data}, we will perform the plus construction with respect to open covers $(U_i)_{i\in I}$ of  $X$.   We detail this construction for $\congen\Gamma-$ and then describe the additional constraints for the sub-presheaves.  

\subsubsection{Sheafification I: Objects}

An object $(A_i,B_i,\varphi_{ij},g_{ij},a_{ijk})$ in $(\congen\Gamma-)^{+}(X)$ consists of an open cover $(U_i)_{i\in I}$ of  $X$, and
\begin{enumerate}[(a)]

\item
on every open set $U_i$, a 1-form $A_i \in \Omega^1(U_i,\mathfrak{g})$ and a 2-form $B_i \in \Omega^2(U_i,\mathfrak{h})$,

\item
on every two-fold intersection $U_i \cap U_j$, a smooth map $g_{ij}:U_i \cap U_j \to G$ and a 1-form $\varphi_{ij} \in \Omega^1(U_i \cap U_j,\mathfrak{h})$, and

\item
on every three-fold intersection $U_i \cap U_j \cap U_k$, a smooth map $
a_{ijk}:U_i \cap U_j \cap U_k \to H$. 

\end{enumerate}  
The cocycle conditions are the following:
\begin{enumerate}

\item
Over every two-fold intersection $U_i \cap U_j$:
\begin{align}
A_j + t_{*} (\varphi_{ij}) &= \Ad_{g_{ij}}(A_i) - g_{ij}^{*}\bar\theta 
\label{eq:transconnco}
\end{align} 

\item
Over every three-fold intersection $U_i \cap U_j \cap U_k$:
\begin{align}
\label{eq:transtrans}
g_{ik} &= (t \circ a_{ijk}) \cdot g_{jk}g_{ij}
\\
\label{eq:transconn2co}
\Ad_{a_{ijk}}^{-1}(\varphi_{ik}) +(\tilde \alpha_{a_{ijk}})_{*}(A_k) &=  \varphi_{jk}+(\alpha_{g_{jk}})_{*} (\varphi_{ij}) -a_{ijk}^{*}\theta
\end{align}

\item
Over every four-fold intersection $U_i \cap U_j \cap U_k \cap U_l$:
\begin{equation}
\label{31c}
a_{ikl} \alpha(g_{kl},a_{ijk})=a_{ijl}a_{jkl}
\end{equation}
\end{enumerate} 
The relevant curvature forms are defined as follows: 
\begin{align*}
\tilde F_i& := \mathrm{d}A_i +\tfrac{1}{2} [A_i
\wedge A_i] - t_{*} ( B_i) \in \Omega^2(U_i,\mathfrak{g})
\\
H_i&:=\mathrm{d}B_i + \alpha_{*}( A_i\wedge B_i)\in \Omega^3(U_i,\mathfrak{h})
\\
k_{ij} &:= \mathrm{d}\varphi_{ij} + \tfrac{1}{2}[\varphi_{ij} \wedge \varphi_{ij}]+ \alpha_{*}(A_j \wedge \varphi_{ij})+B_j- (\alpha_{g_{ij}})_{*} (B_i)\in \Omega^2(U_i \cap U_j,\mathfrak{h})
\end{align*}
\cref{prop:gammaconcurv} implies
\begin{align*}
\mathrm{d} H_i &= \alpha_{*}(\tilde F_i\wedge B_i)-\alpha_{*}(A_i\wedge H_i)   \text{.}
\end{align*}
\cref{eq:ttwcurv,eq:transcurv} imply on double intersections:
\begin{align}
\tilde F_j &= \Ad_{g_{ij}}(\tilde F_i)-t_{*}(k_{ij})  
\\
H_j&=(\alpha_{g_{ij}})_{*}(H_i)+\alpha_{*}(A_j\wedge k_{ij})-\alpha_{*}(\tilde F_j\wedge\varphi_{ij})-[k_{ij}\wedge\varphi_{ij}]+\mathrm{d}k_{ij} 
\end{align}
\cref{eq:twcurvunder2,gt2} implies on triple intersections
\begin{equation*}
k_{jk}+(\alpha_{g_{jk}})_{*}(k_{ij}) =\Ad_{a_{ijk}}^{-1}(k_{ik})+(\tilde\alpha_{a_{ijk}})_{*}(\tilde F_j)\text{.}
\end{equation*}
The following additional conditions characterize the three sub-bigroupoids:
\begin{itemize}
\item 
In $(\conadj{\Gamma\!,\kappa}-)^{+}(X)$, the additional condition
\begin{align*}
\label{cocycles-obj-adj}
k_{ij}= \kappa(g_{ij},\tilde F_i)
\end{align*}
is imposed, and the adjusted curvature $H_i^{\kappa}:=H_i + \kappa_{*}(A_i\wedge \tilde F_i) \in \Omega^3(U_i,\mathfrak{h})$ is considered. \cref{adjusted-curvature} now implies that
\begin{equation*}
\mathrm{d}H^{\kappa}_i =b_{\kappa}(F_i \wedge F_i)   \text{,}
\end{equation*}
with $F_i:= p_{*}(\mathrm{d}A_i+\tfrac{1}{2}[A_i \wedge A_i])=p_{*}(\tilde F_i) \in \Omega^2(U_i,\mathfrak{f})$,
and \cref{adjusted-curvature-relation} implies $H^{\kappa}_i = H^{\kappa}_j$, so that there exists a globally defined curvature form $H^{\kappa}\in \Omega^3(X,\mathfrak{h})$ such that $H^{\kappa}|_{U_i}=H^{\kappa}_i$. \Cref{transformation-of-F} implies
\begin{equation*}
\mathrm{Ad}_{p(g_{ij})}(F_i) = F_{j}\text{.}
\end{equation*}

\item
In $(\conaa{\Gamma\!,u,\kappa}-)^{+}(X)$, the further additional condition $u(\tilde F_i)=0$ is imposed, which implies that $\tilde F_i = q(F_i)$ and ensures that $H^{\kappa}$ is $\mathfrak{a}$-valued. 

\item
In $(\conff\Gamma-)^{+}$, we impose 
$\tilde F_i=0$,
implying $k_{ij}=0$ and $F_i=0$, as well as $\mathrm{d}H_i=\alpha_{*}(A_i \wedge H_i)$. 

\end{itemize}

\subsubsection{Sheafification II: 1-Morphisms}

For the 1-morphisms, we have to work on the common refinement of the open covers of the source and target objects. We will tacitly assume that we have already passed to this common refinement. 
Then, a 1-morphism $(\phi_i,h_i,e_{ij}):(A_i,B_i,\varphi_{ij},g_{ij},a_{ijk}) \to (A_i',B_i',\varphi_{ij}',g_{ij}',a_{ijk}')$ in $(\congen\Gamma-)^{+}(X)$ consists of:
\begin{enumerate}
\item[(a)]
on every open set $U_i$, a smooth map $h_{i}:U_{i} \to G$ and a 1-form $\phi_{i}\in\Omega^1(U_i,\mathfrak{h})$, and

\item[(b)]
on every two-fold intersection $U_i \cap U_j$, a smooth map $e_{ij}:U_i \cap U_j \to H$.

\end{enumerate}
The following conditions have to be satisfied:
\begin{enumerate}
\item 
Over every open set $U_i$:
\begin{align}
\label{32a}
A'_i+ t_{*}(\phi_i) &= \Ad_{h_i}(A_i)  - h_i^{*}\bar\theta
\end{align}

\item
Over every two-fold intersection $U_i \cap U_j$:
\begin{align}
\label{32c}
g_{ij}'h_i &= t(e_{ij})h_jg_{ij}
\\
\label{32d}
\Ad_{e_{ij}}^{-1}(\varphi_{ij}'+ (\alpha_{g_{ij}'})_{*}(\phi_i))   &= \phi_j + (\alpha_{h_j})_{*}(\varphi_{ij})- (\tilde\alpha_{e_{ij}})_{*}(A'_{j}) 
- e_{ij}^{*}\theta
\end{align}

\item
Over every three-fold intersection $U_i \cap U_j \cap U_k$:
\begin{equation}
\label{32f}
a_{ijk}'\alpha(g'_{jk},e_{ij})e_{jk} = e_{ik}\alpha(h_k,a_{ijk})
\end{equation}

\end{enumerate}
We define the curvature form
\begin{equation*}
c_{i} := \mathrm{d}\phi_{i} + \tfrac{1}{2}[\phi_{i} \wedge \phi_{i}]+ \alpha_{*}(A'_i \wedge \phi_{i})+B'_i- (\alpha_{h_{i}})_{*} (B_i)\in \Omega^2(U_i,\mathfrak{h})\text{.} 
\end{equation*}
\cref{eq:ttwcurv,eq:transcurv} imply on double intersections
\begin{align}
\tilde F_i' &= \Ad_{h_{ij}}(\tilde F_i)-t_{*}(c_{i})  
\\
\label{curvature-of-cocycles-under-morphism}
H_i'&=(\alpha_{h_{i}})_{*}(H_i)+\alpha_{*}(A'_i\wedge c_{i})-\alpha_{*}(\tilde F'_i\wedge\phi_{i})-[c_{i}\wedge\phi_{i}]+\mathrm{d}c_{i}
\text{.} 
\end{align}
\cref{eq:twcurvunder2,gt2} imply on double intersections
\begin{equation*}
c_j+(\alpha_{h_{j}})_{*}(k_{ij}) =\Ad_{e_{ij}}^{-1}(k_{ij}'+(\alpha_{g_{ij}'})_{*}(c_i))+(\tilde\alpha_{e_{ij}})_{*}(\tilde F'_j)\text{.}
\end{equation*}
The following conditions apply to the three sub-bigroupoids we consider:
\begin{itemize}
\item 
In $(\conadj{\Gamma\!,\kappa}-)^{+}$, the additional condition
\begin{align}
\label{eq:equivtranscurvadj}
c_i=\kappa(h_i,\tilde F_i) &
\end{align}
is imposed, which implies via \cref{adjusted-curvature-relation} that \cref{curvature-of-cocycles-under-morphism} is replaced by $H_i^{\kappa}=(H_i')^{\kappa}$. Taking into account that these forms are actually globally defined, we have $H^{\kappa}=(H')^{\kappa}$. 

\item
There are no further conditions in $(\conaa{\Gamma\!,u,\kappa}-)^{+}$.  

\item
There are no further conditions in  $(\conff{\Gamma}-)^{+}$. We remark that the vanishing of the fake curvature implies via \cref{eq:equivtranscurvadj} that $c_i=0$. 

\end{itemize}

\subsubsection{Sheafification III: 2-Morphisms}

Again, we work here on the common refinement of the source and target objects. 
A 2-morphism in $(\congen\Gamma-)^{+}(X)$,
\begin{equation*}
\alxydim{}{(A_i,B_i,\varphi_{ij},g_{ij},a_{ijk}) \ar@/_2pc/[r]_{(\phi_i',h_i',e_{ij}')}="2"  \ar@/^2pc/[r]^{(\phi_i,h_i,e_{ij})}="1" \ar@{=>}"1";"2"|*+{m_{i}} & (A_i',B_i',\varphi_{ij}',g_{ij}',a_{ijk}'),}
\end{equation*}
has on each open set $U_i$ a smooth map $m_i : U_i \to H$ such that:
\begin{enumerate}

\item 
Over every open set $U_i$:
\begin{align}
h_i' &= (t \circ m_i) \cdot h_i
\\
\Ad_{m_i}^{-1}(\phi_i') +(\tilde\alpha_{m_i})_{*}(A'_i) &=  \phi_i -m_i^{*}\theta\text{.}
\end{align}

\item
Over every 2-fold intersection $U_i \cap U_j$:
\begin{equation}
e_{ij}'m_j=\alpha(g_{ij}',m_i)e_{ij}\text{.}
\end{equation}

\end{enumerate}
On the level of curvatures, \cref{eq:twcurvunder2} implies $c_i=\mathrm{Ad}^{-1}_{m_i}(c_i') + (\tilde\alpha_{m_i})_{*}(\tilde F'_i)$.
There are no further conditions for any of our sub-bigroupoids.

\begin{remark}
The formulas derived above from the principle of sheafification coincide with the corresponding formulas in \cite[\S 2.2]{Waldorf2016}. They are equivalent to the setting of \cite{Rist2022} and \cite{Tellez2023}, under a transformation involving a shuffling of indices.  
\end{remark}

\begin{remark}
We obtain a sequence of 2-functors
\begin{equation*}
(\conff\Gamma-)^{+}(X)\to (\conaa{\Gamma\!,u,\kappa}-)^{+}(X)\to (\conadj{\Gamma\!,\kappa}-)^{+}(X)\to (\congen\Gamma-)^+(X)\text{;} 
\end{equation*}
where the first two are fully faithful and the third is essentially surjective.
On the level of isomorphism classes of objects, we obtain the non-abelian differential cohomology groups
\begin{align*}
\hat \h^1(X,\Gamma)^{\ff} &:= \hc 0(\conff\Gamma-)^{+}(X) 
\\
\hat \h^1(X,(\Gamma\!,u,\kappa))^{\aptadj} &:= \hc 0(\conaa{\Gamma\!,u,\kappa}-)^{+}(X) 
\\
\hat \h^1(X,(\Gamma\!,\kappa))^{\adj} &:= \hc 0(\conadj{\Gamma\!,\kappa}-)^{+}(X) 
\\
\hat \h^1(X,\Gamma)^{\gen} &:= \hc 0(\congen\Gamma-)^{+}(X) 
\end{align*}
inducing maps
\begin{equation*}
\hat \h^1(X,\Gamma)^{\ff} \incl \hat \h^1(X,(\Gamma\!,u,\kappa))^{\aptadj} \incl \hat \h^1(X,(\Gamma\!,\kappa))^{\adj} \to \hat\h^1(X,\Gamma)^{\gen}\text{,} 
\end{equation*}
such that the first two are injective and the third is surjective.
\end{remark}

\begin{remark}
Since sheafification is functorial, the shift by 2-form (\cref{shift-by-2-form}) and the covariance in the adjustments (\cref{covariance-in-the-adjustment}) induce corresponding functors between the sheaves, and corresponding bijections in cohomology.
\begin{equation*}
\phi_{*}: \hat\h^1(X,(\Gamma\!,u,\kappa))^{\aptadj}  \to \hat\h^1(X,(\Gamma\!,u',\kappa'))^{\aptadj} \text{.}
\end{equation*}
In that sense, the adapted and adjusted non-abelian differential cohomology only depends on the isomorphism class of the adjustment. 
\end{remark}

\setsecnumdepth{2}
 
\section{Connections on 2-group bundles} 

\label{section-2-group-bundles}

Any theory of bundle gerbes is based on some monoidal category of bundles. A major obstacle in establishing the theory of \quot{non-abelian} bundle gerbes was that principal $H$-bundles, for $H$ a non-abelian Lie group, do not form a monoidal category. However, principal $H$-\emph{bi}-bundles form a monoidal category. The first version of non-abelian bundle gerbes set up by Aschieri-Cantini-Jur\v co used these \cite{aschieri}, see  \cref{prop:equivbibun}. 

Here, we use a generalization, called \quot{2-group bundles} or \emph{principal $\Gamma\!$-bundles}, for $\Gamma$ a Lie 2-group. The monoidal category of principal $\Gamma\!$-bundles is standard by now, and recalled in \cref{2-group-bundles}.
The theory of connections on principal $\Gamma\!$-bundles is  obscure at first glance, in particular, the treatment of their curvature and the involvement of adjustments. We describe this theory  in full detail in \cref{curvature-on-2-group-bundles,tensor-product-of-2-group-bundles,covariance-of-Gamma-bundles}. Finally, we show in \cref{sec:twistedconnectionsanddeligne} how (various versions of) principal $\Gamma\!$-bundles are related to the (various versions of)   presheaves of   $\Gamma\!$-connections  defined in \cref{sec:pre2stackificationofconGamma}. 

\subsection{2-group bundles}

\label{2-group-bundles}

Let $\Gamma=(H \stackrel{t}\to G \stackrel{\alpha}\to \mathrm{Aut}(H))$ be a  Lie 2-group, with $F:=\pi_0(\Gamma)$ and projection $p: G \to F$.
Based on \cite[Section 2]{Nikolaus}, we describe the monoidal category $\bun \Gamma X$ of principal $\Gamma\!$-bundles over a smooth manifold $X$. 
\begin{definition}
A \emph{principal $\Gamma\!$-bundle over $X$} is a principal $H$-bundle $P$ over $X$, together with a smooth map $\phi:P \to G$ satisfying $\phi(ph)=t(h)^{-1}\phi(p)$ for all $p\in P$ and $h\in H$. A morphism of principal $\Gamma\!$-bundles $\varphi: (P,\phi) \to (P',\phi')$ is a morphism $\varphi:P \to P'$ of principal $H$-bundles such that $\phi'\circ\varphi=\phi$. 
\end{definition}

The map $\phi$ is called the \emph{anchor} of $P$, and omitted from the notation when no explicit mention is necessary.
Principal $\Gamma\!$-bundles over $X$ form a category $\bun\Gamma X$. Moreover, the assignment $X \mapsto \bun\Gamma X$ forms a stack $\bun\Gamma-$ on the site of smooth manifolds. 

\begin{remark}
\begin{enumerate}[(i)]

\item
\label{anchor-push}
The map $p \circ \phi:P\to F$ descends to a  map $p_{*}P: X \to F$, which is smooth if $\Gamma$ is smoothly separable. The notation will be explained later (see \cref{covariance-of-Gamma-bundles}).

\item 
\label{trivial-2-group-bundle}
A principal $\Gamma\!$-bundle $(P,\phi)$ is called \emph{trivial}, if $P$ is  the trivial principal $H$-bundle, i.e. $P=X\times H$. In this case, $\phi$ can be identified with a smooth map $g:X \to G$, via the correspondence $\phi(x,h):=t(h)^{-1}g(x)$.
We write $\trivlin_{g} := (X \times H,\phi)$ and call this the \emph{trivial principal $\Gamma\!$-bundle with anchor $g$}. We have a bijection 
\begin{equation*}
\hom_{\bun\Gamma X}(\trivlin_{g_1},\trivlin_{g_2})  \cong \{h\in C^{\infty}(X,H) \sep (t \circ h) \cdot g_1=g_2 \}\text{,}
\end{equation*}
under which a smooth map $h$ corresponds to the bundle morphism $\varphi(x,h') := (x,h(x)h')$,  
and composition corresponds to the point-wise multiplication. 

\end{enumerate}
\end{remark}

The tensor product $P_1 \otimes P_2$ of two principal $\Gamma\!$-bundles  $(P_1,\phi_1)$ and $(P_2,\phi_2)$ over $X$ has the total space $(P_1 \times_X P_2)/\sim$ with $(p_1h,p_2)\sim(p_1,p_2\alpha(\phi_1(p_1)^{-1},h))$.
The principal $H$-action on $P_1 \otimes P_2$ is given by $(p_1,p_2)\cdot h = (p_1h,p_2)$, and the anchor  sends $(p_1,p_2)$ to $\phi_1(p_1)\phi_2(p_2)$. The tensor product  is associative with associators induced from the canonical diffeomorphisms $(P_1 \times_X P_2) \times_X P_3 \cong P_1\times_X(P_2 \times_X P_3)$.  Moreover, $\trivlin_1$ is  a unit object  with respect to the tensor product, expressed by bundle isomorphisms $P \otimes \trivlin_1\cong P$ and $\trivlin_1\otimes P \cong P$ given by 
\begin{equation}
\label{eq:unitors}
(p,(x,h)) \mapsto p\alpha(\phi(p),h)
\quand
((x,h),p) \mapsto ph\text{,}
\end{equation}
respectively.
Thus, $\bun\Gamma X$ forms a monoidal category, and $\bun\Gamma-$ is a monoidal stack. 
In general, $\bun\Gamma X$ is neither braided nor symmetric monoidal; however, one can show that if $\Gamma$ (viewed as a monoidal category) is braided/symmetric, then  $\bun\Gamma X$ inherits these properties.  

\begin{remark}
\begin{enumerate}[(i)]
\item 
\label{tensor-product-of-trivial-bundles}
The tensor product of trivial principal $\Gamma\!$-bundles corresponds canonically to the product of  maps. Indeed,  there is a canonical isomorphism $\gamma_{g_1,g_2}:\trivlin_{g_1} \otimes \trivlin_{g_2}\to\trivlin_{g_1g_2}$,  given by
\begin{equation*}
((x,h_1),(x,h_2)) \mapsto (x,\alpha(g_1(x),h_2)h_1)\text{.}
\end{equation*} 
These canonical morphisms are associative, in the sense that the diagram
\begin{equation*}
\alxydim{@C=4em}{\trivlin_{g_1} \otimes \trivlin_{g_2} \otimes \trivlin_{g_3} \ar[r]^-{\id \otimes \gamma_{g_2,g_3}} \ar[d]_{\gamma_{g_1,g_2} \otimes \id}  & \trivlin_{g_1} \otimes \trivlin_{g_2g_3} \ar[d]^{\gamma_{g_1,g_2g_3}} \\ \trivlin_{g_1g_2} \otimes \trivlin_{g_3} \ar[r]_{\gamma_{g_1g_2,g_3}} & \trivlin_{g_1g_2g_3}}
\end{equation*}
is commutative.

\item
\label{tensor-product-with-trivial-bundle}
The tensor product of a trivial principal $\Gamma\!$-bundle with anchor $g:X \to G$ and an arbitrary principal $\Gamma\!$-bundle $(P,\phi)$ can be identified canonically with a \emph{shifted bundle}. We denote by ${}^{g\!}P$  the  principal $H$-bundle with the same total space as $P$, but principal action $p \cdot_g h := p\alpha(g(x)^{-1},h)$ and with anchor ${}^{g\!}\phi(p) := g(x)\phi(p)$. Moreover, we denote by $P^{g}$ the same principal $H$-bundle as $P$ but with anchor $\phi^{g}(p) := \phi(p)g(x)$. Then, there are canonical isomorphisms 
\begin{equation*}
\lambda_g: \trivlin_g \otimes P \to {}^{g\!}P
\quand
\rho_g: P \otimes \trivlin_g \to P^{g}\text{,} 
\end{equation*}
given by $\lambda((x,h),p) := p\cdot_g h$ and $\rho_g(p,(x,h)) :=p \alpha(\phi(p),h)$.  
In the case $P=\trivlin_{g'}$, one obtains ${}^{g\!}(\trivlin_{g'}) = \trivlin_{gg'}$ and $(\trivlin_{g'})^{g} = \trivlin_{g'g}$; moreover, $\lambda_g=\gamma_{g,g'}=\rho_{g'}$ as isomorphisms $\trivlin_g \otimes \trivlin_{g'} \cong \trivlin_{gg'}$.

\item 
\label{left-and-right-shift-isomorphic}
If the Lie 2-group $\Gamma$ is central, and $g:X \to G$ co-restricts to $t(H) \subset G$, then the left and the right shift gives isomorphic bundles,
via the isomorphism\begin{equation*}
\delta_g: P^{g} \to {}^{g\!}P
\quomma \delta_g(p):= p \cdot  \tilde\alpha_{t^{-1}(g(x))}(\phi(p))^{-1}\text{;} 
\end{equation*}
recall from  \cref{Lie-2-groups} that $\tilde\alpha_{t^{-1}(g)}$ is well-defined for $g\in t(H)$. 
In particular, trivial principal $\Gamma\!$-bundles with anchors in $t(H)$ are central in the sense that there is a canonical braiding isomorphism
\begin{equation*}
\alxydim{}{P \otimes \trivlin_g \ar[r]^-{\rho_g} & P^{g}\ar[r]^-{\delta_g}& {}^{g\!}P \ar[r]^-{\lambda_{g}^{-1}} & \trivlin_g \otimes P\text{,}}
\end{equation*}
for any principal $\Gamma\!$-bundle $P$.

\end{enumerate}
\end{remark}

\begin{lemma}
\label{lem:invertibility}
Every principal $\Gamma\!$-bundle is invertible with respect to the tensor product. \end{lemma}

\begin{proof}
If $P$ is a principal $\Gamma\!$-bundle, we define the dual bundle $P^{*}$ to have  the same total space $P$ and projection, but the inverted anchor $\phi^{*}(p):=\phi(p)^{-1}$ and the new  $H$-action
$p\ast h:=p\alpha(\phi(p),h^{-1})$. 
A bundle isomorphism $\varepsilon :P \otimes P^{*} \to \trivlin_1$ is given by $\varepsilon(p,p^{*}) := (x,h^{-1})$,
where $x$ is the common base point of $p$ and $p^{*}$, and $h\in H$ is the unique element such that $p^{*}=ph$; similarly, one sets up an isomorphism $\eta: \trivlin_1 \to P^{*} \otimes P$, and one can prove that $\varepsilon$ and $\eta$ satisfy the usual zigzag identities.  
\end{proof}

\begin{example}
\label{examples-groupoid-bundles}
\begin{enumerate}[i)]

\item 
Let  $A$ be any abelian Lie group.
\begin{itemize}
\item 
$\Gamma=BA$ gives  the monoidal  stack of ordinary principal $A$-bundles.
We will write $\bun A-$ instead of $\bun{BA}-$.
\item
 $\Gamma = BA \times (\Z_2)_{\dis}$. Then $\bun\Gamma X$ is the  category of $\Z_2$-graded principal $A$-bundles, with objects pairs $(P,\alpha)$ consisting of an ordinary principal $A$-bundle $P$ over $X$ and $\alpha:X \to \Z_2$ a continuous map. Then tensor product is
\begin{equation*}
(P_1,\alpha_1) \otimes (P_2,\alpha_2) = (P_1 \otimes P_2,\alpha_1\alpha_2)\text{.}
\end{equation*}

\item
$\Gamma= \AUT(A)$. Again, $\bun\Gamma X$ is the  category of $\Z_2$-graded principal $A$-bundles, but now with the tensor product
\begin{equation*}
(P_1,\alpha_1) \otimes (P_2,\alpha_2) = (P_1 \otimes P_2^{\alpha_1},\alpha_1\alpha_2)\text{,}
\end{equation*}
where $P^{-1}$ is the dual bundle.

\end{itemize}

\item
\label{Fdis-bundles}
For any Lie group $F$, there is a canonical equivalence of categories 
\begin{equation*}
\bun {F_{\dis}}X \cong C^{\infty}(X,F)_{\dis}
\end{equation*}
given by sending a principal $F_{\dis}$-bundle $P$, which is, by definition  a principal $\{\ast\}$-bundle, to the smooth map $ p_{*}P:X \to F$. 

\end{enumerate}
\end{example}

\begin{remark}
\label{prop:equivbibun}
We compare principal $\Gamma\!$-bundles, for $\Gamma=\AUT(H)$, with  \emph{principal $H$-bibundles}  \cite{aschieri}.   These form a category $\bibun HX$, which is monoidal under the usual tensor product obtained by taking the quotient  of $P_1 \times_X P_2$ by equivalence relation $(p_1h,p_2)\sim (p_1,hp_2)$.
Given a principal $\AUT(H)$-bundle $P$ with anchor map $\phi: P \to \mathrm{Aut}(H)$, then $P$ becomes a principal $H$-bibundle $\lli{\phi}P$ under the given right principal action and the left action $hp := p\phi(p)(h)$.
Clearly, every morphism of principal $\AUT(H)$-bundles is equivariant with respect to both $H$-actions, and thus we have constructed a functor 
\begin{equation*}
\bun{\AUT(H)}X \to \bibun HX
\end{equation*}
We claim that this functor
 establishes an equivalence of categories, and is 
 \quot{anti-monoidal}, i.e., it is monoidal for the opposite monoidal structure on one of the two categories.  
Indeed, our functor is clearly full and faithful. For essential surjectivity, consider a principal $H$-bibundle $P$ and define $\phi:P \to \mathrm{Aut}(H)$ such that $hp = p\phi(p)(h)$ holds for all $p\in P$ and $h\in H$. Then, $P\cong \lli\phi P$ as principal $H$-bibundles. It remains to show that our functor is anti-monoidal. Consider two principal $\AUT(H)$-bundles $P$ and $P'$ over $X$ with anchor maps $\phi$ and $\phi'$, respectively. 
An  isomorphism
\begin{equation*}
\psi: \lli{{\phi\phi'}}(P\otimes P') \to \lli{{\phi'}}P'\otimes \lli\phi P
\end{equation*}
 of principal $H$-bibundles can be defined by by $(p,p')\mapsto (p',p)$. 
It is  easy to see that this is indeed an anti-monoidal structure.
\end{remark}

\setsecnumdepth{2}

\subsection{Connections}

In the following, we define connections on 2-group bundles. Again, we let $\Gamma$ be a Lie 2-group, and we consider a  principal $\Gamma\!$-bundle $P$  over a smooth manifold  $X$, with anchor $\phi:P \to G$. Furthermore, let $A,A'\in \Omega^1(X,\mathfrak{g})$ be two 1-forms. As it will become clearer later, they will play the role of a \quot{source} and a \quot{target}.  
\begin{definition}
\label{connection-on-Gamma-bundle}
An \emph{$A$-$A'$-connection} on $P$ is a  1-form $\omega\in \Omega^1(P,\mathfrak{h})$ such that the following two conditions hold:
An \emph{$A$-$A'$-connection} on $P$ is a 1-form $\omega\in \Omega^1(P,\mathfrak{h})$ satisfying the following two conditions:
\begin{enumerate}[1.)]

\item 
$\rho^{*}\omega = \Ad^{-1}_h(\pr_P^{*}\omega) +(\tilde\alpha_{h})_{*}(A') + h^{*}\theta$

\item
$t_{*}(\omega) =  \Ad_{\phi}(A) - A' - \phi^{*}\bar\theta$

\end{enumerate} 
Here, $\rho:P\times H \to P$ denotes the principal $H$-action on $P$, and $h:P \times H \to H$ denotes the projection onto the second factor. 
A bundle morphism $\varphi: P_1 \to P_2$ is called \emph{connection-preserving} for $A$-$A'$-connections $\omega_1$ and $\omega_2$ if $\varphi^{*}\omega_2=\omega_1$. 
\end{definition}

\begin{remark}
\begin{enumerate}[(i)]

\item
\label{connection-on-trivial-bundle} 
An $A$-$A'$-connection on a trivial $\Gamma\!$-bundle $\trivlin_g$ can be specified by a 1-form $\alpha\in \Omega^1(X,\mathfrak{h})$ satisfying $t_{*}(\alpha) =  \Ad_{g}(A) - A' - g^{*}\bar\theta$. The connection is then given by
\begin{equation}
\label{eq:connoftrivbun}
\omega_{\alpha} := \Ad_h^{-1}(x^{*}\alpha)+(\tilde\alpha_h)_{*}(x^{*}A')+h^{*}\theta\text{,}
\end{equation} 
where $x$ and $h$ denote the projections to the factors of $\trivlin_g=X\times H$.
We denote this $\Gamma\!$-bundle with $A$-$A'$-connection by $\trivlin_{g}^{\alpha}$, although this notation implicitly hides the dependence on the 1-form $A'$ in \cref{eq:connoftrivbun}. 
If a bundle morphism $\trivlin_{g_1}^{\alpha_1}\to\trivlin_{g_2}^{\alpha_2}$ is identified with a smooth map $h:X \to H$ as in \cref{trivial-2-group-bundle}, then it is connection-preserving if and only if $h$ satisfies $\alpha_1=\Ad^{-1}_{h}(\alpha_2)+(\tilde\alpha_{h})_{*}(A')+h^{*}\theta$.

\item
\label{adjusted-shift-of-connection}
Suppose $\Gamma$ is central and equipped with an adjustment $\kappa$ adapted to a splitting $u$. Given a 1-form $\eta \in \Omega^1(X,\mathfrak{g})$, we may shift an $A$-$A'$-connection $\omega$ on a principal $\Gamma\!$-bundle $P$ to 
\begin{equation*}
\omega^{\eta} := \omega + \kappa(\phi, \eta)\text{,}
\end{equation*}
which becomes an $\tilde A$-$\tilde A'$-connection for
$\tilde A := A + \eta$ and $\tilde A' := A'+\tilde q_u(p_{*}P,\eta)$, 
where $p_{*}P:X \to F$ is the smooth map defined in \cref{anchor-push}, and, for a section $q$ of $\mathfrak{G}$, we set
\begin{equation*}
\tilde q: F \times \mathfrak{g} \to \mathfrak{g}
\quad\text{,}\qquad
\tilde q(f,X) := X+q(\mathrm{Ad}_f(p_{*}(X))-p_{*}(X))\text{.}
\end{equation*}
We refer to $\omega^{\eta}$ as the \emph{adjusted shift of $\omega$ by $\eta$}. 
The adjusted shift of connections is compatible with connection-preserving isomorphisms: if $\varphi:P\to P'$ is a morphism between $\Gamma\!$-bundles that is connection-preserving for connections $\omega$ and $\omega'$, respectively, then it remains connection-preserving for the adjusted shifts.
If $\mu\in \Omega^1(X,\mathfrak{g})$ is another 1-form, then we have
\begin{equation*}
(\omega^{\eta})^{\mu}=\omega^{\eta+\mu}\text{.}
\end{equation*}
If $\trivlin_g^{\alpha}$ is a trivial principal $\Gamma\!$-bundle with $A$-$A'$-connection determined by $\alpha\in \Omega^1(X,\mathfrak{h})$, then we have
\begin{equation*}
(\trivlin_g^{\alpha})^{\eta} = \trivlin_g^{\alpha+\kappa(g,\eta)}
\end{equation*}
as bundles with $(A+\eta)$-$(A'+\tilde q(p \circ g,\eta))$-connection. 
We remark that  adjusted shifts of connections are particularly simple when $\eta = q(\beta)$ for a 1-form $\beta\in\Omega^1(X,\mathfrak{f})$, because then $\tilde q(f,q(Z))=q(\Ad_{f}(Z))$. 

\end{enumerate}
\end{remark}

The following result concerns the existence of connections on principal $\Gamma\!$-bundles.

\begin{lemma}
\label{existence-connection-A}
Suppose that $\Gamma$ is central and allows an adapted adjustment.
If $P$ is a principal $\Gamma\!$-bundle and $A\in \Omega^1(X,\mathfrak{g})$, then there exist $A'\in \Omega^1(X,\mathfrak{g})$ and an $A$-$A'$-connection $\omega$ on $P$.
\end{lemma}

\begin{proof}
Since $P$ has an underlying principal $H$-bundle, it admits a connection $\omega$. As an ordinary connection, it satisfies the first condition in \cref{connection-on-Gamma-bundle} for $A':= 0$. We consider $\tilde A:= t_{*}((\alpha_{\phi^{-1}})_{*}(\omega))+\phi^{*}\theta\in \Omega^1(P,\mathfrak{g})$. One can easily verify that this form descends to $X$, and then $\omega$ is a $\tilde A$-$0$-connection. 
Now we perform an adjusted shift by $\eta := A-\tilde A$. By \cref{adjusted-shift-of-connection}, this becomes an $A$-$A'$-connection for $A' := 0 +\tilde q_u(p_{*}P,A-\tilde A)$. 
\end{proof}

At this point, we may define a category $\buncon\Gamma X_{A,A'}$ consisting of principal $\Gamma\!$-bundles with $A$-$A'$-connections and all connection-preserving bundle morphisms. 

\begin{example}
\label{groupoid-bundles-with-connection}
\begin{enumerate}[(i)]
\item
For $\Gamma=BA$, where necessarily $A=A'=0$, $\buncon AX:=\buncon {BA} X_{0,0}$ is precisely the category of ordinary principal $A$-bundles with connection. 

\item
For $\Gamma=F_{\dis}$ we have, for any $A,A'\in\Omega^1(X,\mathfrak{f})$,
an isomorphism of categories\begin{equation*}
\buncon{F_{\dis}}X_{A,A'} \cong \{\phi\in C^{\infty}(X,F)\sep \Ad_{\phi}(A)=A'+\phi^{*}\bar \theta\}_{\dis}\text{,}
\end{equation*} 
given by $P \mapsto p_{*}P$.
\end{enumerate}
\end{example}


\label{curvature-on-2-group-bundles}

The curvature of connections on 2-group bundles is notoriously subtle and requires careful treatment. 

\begin{definition}
\label{curvature-groupoid-connection}
Let $B,B'\in\Omega^2(X,\mathfrak{h})$, and let $\omega$ be an $A$-$A'$-connection on a principal $\Gamma\!$-bundle $P$. The \emph{$B$-$B'$-curvature} of $\omega$ is defined by
\begin{equation*}
\mathrm{curv}_{B,B'}(\omega) := \mathrm{d}\omega + \tfrac{1}{2}[\omega \wedge \omega] + \alpha_{*}(A'\wedge \omega)-(\alpha_{\phi})_{*}(B)+B' \in \Omega^2(P,\mathfrak{h})\text{.} 
\end{equation*}
We say that $\omega$ is \emph{$B$-$B'$-flat} if $\mathrm{curv}_{B,B'}(\omega)=0$. 
If the Lie 2-group $\Gamma$ is equipped with an adjustment $\kappa$, we say that $\omega$ is \emph{$B$-$B'$-adjusted} if $\mathrm{curv}_{B,B'}(\omega)= \kappa(\phi,\mathrm{fcurv}(A,B))$ holds.
\end{definition}

Here, $\mathrm{fcurv}(A,B)$ denotes the fake-curvature of the pair $(A,B)$ as defined in \cref{def:gammacon}, understood to be pulled back along the bundle projection $P \to X$. 
Obviously, if $(A,B)$ is fake-flat, then $\omega$ is $B$-$B'$-adjusted if and only if $\omega$ is $B$-$B'$-flat.
Both the curvature and the conditions of flatness and adjustedness are preserved under connection-preserving bundle morphisms.

\begin{remark}
\begin{enumerate}[(a)]

\item 
A straightforward calculation reveals
\begin{equation}
\label{eq:toftwistedcurvature}
t_{*}(\mathrm{curv}_{B,B'}(\omega))= \Ad_{\phi}(\mathrm{fcurv}(A,B)) - \mathrm{fcurv}(A',B')\text{.}
\end{equation}
Thus, if $\omega$ is $B$-$B'$-flat, then $\mathrm{fcurv}(A',B')=\Ad_{\phi}(\mathrm{fcurv}(A,B))$ on $P$. 
If $\omega$ is $B$-$B'$-adjusted, then we have
\begin{equation}
\label{eq:adjcurv}
t_{*}(\kappa(\phi,\mathrm{fcurv}(A,B)))= \Ad_{\phi}(\mathrm{fcurv}(A,B)) - \mathrm{fcurv}(A',B')\text{.}
\end{equation}

\item
\label{twisted-curvature-on-trivial-bundle}
Consider a trivial principal $\Gamma\!$-bundle $\trivlin_g^{\alpha}$ with connection $\omega_{\alpha}$, determined by a smooth map $g:X \to G$ and a 1-form $\alpha\in \Omega^1(X,\mathfrak{h})$. 
Then, the $B$-$B'$-curvature of $\omega_{\alpha}$ is
\begin{equation*}
\mathrm{curv}_{B,B'}(\omega_{\alpha})= \Ad^{-1}_{\pr_H}(\pr_X^{*}\mathrm{curv}(g,\alpha))
+(\tilde\alpha_{\pr_H})_{*}(\pr_X^{*}\mathrm{fcurv}(A',B')) \in \Omega^2(X \times H,\mathfrak{h})\text{,}
\end{equation*}
where $\mathrm{curv}(g,\alpha)$ is the curvature of the gauge transformation $(g,\alpha):(A,B) \to (A',B')$ in the sense of \cref{curvature-of-a-gauge-transformation}.
This implies that $\omega_{\alpha}$ is $B$-$B'$-adjusted if and only if $(g,\alpha)$ is adjusted in the sense of \cref{def8}, and that $\omega_{\alpha}$ is $B$-$B'$-flat if and only if $(g,\alpha)$ is flat and $(A,B)$ is fake-flat. 

\item
\label{shift-of-curving}
If an $A$-$A'$-connection $\omega$ is $B$-$B'$-adjusted, then it is also $(B+\lambda)$-$(B'+\lambda)$-adjusted, for any $\lambda\in \Omega^2(X,\mathfrak{h})$. 

\item
The $B$-$B'$-curvature satisfies the following transformation rule with respect to the $H$-action $\rho$ on $P$:
\begin{equation}
\label{eq:twistedcurvtrans}
\rho^{*}\mathrm{curv}_{B,B'}(\omega)=\Ad_{\pr_H}^{-1}(\pr_{P}^{*}\mathrm{curv}_{B,B'}(\omega))
  +(\tilde\alpha_{\pr_H})_{*}(\mathrm{fcurv}(A',B'))\text{.}
\end{equation}
Thus, in the fake-flat case, the curvature transforms like the curvature of ordinary connections.

\item\label{curvature-of-adjusted-shift}
We compute the curvature change under an adjusted shift of an $A$-$A'$-connection $\omega$ by a 1-form $\eta\in\Omega^1(X,\mathfrak{g})$ as in \cref{adjusted-shift-of-connection}. We assume that $\omega$ is $B$-$B'$-adjusted for 2-forms  $B,B'\in \Omega^2(X,\mathfrak{h})$, and define
\begin{align*}
\tilde B &:= B +  \tfrac{1}{2}\kappa_{*}(\eta \wedge \eta)+\kappa_{*}( A \wedge\eta)
\\
\tilde B' & := B' +  \tfrac{1}{2}\kappa_{*}( \eta'\wedge  \eta')+\kappa_{*}( A' \wedge \eta')\text{,}
\end{align*}
where $\eta' := \tilde q(p_{*}P,\eta)$.
A direct, albeit lengthy, calculation shows  that the shifted $\tilde A$-$\tilde A'$-connection $\omega^{\eta}$ is $\tilde B$-$\tilde B'$-adjusted. 

\end{enumerate}
\end{remark}

Next we look again at the existence of adjusted connections. 

\begin{lemma}
\label{existence-of-adjusted-connection}
Suppose that $\Gamma$ is central and allows an adapted adjustment. Further, let $(A,B)$ be a $\Gamma\!$-connection and $P$ be a principal $\Gamma\!$-bundle over $X$. Then, there exists a $\Gamma\!$-connection $(A',B')$ and a $B$-$B'$-adjusted $A$-$A'$-connection on $P$.
\end{lemma}

\begin{proof}
By \cref{existence-connection-A}, there exist $A'$ and an $A$-$A'$-connection $\omega$. We want to define $B'$ such that $\omega$ is $B$-$B'$-adjusted; for this purpose, we set
\begin{equation*}
 B' :=\kappa(\phi,\mathrm{fcurv}(A,B))-\mathrm{d}\omega - \tfrac{1}{2}[\omega \wedge \omega] - \alpha_{*}(A'\wedge \omega)+(\alpha_{\phi})_{*}(B)\in \Omega^2(P,\mathfrak{h})\text{.}
\end{equation*}
Then one verifies in a straightforward calculation that $\rho^{*}B' = \pr_P^{*}B'$ on $P \times X$, so that $B'$ descends to $X$.
\end{proof}

Previously, we organized principal $\Gamma\!$-bundles with $A$-$A'$-twisted connection into a category $\buncon\Gamma X_{A,A'}$.
Now, given an adjustment $\kappa$ on $\Gamma$ and two $\Gamma\!$-connections $(A,B)$ and $(A',B')$ we define a full subcategory 
\begin{equation*}
\buncon{\Gamma\!,\kappa} X_{(A,B),(A',B')} \subset \buncon\Gamma X_{A,A'}
\end{equation*}
over all principal $\Gamma\!$-bundles with $B$-$B'$-adjusted $A$-$A'$-connection. 
We have a further full subcategory
\begin{equation*}
\bunconflat\Gamma X_{(A,B),(A',B')}  \subset \buncon{\Gamma\!,\kappa} X_{(A,B),(A',B')} 
\end{equation*}
over those bundles with $B$-$B'$-flat connections, with equality when $(A,B)$ or $(A',B')$ is fake-flat. 
Moreover, we let, for $f:X \to F$ a smooth map, 
\begin{equation*}
\buncon{\Gamma,{\kappa}} X_{(A,B),(A',B')}^{f}  \subset \buncon{\Gamma\!,\kappa} X_{(A,B),(A',B')} 
\end{equation*}
be the full subcategory over those bundles $P$ with $p_{*}P=f$. 

\begin{remark}
\begin{enumerate}
\item 
\Cref{shift-of-curving} says that
\begin{equation*}
\buncon{\Gamma\!,\kappa} X_{(A,B),(A',B')}=\buncon{\Gamma\!,\kappa} X_{(A,B+\lambda),(A',B'+\lambda)}
\end{equation*}
for any $\lambda\in \Omega^2(X,\mathfrak{h})$.

\item
\cref{adjusted-shift-of-connection,curvature-of-adjusted-shift} reveal that the adjusted shift by a 1-form $\eta\in\Omega^1(X,\mathfrak{g})$ is a functor
\begin{equation*}
(..)^{\eta}:\buncon{\Gamma\!,\kappa} X_{(A,B),(A',B')} \to \buncon{\Gamma\!,\kappa} X_{(\tilde A,\tilde B),(\tilde A',\tilde B')} 
\end{equation*}
with 
\begin{align*}
\tilde A &:= A+\eta & \tilde B &:= B+ \tfrac{1}{2}\kappa_{*}(\eta \wedge \eta)+\kappa_{*}( A \wedge\eta)
\\
\tilde A' &:= A'+\eta' & \tilde B' &:= B'+  \tfrac{1}{2}\kappa_{*}(\eta'\wedge \eta')+\kappa_{*}( A' \wedge\eta')
\end{align*}
and $\eta' := \tilde q(p_{*}P,\eta)$.
\end{enumerate}
\end{remark}

\begin{example}
We continue \cref{groupoid-bundles-with-connection}.
\begin{enumerate}[(i)]

\item 
\label{BA-bundles-with-connection}
For $\Gamma=BA$, we recall that $\buncon AX=\buncon{BA}X_{0,0}$ is the category of ordinary principal $A$-bundles and connections. If $B,B'\in\Omega^2(X,\mathfrak{a})$, then 
\begin{equation*}
\buncon AX^{B-B'}:=\buncon{BA}X_{(0,B),(0,B')}
\end{equation*}
is the category of ordinary principal $A$-bundles with connections $\omega$ whose curvature satisfies $\mathrm{curv}(\omega)=\mathrm{d}\omega=B-B'$.

\item
\label{Fdis-bundles-with-connection}
For $\Gamma=F_{\dis}$, where necessarily $B=B'=0$, we have, for any $A,A'\in\Omega^1(X,\mathfrak{f})$, the same result as before,
\begin{equation*}
\buncon{F_{\dis}}X_{A,A'}=\buncon{F_{\dis}}X_{(A,0),(A',0)}\cong \{\phi\in C^{\infty}(X,F)\sep \Ad_{\phi}(A)=A'+\phi^{*}\bar \theta\}_{\dis}\text{.}
\end{equation*}

\end{enumerate}
\end{example}

\subsection{Tensor product}

\label{tensor-product-of-2-group-bundles}

We investigate the fact that the category $\buncon\Gamma X_{A,A'}$ lacks a global monoidal structure, possessing instead a partially defined one. For $A,A',A''\in \Omega^1(X,\mathfrak{g})$, this partial tensor product is a functor
\begin{equation*}
\otimes : \buncon{\Gamma}X_{A',A''} \times \buncon{\Gamma}X_{A,A'} \to \buncon{\Gamma}X_{A,A''}\text{.}
\end{equation*}
The core construction is given by the following lemma, which follows from a direct computation.
\begin{lemma}
\label{tensor-product-connection}
Let $\omega_1\in \Omega^1(P_1,\mathfrak{h})$ be an $A'$-$A''$-connection on $P_1$, and let $\omega_2\in\Omega^1(P_2,\mathfrak{h})$ be an $A$-$A'$-connection on $P_2$. Then, the 1-form $\omega\in\Omega^1(P_1 \times_X P_2,\mathfrak{h})$ defined by
\begin{equation*}
\omega :=  \pr_1^{*}\omega_1 + (\alpha_{\phi_1 \circ \pr_1})_{*}(\pr_2^{*}\omega_2) \text{,}
\end{equation*}
where $\phi_1$ is the anchor of $P_1$,
descends to an $A$-$A''$-connection on $P_1 \otimes P_2$.
\end{lemma}

The connection $\omega$ on $P_1 \otimes P_2$ defined in \cref{tensor-product-connection} is called the \emph{tensor product connection} and is denoted by $\omega_1 \otimes \omega_2$.
One readily verifies that this construction is functorial with respect to connection-preserving isomorphisms. 

\begin{remark}
Let $B, B', B''\in \Omega^2(X,\mathfrak{h})$.
Then, for the curvatures, we obtain:
\begin{equation*}
\mathrm{curv}_{B,B''}(\omega_1 \otimes \omega_2)=\mathrm{curv}_{B',B''}(\omega_1) + \mathrm{curv}_{B,B'}(\omega_2)\text{.}
\end{equation*}
Thus, the tensor product of flat connections is flat, and one easily deduces that the tensor product of adjusted connections is adjusted.
Thus, the tensor product restricts to a functor
\begin{equation*}
\otimes : \buncon{{\Gamma\!,\kappa}}X_{(A',B'),(A'',B'')} \times \buncon{{\Gamma\!,\kappa}}X_{(A,B),(A',B')} \to \buncon{{\Gamma\!,\kappa}}X_{(A,B),(A'',B'')}\text{.}
\end{equation*}
\end{remark}

\begin{remark}
\begin{enumerate}[(a)]

\item 
The trivial $\Gamma\!$-bundle $\trivlin_1$ carries a canonical $A$-$A$-connection $\omega_0$, for any 1-form $A\in\Omega^1(X,\mathfrak{g})$,  induced by $\alpha=0$ via \cref{eq:connoftrivbun}, and hence is denoted by $\trivlin_1^0$. 
It serves as a unit with  respect to the partially defined tensor product in the sense that the bundle isomorphisms of \cref{eq:unitors} are connection-preserving for the tensor product connections $\omega \otimes \omega_0$ on $P\otimes \trivlin^0_1$  and $\omega_{0}\otimes \omega$ on $\trivlin^0_1\otimes P$, hence yielding isomorphisms
\begin{equation*}
P \otimes \trivlin^0_1\cong P\cong \trivlin^0_1 \otimes P
\end{equation*}
in $\buncon\Gamma X_{A,A'}$. 

\item
We recall that every principal $\Gamma\!$-bundle $P$ has a dual bundle $P^{*}$ with respect to the tensor product. If $P$ is equipped with an $A$-$A'$-connection $\omega$, then $P^{*}$ carries the $A'$-$A$-twisted connection $\omega^{*} := -(\alpha_{\phi^{-1}})_{*}(\omega)$. Then, $P \otimes P^{*} \cong \trivlin_1^0$ and $P^{*} \otimes P \cong \trivlin_1^0$. 

\item
\label{tensor-product-of-trivial-bundles-with-connections}
We consider two trivial principal $\Gamma\!$-bundles $\trivlin_{g_1}$ and $\trivlin_{g_2}$, along with 1-forms $A,A',A'' \in \Omega^1(X,\mathfrak{g})$ and 1-forms $\alpha_1,\alpha_2\in \Omega^1(X,\mathfrak{h})$ such that $t_{*}(\alpha_1) =  \Ad_{g_1}(A) - A' - g_1^{*}\bar\theta$ and $t_{*}(\alpha_2) =  \Ad_{g_2}(A') - A'' - g_2^{*}\bar\theta$. The 1-forms $\alpha_i$ define an $A$-$A'$-connection on $\trivlin_{g_1}$ and an $A'$-$A''$-connection on $\trivlin_{g_2}$, respectively; see \cref{connection-on-trivial-bundle}. Then, the isomorphism $\gamma_{g_1,g_2}$ of \cref{tensor-product-of-trivial-bundles} is a connection-preserving isomorphism 
\begin{equation*} 
 \trivlin^{\alpha_1}_{g_1} \otimes \trivlin^{\alpha_2}_{g_2} \cong \trivlin^{\alpha_1 + (\alpha_{g_1})_{*}(\alpha_2)}_{g_1g_2}\text{.}
\end{equation*} 
In particular, when $\alpha$ determines a connection on $\trivlin_g$, then $-(\alpha_g^{-1})_{*}(\alpha)$ is the connection on the dual bundle $\trivlin_{g}^{*}$.  

\item
\label{tensor-product-connection-for-trivial-bundle}
We consider a principal $\Gamma\!$-bundle $P$ with $A$-$A'$-connection $\omega$, and a trivial principal $\Gamma\!$-bundle $\trivlin_g$ with $A'$-$A''$-connection, given by a smooth map $g:X \to G$ and a 1-form $\alpha\in \Omega^1(X,\mathfrak{h})$ such that $t_{*}(\alpha) =  \Ad_{g}(A') - A'' - g^{*}\bar\theta$ according to \cref{connection-on-trivial-bundle}. 
We denote by $P^{(g,\alpha)}$ the principal $\Gamma\!$-bundle $P^{g}$ of \cref{tensor-product-with-trivial-bundle} equipped with the connection $\omega + (\alpha_{\phi})_{*}(\alpha)$. Then, the isomorphism
\begin{equation*}
\rho_g: P \otimes \trivlin_{g}^{\alpha} \to P^{(g,\alpha)}
\end{equation*}
is connection-preserving. 
Similarly, we may consider a trivial principal $\Gamma\!$-bundle $\trivlin_g$ with $A''$-$A$-connection given by a smooth map $g:X \to G$ and a 1-form $\alpha\in \Omega^1(X,\mathfrak{h})$ such that $t_{*}(\alpha) =  \Ad_{g}(A'') - A - g^{*}\bar\theta$ according to \cref{connection-on-trivial-bundle}. 
We denote by $^{(g,\alpha)}P$ the principal $\Gamma\!$-bundle ${}^{g\!}P$ of \cref{tensor-product-with-trivial-bundle} equipped with the connection $\alpha+(\alpha_{g})_{*}(\omega)$. Then, the isomorphism
\begin{equation*}
\lambda_g: \trivlin_{g}^{\alpha} \otimes P \to {}^{(g,\alpha)\!}P
\end{equation*}
is connection-preserving. 

\item
\label{adjusted-shift-of-tensor-product-connection}
\label{shifting-connection-of-tensor-product}
We analyze the adjusted shift (\cref{adjusted-shift-of-connection}) of a tensor product connection, assuming $\Gamma$ is central and equipped with an adjustment $\kappa$ adapted to a splitting $u$. We suppose that $P_1$ carries an $A$-$A''$-twisted connection $\omega_1$ and $P_2$ carries an $A$-$A'$-connection $\omega_2$, so that $P_1 \otimes P_2$ carries the $A$-$A''$-twisted connection $\omega_1\otimes \omega_2$. Its adjusted shift satisfies
\begin{equation*}
(\omega_1\otimes \omega_2)^{\eta}= \omega_1^{q_u(p_{*}P_2,\eta)} \otimes \omega_2^{\eta}\text{.}
\end{equation*} 
\end{enumerate}
\end{remark}

\begin{remark}
\label{ascj-2}
Finally, we revisit the relationship to principal bibundles. Let $P$ be a principal $H$-bibundle over $X$. 
\begin{itemize}

\item 
Following \cite[Def. 8]{aschieri}, a \quot{2-connection} on $P$ is a pair $(\tilde A,\tilde\omega)$ consisting of an $\der(H)$-valued 1-form $\tilde A$ on $X$ and an $\mathfrak{h}$-valued 1-form $\tilde\omega$ on $P$ such that
\begin{enumerate}[(a)]

\item 
$h_p^{*}\tilde\omega=-\bar\theta$ for all $p\in P$, where $h_p: H \to P$ is defined by $h_p(h):=hp$. 

\item
$l_h^{*}\tilde\omega=\Ad_h(\tilde\omega)+(\tilde\alpha_{h^{-1}})_{*}(\tilde A)$ for all $h\in H$, where $l_h: P \to P$ is defined by $l_h(p):=hp$. 

\end{enumerate}
We remark that the 1-form $(\tilde\alpha_{h^{-1}})_{*}(\tilde A)$ at $v\in T_xX$ is
\begin{equation*}
(\tilde\alpha_{h^{-1}})_{*}(\tilde A_p(v))=\partial_th\alpha(e^{t\tilde A_p(v)},h^{-1})=h\partial_te^{t\tilde A_p(v)}(h^{-1})\text{,}
\end{equation*}
and the latter expression is the one that appears in \cite{aschieri}. 
We also remark that the sign in (a) is due to \cite[Eq. 39, Eq. 51]{aschieri}.
Every 2-connection determines another 1-form denoted $\tilde A^{r}\in \Omega^1(X,\der(H))$. In order to define it properly, we consider the map $\varphi: P \to \mathrm{Aut}(H)$ defined such that $\varphi(p)(h)p=ph$, see \cite[Eq. 4]{aschieri}.  
Then, $\tilde A^{r}$ is characterized by the equation 
\begin{equation*}
\tilde A^{r} =\Ad_{\varphi}^{-1} (\tilde A+t_{*}(\tilde\omega))+\varphi^{*}\theta
\end{equation*} 
of $\der(H)$-valued 1-forms on $P$, see \cite[Eq. 53]{aschieri}, and one can check that $\tilde A^r$ descends to $X$.

\item
A morphism $\varphi:P \to P'$ of principal $H$-bibundles is connection-preserving if $\varphi^{*}\tilde\omega'=\tilde\omega$ and $\tilde A=\tilde A'$. This implies that $\tilde A^{r}=\tilde A'^{r}$. The category of principal $H$-bibundles over $X$ with connections relative to 1-forms $\tilde A$ and $\tilde A^r$ is denoted by $\bibuncon HX_{\tilde A,\tilde A^r}$. 

\item
The curvature of a connection $(\tilde A,\tilde\omega)$ is defined in \cite[Eq. 66]{aschieri} to be the pair of 2-forms
\begin{align}
\label{eq:curvbibun}
\tilde\kappa_{\tilde A,\tilde\omega}&:=\mathrm{d}\tilde\omega +\tfrac{1}{2}[\tilde\omega\wedge\tilde\omega]-\alpha_{*}(\tilde A\wedge \tilde \omega)\in\Omega^2(P,\mathfrak{h})\text{,}
\\
K_{\tilde A}&:=\mathrm{d}\tilde A + \tfrac{1}{2}[\tilde A\wedge \tilde A]\in\Omega^2(X,\der(H))\text{.}
\end{align}

\item
For the discussion of tensor products, suppose $P_1$ and $P_2$ are principal $H$-bibundles equipped with connections $(\tilde A_1,\tilde\omega_1)$ and $(\tilde A_2,\tilde\omega_2)$, respectively, such that $\tilde A_2=\tilde A_1^{r}$ \cite[Eq. 81]{aschieri}. Then, the 1-form $\pr_1^{*}\tilde\omega_1 + (\alpha_{\varphi_1})_{*}(\pr_2^{*}\tilde\omega_2)$ on $P_1 \times_X P_2$ descends to a 1-form $\tilde\omega_1\otimes \tilde\omega_2$ on the tensor product $P_1 \otimes P_2$, and the pair $(A_1,\tilde\omega_1\otimes \tilde\omega_2)$ is a connection on $P_1 \otimes P_2$ that determines the 1-form $\tilde A_2^r$. In this manner, the category $\bibuncon HX$ admits a partially defined tensor product. 

\item
We obtain a functor $\buncon{\AUT(H)}X_{A,A'} \to \bibuncon HX_{A,A'}$ that extends the previous functor $P\mapsto \lli{\phi}P$ by sending an $A$-$A'$-connection $\omega$ on a principal $\AUT(H)$-bundle $P$ to the bibundle connection $(\tilde A,\tilde\omega)$ on $\lli{\phi}P$ given by $\tilde A := A$ and $\tilde\omega := -(\alpha_{\phi^{-1}})_{*}(\omega)$. 
The corresponding 1-form $\tilde A^{r}$ is $A'$.
The curvature forms of $(\tilde A,\tilde\omega)$ defined in \cref{{eq:curvbibun}} are then
\begin{equation*}
\tilde\kappa_{\tilde A,\tilde\omega} = - (\alpha_{\phi^{-1}})_{*} (\mathrm{d}\omega+\tfrac{1}{2}[\omega\wedge\omega]+\alpha_{*}(A'\wedge \omega) )
\quand K_{\tilde A}=\mathrm{d}A +\tfrac{1}{2}[A\wedge A]\text{.}
\end{equation*}
Thus, these curvature forms correspond to the forms $\mathrm{curv}_{B,B'}(\omega)$ and $\mathrm{fcurv}(A,B)$ but without the 2-form contributions. 
It is straightforward to verify that this functor constitutes an equivalence and is monoidal with respect to the opposite monoidal structure on one side.

\end{itemize}
\end{remark}

\subsection{Extension and reduction}

\label{covariance-of-Gamma-bundles}

Let $\Gamma$ and $\Gamma'$ be two Lie 2-groups, and let $f:\Gamma \to \Gamma'$ be a Lie 2-group homomorphism; see \cref{Lie-2-groups}. We construct a functor
\begin{equation*}
f_{*}:\bun{\Gamma}X \to \bun{\Gamma'}X
\end{equation*}
as follows. Let $\pi:P \to X$ be a principal $\Gamma\!$-bundle over $X$ with anchor $\phi:P \to G$. Then, $f_{*}(P)$ is the principal $\Gamma'$-bundle over $X$ with total space $(P \times H')/\sim$, where the equivalence relation is defined by $(ph,h')\sim (p,f_1(h)h')$. The projection is given by $(p,h) \mapsto \pi(p)$, the principal $H'$-action by $(p,h')h'' := (p,h'h'')$, and the anchor by $(p,h') \mapsto t(h')^{-1}f_0(\phi(p))$. 
A morphism $\varphi:P_1 \to P_2$ between principal $\Gamma\!$-bundles is mapped to the morphism $f_{*}(\varphi):f_{*}(P_1) \to f_{*}(P_2)$ defined by $(p,h')\mapsto (\varphi(p),h')$. This defines the functor $f_{*}$ as claimed. Furthermore, this functor is monoidal, with the required natural transformation given by
\begin{equation}
\label{monoidal-covariance}
\mu:f_{*}(P_1 \otimes P_2) \to f_{*}(P_1) \otimes f_{*}(P_2), \quad ((p_1,p_2),h')\mapsto ((p_1,h'),(p_2,1))\text{.}
\end{equation}

We write again $A:=\pi_1\Gamma$ and by $F:=\pi_0\Gamma$, and denote by $i: BA \to \Gamma$ the inclusion  and by $p: \Gamma \to F_{\dis}$ the projection. Under the equivalence $\bun {F_{\dis}}X \cong C^{\infty}(X,F)_{\dis}$ of \cref{Fdis-bundles}, the monoidal functor
\begin{equation*}
p_{*}: \bun\Gamma X \to C^{\infty}(X,F)_{\dis} 
\end{equation*}
is obtained by sending $P$ to the smooth map $p_{*}P:X \to F$ defined in \cref{anchor-push}, thereby justifying this notation. The monoidal functor
\begin{equation*}
i_{*}: \bun {A}X \to \bun \Gamma X
\end{equation*}
maps a principal $A$-bundle $P$ to the extended $H$-bundle $(P \times H)/\sim$, where $(pa,h)\sim (p,ah)$; the projection is $(p,h) \mapsto \pi(p)$, the principal $H$-action is $(p,h)h' := (p,hh')$, and the anchor map is $(p,h) \mapsto t(h)^{-1}$. Thus, the functor $i_{*}$ maps into the subcategory $\bun\Gamma X^1$ of principal $\Gamma\!$-bundles $P$ satisfying $p_{*}P=1$. Since anchors multiply under the tensor product, this subcategory is monoidal. 

\begin{proposition}
\label{reduction-to-abelian}
The functor $i_{*}$ establishes an equivalence of monoidal categories
\begin{equation*}
\bun{A}X \cong \bun\Gamma X^1\text{.}
\end{equation*}
Moreover, its image is central in $\bun \Gamma X$ in the sense that $i_{*}$ extends canonically to the Drinfeld center:
\begin{equation*}
\alxydim{}{ & \mathscr{Z}(\bun\Gamma X) \ar[d] \\ \bun{A}X \ar[r]_{i_{*}} \ar@/^1.4pc/[ur] & \bun\Gamma X}
\end{equation*}
\end{proposition}

\begin{proof}
The first assertion follows from \cite[Thm. 4.1.2]{Nikolausa}.
We recall the proof of essential surjectivity. If $(P,\varphi)$ is a principal $\Gamma\!$-bundle such that $p \circ \varphi=1$, we define $P_{red}\subset P$ as the subset of points where $\varphi(p)=1$. Equipped with the restriction of the $H$-action along $i:A \to H$, this forms a principal $A$-bundle. An isomorphism $i_{*}(P_{red})\cong P$ is given by $(p,h) \mapsto ph$.

The lift to Drinfeld center is provided by a  \quot{half-braiding} isomorphism
\begin{equation*}
\beta: P \otimes i_{*}(Q) \cong i_{*}(Q) \otimes P
\end{equation*}
for each principal $A$-bundle $Q$ and every principal $\Gamma\!$-bundle $P$, defined by
\begin{equation*}
(p ,  (q,h)) \mapsto ((q,\alpha(\phi(p),h)),p)\text{.}
\end{equation*}
Verifying all required properties is routine, including naturality in the bundles and the hexagon axioms. 
\end{proof}

\begin{remark}
\label{braiding-abelian}
If $P$ itself is the extension of a principal $A$-bundle $P'$, i.e., $P=i_{*}(P')$, then the braiding isomorphism $ i_{*}(P') \otimes i_{*}(Q) \cong i_{*}(Q) \otimes i_{*}(P')$ corresponds to $i_{*}$ applied to the standard braiding of the symmetric monoidal category $\bun AM$. 
\end{remark}

\begin{remark}
\label{i-star-trivializes}
It is possible that $i_{*}(Q)$ is trivializable (see \cref{the-trivialization} for an example), i.e., there exists an isomorphism $\varphi: i_{*}(Q) \to \trivlin_g$ of principal $\Gamma\!$-bundles over $X$, with $g:X \to G$. In this case, it follows that the image of $g$ is contained in $t(H)$.  
If moreover $\Gamma$ is central, one can check that the following diagram is commutative,
\begin{equation*}
\alxydim{}{P \otimes i_{*}(Q) \ar[r]^{\beta} \ar[d]_{\id \otimes \varphi} & i_{*}(Q) \otimes P \ar[d]^{\varphi \otimes \id} \\ P \otimes \trivlin_g \ar[d]_{\rho_g}  & \trivlin_g \otimes P \ar[d]^{\lambda_g} \\ P^{g} \ar[r]_{\delta_g} & {}^{g\!}P\text{,}}
\end{equation*}
where $P^{g}$ and ${}^{g\!}P$ denote the anchor shifts of \cref{tensor-product-with-trivial-bundle}, and $\delta_g$ is the isomorphism of \cref{left-and-right-shift-isomorphic}. 
\end{remark}

Next, we incorporate connections into this framework. Suppose $P$ is a principal $\Gamma\!$-bundle over $M$, $A_1,A_2\in \Omega^1(M,\mathfrak{g})$, and $\omega$ is an $A_1$-$A_2$-connection on $P$. Let $f:\Gamma\to\Gamma'$ be a Lie 2-group homomorphism compatible with adjustments as specified in \cref{preserve-adjustment}. Let $f_{*}(P)$ be the extended principal $\Gamma'$-bundle over $M$. We define $A_1':=(f_0)_{*}(A_1)$ and $A_2' := (f_0)_{*}(A_2)$. The following result can be established via straightforward calculations.

\begin{lemma}
The 1-form 
\begin{equation*}
\omega' := \Ad^{-1}_{h'}((f_1)_{*}(p^{*}\omega))+(\tilde\alpha_{h'})_{*}(A_2')+h'^{*}\theta \in \Omega^1(P \times H',\mathfrak{h}') 
\end{equation*} 
where $h':P \times H' \to H'$ and $p:P \times H' \to P$ are the projections,
descends to  an $A_1'$-$A_2'$-connection on $f_{*}(P)$. Moreover, if $B_1,B_2 \in \Omega^2(M,\mathfrak{h})$, we set $B_1':=(f_1)_{*}(B_1)$ and $B_2' := (f_1)_{*}(B_2)$; then, we have the following result:
\begin{itemize}

\item 
$\omega'$ is $B_1'$-$B_2'$-flat if $\omega$ is $B_1$-$B_2$-flat and $(A_1,B_1)$ is fake-flat.

\item
$\omega'$ is $B_1'$-$B_2'$-adjusted if $\omega$ is $B_1$-$B_2$-adjusted.

\end{itemize}
\end{lemma}

It is straightforward to check that the morphism $\mu$ of \cref{monoidal-covariance} is connection-preserving for the extended connections.  
Hence, we conclude the following result.

\begin{proposition}
\label{covariance-bundles-with-connections}
Suppose $f: \Gamma \to \Gamma'$ is a Lie 2-group homomorphism that preserves adjustments. It induces a functor
\begin{equation*}
f_{*}:\buncon{\Gamma\!,\kappa} X_{(A_1,B_1),(A_2,B_2)} \to \buncon{{\Gamma',\kappa'}}X_{f_{*}(A_1,B_1),f_{*}(A_2,B_2)}
\end{equation*}
which is monoidal with respect to the partial tensor product.
\end{proposition}

We note that the inclusion functor $i: BA \to \Gamma$ preserves adjustments (the unique adjustment $\kappa=0$ on $BA$ and an arbitrary adjustment $\kappa$ on $\Gamma$). In the case of extension along $i: BA \to \Gamma$, the extended connection simplifies to
\begin{equation*}
\omega' := p^{*}\omega+h^{*}\theta \text{,}
\end{equation*}
where $\omega$ is a $0$-$0$-connection on a principal $BA$-bundle. If $B,B'\in \Omega^2(X,\mathfrak{a})$ are 2-forms such that $\mathrm{curv}(\omega)=B-B'$ (see \cref{BA-bundles-with-connection}), then  $\omega'$ is $i_{*}B$-$i_{*}B'$-flat.
 We remark that all connections in $\buncon\Gamma X_{(0,i_{*}B),(0,i_{*}B')}$ are automatically $i_{*}B$-$i_{*}B'$-flat, since they are -- by definition of the category -- $i_{*}B$-$i_{*}B'$-adjusted, and  $\mathrm{fcurv}(0,i_{*}B)=0$. We have the following generalization of \cref{reduction-to-abelian}.

\begin{proposition}
\label{bundle-morphisms-under-i}
For any $B,B'\in \Omega^2(X,\mathfrak{a})$, the functor $i_{*}$ establishes an equivalence of categories
\begin{equation*}
\buncon{A}X^{B-B'} \cong \buncon\Gamma X_{(0,i_{*}B),(0,i_{*}B')}^1\text{.}
\end{equation*}
Moreover, it is monoidal with respect to the partial tensor product. 
\end{proposition}

\begin{proof}
Suppose $\omega'$ is a $i_{*}B$-$i_{*}B'$-adjusted $0$-$0$-connection    on a principal $\Gamma\!$-bundle $(P,\phi)$ with $p\circ \phi=1$. We consider the principal $A$-bundle $P_{red}$ and $\omega := \omega'|_{P_{red}}$.
The second condition in \cref{connection-on-Gamma-bundle} implies  $t_{*}(\omega) =  0$, so that $\omega$ can be regarded as being $\mathfrak{a}$-valued.
The second condition shows then that $\omega$ is an ordinary connection on $P_{red}$, and indeed, the isomorphism $i_{*}(P_{red})\to P:(p,h) \mapsto ph$ is connection-preserving.
This shows essential surjectivity. Faithfulness is as in \cref{reduction-to-abelian}, and it is full because the restriction of any connection-preserving bundle isomorphism $P \to P'$ to $P_{red} \to P_{red}'$ is connection-preserving. The property of being monoidal is just a special case of \cref{covariance-bundles-with-connections}.
\end{proof}

A limitation of the equivalence in \cref{bundle-morphisms-under-i} is that the image of this functor is not central, as it was without connections (\cref{reduction-to-abelian}): we cannot even tensor $i_{*}(Q)$ from both sides with a general principal $\Gamma\!$-bundle, unless it is equipped with a 0-0-connection. To address this issue, we now assume that $\Gamma$ is central and equipped with a splitting $u$ and an adjustment $\kappa$ adapted to $u$. 

\begin{definition}
\label{extension-and-adjusted-shift}
Let $Q$ be a principal $A$-bundle with connection, and let $\eta\in\Omega^1(X,\mathfrak{g})$. We denote by $i_{*}^{\eta}(Q)$ the principal $\Gamma\!$-bundle $i_{*}(Q)$ with its extended connection, subject to an adjusted shift along $\eta$; it is called the \emph{adjusted extension} of $Q$. 
\end{definition}

The connection on $i_{*}^{\eta}(Q)$
is by \cref{adjusted-shift-of-connection} an $\eta$-$\eta$-connection and by \cref{curvature-of-adjusted-shift}
$\tilde B$-$\tilde B'$-adjusted, where
\begin{align*}
\tilde B := i_{*}B+ \tfrac{1}{2}\kappa_{*}( \eta\wedge  \eta)
\quand
\tilde B' :=i_{*}B'+   \tfrac{1}{2}\kappa_{*}(\eta\wedge\eta)\text{.} 
\end{align*}
Observe that the fake-curvatures are
\begin{equation*}
\mathrm{fcurv}(\eta,\tilde B) = \mathrm{fcurv}(\eta,\tilde B')=  \mathrm{d}\eta+\tfrac{1}{2}qp([\eta\wedge\eta])\in \Omega^2(X,\mathfrak{g})\text{.}
\end{equation*}
We collect these results and further consequences in the following proposition.
\begin{proposition}
\label{reduction-with-shifted-connection}
Suppose $\Gamma$ is central and equipped with an adapted adjustment. Let $\eta\in\Omega^1(X,\mathfrak{g})$. Then, the assignment $Q \mapsto i_{*}^{\eta}(Q)$ defines an equivalence of categories,
\begin{equation*}
\buncon{A}X^{B-B'} \cong \buncon\Gamma X^{1}_{(\eta,\tilde B),(\eta,\tilde B')}\text{.}
\end{equation*}
Moreover, it has the following properties: 
\begin{enumerate}[(i)]

\item
\label{shifting-commutes-with-tensor-product}
It is monoidal: if $Q_1$ and $Q_2$ are principal $A$-bundles with connections, then the isomorphism $\mu$ from \cref{monoidal-covariance} yields a connection-preserving isomorphism
\begin{equation*}
i_{*}^{\eta}(Q' \otimes Q) \cong i_{*}^{\eta}(Q') \otimes i_{*}^{\eta}(Q)\text{.}
\end{equation*}

\item
\label{tensor-product-with-abelian-bundle}
Its image is central: if $Q$ is a principal $A$-bundle with connection, and $P$ is a principal $\Gamma\!$-bundle with $A$-$A'$-connection, then the braiding isomorphism $\beta$ from \cref{reduction-to-abelian} yields a connection-preserving isomorphism
\begin{equation*}
P \otimes i_{*}^{A}(Q) \cong i_{*}^{A'}(Q) \otimes P\text{.}
\end{equation*}
\end{enumerate}
\end{proposition} 

\begin{proof}
(i) follows directly from the fact that $i_{*}$ is monoidal (\cref{bundle-morphisms-under-i})  and of the relation between tensor product and adjusted shift (\cref{adjusted-shift-of-tensor-product-connection}). (ii) follows from a direct calculation. 
\end{proof}

\subsection{Local data}

\label{sec:twistedconnectionsanddeligne}

We perform the Hom-set closure described in \cref{sec:homsetclosure} to the separated presheaf of bicategories  $\conff \Gamma-$, $\conadj{\Gamma\!,\kappa}-$, and $\congen\Gamma-$  of \cref{rem:gammaconn:prepre}, and and demonstrate below in \cref{lem:twistedconnequiv} that this yields the correct local formalism for principal $\Gamma\!$-bundles with connections. We omit a separate discussion of $\conaa{\Gamma\!,u,\kappa}-$, as the distinction to  $\conadj{\Gamma\!,\kappa}-$  lies solely in the objects.

\noindent 
We detail the resulting bicategory $\prestack{\congen\Gamma-}(X)$:
\begin{itemize}

\item
The objects of $\prestack{\congen\Gamma-}(X)$ are, as before, $\Gamma\!$-connections on $X$.

\item 
A 1-morphism in  $\prestack{\congen\Gamma-}(X)$  from $(A,B)$ to $(A',B')$  is a quadruple $(\pi, \varphi,g,a)$  consisting of a surjective submersion $\pi:Y \to X$, 
 a gauge transformation $(g,\varphi): \pi^{*}(A,B) \to \pi^{*}(A',B')$ on $Y$,
and a gauge 2-transformation  $a:\pr_1^{*}(g,\varphi) \Rightarrow \pr_2^{*}(g,\varphi)$ on $Y^{[2]} := Y \times_X Y$,
satisfying a cocycle condition over $Y^{[3]}$:
\begin{equation*}
a(y',y'') \cdot a(y,y') = a(y,y'')\text{.}
\end{equation*}

\item
A 2-morphism in   $\prestack{\congen\Gamma-}(X)$  from $(\pi_1,\varphi_1,g_1,a_1)$ to $(\pi_2,\varphi_2,g_2,a_2)$ is a gauge 2-transformation $h: \pr_{Y_1}^{*}(g_1,\varphi_1) \Rightarrow \pr_{Y_2}^{*}(g_2,\varphi_2)$ over $Y_1 \times Y_2$
such that 
\begin{equation*}
a_{2}(y_2,y_2')\cdot h(y_1,y_2) = h(y_1',y_2') \cdot a_1(y_1,y_1')
\end{equation*}
holds for all $(y_1,y_1',y_2,y_2')\in Y_1^{[2]} \times_X Y_2^{[2]}$.
The vertical composition of 2-morphisms $h_{12}$ and $h_{23}$ is defined by considering the gauge 2-transformation
\begin{equation*}
\pr_{23}^{*}h_{23} \bullet \pr_1^{*}h_{12} :\pr_{Y_1}^{*}(g_1,\varphi_1) \Rightarrow \pr_{Y_3}^{*}(g_3,\varphi_3)
\end{equation*}

over $Y_1 \times_X Y_2 \times_X Y_3$, which then descends, by virtue of  \cref{rem:gammaconn:prepre}, along the projection $\pr_{13}:Y_1 \times_X Y_2 \times_X Y_3 \to Y_1 \times_X Y_3$ to a unique gauge 2-transformation over $Y_1 \times_X Y_3$.  
\item
The composition of 1-morphisms 
\begin{equation*}
(\pi,\varphi,g,a):(A,B) \to (A',B')
\quand
(\pi',\varphi',g',a'):(A',B') \to (A'',B'')
\end{equation*}
is defined by setting 
\begin{equation*}
\tilde Y := Y \times_X Y'
\quand
\tilde\varphi := (\alpha_{g' \circ \pr_2})_{*}(\pr_1^{*}\varphi) + \pr_2^{*}\varphi'
\end{equation*}
and defining the maps
\begin{equation*}
\tilde g(y,y') := g'(y') \cdot g(y)
\quand
\tilde a((y_1,y_1'),(y_2,y_2')) := a'(y_1',y_2')\cdot \alpha(g'(y_1'),a(y_1,y_2))\text{.}
\end{equation*}
The horizontal composition of 2-morphisms is defined in a straightforward way accordingly. 

\end{itemize}

Regarding the sub-presheaves $\conff \Gamma-$ and $\conadj{\Gamma\!,\kappa}-$, the following conditions must be imposed:
\begin{itemize}

\item 
For $\prestack{\conadj {\Gamma\!,\kappa} X}$,  the 1-morphisms $(\pi,\varphi,g,a)$ contain an \emph{adjusted} gauge transformation $(g,\varphi)$. 

\item
For $\prestack{\conff \Gamma X}$, we restrict to \emph{fake-flat} $\Gamma\!$-connections and \emph{flat} gauge transformations $(g,\varphi)$.
\end{itemize}

Next, we  establish a \quot{reconstruction} procedure for principal $\Gamma\!$-bundles with connection from local data. We consider a 1-morphism $(\pi,\varphi,g,a)$  in $\prestack{\congen \Gamma -}(X)$, between $\Gamma\!$-connections $(A,B)$ and $(A',B')$.
We construct a principal $H$-bundle in the standard way,  
\begin{equation*}
P := (Y \times H)/\sim
\quere
(y_1,h) \sim (y_2,a(y_1,y_2)h)\text{,}
\end{equation*}
equipped with the right action $(y,h) \cdot h' := (y,hh')$. The map
\begin{equation*}
\phi: Y \times H \to G: (y,h) \mapsto t(h)^{-1} g(y)
\end{equation*}
descends to a smooth map $\phi: P \to G$, rendering $(P,\phi)$ a principal $\Gamma\!$-bundle. 

We define the 1-form 
\begin{equation*}
\omega := \Ad_{\pr_H}^{-1}(\pr_Y^{*}\varphi) + (\tilde\alpha_{\pr_H})_{*}(\pr_X^{*}A')  + \pr_H^{*} \theta \in \Omega^1(Y \times H,\mathfrak{h})\text{,} 
\end{equation*}
where $\pr_H$, $\pr_Y$ and $\pr_X$ are the projections from $Y \times H$ to the indexed manifolds. 
It is straightforward to check that $\omega$ descends to a 1-form $\omega\in\Omega(P,\mathfrak{ h})$ that is an $A$-$A'$-connection on $P$. 
A direct, albeit lengthy, calculation demonstrates that
\begin{equation*}
\mathrm{curv}_{B,B'}(\omega)= \pr_Y^{*}\mathrm{curv}(g,\varphi)+ (\tilde\alpha_{\pr_H})_{*}(\Ad_{g \circ \pr_Y}(\mathrm{fcurv}(A,B)))
\text{.}
\end{equation*}

If the morphism is in $\prestack{\conadj{\Gamma\!,\kappa} -}(X)$ then the last equation simplifies to
\begin{equation*}
\mathrm{curv}_{B,B'}(\omega) =\kappa(\phi,\mathrm{fcurv}(A,B))\text{,}
\end{equation*} 
i.e., $\omega$ is $B$-$B'$-adjusted.
If the morphism is in $\prestack{\conff\Gamma -}(X)$ we obtain immediately  $\mathrm{curv}_{B,B'}(\omega)=0$, i.e., $\omega$ is $B$-$B'$-flat.

We now consider a 2-morphism $h$  between  two 1-morphisms $(\pi_1,\varphi_1,g_1,a_1)$ and $(\pi_2,\varphi_2,g_2,a_2)$ in $\prestack{\congen\Gamma-}(X)$. We denote by $(P_1,\phi_1,\omega_1)$ and $(P_2,\phi_2,\omega_2)$ the corresponding reconstructed principal $\Gamma\!$-bundles with their $A$-$A'$-connections. We consider the map
\begin{equation*}
f': Y_1 \times H \to P_2 : (y_1,h') \mapsto (y_2,h(y_1,y_2) \cdot h')\text{,}
\end{equation*} 
where $y_2\in Y_2$ is any point satisfying $\pi_2(y_2)=\pi_1(y_1)$; the definition of $f'$ is independent of that choice. Furthermore, $f'$ descends to $P_1$ to a well-defined  bundle morphism $f:P_1\to P_2$, which  moreover preserves the $A$-$A'$-connections. 
It is straightforward to show that the vertical composition of 2-morphisms corresponds to the composition of bundle morphisms. Thus, we have defined a \quot{reconstruction} functor
\begin{align}
\label{reconstruction-functor}
\hom_{\prestack{\congen \Gamma -}(X)}((A,B),(A',B')) &\to \buncon {\Gamma} X_{A,A'}
\end{align}
that restricts to functors
\begin{align}
\label{reconstruction-functor-adj}
\hom_{\prestack{\conadj {\Gamma\!,\kappa} -}(X)}((A,B),(A',B')) &\to \buncon {\Gamma\!,\kappa} X_{(A,B),(A',B')}
\\
\label{reconstruction-functor-ff}
\hom_{\prestack{\conff \Gamma -}(X)}((A,B),(A',B')) &\to \bunconflat \Gamma X_{(A,B),(A',B')}\text{.}
\end{align}

\begin{proposition}
\label{lem:twistedconnequiv}
The reconstruction procedure possesses the following properties:
\begin{enumerate}[(a)]
\item 
All functors defined in \cref{reconstruction-functor,reconstruction-functor-adj,reconstruction-functor-ff} map the composition of 1-morphisms to the partial tensor product of principal $\Gamma\!$-bundles.

\item
For any pair of $\Gamma\!$-connections $(A,B)$ and $(A',B')$ on $X$, the functor \cref{reconstruction-functor-adj} is an equivalence of categories. 

\item
For any pair of fake-flat $\Gamma\!$-connections $(A,B)$ and $(A',B')$ on $X$, the functor \cref{reconstruction-functor-ff} is an equivalence of categories. 
\end{enumerate}
\end{proposition}

\begin{proof}
The compatibility between composition and tensor product is straightforward to check in the general case of \cref{reconstruction-functor}. 
We  show essential surjectivity of the functor \cref{reconstruction-functor-adj}. Suppose we have a principal $\Gamma\!$-bundle $(P,\phi)$,  equipped with a $A$-$A'$-connection $\omega$. We extract local data with respect to the bundle projection $\pi:P \to X$, by setting $g:=\phi:P \to G$, $\varphi := \omega$, and $a :P\times_X P \to H$ defined by $a(ph,p):=h$. 
We have
\begin{equation*}
\mathrm{curv}(g,\varphi)=\mathrm{curv}_{B,B'}(\omega)=\kappa(\phi,\mathrm{fcurv}(A,B))=\kappa(g,\mathrm{fcurv}(A,B))
\end{equation*}
which implies that the cocycle $(\pi,g,\varphi,a)$ is a 1-morphism in $\prestack{\conadj \Gamma -}(X)$.  
Then, let $(\tilde P,\tilde\phi,\tilde\omega)$ be principal $\Gamma\!$-bundle reconstructed from  $(\pi,g,\varphi,a)$. A morphism $f: \tilde P \to P$ is defined by $f(p,h):= ph$.
Moreover, we get $f^{*}\omega=\tilde\omega$.
This shows that   $(\pi,g,\varphi,a)$ is an essential preimage of $(P,\phi)$ under the functor  \cref{reconstruction-functor-adj}.
In the flat case, we have $\mathrm{curv}(g,\varphi)=\mathrm{curv}_{B,B'}(\omega)=0$, which implies that the cocycle is in $\prestack{\conff \Gamma -}(X)$ and hence also is an essential preimage of $(P,\phi)$ under the functor  \cref{reconstruction-functor-ff}.  

Next, we prove that both functors are fully faithful; the proofs are identical. We assume that $(\pi_1,\varphi_1,g_1,a_1)$ and $(\pi_2,\varphi_2,g_2,a_2)$ are objects, and we let $(P_1,\phi_1,\omega_1)$ and $(P_2,\phi_2,\omega_2)$ be the corresponding reconstructed principal $\Gamma\!$-bundles with connections. Suppose $\varphi: P_1 \to P_2$ is a morphism with $\varphi^{*}\omega_2=\omega_1$. We define a 2-morphism $h:(\pi_1,\varphi_1,g_1,a_1) \to (\pi_2,\varphi_2,g_2,a_2)$ by letting $h(y_1,y_2)\in H$ be the unique element such that 
\begin{equation*}
\varphi(y_1,1) = (y_2,1)\cdot h(y_1,y_2)\text{.}
\end{equation*}
It is straightforward to show that this is a morphism in  $\hom_{\prestack{\conadj\Gamma-}(X)}((A,B),(A',B'))$, 
and that the assignment $\varphi \mapsto h$ is a two-sided inverse to reconstruction. 
\end{proof}

\begin{remark}
We have inclusions
\begin{equation*}
{\conadj \Gamma -}(X) \subset \,\prestack{\conadj \Gamma -}(X)
\quand
{\conff \Gamma -}(X) \subset \,\prestack{\conff \Gamma -}(X)
\end{equation*}
induced by using identity surjective submersions, $\pi=\id_X$. Under reconstruction, these inclusions correspond to \emph{trivial} principal $\Gamma\!$-bundles: a morphism $(g,\varphi): (A,B) \to (A',B')$ in ${\conadj \Gamma -}(X)$ is mapped to $\trivlin_{g}^{\varphi}$, which is indeed $B$-$B'$-adjusted, and is $B$-$B'$-flat if and only if $(g,\varphi)$ is flat and $(A,B)$ is fake-flat (see \cref{twisted-curvature-on-trivial-bundle}). 
\end{remark}

\begin{corollary}
Let $(A,B)$ and $(A',B')$ be $\Gamma\!$-connections on $X$.
The presheaves of categories over $X$,
\begin{equation*}
\buncon {\Gamma\!,\kappa} {..}_{(A,B),(A',B')}
\quand
\bunconflat \Gamma {..}_{(A,B),(A',B')},
\end{equation*}
are stacks on the site $\man/X$ of manifolds over $X$.
\end{corollary}

\begin{proof}
The equivalences established in \cref{lem:twistedconnequiv} transfer this result from the abstract properties of the plus construction: $\prestack{\conadj {\Gamma\!,\kappa}-}$ and $\prestack{\conff {\Gamma} -}$ are pre-2-stacks by \cref{lem:pre2stackification}, and \cref{pre-2-stack} demonstrates that the Hom-presheaves of pre-2-stacks are stacks.
\end{proof}

\setsecnumdepth{2}

\section{Bundle gerbes with connections}

\label{sec:bundlegerbes}

In this section, we address the central topic of this article: the definition of non-abelian bundle gerbes with adjusted connections. $\Gamma\!$-bundle gerbes with connection were introduced by Aschieri, Cantini, and Jur?o \cite{aschieri} for the case of $\Gamma=\AUT(H)$. In \cite{Nikolaus}, we defined $\Gamma\!$-bundle gerbes for arbitrary Lie 2-groups as the result of the sheafification of the pre-2-stack $B\bun\Gamma-$. Here, we extend this framework to include connections.

We discuss bundle gerbes with connections exclusively within the adjusted setting, as this provides the most general working framework. Consequently, throughout this section, $\Gamma$ denotes a central Lie 2-group equipped with an adjustment $\kappa$, and $M$ is a smooth manifold.

\subsection{Main definition}

We reformulate the constructions from \cref{section-2-group-bundles} as follows. We consider a bicategory with:
\begin{enumerate}[(a)]

\item 
objects: all $\Gamma\!$-connections $(A,B)$ on $M$.

\item
1-morphisms from $(A,B)$ to $(A',B')$: principal $\Gamma\!$-bundles over $M$ equipped with a $B$-$B'$-adjusted $A$-$A'$-connection. Composition is given by the partially defined tensor product.

\item
2-morphisms: connection-preserving bundle morphisms. Vertical composition corresponds to the composition of bundle morphisms, while horizontal composition arises from the functoriality of the partially defined tensor product. 

\end{enumerate}
We call this bicategory the bicategory of \emph{trivial $\Gamma\!$-bundle gerbes with connections}, and denote it by $\trivgrbcon {\Gamma\!,\kappa} M$. Under the pullback along smooth maps between manifolds, $\trivgrbcon {\Gamma\!,\kappa} -$ becomes a presheaf of bicategories.   

\begin{definition}
\label{bundle-gerbes}
The presheaf of $\Gamma\!$-bundle gerbes with connection is defined as an application of the plus construction,
\begin{equation*}
\grbcon {\Gamma\!,\kappa} - := (\trivgrbcon {\Gamma\!,\kappa} -)^{+}\text{.}
\end{equation*}
\end{definition}

\begin{theorem}
\label{th:equivcongrb}
The presheaf $\grbcon{\Gamma\!,\kappa}-$ is a sheaf of bicategories. 
\end{theorem}

\begin{proof}
\Cref{lem:twistedconnequiv} provides an equivalence
\begin{equation*}
\trivgrbcon {\Gamma\!,\kappa} - \cong \prestack{\conadj{\Gamma\!,\kappa} -}
\end{equation*}
and the right-hand side is a pre-2-stack by \cref{lem:pre2stackification}. \Cref{th:2stackification} then establishes the claim.
\end{proof}

\begin{theorem}
\label{classification}
$\Gamma\!$-bundle gerbes with connections are classified up to isomorphism by the adjusted non-abelian differential cohomology:
\begin{align*}
\hc 0 \grbcon {\Gamma\!,\kappa} M  &\cong \hat \h^1(M,(\Gamma\!,\kappa))^{\adj}  
\end{align*}
\end{theorem}

\begin{proof}
There exist canonical isomorphisms
\begin{align*}
\grbcon {\Gamma\!,\kappa} M =(\trivgrbcon {\Gamma\!,\kappa} -)^{+}(M)\cong \overline{\conadj {\Gamma\!,\kappa} -}^{+}(M)\cong {(\conadj {\Gamma\!,\kappa} -)}^{+}(M)
\end{align*}
of bicategories; the last equivalence follows from \cref{sheafification-shortcut}.
\end{proof}

In the following, we elaborate on the abstract definition \cref{bundle-gerbes} in full detail. Ignoring the differential form data, \cref{def:grb} coincides with the non-abelian bundle gerbes defined in \cite{Nikolaus}. 

\begin{definition} \label{def:grb}
A \emph{$\Gamma\!$-bundle gerbe with connection over $M$} is a tuple 
\begin{equation*}
\mathcal{G}=(Y,\pi,(A,B),P,\mu)
\end{equation*}
consisting of a smooth manifold $Y$, of a surjective submersion $\pi\maps Y \to M$, of a $\Gamma\!$-connection $(A,B)$ on $Y$, with a 1-form $A\in\Omega^1(Y,\mathfrak{g})$ and a 2-form $B\in\Omega^2(Y,\mathfrak{h})$, of a principal $\Gamma\!$-bundle $P$ over $Y^{[2]}$ with  $\pr_1^{*}B$-$\pr_2^{*}B$-adjusted $\pr_1^{*}A$-$\pr_2^{*}A$-connection,  and of a connection-preserving morphism
    \begin{equation*}
    \mu: \pr_{23}^{*}P \otimes \pr_{12}^{*}P \to \pr_{13}^{*}P
    \end{equation*}
of $\Gamma\!$-bundles over $Y^{[3]}$ that is associative in the sense that, in the fibres over  each point $(y_1,y_2,y_3)\in Y^{[3]}$, the following diagram is commutative:
\begin{equation*}
\alxydim{@C=7em}{P_{y_3,y_4} \otimes P_{y_2,y_3} \otimes P_{y_1,y_2} \ar[d]_{\mu_{y_2,y_3,y_4} \otimes \id} \ar[r]^-{\id \otimes \mu_{y_1,y_2,y_3}} & P_{y_3,y_4} \otimes P_{y_1,y_3} \ar[d]^{\mu_{y_1,y_3,y_4}} \\ P_{y_2,y_4} \otimes P_{y_1,y_2} \ar[r]_-{\mu_{y_1,y_2,y_4}} & P_{y_1,y_4}}
\end{equation*}
\end{definition}

\begin{remark}
\label{connections-on-bundle-gerbes}
\begin{enumerate}[(i)]

\item   
\label{connections-on-bundle-gerbes:1}
A $\Gamma\!$-bundle gerbe $\mathcal{G}$ with connection is called \emph{adapted} to a splitting $u$ if the $\Gamma\!$-connection $(A,B)$ is adapted to $u$, i.e., $u(\mathrm{fcurv}(A,B))=0$. Adapted bundle gerbes form a full sub-bicategory 
\begin{equation*}
\grbconaa {\Gamma\!,u,\kappa} M\subset \grbcon{\Gamma\!,\kappa}M\text{.}
\end{equation*}
Recall that, in the present adapted setting, the only remaining information from the fake-curvature is the differential form 
\begin{equation*}
F_A = p_{*}(\mathrm{fcurv}(A,B)) =p_{*}(\mathrm{d}A+\tfrac{1}{2}[A \wedge A])\in \Omega^2(Y,\mathfrak{f})\text{.}
\end{equation*}
In general, $F_A$ does not descend to $M$ but transforms on $Y^{[2]}$ according to $\pr_1^{*}F_A = \mathrm{Ad}_{p_{*}P}(\pr_2^{*}F_A)$. 

\item 
\label{connections-on-bundle-gerbes:2}
A $\Gamma\!$-bundle gerbe $\mathcal{G}$ with connection is called \emph{fake-flat} if $(A,B)$ is a fake-flat $\Gamma\!$-connection. Note that, due to the adjustedness of the connection $\omega$ on $P$, this implies
\begin{equation*}
\mathrm{curv}_{\pr_1^{*}B,\pr_2^{*}B}(\omega) = \kappa(\phi,\mathrm{fcurv}(\pr_1^{*}A,\pr_2^{*}A))=0\text{,}
\end{equation*}
meaning that the connection $\omega$ on $P$ is $\pr_1^{*}B$-$\pr_2^{*}B$-flat. Fake-flat bundle gerbes form a full sub-bicategory 
\begin{equation*}
\grbconff\Gamma M \subset \grbconaa{\Gamma\!,u,\kappa}M\subset \grbcon{\Gamma\!,\kappa} M\text{.}
\end{equation*}
   
\item
\label{connections-on-bundle-gerbes:3}
The \emph{curvature} of $\mathcal{G}$ is the unique 3-form $\mathrm{curv}_{\kappa}(\mathcal{G})\in \Omega^3(M,\mathfrak{h})$ satisfying 
\begin{equation*}
\pi^{*}\mathrm{curv}_{\kappa}(\mathcal{G}) = \mathrm{d}B+\kappa_{*}(A\wedge \mathrm{d}A )+\tfrac{1}{2}\kappa_{*}(A\wedge  [A \wedge A])
\end{equation*}
on $Y$. If $\mathcal{G}$ is adapted or fake-flat, this is an $\mathfrak{a}$-valued 3-form. We say that $\mathcal{G}$ is flat if it is fake-flat and $\mathrm{curv}_{\kappa}(\mathcal{G})=0$. 

\end{enumerate}
\end{remark}

\begin{remark}
\label{stackification-other-gerbes}
In analogy with \cref{th:equivcongrb,classification}, adapted and fake-flat bundle gerbes form sheaves of bicategories,
\begin{equation*}
\grbconaa{\Gamma\!,u,\kappa}{...} = (\prestack{\conaa{\Gamma\!,u,\kappa} -})^{+} 
\quand
\grbconff{\Gamma}{...} = (\prestack{\conff{\Gamma} -})^{+} \text{,}
\end{equation*}
and they are classified by the corresponding versions of non-abelian differential cohomology:
\begin{align*}
\hc 0 \grbconff \Gamma M  &\cong \hat \h^1(M,\Gamma)^{\ff}  
\\
\hc 0 \grbconaa {\Gamma\!,u,\kappa} M  &\cong \hat \h^1(M,(\Gamma\!,u,\kappa))^{\aptadj}  
\end{align*}
\end{remark}

\begin{example}
\label{trivial-gamma-bundle-gerbe}
For a $\Gamma\!$-connection $(A,B)$ on $M$, we define a $\Gamma\!$-bundle gerbe  
\begin{equation*}
\mathcal{I}_{A,B}:=(M,\id_M,(A,B),\trivlin_1^0,\gamma_{1,1})
\end{equation*}
consisting of the identity surjective submersion $\id_M:M \to M$, the given $\Gamma\!$-connection $(A,B)$, the trivial principal $\Gamma\!$-bundle $\trivlin_1^0$ over $M \times_M M = M$ with connection $\omega_0$ induced by the zero 1-form  (see \cref{connection-on-trivial-bundle}),
and the  bundle morphism $\gamma_{1,1}: \trivlin_1^{0} \otimes \trivlin_1^{0} \to \trivlin_1^0$ of \cref{tensor-product-of-trivial-bundles}. According to \cref{twisted-curvature-on-trivial-bundle}, the $B$-$B$-curvature of $\omega_0$ is, at a point $(x,h)\in \trivlin_1$, 
\begin{multline*}
\mathrm{curv}_{B,B}(\omega_{0})_{x,h}= (\tilde\alpha_{h})_{*}(\mathrm{fcurv}(A_x,B_x))\\=\kappa(t(h)^{-1},\mathrm{fcurv}(A_x,B_x))=\kappa(\phi_{\trivlin_1}(x,h),\mathrm{fcurv}(A_x,B_x))\text{.}
\end{multline*}
This confirms that $\omega_0$ is $B$-$B$-adjusted, and thus $\mathcal{I}_{A,B}$ is a $\Gamma\!$-bundle gerbe with connection. 
The construction $(A,B)\mapsto \mathcal{I}_{A,B}$ implements the canonical inclusion
\begin{equation*}
\trivgrbcon {\Gamma\!,\kappa} M \to \grbcon {\Gamma\!,\kappa}M
\end{equation*}
of a presheaf into its sheafification.
Since this inclusion is fully faithful, we obtain an equivalence
\begin{equation*}
\buncon \Gamma M_{(A,B),(A',B')} = \Hom_{\trivgrbcon {\Gamma\!,\kappa} M}((A,B),(A',B')) \cong \Hom_{\grbcon {\Gamma\!,\kappa} M}(\mathcal{I}_{A,B},\mathcal{I}_{A',B'})\text{.}
\end{equation*}
In other words, morphisms between the trivial $\Gamma\!$-bundle gerbes with connection $\mathcal{I}_{A,B}$ and $\mathcal{I}_{A',B'}$ are precisely principal $\Gamma\!$-bundles with $B$-$B'$-adjusted $A$-$A'$-connections. 
\end{example}

\begin{remark}
\label{re:cantriv}
Restricting the isomorphism $\mu$ of a $\Gamma\!$-bundle gerbe to elements of the form $(y,y,y)\in Y^{[3]}$ and $(y,y',y')\in Y^{[3]}$ yields a connection-preserving isomorphism $t: \Delta^{*}P \to \trivlin_1^0$ of principal $\Gamma\!$-bundles over $Y$, where $\Delta: Y \to Y^{[2]}$ is the diagonal map. We denote by $t_y\in P_{y,y}$ the element corresponding to $((y,y),1)\in \trivlin_1^0$ under the isomorphism $t$. The isomorphism $t$ is uniquely characterized by the property that the elements $t_y$ are \quot{neutral} with respect to the \quot{multiplication} $\mu$, i.e.,
\begin{equation*}
\mu(t_{y_2} \otimes p)=p
\quand
\mu(p \otimes t_{y_1})=p
\end{equation*}
for all $p\in P_{y_1,y_2}$.
Furthermore, $s^{*}P$ is canonically isomorphic to the dual bundle $P^{*}$ as principal $\Gamma\!$-bundles with connection, where $s:Y^{[2]} \to Y^{[2]}$ swaps the two factors.   
\end{remark}

\begin{remark}
\label{re:recongrb}
We explicitly describe the reconstruction 2-functor
\begin{equation*}
(\conadj{\Gamma\!,\kappa}-)^{+} \to \grbcon{\Gamma\!,\kappa}-
\end{equation*}
underlying the classification result of \cref{classification}.
An object in $(\conadj{\Gamma\!,\kappa}-)^{+}(M)$ is a quadruple $(\pi,(A,B),(g,\varphi),a)$ consisting of a surjective submersion $\pi:Y \to M$, a $\Gamma\!$-connection $(A,B)$ on $Y$, a gauge transformation $(g,\varphi): \pr_1^{*}(A,B) \to \pr_2^{*}(A,B)$ over $Y^{[2]}$, and a gauge 2-transformation $a: \pr_{23}^{*}(g,\varphi)\circ \pr_{12}^{*}( g,\varphi)\Rightarrow \pr_{13}^{*}(g,\varphi)$ over $Y^{[3]}$ satisfying a cocycle condition over $Y^{[4]}$. 
The corresponding bundle gerbe comprises the surjective submersion $\pi$, the $\Gamma\!$-connection $(A,B)$ over $Y$, the trivial $\Gamma\!$-bundle $\trivlin_{g}^{\varphi}$ over $Y^{[2]}$ equipped with a $\pr_1^{*}A$-$\pr_2^{*}A$-connection (\cref{connection-on-trivial-bundle}) that is $\pr_1^{*}B$-$\pr_2^{*}B$-adjusted (\cref{twisted-curvature-on-trivial-bundle}), and the morphism between trivial $\Gamma\!$-bundles over $Y^{[3]}$ given by the smooth map $a$.
\end{remark}

\begin{remark}
If $\mathcal{G}=(Y,\pi,(A,B),P,\mu)$ is a bundle gerbe with connection over $M$, and $\lambda\in \Omega^2(M,\mathfrak{h})$, then by \cref{shift-of-curving} the tuple $\mathcal{G}^{\lambda}=(Y,\pi,(A,B+\lambda),P,\mu)$ defines another bundle gerbe with connection. Fake-flatness and adaptedness are preserved by such shifts if $\lambda \in \Omega^2(M,\mathfrak{a})$. 
\end{remark}

\begin{remark}
\label{rem:aschieri}
The definition of the underlying bundle gerbe coincides (for $\Gamma=\AUT(H)$) with  \cite[Def. 11]{aschieri} under the equivalence of \cref{prop:equivbibun}. The fact that this equivalence is monoidal for the opposite monoidal structure is reflected by the fact that in \cite[Def. 11]{aschieri}
the isomorphism $\mu$ goes $\pr_{12}^{*}P \otimes \pr_{23}^{*}P \to \pr_{13}^{*}P$. The connection data considered in \cite[Def. 22]{aschieri} consists in the first place of a bibundle connection $(\tilde A,\tilde \omega)$ on $P$ such that $\pr_{12}^{*}\tilde A^{r}=\pr_{23}^{*}\tilde A$, and such that $\mu$ is connection-preserving. 
Secondly, a \quot{curving} form $\tilde B\in \Omega^2(Y,\mathfrak{h})$ is considered, together with the 2-form (see \cite[Eq. 147]{aschieri})
\begin{equation}
\label{eq:defdelta}
\delta :=-(\alpha_{\varphi^{-1}})_{*}(\pr_1^{*}\tilde B)  + \pr_2^{*}\tilde B+ \tilde\kappa_{\tilde A,\tilde\omega} \in \Omega^2(P,\mathfrak{h})\text{.}
\end{equation}
The relationship to our definition is as follows. Suppose an $\AUT(H)$-bundle gerbe with connection as in \cref{def:grb} is given. Then, set $\tilde B:= B$, and let the connection $(\tilde A,\tilde\omega)$ on $P_{\phi}$ correspond to the given connection $\omega$ on $P$, i.e., $\tilde A := \pr_1^{*}A$ and $\tilde\omega := -(\alpha_{\phi^{-1}})_{*}(\omega)$. This provides the first part. There is no useful expression for $\delta$ in this context; however, if one were to replace \cref{eq:defdelta} by
\begin{equation*}
\delta' :=\pr_1^{*}\tilde B  - (\alpha_{\varphi})_{*}(\pr_2^{*}\tilde B)+ \tilde\kappa_{\tilde A,\tilde\omega}\text{,}
\end{equation*}
then we would obtain $\delta' = -(\alpha_{\phi^{-1}})_{*}\mathrm{curv}_{\pr_1^{*}B,\pr_2^{*}B}(\omega)$; thus, $\delta$ essentially represents the twisted curvature of $\omega$. As noted in \cite{aschieri}, the condition $\delta=0$ cannot generally be achieved, and no other condition is formulated there. Our formalism imposes such a condition via adjustments.   
\end{remark}

\subsection{Bicategorical structure}

We elaborate on the consequences of \cref{bundle-gerbes} for the 1-morphisms and 2-morphisms in the bicategory of $\Gamma\!$-bundle gerbes with connections.

\begin{definition}
\label{definition-of-1-morphism}
Let $\mathcal{G}_1=(\pi_1,(A_1, B_1),P_1,\mu_1)$ and $\mathcal{G}_2=(\pi_2,(A_2, B_2),P_2,\mu_2)$ be $\Gamma\!$-bundle gerbes with connections over $M$. A \emph{1-morphism}
$\mathcal{P}:\mathcal{G}_1 \to \mathcal{G}_2$
is a tuple $\mathcal{P}=(Z,\zeta,Q,\nu)$ consisting of:
\begin{itemize}

\item 
a smooth manifold $Z$ and a surjective submersion $\zeta: Z \to Y_1 \times_M Y_2$,

\item
a principal $\Gamma\!$-bundle $Q$ over $Z$ with  $B_1$-$B_2$-adjusted $A_1$-$A_2$-connection (here, all forms are understood to be pulled back along $Z \stackrel\zeta\to Y_1 \times_M Y_2 \stackrel{\pr_i}\to Y_i$),
and
\item
a connection-preserving bundle isomorphism
\begin{equation*}
\nu:  \pr_2^{*}Q \otimes (\zeta ^2)^{*}\pr_{13}^{*}P_1 \to (\zeta ^2)^{*}\pr_{24}^{*}P_2 \otimes \pr_1^{*}Q
\end{equation*}
over $Z^{[2]} = Z \times_M Z$. Here, $\pr_i: Z^{[2]} \to Z$ denote the projections, while $\zeta^2$ is the map $Z^{[2]} \to Y_1 \times_M Y_2 \times_M Y_1 \times_M Y_2$. The projections $\pr_{13}$ and $\pr_{24}$ map to the respective indices in $Y_1^{[2]}$ and $Y_2^{[2]}$, respectively. 

\end{itemize}
This structure is subject to the condition that $\nu$ acts as a \quot{homomorphism} with respect to $\mu_1$ and $\mu_2$. Specifically, in the fiber over any point $(z,z',z'')\in Z^{[3]}$, the following diagram must commute:

\begin{equation}\label{diag:1Morphisms}
\begin{gathered}
\xymatrix@C=8em{
Q_{z^{\prime\prime}} \otimes (P_1)_{y_1',y_1''} \otimes (P_1)_{y_1,y_1'}  \ar[r]^-{ \id \otimes (\mu_1)_{y_1,y_1',y_1''}} \ar[d]_{\nu_{z^\prime,z^{\prime\prime}}\otimes \id} &
Q_{z^{\prime\prime}} \otimes (P_1)_{y_1,y_1''}   \ar[dd]^{\nu_{z,z''}} \\ 
 (P_2)_{y_2',y_2''} \otimes Q_{z'} \otimes (P_1)_{y_2,y_2'} \ar[d]_{\id \otimes \nu_{z,z'}} & \\ 
  (P_2)_{y_2',y_2''} \otimes (P_2)_{y_2,y_2'}  \otimes Q_{z} \ar[r]_-{(\mu_2)_{y_2,y_2',y_2''} \otimes \id} & 
 (P_2)_{y_2,y_2''} \otimes Q_{z} 
 }
\end{gathered}
\end{equation}
where $\zeta(z)=: (y_1,y_2)$, $\zeta(z')=:(y_1',y_2')$, and $\zeta(z'')=:(y_1'',y_2'')$.
 \end{definition}

\begin{remark}
Disregarding the connection data, \Cref{definition-of-1-morphism} is nearly identical to the definition of morphisms between $\Gamma\!$-bundle gerbes given in \cite{Nikolaus}. However, the treatment in \cite{Nikolaus} suffers from minor inaccuracies discussed in \cite{nikolaus2}, see \cref{incompleteness-ns}.
\end{remark}

\begin{remark}
\label{fake-curvature-under-isomorphisms}
Suppose $\mathcal{P}:\mathcal{G}_1 \to \mathcal{G}_2$ is a 1-morphism, $\mathcal{P}=(Z,\zeta,Q,\nu)$, let $\phi:Q \to G$ denote the anchor of $Q$, and let $\omega$ be the connection on $Q$. 
\begin{enumerate}[(i)]

\item
\Cref{eq:adjcurv} implies the following identity:
\begin{equation*}
t_{*}(\kappa(\phi,\zeta^{*}\pr_{1}^{*}\mathrm{fcurv}(A_1,B_1)))= \Ad_{\phi}(\zeta^{*}\pr_{1}^{*}\mathrm{fcurv}(A_1,B_1)) - \zeta^{*}\pr_{2}^{*}\mathrm{fcurv}(A_2,B_2)\text{.}
\end{equation*}
We conclude the following:
\begin{itemize}[leftmargin=2em,labelsep=0.5em]

\item 
If $\mathcal{G}_1$ is fake-flat, then $\mathcal{G}_2$ is also fake-flat (and the converse follows from the invertibility of 1-morphisms, see \cref{LemmaInvertibilityA}). 

\item
If $\kappa$ is adapted to a splitting $u$, then \cref{eq:adjcurv} implies
\begin{equation*}
\zeta^{*}\pr_{1}^{*}u(\mathrm{fcurv}(A_1,B_1)) = \zeta^{*}\pr_{2}^{*} u(\mathrm{fcurv}(A_2,B_2))\text{.}
\end{equation*}
Thus, if one of the connections is adapted to $u$, the other is adapted as well. In this case, the relation between the fake-curvatures simplifies to
\begin{equation*}
\Ad_{\phi}(\zeta^{*}\pr_{1}^{*}\mathrm{fcurv}(A_1,B_1))=\zeta^{*}\pr_{2}^{*}\mathrm{fcurv}(A_2,B_2)\text{.}
\end{equation*}
\end{itemize}
In summary, fake-flatness and adaptedness of bundle gerbe connections are preserved under 1-morphisms. 

\item 
If $\mathcal{G}_1$ is fake-flat (and thus $\mathcal{G}_2$ as well), we have
\begin{equation*}
\mathrm{curv}_{\zeta^{*}\pr_1^{*}B_1,\zeta^{*}\pr_2^{*}B_2}(\omega) =\kappa(\phi,\zeta^{*}\pr_{1}^{*}\mathrm{fcurv}(A_1,B_1))=0\text{,} 
\end{equation*}
meaning that $\omega$ is $\zeta^{*}\pr_1^{*}B_1$-$\zeta^{*}\pr_2^{*}B_2$-flat. This confirms that $\grbconff\Gamma M$, defined in \cref{connections-on-bundle-gerbes:2} as the \emph{full} subcategory of $\grbcon{\Gamma\!,\kappa} M$ over the fake-flat bundle gerbes, arises via sheafification from a completely fake-flat setting, 
\begin{equation*}
\grbconff{\Gamma}{M} = (\prestack{\conff{\Gamma} -})^{+}(M)\text{,}
\end{equation*}
as claimed in \cref{stackification-other-gerbes}.  

\end{enumerate}

\end{remark}

\begin{example}
\label{identity-1-morphism}
The \emph{identity 1-morphism} of a $\Gamma\!$-bundle gerbe $\mathcal{G}=(Y,\pi,(A,B),P,\mu)$ is $\id_{\mathcal{G}} :=(Y^{[2]},\id_{Y^{[2]}},P,\nu_{\mu})$ with 
\begin{equation*}
(\nu_{\mu})_{y_1,y_2,y_1',y_2'} := \mu^{-1}_{y_1,y_1',y_2'}\circ \mu_{y_1,y_2,y_2'}\text{,}
\end{equation*}
for a point $(y_1,y_2,y_1',y_2')\in Y^{[2]} \times_M Y^{[2]}$. 
\end{example}

The following proposition demonstrates that connections can be induced from a second $\Gamma\!$-bundle gerbe along a 1-morphism. 

\begin{proposition}
\label{connection-transport-along-1-morphism}
Suppose $\mathcal{P}: \mathcal{G}_1 \to \mathcal{G}_2$ is a 1-morphism of bundle gerbes without connection, $\mathcal{P}=(Z,\zeta,Q,\nu)$, and suppose $\mathcal{G}_1$ is equipped with a connection. Then, there exist connections on $\mathcal{G}_2$ and on $Q$ such that $\mathcal{P}$ becomes connection-preserving.  
\end{proposition}

\begin{proof}
By \cref{existence-of-adjusted-connection}, there exists a $\Gamma\!$-connection $(A',B')$ on $Z$ and a $\zeta^{*}\pr_1^{*}B_1$-$B'$-adjusted $\zeta^{*}\pr_1^{*}A_1$-$A'$-connection $\omega$ on $Q$.
We next show that there exists a unique connection $\omega_2'$ on the pullback bundle $(\zeta \times \zeta)^{*}\pr_{24}^{*}P_2$ over $Z^{[2]}$ such that $\nu$ is connection-preserving.

Let $(z,z')\in Z^{[2]}$ with $\zeta(z)=: (y_1,y_2)$ and $\zeta(z')=:(y_1',y_2')$. Let $q\in Q_z$, $q'\in Q_{z'}$, and $p_1\in P_1|_{y_1,y_1'}$. Since $\nu$ is an isomorphism, there exists a unique $p_2\in P_2|_{y_2,y_2'}$ such that $\nu(q'\otimes p_1)=p_2 \otimes q$. Assuming the existence of $\omega_2'$ such that $\nu$ preserves connections, the equation
\begin{equation*}
\omega_{q'} + (\alpha_{\phi(q')})_{*}(\omega_1|_{p_1}) =\omega_2'|_{z,z',p_2}+(\alpha_{\phi_2(p_2)})_{*}(\omega_q)\text{.} 
\end{equation*}
holds. Conversely, this equation can be used to define $\omega_2'$, and one can check that it is a  $\pr_1^{*}A'$-$\pr_2^{*}A'$-connection and that it is
$\pr_1^{*}B'$-$\pr_2^{*}B'$-adjusted.

We next verify that $\omega_2'$ descends along $(\zeta \times \zeta)^{*}\pr_{24}^{*}P_2 \to P_2$. To this end, we fix a point $p_2\in P_2$ in the fibre over a point $(y_2,y_2')\in Y_2^{[2]}$, and two preimages of this point. These may be given as $(z,z',p_2)$ and $(\tilde z,\tilde z',p_2)$, where  $(z,z',\tilde z,\tilde z')\in Z^{[4]}$ are points with $\zeta(z)=(y_1,y_2)$, $\zeta(z')=(y_1',y_2')$, $\zeta(\tilde z)=(\tilde y_1,y_2)$, and $\zeta(\tilde z')=(\tilde y_1',y_2')$.
We start with some preparation. To leverage the various bundle bundle morphisms, we choose $p_1\in P_1|_{y_1,y_1'}$, $p_1''\in P_1|_{y_1,\tilde y_1}$, as well as $q\in Q_z$, $\tilde q'\in Q_{\tilde z'}$. Then, there exist unique
\begin{enumerate}[1.]

\item 
\label{desc-eq-1}
$q'\in Q_{z'}$ such that $\nu_{z,z'}(q' \otimes p_1)=p_2 \otimes q$,
and\item
\label{desc-eq-2}
$\tilde p_1' \in P_1|_{y_1,\tilde y_1'}$  such that $\nu_{z,\tilde z'}(\tilde q' \otimes \tilde p_1')=p_2 \otimes q$.

\end{enumerate}
Given these points, there exist unique:
\begin{enumerate}[1.,start=3]

\item 
\label{desc-eq-4}
$p_1' \in P_1|_{y_1',\tilde y_1'}$ such that $\mu_1(p_1 '\otimes p_1)=\tilde p_1'$, and

\item
\label{desc-eq-5}
$\tilde p_1\in P_1|_{\tilde y_1,\tilde y_1'}$ such that $\mu_1(\tilde p_1 \otimes p_1'')=\tilde p_1'$.

\end{enumerate}
Lastly, there exists a unique:
\begin{enumerate}[1.,start=5]
\item 
\label{desc-eq-6}
$\tilde q\in Q_{\tilde z}$ such that $\nu_{\tilde z,\tilde z'}(\tilde q' \otimes \tilde p_1)=p_2 \otimes \tilde q$.  

\end{enumerate}
The commutativity of Diagram \cref{diag:1Morphisms} yields two further equalities:
\begin{enumerate}[1.,start=6]

\item 
\label{desc-eq-7}
$\nu_{z',\tilde z'}(\tilde q' \otimes p_1')= t_{y_2'} \otimes q'$, and

\item
\label{desc-eq-8}
$\nu_{z,\tilde z}(\tilde q \otimes p_1'')= t_{y_2} \otimes  q\text{.}$
\end{enumerate}
Here, $t_{y_2}$ and $t_{y_2'}$ are the canonical elements in the fiber over the diagonal (see \cref{re:cantriv}).
Since $\nu$ and $\mu_1$ preserve anchors and connections, each of the seven equalities above yields a corresponding equation for anchors and connections. For instance, \cref{desc-eq-1} gives
\begin{equation*}
\phi(q')\phi_1(p_1)=\phi_2(p_2)\phi(q) \quand \omega_{q'} + (\alpha_{\phi(q')})_{*}(\omega_1|_{p_1})= \omega_2'|_{z,z',p_2} +(\alpha_{\phi_2(p_2)})_{*}(\omega_q) \text{.}
\end{equation*} 
Throughout the following, $\phi$, $\phi_1$, and $\phi_2$ denote the anchors of $Q$, $P_1$, and $P_2$, respectively. 
Finally, since in \cref{diag:1Morphisms} all isomorphisms except  (the pullback to $Z^{[3]}$ of)  $\mu_2$ are known to be connection-preserving, it follows that also (the pullback to $Z^{[3]}$ of)  $\mu_2$ is connection-preserving. We evaluate that condition over three different points in $Z^{[3]}$:
\begin{enumerate}[1.,start=8]

\item 
\label{omega-2-1}
over $(z,z',\tilde z')$, we obtain $\omega_2'|_{z,\tilde z',p_2} = \omega_2'|_{z',\tilde z',t_{y_2'}} + \omega_2'|_{z,z',p_2}$,

\item
\label{omega-2-2}
over $(z,\tilde z,\tilde z')$, we obtain $\omega_2'|_{z,\tilde z',p_2} = \omega_2'|_{\tilde z,\tilde z',p_2} +(\alpha_{\phi_2(p_2)})_{*} (\omega_2'|_{z,\tilde z,t_{y_2}})$, and

\item
\label{omega-2-3}
over $(z, z,z)$, we obtain $\omega_2'|_{z,z,p_2} = \omega_2'|_{z,z,t_{y_2}} + \omega_2'|_{z,z,p_2}$.  

\end{enumerate}
We are now in a position to prove that $\omega_2'$ descends:
\begin{align*}
\omega'_2|_{z,z',p_2}
&\eqcref{desc-eq-1}\omega_{q'} + (\alpha_{\phi(q')})_{*}(\omega_1|_{p_1}) -(\alpha_{\phi_2(p_2)})_{*}(\omega_q)
\\&\eqcref{desc-eq-7}\omega_{\tilde q'} + (\alpha_{\phi(\tilde q')})_{*}(\omega_1|_{p_1'})+ (\alpha_{\phi(q')})_{*}(\omega_1|_{p_1})-\omega_2'|_{z',\tilde z',t_{y_2'}}  -(\alpha_{\phi_2(p_2)})_{*}(\omega_q)  
\\&\eqcref{desc-eq-7,omega-2-1}\omega_{\tilde q'} + (\alpha_{\phi(\tilde q')})_{*}(\omega_1|_{p_1'}+ (\alpha_{\phi(p_1')})_{*}(\omega_1|_{p_1}))-\omega_2'|_{z,\tilde z',p_2}  + \omega_2'|_{z,z',p_2}-(\alpha_{\phi_2(p_2)})_{*}(\omega_q) \\&\eqcref{desc-eq-4} \omega_{\tilde q'} + (\alpha_{\phi(\tilde q')})_{*}( \omega_1|_{\tilde p_1'}) -(\omega_2'|_{z,\tilde z',p_2} -\omega_2'|_{\tilde z,\tilde z',p_2})-(\alpha_{\phi_2(p_2)})_{*}(\omega_{q})
\\&\eqcref{omega-2-2,desc-eq-5} \omega_{\tilde q'} + (\alpha_{\phi(\tilde q')})_{*}(\omega_1|_{\tilde p_1}+(\alpha_{\phi_1(\tilde p_1)})_{*}(\omega_1|_{p_1''})) -(\alpha_{\phi_2(p_2)})_{*}(\omega_2'|_{z,\tilde z,t_{y_2}} )-(\alpha_{\phi_2(p_2)})_{*}(\omega_{q})
\\&\eqcref{desc-eq-6} \omega_{\tilde q'} + (\alpha_{\phi(\tilde q')})_{*}(\omega_1|_{\tilde p_1}) -(\alpha_{\phi_2(p_2)})_{*}(\omega_2'|_{z,\tilde z,t_{y_2}} -(\alpha_{\phi(\tilde q)})_{*}(\omega_1|_{p_1''})+\omega_{q})
\\&\eqcref{desc-eq-8} \omega_{\tilde q'} + (\alpha_{\phi(\tilde q')})_{*}(\omega_1|_{\tilde p_1}) -(\alpha_{\phi_2(p_2)})_{*}(\omega_{\tilde q})
\\&\eqcref{desc-eq-6} \omega'_2|_{\tilde z,\tilde z',p_2}
\end{align*}
This shows that $\omega_2'$ descends to a 1-form $\omega_2$ on $P_2$. Before verifying that it defines a connection, we observe that, since it is now proved that $\omega_2'$ does not depend on the $z$'s, \cref{omega-2-3} implies that $\omega_2|_{z,z',t_{y_2}}=0$.
We also note that the descent already implies that $\mu_2$ is connection-preserving. 

Next, we prove that the $\Gamma\!$-connection $(A',B')$ descends along $\pr_{2} \circ \zeta: Z \to Y_2$. 
For this purpose, consider $(z,z')\in Z^{[2]}$ with $\zeta(z)=:(y_1,y_2)$ and $\zeta(z')=:(y_1',y_2)$.
Choose $q\in Q_z$ and $q'\in Q_{z'}$, and let $p_1\in P_1|_{y_1,y_1'}$ be the unique point such that $\nu(q' \otimes p_1) = t_{y_2} \otimes q$. The fact that $\nu$ preserves anchors and connections yields
\begin{equation*}
\phi(q')\phi_1(p_1)=\phi(q) \quand \omega_{q'} + (\alpha_{\phi(q')})_{*}(\omega_1|_{y_1,y_1'}) =  \omega_q\text{.}
\end{equation*}
At our disposal we have further the second equation in \cref{connection-on-Gamma-bundle}  for the connection $\omega$ at $q\in Q$ and at $q'\in Q$,
and for the connection
$\omega_1$ at $p_1\in P_1$.
Using these it is straightforward to prove that $A'_{z}=A'_{z'}$, so that $A'$ descends to a 1-form $A\in \Omega^1(Y_2,\mathfrak{g})$. 

The proof that $B'$ descends follows from a similar calculation using that $\omega$ is adjusted, 
\begin{equation*}
B'_{z} = \kappa(\phi(q),\mathrm{fcurv}(A_{y_1},B_{y_1}))-\mathrm{d}\omega_q - \tfrac{1}{2}[\omega_q \wedge \omega_q] - \alpha_{*}(A'_z\wedge \omega_q)+(\alpha_{\phi(q)})_{*}(B_{y_1}) 
\end{equation*}
and similarly over $z'$. 
Consequently, the $\Gamma\!$-connection $(A',B')$ descends to a $\Gamma\!$-connection $(A_2,B_2)$ on $Y_2$. Now,  $\omega_2'$ was a  $\pr_1^{*}A'$-$\pr_2^{*}A'$-connection and 
$\pr_1^{*}B'$-$\pr_2^{*}B'$-adjusted, and since all forms are by now identified as being pulled back from $Y_1$, due to the injectivity of the pullback along a surjective submersion,
$\omega_2$ is a  $\pr_1^{*}A_2$-$\pr_2^{*}A_2$-connection and $\pr_1^{*}B_2$-$\pr_2^{*}B_2$-adjusted.  
\end{proof}

Consider three $\Gamma\!$-bundle gerbes with connections and two 1-morphisms 
\begin{equation*}
\xymatrix{\mathcal{G}_1 \ar[r]^{\mathcal{P}_{12}} & \mathcal{G}_2 \ar[r]^{\mathcal{P}_{23}} & \mathcal{G}_3\text{,}}
\end{equation*}
retaining the notation from the preceding definitions. The \emph{composition} $\mathcal{P}_{23} \circ \mathcal{P}_{12} = (Z,\zeta, Q,\nu)$ is defined as follows. We set $Z:= Z_{12} \times_{Y_2} Z_{23}$, and consider a point $z:=(z_{12},z_{23})\in Z$ with $\zeta_{12}(z_{12})=:(y_1,y_2)$ and $\zeta_{23}(z_{23})=:(y_2,y_3)$. The surjective submersion $\zeta:Z \to Y_1 \times_M Y_3$ is then given by $\zeta(z) := (y_1,y_3)$. The principal $\Gamma\!$-bundle $Q$ with connection over $Z$ is the tensor product
\begin{equation*}
Q:=\pr_{Z_{23}}^{*}Q_{23} \otimes \pr_{Z_{12}}^{*}Q_{12}\text{.}
\end{equation*}
Finally, the isomorphism $\nu$ over a point $((z_{12},z_{23}),(z_{12}',z_{23}'))\in Z^{[2]}$ is defined by
\begin{equation*}
\xymatrix@C=4em{
\hspace{-10em}Q_{z_{12}',z_{23}' } \otimes_{} P_1|_{y_1,y_1'}=(Q_{23})_{z'_{23}} \otimes_{} (Q_{12})_{z_{12}'} \otimes_{} P_1|_{y_1,y_1'}\ar[d]^-{\id\otimes( \nu_{12})_{z_{12},z_{12}'}} \\
(Q_{23})_{z_{23}'}\otimes_{} P_2|_{y_2,y_2'} \otimes_{} (Q_{12})_{z_{12}}  \ar[d]^-{(\nu_{23})_{z_{23},z_{23}'} \otimes \id}
\\
P_3 |_{y_3,y_3'}\otimes_{} (Q_{23})_{z_{23}} \otimes_{} (Q_{12})_{z_{12}}=P_3 |_{y_3,y_3'}\otimes_{} Q_{z_{12},z_{23}}\text{.}\hspace{-10em}  }
\end{equation*}

We consider two 1-morphisms between the same $\Gamma\!$-bundle gerbes with connections,
\begin{equation*}
\xymatrix{\mathcal{G}_1 \ar@/^1pc/[r]^{\mathcal{P}}\ar@/_1pc/[r]_{\mathcal{P}'} & \mathcal{G}_2\text{.}}
\end{equation*}
To simplify notation, we set $Y_{12} := Y_1 \times_M Y_2$. 
Then, a \emph{2-morphism} $\mathcal{P} \Rightarrow \mathcal{P}'$ is represented by a triple $(W,\rho,\varphi)$ consisting of:
\begin{itemize}

\item 
a smooth manifold $W$ and a surjective submersion $\rho: W \to Z \times_{Y_{12}} Z'$, and

\item
a connection-preserving isomorphism $\varphi:\rho^{*}\pr_Z^{*}Q \to \rho^{*}\pr_{Z'}^{*}Q'$ of $\Gamma\!$-bundles over $W$.
\end{itemize}
This structure must satisfy the condition that $\varphi$ commutes with the isomorphisms $\nu$ and $\nu'$, i.e., the diagram
\begin{equation*}
\xymatrix@C=4em{
Q_{\tilde z} \otimes_{} (P_1)_{y_1,\tilde y_1} \ar[d]_{\varphi_{\tilde w} \otimes \id} \ar[r]^-{\nu_{z,\tilde z}} & (P_2)_{y_2,\tilde y_2} \otimes Q_{z} \ar[d]^{\id \otimes \varphi_w}
 \\ 
 Q'_{\tilde z'} \otimes_{} (P_1)_{y_1',\tilde y_1'}  \ar[r]_-{\nu'_{z',\tilde z'}} &(P_2)_{y_2',\tilde y_2'} \otimes_{} Q'_{z'}
 }
\end{equation*}
commutes for all $(w,\tilde w) \in W \times_M W$, where $\rho(w)=(z,z')$ and $\rho(\tilde w)=(\tilde z, \tilde z')$.
Two triples $(W,\rho,\varphi)$ and $(W',\rho',\varphi')$ are identified if the pullbacks of $\varphi$ and $\varphi'$ coincide over $W \times_{Z \times_{Y_{12}} Z'} W'$. 
Since the choice of $\rho$ is unimportant in this sense, we usually denote the 2-morphism simply by $\varphi$.

We consider three 1-morphisms between the same $\Gamma\!$-bundle gerbes and two 2-morphisms:
\begin{equation*}
\xymatrix@C=5em{\mathcal{G}_1 \ar[r]|*+{\mathcal{P}'}="2" \ar@/^2.5pc/[r]^*+{\mathcal{P}}="1"\ar@/_2.5pc/[r]_*+{\mathcal{P}''}="3" & \mathcal{G}_2\text{.} \ar@{=>}"1";"2"|-{\varphi} \ar@{=>}"2";"3"|-{\varphi'}}
\end{equation*}
Suppose the 1-morphisms $\mathcal{P}$, $\mathcal{P}'$, and $\mathcal{P}''$ come with surjective submersions $\zeta:Z \to Y_{12}$, $\zeta': Z'\to Y_{12}$, and $\zeta'': Z'' \to Y_{12}$, respectively, and that the 2-morphisms $\varphi$ and $\varphi'$ come with surjective submersions $\rho: W \to Z \times_{Y_{12}} Z'$ and $\rho': W' \to Z' \times_{Y_{12}} Z''$, respectively. 
We consider $W \times_{Z'} W'$ equipped with the surjective submersion $(w,w') \mapsto (z,z'')$, where $(z,z'):=\rho(w)$ and $(z',z'') := \rho'(w')$. Then the vertical composition $\varphi' \bullet \varphi: \mathcal{P} \Rightarrow \mathcal{P}''$ is the isomorphism over $W \times_{Z'} W'$ defined fibrewise over $(w,w')$ by
\begin{equation*}
\xymatrix{Q_{z} \ar[r]^{\varphi_{w}} & Q'_{z'} \ar[r]^{\varphi'_{w'}} & Q''_{z''}}\text{.}
\end{equation*}

The identity 2-morphism $\id_{\mathcal{P}}$ of a 1-morphism $\mathcal{P}=(Z,\zeta,Q,\nu)$ is obtained by restricting the intertwiner $\nu$ to $W:=Z \times_{Y_{12}} Z \subset Z \times_M Z$. Over $(z_1,z_2)$ with $\zeta(z_1)=\zeta(z_2)=(y_1,y_2)$, this becomes an isomorphism 
\begin{equation*}
Q_{z_1}\otimes_{} (P_1)_{y_1,y_1} \to (P_2)_{y_2,y_2} \otimes_{} Q_{z_2} \text{.}
\end{equation*}
Under the canonical isomorphism $\Delta^{*}P_i \cong \trivlin_1^0$ of \cref{re:cantriv}, this yields  an isomorphism $\varphi$ with $\varphi_{z_1,z_2}: Q_{z_1} \to Q_{z_2}$, and the pair $(\id_W,\varphi)$ is the identity 2-morphism of $\mathcal{P}$.

Consider the following $\Gamma\!$-bundle gerbes with connections, 1-morphisms, and 2-morphisms,
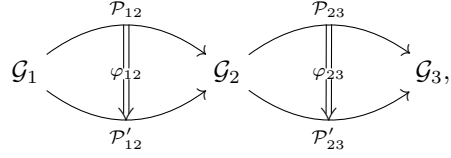
\begin{equation*}
\xymatrix@C=6em{\mathcal{G}_1 \ar@/^1.5pc/[r]^-{\mathcal{P}_{12}}="1"\ar@/_1.5pc/[r]_{\mathcal{P}_{12}'}="2" & \mathcal{G}_2 \ar@/^1.5pc/[r]^-{\mathcal{P}_{23}}="3" \ar@/_1.5pc/[r]_-{\mathcal{P}_{23}'}="4" &\mathcal{G}_3\text{,} \ar@{=>}"1";"2"|{\varphi_{12}}\ar@{=>}"3";"4"|{\varphi_{23}}}
\end{equation*}
with all structure labelled as above. The horizontal composition of $\varphi_{12}$ and $\varphi_{23}$, denoted $\varphi_{23}\circ \varphi_{12}$, is defined by the smooth manifold $W:=W_{12} \times_{Y_2} W_{23}$ equipped with the surjective submersion $(w_{12},w_{23})\mapsto ((z_{12},z_{23}),(z_{12}',z_{23}'))$, where $(z_{12},z_{12}'):= \rho_{12}(w_{12})$ and $(z_{23},z_{23}'):=\rho_{23}(w_{23})$, and the isomorphism of principal $\Gamma\!$-bundles over $W$, given in the fibre over $(w_{12},w_{23})\in W$ by
\begin{equation*}
\xymatrix{(\varphi_{23})_{w_{23}} \otimes (\varphi_{12})_{w_{12}}:(Q_{23})_{z_{23}} \otimes_{} (Q_{12})_{z_{12}} \to (Q_{23}')_{z_{23}'}\otimes_{} (Q_{12}')_{z_{12}'}\text{.}}
\end{equation*}   
This completes the description of the bicategory $\grbcon{\Gamma\!,\kappa}M$ of $\Gamma\!$-bundle gerbes.

We present three technical results, which follow directly from the plus construction \cite{nikolaus2} and are well-known in the context of, e.g., abelian bundle gerbes.
\begin{lemma}
\begin{enumerate}[ (a)]

\item \label{LemmaInvertibilityA}
Every 1-morphism $\mathcal{P}=(Z,\zeta,Q,\nu)$ is invertible.

\item \label{LemmaInvertibilityB}
Every 2-morphism $(W,\rho,\varphi)$ is invertible.

\end{enumerate}  
\end{lemma}

Our second result shows that the surjective submersions of 1-morphisms and 2-morphisms can be assumed to be identities.

\begin{lemma}\
\label{lem:canonicalrefinements}
\begin{enumerate}[ (a)]

\item
\label{lem:canonicalrefinements:a}
Every 1-morphism is  2-isomorphic to one with $Z=Y_1 \times_M Y_2$ and $\zeta=\id_{Z}$.

\item 
\label{lem:canonicalrefinements:b}
Every 2-morphism can be represented by a pair $(\rho,\varphi)$ with $W= Z \times_{Y_{12}} Z'$ and $\rho=\id_W$.

\end{enumerate}
\end{lemma}

The reason our definitions permit general $Z$ and $W$ is that all kinds of compositions result in such more general choices, and in practice it is often easier to keep those rather than performing the descent procedure used to prove  \cref{lem:canonicalrefinements}. 
Our third result allows one to refine the surjective submersion of a $\Gamma\!$-bundle gerbe without changing its isomorphism class.

\begin{lemma}
\label{lem:refinementofss}
If $\mathcal{G}=(\pi,(A,B),P,\mu)$ is a $\Gamma\!$-bundle gerbe with connection, with surjective submersion $\pi:Y \to M$, and $\rho:Y' \to Y$ is a smooth map such that $\pi':= \pi\circ\rho$ is again a surjective submersion, then $\mathcal{G}^{\rho}:=(\pi',\rho^{*}(A,B),(\rho^{[2]})^{*}P,(\rho^{[3]})^{*}\mu)$ is another $\Gamma\!$-bundle gerbe, and there exists a canonical isomorphism $\mathcal{G}^{\rho}\cong \mathcal{G}$. \end{lemma}

A proof of \Cref{lem:refinementofss} can alternatively be derived with the methods introduced in the next section, identifying the isomorphism $\mathcal{G}^{\rho}\cong \mathcal{G}$ as a refinement in the sense of \cref{refinement}.

\begin{example}
\label{BA-bundle-gerbes}
Let $\Gamma=BA$. We recall that $BA$ has a unique adjustment $\kappa=0$ and a unique splitting $u=0$. Every $BA$-bundle gerbe with connection is automatically adapted and fake-flat. Moreover, up to canonical changes of convention, they coincide with the usual abelian $A$-bundle gerbes as defined in \cite{waldorf1}: they consist of a surjective submersion $\pi:Y \to M$, a 2-form $B \in \Omega^2(Y,\mathfrak{a})$, a principal $A$-bundle $P$ over $Y^{[2]}$ with connection $\omega$, and a connection-preserving bundle isomorphism $\mu$.
There are two conventional differences. The first lies in the curvature condition, which here reads
\begin{equation}
\label{relation-between-B-and-omega}
\mathrm{curv}_{\pr_1^{*}B,\pr_2^{*}B}(\omega) := \mathrm{d}\omega - \pr_1^{*}B + \pr_2^{*}B=0\text{,}
\end{equation}
whereas in \cite{waldorf1} the sign is reversed. Second, the isomorphism $\mu$ is defined as $\mu: \pr_{23}^{*}P \otimes \pr_{12}^{*}P \to \pr_{13}^{*}P$, whereas in \cite{waldorf1} the order of the tensor product is permuted. Transitioning between conventions corresponds to replacing $P$ with $P^{*}$.
For brevity, we write
\begin{equation*}
\grbcon A- :=\grbcon{BA,0}-\text{,}
\end{equation*}
and use the term \quot{$A$-bundle gerbe} from now on as a synonym for \quot{$BA$-bundle gerbe}. We also omit the notation $A=0$ in the tuple notation of \cref{def:grb} for bundle gerbes. Moreover, concerning the trivial bundle gerbes of \cref{trivial-gamma-bundle-gerbe}, we write $\mathcal{I}_{B} := \mathcal{I}_{0,B}$ for any $B\in \Omega^2(M,\mathfrak{a})$. 
\end{example}

\begin{example}
\label{dis-bundle-gerbes}
$\Gamma = F_{\dis}$. We recall that $F_{\dis}$ has a unique adjustment $\kappa=0$ and a unique splitting $u=0$.  Every $F_{\dis}$-bundle gerbe with connection is adapted. 
There is a  canonical equivalence
\begin{equation*}
\grbcon {F_{\dis},0}-  \cong (\buncon F-)_{\dis}
\end{equation*}
of sheaves of bicategories, where on the right-hand side, $\buncon F- $ denotes the sheaf of principal $F$-bundles with connection, viewed as a sheaf of bicategories having only identity 2-morphisms.
The equivalence restricts to
\begin{equation*}
\grbconff {F_{\dis},0}{...}  \cong (\bunconflat F-)_{\dis}\text{,}
\end{equation*}
i.e., to one between \emph{fake-flat} $F_{\dis}$-bundle gerbes and \emph{flat} $F$-bundles.
These equivalences are obtained as follows.

Abstractly, we let $\con FX$  be the category of ordinary $F$-connections on $X$, i.e., the objects are 1-forms $A\in \Omega^1(X,\mathfrak{f})$ and the morphisms $A_1 \to A_2$ are gauge transformations: smooth maps $f: X \to F$ such that $A_2 = \mathrm{Ad}_f^{-1}(A_1) + f^{*}\theta$. Then, the assignment $X \mapsto \con FX$ is a presheaf of groupoids, and   $(\con F-)^{+} \cong \buncon F-$. 
We may re-interpret \cref{Fdis-bundles-with-connection} so that $\buncon {F_{\dis}}X_{A_1,A_2} \cong \hom_{\con F-}(A_1,A_2)_{\dis}$, and this implies that
\begin{equation*}
\trivgrbcon {F_{\dis},0} M \cong (\con F-)_{\dis}\text{.}
\end{equation*}
Applying the plus construction yields the claimed equivalence.

Concretely, for a $F_{dis}$-bundle gerbe $\mathcal{G}=(Y,\pi,(A,0),P,\mu)$, since $H=\mathfrak{h}=0$, we necessarily have $P=\trivlin_g^{0}$, with $g: Y^{[2]} \to F$. The condition that the connection $\omega_0=0$ on $P$ is an $\pr_1^{*}A$-$\pr_2^{*}A$-connection implies, by \cref{connection-on-Gamma-bundle}, that $0=\Ad_{g}(\pr_1^{*}A)-\pr_2^{*}A-g^{*}\bar\theta$.  The requirement that the isomorphism $\mu$ preserves anchors requires that $g$ satisfies the usual cocycle condition over $Y^{[3]}$, $g(y_2,y_3)g(y_1,y_2)=g(y_1,y_3)$. Thus, $(A,g)$ constitutes descent data for a principal $F$-bundle with connection along the surjective submersion $\pi:Y \to M$. The clutching construction canonically provides a principal $F$-bundle
\begin{equation*}
P := (Y \times F) / \sim
\quith
(y,f) \sim (y',g(y,y')f)
\end{equation*} 
with connection $\omega\in \Omega^1(P,\mathfrak{f})$ given on $Y \times F$ by $\omega := \mathrm{Ad}^{-1}_{\pr_F}(\pr_Y^{*}A) + \pr_F^{*}\theta$. Clearly, the condition for fake-flatness of $\mathcal{G}$ coincides with the condition for $\omega$ being flat.
Completing this to a 2-functor is straightforward.
\end{example}

\subsection{Framing}

The definition of morphisms between bundle gerbes is intricate. In this section, we present a method to induce 1-morphisms from simpler structures, \emph{refinements} and \emph{strict 1-morphisms}. 
Bundle gerbes with refinements as morphisms even form  a (1-)category $\grbcon{\Gamma\!,\kappa} M^{\refine}$, which is  defined as follows.
The objects are all $\Gamma\!$-bundle gerbes with connection over $M$. 
The morphisms are defined as follows.

\begin{definition}
\label{refinement}
Let $\mathcal{G}_1=(Y_1,\pi_1,(A_1, B_1),P_1,\mu_1)$ and $\mathcal{G}_2=(Y_2,\pi_2,(A_2, B_2),P_2,\mu_2)$ be $\Gamma\!$-bundle gerbes with connection over $M$. 
A \emph{refinement} $\mathcal{R}:\mathcal{G}_1\to \mathcal{G}_2$ is a
 triple $\mathcal{R}=(\rho,(g,\varphi),\nu)$ consisting of a smooth map $\rho:Y_1 \to Y_2$ such that $\pi_2\circ \rho=\pi_1$, a gauge transformation $(g,\varphi): (A_1,B_1)\to \rho^{*}(A_2,B_2)$, and 
of a connection-preserving bundle isomorphism 
    \begin{equation*}
    \nu: {}^{\pr_2^{*}(g,\varphi)} P_1 \to  (\rho^{[2]})^* (P_2)^{\pr_1^{*}(g,\varphi)} 
    \end{equation*}
 over $Y_1^{[2]}$
that  renders the diagram
\begin{equation}
\label{eq:refinementdiag}
\begin{aligned}
\xymatrix@C=6em{ (P_1)_{y',y''} \otimes_{} (P_1)_{y,y'}\ar[d]_{\nu_{y',y''} \otimes \id} \ar[r]^-{\id \otimes \mu_1} & (P_1)_{y,y''} \ar[dd]^{\nu_{y,y''}}\\ (P_2)_{\rho(y'),\rho(y'')} \otimes (P_1)_{y,y'}\ar[d]_{\id \otimes \nu_{y,y'}} \\ (P_2)_{\rho(y'),\rho(y'')} \otimes_{} (P_2)_{\rho(y),\rho(y')}  \ar[r]_-{\mu_2\otimes \id} & (P_2)_{\rho(y),\rho(y'')}}
\end{aligned}
\end{equation}
commutative for all $(y,y',y'')\in Y_1^{[3]}$.
\end{definition}

Here, we employ the left and right shifted $\Gamma\!$-bundles $P^{(g,\varphi)}$ and $^{(g,\varphi)\!}P$ introduced in \cref{tensor-product-connection-for-trivial-bundle}.
Given two refinements,
\begin{equation*}
\mathcal{R}_{12}=(\rho_{12},(g_{12},\varphi_{12}),\nu_{12}):\mathcal{G}_1 \to \mathcal{G}_2
\quand
\mathcal{R}_{23}=(\rho_{23},(g_{23},\varphi_{23}),\nu_{23}):\mathcal{G}_2 \to \mathcal{G}_3\text{,}
\end{equation*} 
their composition is given by 
\begin{equation*}
\mathcal{R}_{23} \circ \mathcal{R}_{12} := (\rho_{23} \circ \rho_{12},\rho_{12}^{*}(g_{23},\varphi_{23})\circ (g_{12},\varphi_{12}),(\rho_{12}^{[2]})^{*}\nu_{23} \circ \nu_{12})\text{,}
\end{equation*}
and the identity morphism of $\mathcal{G}$ is $(\id_Y,(1,0),\id_{P})$. 
This defines  the category  $\grbcon\Gamma M^{\refine}$. 
Next, we introduce the slightly weaker notion of strict 1-morphisms. 

\begin{definition}
\label{strict-1-morphism}
Let $\mathcal{G}_1=(Y_1,\pi_1,(A_1, B_1),P_1,\mu_1)$ and $\mathcal{G}_2=(Y_2,\pi_2,(A_2, B_2),P_2,\mu_2)$ be $\Gamma\!$-bundle gerbes with connection over $M$. 
A \emph{strict 1-morphism} $\mathcal{R}:\mathcal{G}_1\to \mathcal{G}_2$ is a
 triple $\mathcal{R}=(\rho,Q,\nu)$ consisting of a smooth map $\rho:Y_1 \to Y_2$ such that $\pi_2\circ \rho=\pi_1$, of a principal $\Gamma\!$-bundle $Q$ over $Y_1$ with  $B_1$-$\rho^{*}B_2$-adjusted $A_1$-$\rho^{*}A_2$-connection,  and 
of a connection-preserving bundle isomorphism 
    \begin{equation*}
    \nu:\pr_2^{*}Q \otimes P_1 \to  (\rho^{[2]})^* P_2 \otimes \pr_1^{*}Q
    \end{equation*}
 over $Y_1^{[2]}$
that  renders the diagram
\begin{equation}
\label{eq:refinementdiag-strict} 
\begin{aligned}
\xymatrix@C=6em{Q{_{y''}}\otimes (P_1)_{y',y''} \otimes_{} (P_1)_{y,y'}\ar[d]_{\nu_{y',y''} \otimes \id} \ar[r]^-{\id \otimes \mu_1} & Q{_{y''}}\otimes(P_1)_{y,y''} \ar[dd]^{\nu_{y,y''}}\\ (P_2)_{\rho(y'),\rho(y'')} \otimes Q{_{y'}}\otimes (P_1)_{y,y'}\ar[d]_{\id \otimes \nu_{y,y'}} \\ (P_2)_{\rho(y'),\rho(y'')} \otimes_{} (P_2)_{\rho(y),\rho(y')}  \ar[r]_-{\mu_2\otimes \id} & (P_2)_{\rho(y),\rho(y'')}\otimes Q{_{y}}}
\end{aligned}
\end{equation}
commutative for all $(y,y',y'')\in Y_1^{[3]}$.
\end{definition}

Given two strict 1-morphisms,
\begin{equation*}
\mathcal{R}_{12}=(\rho_{12},Q_{12},\nu_{12}):\mathcal{G}_1 \to \mathcal{G}_2
\quand
\mathcal{R}_{23}=(\rho_{23},Q_{23},\nu_{23}):\mathcal{G}_2 \to \mathcal{G}_3\text{,}
\end{equation*} 
their composition is given by 
\begin{equation*}
\mathcal{R}_{23} \circ \mathcal{R}_{12} := (\rho_{23} \circ \rho_{12},\rho_{12}^{*}Q_{23}\otimes  Q_{12},(\rho_{12}^{[2]})^{*}\nu_{23} \circ \nu_{12})\text{.}
\end{equation*}
If $\mathcal{R}=(\rho,Q,\nu)$ and $\mathcal{R}'=(\rho,Q',\nu')$ are strict 1-morphisms $\mathcal{G}_1 \to \mathcal{G}_2$ (sharing the same map $\rho$), then a \emph{strict 2-morphism} is a connection-preserving bundle isomorphism $\varphi:Q \to Q'$ over $Y_1$ such that the diagram
\begin{equation}
\label{condition-for-2-morphisms-strict}
\xymatrix@C=4em{
Q_{y'} \otimes_{} (P_1)_{y,y'} \ar[d]_{\varphi_{y'} \otimes \id} \ar[r]^-{\nu_{y,y'}} & (P_2)_{\rho(y),\rho(y')} \otimes Q_{y} \ar[d]^{\id \otimes \varphi_y}
 \\ 
 Q'_{y'} \otimes_{} (P_1)_{y,y'}  \ar[r]_-{\nu'_{y,y'}} &(P_2)_{\rho(y),\rho(y')} \otimes_{} Q'_{y}
 }
\end{equation}
over $Y_1^{[2]}$ commutes. 
Completing these structures to a bicategory, denoted $\grbcon{\Gamma\!,\kappa} M^{strict}$, is straightforward.

We now define \quot{framing} functors
\begin{equation}
\label{framing-functors}
\grbcon{\Gamma\!,\kappa} M^{\refine} \to \grbcon{\Gamma\!,\kappa} M^{strict} \to \grbcon{\Gamma\!,\kappa} M
\end{equation}
as follows.
They act, of course, as the identity on the level of objects.

On the level of 1-morphisms, we associate to a refinement $\mathcal{R}=(\rho,(g,\varphi),\nu):\mathcal{G}_1 \to \mathcal{G}_2$ the strict 1-morphism $(\rho,\trivlin_g^{\varphi},\nu)$. This respects composition up to a strict 2-isomorphism. 
Given a strict 1-morphism $(\rho,Q,\nu)$, we associate the following 1-morphism $(Z,\zeta,\tilde Q,\tilde\nu)$. We define $Z:= Y_1 \times_M Y_2$ and $\zeta=\id_Z$. Consider the smooth map $\tilde \rho: Z \to Y_2^{[2]}$ given by $\tilde\rho(y_1,y_2):=(\rho(y_1),y_2)$. 
We define $\tilde Q := \tilde\rho^{*}P_2 \otimes \pr_1^{*}Q$ over $Z$. Note that $\tilde\rho^{*}P_2$ carries a $\pr_1^{*}\rho^{*}A_2$-$\pr_2^{*}A_2$-connection and that $\pr_1^{*}Q$ carries a $\pr_1^{*}A_1$-$\pr_1^{*}\rho^{*}A_2$-connection. Consequently, the tensor product connection on $\tilde Q$ is a $\pr_1^{*}A_1$-$\pr_2^{*}A_2$-connection, as required.
Finally, we define the isomorphism $\tilde\nu$ fibrewise over a point $((y_1,y_2),(y_1',y_2'))\in Z^{[2]}$ by
\begin{equation*}
\xymatrix{
\hspace{-13em}  \tilde Q_{y_1',y_2'} \otimes_{} (P_1)_{y_1,y_1'}=(P_2)_{\rho(y_1'),y_2'} \otimes_{} Q_{y_1'} \otimes (P_1)_{y_1,y_1'} 
 \ar[d]^{\id \otimes \nu_{y_1,y_1'}} 
\\ 
(P_2)_{\rho(y_1^\prime), y_2^\prime} \otimes_{} (P_2)_{\rho(y_1),\rho(y_1^\prime)} \otimes Q_{y_1} 
 \ar[d]^{(\mu_2)_{\rho(y_1),\rho(y_1'),y_2'}\otimes \id}
\\ 
(P_2)_{\rho(y_1),y_2'} \otimes Q_{y_1} \ar[d]^{(\mu_2)_{\rho(y_1),y_2,y_2'}^{-1} \otimes \id} 
\\
(P_2)_{y_2,y_2'} \otimes_{} (P_2)_{\rho(y_1),y_2}\otimes Q_{y_1}
= (P_2)_{y_2,y_2'} \otimes_{} \tilde Q_{y_1,y_2}\text{.}  \hspace{-13em} }
\end{equation*}
It is straightforward to verify that the tuple $(Z,\zeta,\tilde Q,\tilde \phi)$ defined above constitutes a 1-morphism $\mathcal{G}_1 \to \mathcal{G}_2$ and that composition is preserved.  

Finally, given a strict 2-morphism $\varphi:\mathcal{R} \Rightarrow \mathcal{R}'$, with $\mathcal{R}=(\rho,Q,\nu)$ and $\mathcal{R}'=(\rho,Q',\nu')$, we assign the 2-morphism $(W,\rho,\tilde \varphi):(Z,\zeta,\tilde Q,\tilde\nu) \Rightarrow (Z,\zeta,\tilde Q',\tilde\nu')$ with $W := Z=Y_1 \times_M Y_2$, $\rho=\id_Z$, and $\tilde \varphi := \id_{P_2} \otimes \pr_1^{*}\varphi: \tilde Q \to \tilde Q'$. 

\begin{remark}
The existence of the framing functors from \cref{framing-functors} implies, by \cref{LemmaInvertibilityA}, that every refinement and every strict 1-morphism is \quot{weakly} invertible, i.e., invertible by a general 1-morphism. 
\end{remark}

\begin{remark}
Under the framing functors from \cref{framing-functors}, the identity refinement of a bundle gerbe $\mathcal{G}$ induces the identity 1-morphism $\id_{\mathcal{G}}$ defined in \cref{identity-1-morphism}.
\end{remark}

\begin{remark}
A trivialization of a $\Gamma\!$-bundle gerbe $\mathcal{G}$ with connection over $M$ consists of a $\Gamma\!$-connection $(A,B)$ on $M$ together with a 1-morphism $\mathcal{T}: \mathcal{G} \to \mathcal{I}_{A,B}$. By \cref{lem:canonicalrefinements:a}, every trivialization is 2-isomorphic to a \emph{strict} 1-morphism. 
\end{remark}

\subsection{Extension and reduction}

We consider a second Lie 2-group $\Gamma'$, equipped with an adjustment $\kappa'$, and a Lie 2-group homomorphism $f: \Gamma \to \Gamma'$ that preserves adjustments (\cref{preserve-adjustment}). By \cref{covariance-bundles-with-connections}, we have a functor
\begin{equation*}
f_{*}:\buncon{\Gamma\!,\kappa} X_{(A_1,B_1),(A_2,B_2)} \to \buncon{{\Gamma'\!\!,\kappa'}}X_{F_{*}(A_1,B_1),F_{*}(A_2,B_2)}
\end{equation*}
that is monoidal with respect to the partially defined tensor products. Consequently, it induces a morphism
\begin{equation*}
\trivgrbcon {\Gamma\!,\kappa}- \to \trivgrbcon {\Gamma'\!\!,\kappa'}-
\end{equation*}
of presheaves of bicategories, and by functoriality of the plus construction, a functor
\begin{equation*}
f_{*}:\grbcon{\Gamma\!,\kappa}- \to \grbcon{\Gamma'\!\!,\kappa'}-
\end{equation*}
of sheaves. Specifically, it applies $f_{*}$ to all $\Gamma\!$-connections, principal $\Gamma\!$-bundles, and bundle morphisms. We note that $f_{*}$ preserves fake-flatness and flatness. 

In the following, we apply this general covariance to the Lie 2-group homomorphisms
\begin{equation*}
BA \stackrel i\to \Gamma \stackrel p\to F_{\dis}
\end{equation*} 
associated with a smoothly separable, central Lie 2-group $\Gamma$, which are adjustment-preserving (\cref{morphisms-of-central-extension-are-adjustment-preserving}) for an arbitrary adjustment on $\Gamma$ and the zero adjustments on $BA$ and $F_{\dis}$.

We begin with $p: \Gamma \to F_{\dis}$. Using \cref{dis-bundle-gerbes} we identify $F_{\dis}$-bundle gerbes canonically with principal  $F$-bundles, so that extension becomes a functor
\begin{equation*}
p_{*}: \grbcon{\Gamma\!,\kappa}M \to \buncon FM_{\dis}\text{.}
\end{equation*}
Via this functor, higher gauge theory has an underlying \quot{ordinary gauge theory}. The following lemma demonstrates  that higher gauge theory should not be restricted to a fake-flat setting:

\begin{lemma}
\label{flatness-of-underlying-bundle}
Let $\mathcal{G}$ be a $\Gamma\!$-bundle gerbe with connection. Then, $\mathcal{G}$ is fake-flat if and only if $\mathcal{G}$ is adapted to a splitting $u$ and $p_{*}(\mathcal{G})$ is flat. In particular, we have
\begin{equation*}
\mathcal{G}\text{ fake-flat} \quad\Rightarrow\quad p_{*}(\mathcal{G})\text{ flat.}
\end{equation*}
\end{lemma}

\begin{proof}
Let $(A,B)$ be the $\Gamma\!$-connection of $\mathcal{G}$. By \cref{dis-bundle-gerbes}, $\omega_{y,f}=\mathrm{Ad}_f^{-1}(p_{*}(A_y))+f^{*}\theta$ defines the connection on $p_{*}(\mathcal{G})$. Its curvature is
\begin{equation*}
(\mathrm{d}\omega + \tfrac{1}{2}[\omega\wedge \omega])_{y,f} = \mathrm{Ad}_f^{-1}(p_{*}(\mathrm{d}A + \tfrac{1}{2}[A \wedge A]))=\mathrm{Ad}_f^{-1}(p_{*}(\mathrm{fcurv}(A,B)))\text{.}
\end{equation*}
Hence, if $(A,B)$ is fake-flat, then $\omega$ is flat. Conversely, if $\omega$ is flat, then $p_{*}(\mathrm{fcurv}(A,B))=0$. If $(A,B)$ is adapted, we conclude that $\mathrm{fcurv}(A,B)=0$.  
\end{proof}

Next, we consider the Lie 2-group homomorphism $i:BA \to \Gamma$, which induces a functor
\begin{equation*}
i_{*}:\grbcon AM \to \grbconff{\Gamma}M\text{,}
\end{equation*} 
see \cref{BA-bundle-gerbes}.
Explicitly, if $\mathcal{G}=(Y,\pi,b,P,\mu)$ is an $A$-bundle gerbe with connection over $M$, we have 
\begin{equation*}
i_{*}(\mathcal{G}) = (Y,\pi,(0,i_{*}(b)),i_{*}(P),i_{*}(\mu))\text{.}
\end{equation*}
The connection on this bundle gerbe is fake-flat, and the curvature is
\begin{equation*}
\mathrm{curv}(i_{*}(\mathcal{G}))=\mathrm{curv}_{\kappa}(i_{*}(\mathcal{G}))=i_{*}\mathrm{curv}(\mathcal{G})\text{.}
\end{equation*}

We now examine the interplay between the functors $i$ and $p$ in the central extension $BA \to \Gamma \to F_{\dis}$. 
To this end, we additionally consider a splitting $u$ on $\Gamma$ and assume that $\kappa$ is adapted to $u$. We denote by $\grbcon{\Gamma\!,u,\kappa}M^{\pflat}$ the bicategory with:
\begin{itemize}

\item 
objects: pairs $(\mathcal{G},\sigma)$ consisting of a $\Gamma\!$-bundle gerbe $\mathcal{G}$ with adapted connection and a flat global section $\sigma$ of the principal $F$-bundle $p_{*}(\mathcal{G})$. 

\item
morphisms between $(\mathcal{G}_1,\sigma_1)$ and $(\mathcal{G}_2,\sigma_2)$: 1-morphisms $\mathcal{A}:\mathcal{G}_1 \to \mathcal{G}_2$ in $\grbcon{\Gamma\!,u,\kappa}M$ such that $p_{*}(\mathcal{A})\circ \sigma_1 = \sigma_2$. 

\item
2-morphisms: all 2-morphisms in $\grbcon{\Gamma\!,u,\kappa}M$. 

\end{itemize}
This generalizes a bicategory defined in \cite[Def. 5.2.1]{Nikolausa} to a setting with connections. As $p \circ i$ is the constant trivial functor, $i_{*}$ admits a canonical lift along the projection functor $\grbcon{\Gamma\!,u,\kappa}M^{\pflat} \to \grbconff \Gamma M$, yielding a functor
\begin{equation*}
i_{*}^{\pflat}: \grbcon AM \to \grbcon{\Gamma\!,u,\kappa}M^{\pflat}\text{.}
\end{equation*}   
The following result generalizes \cite[Thm. 4.1.3]{Nikolausa} to a setting with connections. 

\begin{theorem}
\label{gerbe-morphisms-under-i}
Let $\Gamma$ be a smoothly separable, central Lie 2-group equipped with a splitting $u$ and an adapted adjustment $\kappa$. 
The functor $i^{\pflat}_{*}$ is an equivalence of categories
\begin{equation*}
\grbcon{A}{M} \cong \grbcon{\Gamma\!,u,\kappa}M^{\pflat}\text{.} 
\end{equation*}
\end{theorem}

\begin{proof}
We begin by establishing essential surjectivity. Let $\mathcal{G}=(Y,\pi,(A,B),P,\mu)$ be a $\Gamma\!$-bundle gerbe with adapted connection over $M$, and let $\sigma$ be a flat section of $p_{*}(\mathcal{G})$. The principal $F$-bundle $p_{*}(\mathcal{G})$ was obtained in  \cref{dis-bundle-gerbes} by the clutching construction performed to  $\pi:Y \to M$ and the map $f:=p_{*}P: Y^{[2]} \to F$. Thus, the section $\sigma$ can be identified with a  smooth map $\tilde \sigma:Y \to F$ such that $\tilde\sigma(y')=f(y,y')\tilde\sigma(y)$ for all $(y,y')\in Y^{[2]}$.   
Flatness of $\sigma$ is equivalent to the condition $p_{*}(A)=-\tilde \sigma^{*}\bar\theta$. We consider the pullback diagram
\begin{equation*}
\alxydim{}{\mquad Y':=Y \times_F G \ar[r]^-{g} \ar[d]_{\rho} & G \ar[d]^{p} \\ Y \ar[r]_{\tilde \sigma} &F}
\end{equation*} 
which exists since $p$ is a surjective submersion. Now, $\rho^{*}\tilde\sigma^{*}\bar\theta = g^{*}p^{*}\bar\theta = p_{*}(g^{*}\bar\theta)$, and the flatness condition above yields $p_{*}(\rho^{*}A+g^{*}\bar\theta)=0$. Thus, there exists $\varphi\in \Omega^1(Y',\mathfrak{h})$ such that 
\begin{equation}
\label{red-1}
t_{*}(\varphi)+\rho^{*}A+g^{*}\bar\theta=0\text{.}
\end{equation}
Using the splitting $u$, we may set $\varphi :=- u(\rho^{*}A+g^{*}\bar\theta)$ to obtain a canonical choice. A straightforward calculation reveals that
\begin{equation*}
t_{*}(\mathrm{d}\varphi + \tfrac{1}{2}[\varphi \wedge \varphi]+ \alpha_{*}(\rho^{*}A \wedge \varphi)+\rho^{*} B)=-\rho^{*}\mathrm{fcurv}(A,B)\text{,}
\end{equation*}
which vanishes by virtue of \cref{flatness-of-underlying-bundle}. Hence, there exists a unique $b\in \Omega^2(Y',\mathfrak{a})$ such that 
\begin{equation}
\label{red-2}
i_{*}(b) = \mathrm{d}\varphi + \tfrac{1}{2}[\varphi \wedge \varphi]+ \alpha_{*}(\rho^{*}A \wedge \varphi)+\rho^{*} B\text{.}
\end{equation}
\Cref{red-1,red-2} then show that
\begin{equation*}
(g,\varphi): (0,i_{*}(b)) \to \rho^{*}(A,B)
\end{equation*}  
constitutes a flat gauge transformation. 

We consider the following $\Gamma\!$-bundle gerbe 
\begin{equation*}
\mathcal{G}'=(Y',\pi',(0,i_{*}(b)), P',\mu')\text{,}
\end{equation*}
with $\pi':=\pi \circ \rho$, with a principal $\Gamma\!$-bundle $P'$ with connection defined by
\begin{equation*}
P' := (\pr_2^{*}\trivlin_g^{\varphi})^{*} \otimes (\rho^{[2]})^{*}P \otimes \pr_1^{*}\trivlin^{\varphi}_g\text{,}
\end{equation*}
and with a bundle morphism $\mu'$ defined using the unit map $(\pr_2^{*}\trivlin_g^{\varphi})^{*} \otimes (\pr_2^{*}\trivlin_g^{\varphi}) \cong \trivlin_1^0$. 
The rules for the tensor product show that $P'$ has a $0$-$0$-connection that is $\pr_1^{*}i_{*}(b)$-$\pr_2^{*}i_{*}(b)$-flat. Thus, $\mathcal{G}'$ is a bundle gerbe with connection, and it is fake-flat. It is readily verified that $\mathcal{R}=(\rho, \trivlin_g^{\varphi},\nu)$, with
\begin{equation*}
\nu:\pr_2^{*}\trivlin_g^{\varphi} \otimes P' \to  (\rho^{[2]})^* P \otimes \pr_1^{*}\trivlin_g^{\varphi}
\end{equation*}
defined by utilizing the unit map $(\pr_2^{*}\trivlin_g^{\varphi})^{*} \otimes (\pr_2^{*}\trivlin_g^{\varphi}) \cong \trivlin_1^0$ once more, is a strict isomorphism $\mathcal{R}:\mathcal{G}' \to \mathcal{G}$.

We compute the map $p_{*}P': Y'^{[2]} \to F$ of the new principal $\Gamma\!$-bundle $P'$, obtaining
\begin{equation*}
p_{*}P'(y_1',y_2') = p(g(y_2'))\cdot p_{*}P(\rho(y_1'),\rho(y_2'))\cdot p(g(y_1'))
= \tilde \sigma(\rho(y_2'))\cdot f(\rho(y_1'),\rho(y_2'))\cdot \tilde\sigma(\rho(y_1'))
=1\text{.}
\end{equation*}
Thus, by \cref{bundle-morphisms-under-i}, there is a principal $A$-bundle $P^{\mathrm{red}}$ together with a connection-preserving bundle isomorphism $i_{*}(P^{\mathrm{red}}) \cong P'$, and $\mu'$ restricts to a suitable morphism between $A$-bundles. This shows that $\mathcal{G}''=(Y',\pi',b,P^{\mathrm{red}},\mu'')$ is an $A$-bundle gerbe with connection, with an obvious refinement $i_{*}(\mathcal{G}'')\cong \mathcal{G}'$. This completes the proof that the functor $i^{\pflat}_{*}$ is essentially surjective.

Next we demonstrate  that $i_{*}^{\pflat}$ is full and faithful. Let $\mathcal{G}_1,\mathcal{G}_2$ be $A$-bundle gerbes with connection over $M$. We have to show that  the Hom-functor
\begin{equation*}
\hom_{\grbcon AM}(\mathcal{G}_1,\mathcal{G}_2) \to \hom_{\grbcon {\Gamma\!,u,\kappa} M^{\pflat}}(i_{*}^{\pflat} (\mathcal{G}_1),i_{*}^{\pflat}(\mathcal{G}_2))
\end{equation*}
is an equivalence of categories. We first show again that it is essentially surjective. Suppose that $\mathcal{A}=(Z,\zeta,Q,\nu):i_{*}^{\pflat}(\mathcal{G}_1) \to i_{*}^{\pflat}(\mathcal{G}_2)$ is a morphism in $\grbcon {\Gamma\!,u,\kappa} M^{\pflat}$. The condition on the morphisms of $\grbcon {\Gamma\!,u,\kappa} M^{\pflat}$ requires that $p_{*}(\mathcal{A})=1$. We recall that $p_{*}(\mathcal{A})$ pulls back to the map $p_{*}Q: Z \to F$, which is hence trivial. By \cref{bundle-morphisms-under-i}, we have $Q=i_{*}(Q^{\mathrm{red}})$ for a principal $A$-bundle $Q^{\mathrm{red}}$ with connection, and there exists a unique bundle isomorphism $\nu^{\mathrm{red}}$ over $Z^{[2]}$ such that $(Z,\zeta,Q^{\mathrm{red}},\nu^{\mathrm{red}})$ is an isomorphism between $\mathcal{G}_1$ and $\mathcal{G}_2$ in $\grbcon AM$. This shows essential surjectivity. Full faithfulness follows directly from the fact that the functor $i_{*}$ in \cref{bundle-morphisms-under-i} is full and faithful.
\end{proof}

\begin{corollary}
\label{corr-1}
If $\mathcal{G}$ is a fake-flat $\Gamma\!$-bundle gerbe, then every point $x\in M$ has an open neighborhood $U$ with a 2-form $\rho\in \Omega^2(U,\mathfrak{a})$ such that $\mathcal{G}|_U \cong i_{*}(\mathcal{I}_{\rho})$. 
\end{corollary}

\begin{proof}
$p_{*}(\mathcal{G})$ is flat and thus has a flat local section supported on a sufficiently small open neighborhood $U$. 
Thus, $\mathcal{G}|_U \cong i_{*}(\mathcal{G}_U)$ for an $A$-bundle gerbe $\mathcal{G}_U$ with connection over $U$. By possibly shrinking $U$ further, the bundle gerbe $\mathcal{G}_U$ itself becomes trivializable, so that $\mathcal{G}_U \cong \mathcal{I}_{\rho}$ for a 2-form $\rho$. 
\end{proof}

\begin{corollary}
\label{corr-2}
Suppose $\mathcal{G}$ is a $\Gamma\!$-bundle gerbe with an adapted connection over $M$. Then, a flat section $\sigma:M \to p_{*}(\mathcal{G})$ determines an $A$-bundle gerbe $\mathcal{G}^{\mathrm{red}}$ with connection over $M$ together with an isomorphism $i_{*}(\mathcal{G}^{\mathrm{red}}) \cong \mathcal{G}$. 
\end{corollary}

\begin{corollary}
\label{exact-sequence-1}
The sequence 
\begin{equation*}
\hat\h^1(M,A) \to \hat \h^1(M,(\Gamma\!,u,\kappa))^{\aptadj} \to \hat\h^0(M,F)
\end{equation*}
in differential cohomology, induced by the functors $i_{*}$ and $p_{*}$, is an exact sequence of pointed sets. 
\end{corollary}

In the remainder of this section, we extend the concept of an \emph{adjusted extension} from principal $\Gamma\!$-bundles (see \cref{extension-and-adjusted-shift}) to bundle gerbes. One key result (\cref{gerbe-morphisms-under-i-non-flat}) generalizes \cref{gerbe-morphisms-under-i} to cases where $p_{*}(\mathcal{G})$ is not necessarily flat. The following lemma introduces the basic construction. 

\begin{lemma}
\label{extension-and-adjusted-shift-lemma}
Suppose $\mathcal{G}=(Y,\pi,b,P,\mu)$ is an $A$-bundle gerbe with connection. 
Let $A\in\Omega^1(M,\mathfrak{g})$ and $\chi\in \Omega^2(M,\mathfrak{a})$ be differential forms.  
We consider the $\Gamma\!$-connection $(A,B^{(A,\chi)})$ on $M$ with
\begin{equation*}
B^{(A,\chi)}:= i_{*}(\chi)+\tfrac{1}{2}\kappa_{*}(A \wedge A)\in \Omega^2(M,\mathfrak{h})\text{.}
\end{equation*}
Then, the tuple
\begin{equation*}
i_{*}^{(A,\chi)}(\mathcal{G}):=(Y,\pi,i_{*}(0,b)+\pi^{*}(A,B^{(A,\chi)}),i_{*}^{A}(P),i_{*}(\mu))\text{,}
\end{equation*}
defines a $\Gamma\!$-bundle gerbe over $M$ with connection. Moreover, if $\kappa$ is adapted to a splitting $u$ and $A=q_u(\beta)$ for a 1-form $\beta\in \Omega^1(M,\mathfrak{f})$, then we have:
\begin{itemize}
\item 
$i_{*}^{(A,\chi)}(\mathcal{G})$ is adapted to $u$,

\item
its fake-curvature is given by $\pi^{*}\mathrm{fcurv}(A,B^{(A,\chi)})=\pi^{*}q_u(\mathrm{d}\beta+\tfrac{1}{2}[\beta\wedge\beta])$, and

\item
its adjusted curvature is
\begin{align*}
\mathrm{curv}_{\kappa}(i_{*}^{(A,\chi)}(\mathcal{G}))
=  i_{*}(\mathrm{curv}(\mathcal{G}) + \mathrm{d}\chi+b_{\kappa}(\beta\wedge \mathrm{d}\beta)+\tfrac{1}{3}b_{\kappa}(\beta\wedge  [\beta \wedge \beta]))\text{.}
\end{align*}
\end{itemize}
\end{lemma}

\begin{proof}
Observe that $i_{*}^{A}(P)$ carries a $\pr_1^*\pi^{*}A$-$\pr_2^{*}\pi^{*}A$-connection by \cref{adjusted-shift-of-connection} and is $\pr_1^{*}(i_{*}(b)+\pi^{*}B^{(A,\chi)})$-$\pr_2^{*}(i_{*}(b)+\pi^{*}B^{(A,\chi)})$-adjusted by \cref{curvature-of-adjusted-shift}, as required. The adjusted shift by $A$ satisfies, by \cref{shifting-connection-of-tensor-product}, $\pr_{23}^{*}i_{*}^{A}(P) \otimes \pr_{12}^{*}i_{*}^{A}(P)=i_{*}^A(\pr_{23}^{*}P \otimes \pr_{12}^{*}P)$, and since $\mu$ was connection-preserving, it remains so after extension and shift (\cref{adjusted-shift-of-connection}). This confirms that $i_{*}^{(A,\chi)}(\mathcal{G})$ is a bundle gerbe with connection. 

If $\kappa$ is adapted to a splitting $u$, we have:
\begin{align*}
\mathrm{fcurv}(A,B^{(A,\chi)}) &= \mathrm{d}A+\tfrac{1}{2}[A\wedge A]-\tfrac{1}{2}t_{*}\kappa_{*}(A \wedge A)
\\&= \mathrm{d}A+\tfrac{1}{2}[A\wedge A]- \tfrac{1}{2} t_{*}u([A\wedge  A])
\\&= \mathrm{d}A+\tfrac{1}{2}qp([A\wedge A])\text{.}
\end{align*}
The adjusted curvature is
\begin{align*}
\mathrm{curv}_{\kappa}(i_{*}^{(A,\chi)}(\mathcal{G}))
&= \mathrm{d}(i_{*}(b)+\pi^{*}B^{(A,\chi)})+\pi^{*}(\kappa_{*}(A\wedge \mathrm{d}A )+\tfrac{1}{2}\kappa_{*}(A\wedge  [A \wedge A]))
\\&= i_{*}(\mathrm{d}(b+\pi^{*}\chi))-\pi^{*}(\kappa_{*}^{skew}(A \wedge \mathrm{d}A)+\kappa_{*}(A\wedge \mathrm{d}A )+\tfrac{1}{2}\kappa_{*}(A\wedge  [A \wedge A]))
\\&= i_{*}(\mathrm{d}(b+\pi^{*}\chi))+\pi^{*}(\kappa_{*}^{sym}(A\wedge \mathrm{d}A )+\tfrac{1}{2}\kappa_{*}(A\wedge  [A \wedge A]))\text{.}
\end{align*}
For $A=q_u(\beta)$, this implies, together with \cref{kappa-with-differential-forms}, the claimed formulas.
\end{proof}

\begin{definition}
\label{extension-and-adjusted-shift-gerbes}
\label{shifting-connections}
The $\Gamma\!$-bundle gerbe $i_{*}^{(A,\chi)}(\mathcal{G})$ with connection of \cref{extension-and-adjusted-shift-lemma} is called the \emph{adjusted extension} of $\mathcal{G}$ along $(A,\chi)$.
\end{definition}

\begin{remark}
\label{remark-adjusted-extension}
\begin{enumerate}[(i)]
\item 
\label{remark-adjusted-extension:1}
The adjusted extension along $(0,\chi)$ coincides with the ordinary shift-of-curving operation available for abelian bundle gerbes, implemented by tensoring with a trivial $A$-bundle gerbe $\mathcal{I}_{\chi}$, i.e., 
\begin{equation*}
i_{*}^{(A,\chi)}(\mathcal{G})=i_{*}^{(A,0)}(\mathcal{G}\otimes \mathcal{I}_{\chi})\text{.}
\end{equation*}

\item
\label{remark-adjusted-extension:2}
We have 
\begin{equation*}
i_{*}^{(A,\chi)}(\mathcal{I}_0)=i_{*}^{(A,0)}(\mathcal{I}_\chi)=\mathcal{I}_{A,B^{(A,\chi)}}\text{,}
\end{equation*}
where $B^{(A,\chi)}$ is defined in \cref{extension-and-adjusted-shift-lemma}, and $\mathcal{I}_{A,B}$ is the trivial $\Gamma\!$-bundle gerbe of \cref{trivial-gamma-bundle-gerbe}. 
Consequently, a trivial $\Gamma\!$-bundle gerbe $\mathcal{I}_{A,B}$, with $B=B^{(\chi,A)}$ for some $\chi\in \Omega^2(M,\mathfrak{a})$, is the adjusted extension of the trivial $A$-bundle gerbe. 

\item
\label{remark-adjusted-extension:3}
We have $p_{*}(i_{*}^{(A,\chi)}(\mathcal{G}))=\trivlin_1^{p_{*}A}$, the trivial $F$-bundle with connection 1-form induced by $p_{*}A$. 

\end{enumerate}
\end{remark}

Although an \quot{adjusted shift} of a connection on a \emph{general} $\Gamma\!$-bundle gerbe is not feasible (only for those in the image of the extension functor $i_{*}$, as above), one may perform an adjusted shift of \emph{general} 1-morphisms between adjusted extensions of bundle gerbes.
First, let $\mathcal{A}:i_{*}(\mathcal{G}_1) \to i_{*}(\mathcal{G}_2)$ be a 1-morphism in $\grbcon{\Gamma\!,\kappa}M$. Then, $p_{*}\mathcal{A}$ is a morphism of trivial principal $F$-bundles, and can hence be identified with a smooth map $p_{*}\mathcal{A}:M \to F$. 

\begin{lemma}
\label{shift-of-1-morphisms}
Suppose $\mathcal{G}_1$ and $\mathcal{G}_2$ are $A$-bundle gerbes with connections over $M$, and suppose $A_1,A_2 \in \Omega^1(M,\mathfrak{g})$ and $\chi_1,\chi_2\in \Omega^2(M,\mathfrak{a})$ are differential forms. We assume that
\begin{equation*}
\mathcal{A}:i_{*}^{( A _1,\chi_1)}(\mathcal{G}_1) \to i_{*}^{( A _2,\chi_2)}(\mathcal{G}_2)
\end{equation*}
is a 1-morphism in $\grbcon{\Gamma\!,\kappa}M$, $\mathcal{A}=(Z,\zeta,Q,\nu)$, and that $\eta \in\Omega^1(M,\mathfrak{g})$ is a 1-form. Then,
$\mathcal{A}^{\eta}:=(Z,\zeta,Q^{\eta},\nu)$, where $Q^{\eta}$ denotes the adjusted shift of the connection on $Q$ along (the pullback to $Z$ of) $\eta$, constitutes a 1-morphism
\begin{equation*}
\mathcal{A}^{ \eta }: i_{*}^{( A _1',\chi_1')}(\mathcal{G}_1)\to i_{*}^{( A' _2,\chi_2')}(\mathcal{G}_2)
\end{equation*}
in $\grbcon{\Gamma\!,\kappa}M$, where
\begin{align*}
A_1'&:= A_1+\eta 
&\chi_1' &:= \chi_1+b_{\kappa}(p_{*}( A _1)\wedge p_{*}( \eta )) 
\\
A_2' &:= A_2+ \tilde q_u(p_{*}\mathcal{A},\eta)
& \chi_2' &:= \chi_2+b_{\kappa}(p_{*}( A _2)\wedge p_{*}(\tilde q_u(p_{*}\mathcal{A},\eta)))\text{,}
\end{align*}
where $\tilde q_u$ is defined in \cref{adjusted-shift-of-connection}, and $b_{\kappa}$ is the symmetric bilinear form associated with the adjustment $\kappa$. 
Moreover,
\begin{enumerate}[(a)]
\item 
\label{shift-of-1-morphisms:a}
if $\mu\in \Omega^1(M,\mathfrak{g})$ is another 1-form, then
\begin{equation*}
(\mathcal{A}^{ \eta })^{\mu} = \mathcal{A}^{ \eta  + \mu}\text{,}
\end{equation*}

\item
\label{shift-of-1-morphisms:b}
if $\mathcal{B}: i_{*}^{( A _2,\chi_2)}(\mathcal{G}_2) \to i_{*}^{( A _3,\chi_3)}(\mathcal{G}_3)$ is another 1-morphism, then
\begin{equation*}
(\mathcal{B} \circ \mathcal{A})^{ \eta } = \mathcal{B}^{\tilde q_u(p_{*}\mathcal{A}, \eta)} \circ \mathcal{A}^{ \eta }\text{.}
\end{equation*}

\end{enumerate}
\end{lemma}

\begin{proof}
Suppose $\mathcal{G}_i=(Y_i,\pi_i,b_i,P_i,\mu_i)$.
We recall that $i_{*}^{(A_i,\chi_i)}(\mathcal{G}_i)$ carries the $\Gamma\!$-connection $( A _i, \tilde B_i)$ with
\begin{align*}
\tilde B_i &:=i_{*}(b_i)+B^{(A_i,\chi_i)}= i_{*}(b_i+\chi_i) +\tfrac{1}{2}\kappa_{*}( A _i \wedge  A _i)\text{,} 
\end{align*}
and that it carries the principal $\Gamma\!$-bundle $i_{*}^{A_i}(P_i)$, with $A_i$ understood to be pulled back from $M$.
The principal $\Gamma\!$-bundle $Q$ of $\mathcal{A}$ has an $A_1$-$A_2$ connection of curvature 
\begin{equation*}
\mathrm{curv}_{\tilde B_1,\tilde B_2}(Q) = \kappa(\chi,\mathrm{fcurv}( A _1,\tilde B_1))\text{.}
\end{equation*}
Finally, the bundle morphism $\nu$ of $\mathcal{A}$ goes
\begin{equation*}
\nu:  \pr_2^{*}Q \otimes i_{*}^{ A _1}(P_1) \to i_{*}^{ A _2}(P_2) \otimes \pr_1^{*}Q\text{.}
\end{equation*}

The adjusted shift $Q^{\eta}$ has, by \cref{adjusted-shift-of-connection,curvature-of-adjusted-shift}, a $A_1'$-$A_2'$-connection that is $\tilde B_1'$-$\tilde B_2'$-adjusted, where \begin{align*}
\tilde B_1' &:= \tilde B_1 + \tfrac{1}{2}\kappa_{*}( \eta  \wedge  \eta )+\kappa_{*}(  A _1 \wedge \eta ) 
\\
\tilde B_2' &:= \tilde B_2 + \tfrac{1}{2}\kappa_{*}( \tilde q_u(p_{*}\mathcal{A},\eta)\wedge   \tilde q_u(p_{*}\mathcal{A},\eta))+\kappa_{*}(  A _2 \wedge  \tilde q_u(p_{*}\mathcal{A},\eta))\text{.} 
\end{align*}
By \cref{adjusted-shift-of-connection}, 
\begin{equation*}
\nu:  (\pr_2^{*}Q \otimes i_{*}^{ A _1}(P_1) )^{ \eta }\to (i_{*}^{ A _2}(P_2) \otimes \pr_1^{*}Q)^{ \eta }
\end{equation*}
is connection-preserving. Using \cref{shifting-connection-of-tensor-product} on the left and on the right, this may be viewed as a connection-preserving isomorphism
\begin{equation}
\label{lsjdfoiadfja}
\nu: \pr_2^{*}Q^{ \eta } \otimes i_{*}^{ A _1+ \eta }(P_1) \to i_{*}^{ A _2+  \tilde q_u(p_{*}\mathcal{A},\eta)}(P_2) \otimes \pr_1^{*}Q^{ \eta }\text{.}
\end{equation}

Now we are in a position to show that $\mathcal{A}^{\eta}:=(Z,\zeta,Q^{\eta},\nu)$ is a 1-morphism from $i_{*}^{( A _1',\chi_1')}(\mathcal{G}_1)$ to $i_{*}^{( A' _2,\chi_2')}(\mathcal{G}_2)$. Due to the definitions of $A_1'$ and $A_2'$, $\nu$ in \cref{lsjdfoiadfja} goes precisely between the correct bundles with connections. It then remains to prove that the 2-forms of $i_{*}^{( A_i',\chi_i')}(\mathcal{G}_i)$, namely, $i_{*}(b_i)+B^{(A_i',\chi_i')}$, coincide with the 2-forms $\tilde B'_i$ above. 
Indeed, 
\begin{align*}
 i_{*}(b_1)+B^{(A_1',\chi_1')} 
&=i_{*}(b_1+\chi_1')+\tfrac{1}{2}\kappa_{*}(A_1'\wedge A_1')
\\&=i_{*}(b_1+\chi_1+b_{\kappa}(p_{*}( A _1)\wedge p_{*}( \eta ))) +\tfrac{1}{2}\kappa_{*}(( A _1+ \eta ) \wedge ( A _1+ \eta ))
\\&= \tilde B_1 +\kappa_{*}^{sym}( A _1\wedge \eta )+\kappa_{*}^{skew}( A _1 \wedge  \eta )+\tfrac{1}{2}\kappa_{*}( \eta  \wedge  \eta )
\\&= \tilde B_1'
\end{align*}
and, similar, $i_{*}(b_2)+B^{(A_2',\chi_2')}=\tilde B'_2$.
Claims (a) and (b) follow readily from the properties of the adjusted shift of connections on principal $\Gamma\!$-bundles. 
\end{proof}

\begin{remark}
\label{adjusted-shift-of-i-of-a-morphism}
The adjusted shift applies, in particular, to the case when $\mathcal{A}=i_{*}(\mathcal{B})$ for an isomorphism $\mathcal{B}: \mathcal{G}_1 \to \mathcal{G}_2$ of $A$-bundle gerbes. Then, $i_{*}(\mathcal{B})^{\eta}: i_{*}^{(\eta,\chi)}(\mathcal{G}_1) \to i_{*}^{(\eta,\chi)}(\mathcal{G}_2)$ for any $\chi\in \Omega^2(M,\mathfrak{a})$. Since adjusted shifts  preserve connection-preserving morphisms, the extension $i_{*}(\beta)$ of a 2-morphism $\varphi: \mathcal{B} \Rightarrow \mathcal{B}'$ yields a 2-morphism $i_{*}(\varphi): i_{*}(\mathcal{B})^{\eta} \Rightarrow i_{*}(\mathcal{B}')^{\eta}$. Consequently, the adjusted extension defines a functor
\begin{equation*}
i_{*}^{(\eta,\chi)}: \grbcon AM \to  \grbcon{\Gamma\!,u,\kappa}M\text{.} 
\end{equation*}
\end{remark}

Next, we consider, for a 1-form $\beta\in \Omega^1(M,\mathfrak{f})$, a bicategory $\grbcon{\Gamma\!,u,\kappa}M^{\pcurved\beta}$ analogous  to $\grbcon{\Gamma\!,u,\kappa}M^{\pflat}$,  with the new condition that the section $\sigma$ does not need to be flat but has \quot{covariant derivative} $\beta$, i.e.  $\beta = \sigma^{*}\omega  \in \Omega^1(M,\mathfrak{f})$, where $\omega$ is the connection on $p_{*}(\mathcal{G})$. Now, it is the \emph{adjusted} extension functor
\begin{equation*}
i_{*}^{(q_u(\beta),0)}: \grbcon AM \to \grbcon{\Gamma\!,u,\kappa}M 
\end{equation*}   
that lifts canonically to a functor $i^{\pcurved{\beta}}_{*}$ with values in the bicategory $\grbcon{\Gamma\!,u,\kappa}M^{\pcurved\beta}$, as $p_{*}(i^{(q_u(\beta),0)}(\mathcal{G}))=\trivlin_1^{\beta}$ by \cref{remark-adjusted-extension:3}, and the canonical section $x \mapsto (x,1)$ has covariant derivative $\beta$.  

\begin{theorem}
\label{gerbe-morphisms-under-i-non-flat}
Let $\Gamma$ be a smoothly separable, central Lie 2-group equipped with a splitting $u$ and an adapted adjustment $\kappa$. Let $\beta\in \Omega^1(M,\mathfrak{f})$ be a 1-form.
The functor $i^{\pcurved{\beta}}_{*}$ is an equivalence of categories
\begin{equation*}
\grbcon{A}{M} \cong \grbcon{\Gamma\!,u,\kappa}M^{\pcurved\beta}\text{.} 
\end{equation*}
\end{theorem}

\begin{proof}
We adapt the proof of \cref{gerbe-morphisms-under-i}. The relationship between the map $\tilde\sigma:Y \to F$ (see the proof of \cref{gerbe-morphisms-under-i}) and the section $\sigma$ is now given by $\pi^{*}\beta = \mathrm{Ad}^{-1}_{\tilde\sigma}(p_{*}(A))+\tilde\sigma^{*}\theta$. 
We set $\eta := q_u(\beta)\in \Omega^1(M,\mathfrak{g})$. We now pull back along $\rho:Y' \to Y$, again using results and notation of the proof of \cref{gerbe-morphisms-under-i}, obtaining
\begin{equation*}
p_{*}(\rho^{*}\pi^{*}\eta )= \rho^{*}(\mathrm{Ad}^{-1}_{\tilde\sigma}(p_{*}(A))+\tilde\sigma^{*}\theta)=p_{*}(\mathrm{Ad}^{-1}_{g}(\rho^{*}A)+g^{*}\theta)\text{.}
\end{equation*}
Thus, defining $\varphi := u(\mathrm{Ad}_g(\pi'^{*}\eta) -\rho^{*}A-g^{*}\bar\theta)\in \Omega^1(Y',\mathfrak{h})$, we obtain the equation
\begin{equation*}
 \rho^{*}A+t_{*}\varphi= \mathrm{Ad}_g(\pi'^{*}\eta) -g^{*}\bar\theta
\end{equation*}
which we recognize as the condition for a gauge transformation
\begin{equation*}
(g,\varphi): (\pi'^{*}\eta,B') \to (\rho^{*}A,\rho^{*}B)\text{,}
\end{equation*}
where $B'\in \Omega^2(Y',\mathfrak{h})$ is to be determined from the condition that $(g,\varphi)$ be adjusted. Invoking \cref{remark-adj-gt-1}, we set
\begin{equation*}
B' :=\mathrm{d}\varphi + \tfrac{1}{2}[\varphi \wedge \varphi]+ \alpha_{*}(\rho^{*}A \wedge \varphi)+\rho^{*} B-\kappa(g,\pi'^{*}(\mathrm{d}\eta + \tfrac{1}{2}[\eta\wedge \eta]))\text{.} 
\end{equation*}
We now consider the following $\Gamma\!$-bundle gerbe 
\begin{equation*}
\mathcal{G}'=(Y',\pi',(\pi'^{*}\eta,B'), P',\mu')\text{,}
\end{equation*}
with $P'$ and $\mu'$ defined as in the proof of \cref{gerbe-morphisms-under-i}. 
However,  $P'$ now carries a $\pr_1^{*}\pi'^{*}\eta$-$\pr_2^{*}\pi'^{*}\eta$-connection that is $\pr_1^{*}B'$-$\pr_2^{*}B'$-adjusted. As in the proof  of  \cref{gerbe-morphisms-under-i}, there is a strict 1-morphism  $\mathcal{R}:=(\rho, \trivlin_g^{\varphi},\nu):\mathcal{R}:\mathcal{G}' \to \mathcal{G}$.

Next, we compute, generalizing the calculation from the proof of \cref{gerbe-morphisms-under-i}, 
\begin{equation*}
t_{*}(B')=\tfrac{1}{2}\pi'^{*} t_{*}u([q_u(\beta)\wedge q_u(\beta)])=\tfrac{1}{2}\pi'^{*}t_{*}\kappa_{*}(\eta \wedge \eta)\text{.} 
\end{equation*} 
Thus, there exists a unique $b \in \Omega^2(Y',\mathfrak{a})$ such that
\begin{equation*}
B' = i_{*}(b) +\tfrac{1}{2}\pi'^{*}\kappa_{*}(\eta \wedge \eta)\text{.}  
\end{equation*}
Furthermore, as in \cref{gerbe-morphisms-under-i}, $p_{*}P'=1$.
By \cref{reduction-with-shifted-connection}, there exists a principal $A$-bundle $P^{\mathrm{red}}$ with connection of curvature $\pr_1^{*}b - \pr_2^{*}b$, together with a connection-preserving bundle isomorphism $i_{*}^{\eta}(P^{\mathrm{red}}) \cong P'$. This shows that $\mathcal{G}''=(Y',\pi',b,P^{\mathrm{red}},\mu^{\mathrm{red}})$ is an $A$-bundle gerbe, and $i_{*}^{(\eta,0)}(\mathcal{G}'')\cong \mathcal{G}'$. This shows that the functor $i_{*}^{\pcurved\beta}$ is essentially surjective. The proof that $i_{*}^{\pcurved\beta}$ is full and faithful proceeds analogously to the one for $i_{*}^{\pflat}$ in \cref{gerbe-morphisms-under-i}.
\end{proof}

We also generalize \cref{corr-1,corr-2}, yielding the following results:

\begin{corollary}
If $\mathcal{G}$ is a $\Gamma\!$-bundle gerbe with adapted connection, then every point $x\in M$ has an open neighborhood $U$ with a 1-form $\eta\in \Omega^1(U,\mathfrak{g})$ and a 2-form $\rho\in \Omega^2(U,\mathfrak{a})$ such that $\mathcal{G}|_U \cong i_{*}^{(\eta,\rho)}(\mathcal{I}_{0})$. 
\end{corollary}

\begin{corollary}
Suppose $\mathcal{G}$ is a $\Gamma\!$-bundle gerbe with an adapted connection over $M$. Then, any section $\sigma:M \to p_{*}(\mathcal{G})$ determines a 1-form $\eta\in \Omega^1(M,\mathfrak{g})$ and a 2-form $\rho\in \Omega^2(M,\mathfrak{a})$, and an $A$-bundle gerbe $\mathcal{G}^{\mathrm{red}}$ with connection over $M$ together with an isomorphism $i_{*}^{(\eta,\rho)}(\mathcal{G}^{\mathrm{red}}) \cong \mathcal{G}$.
\end{corollary}

\subsection{The trivialization trick}

The bicategory $\grbcon{A}M$ of $A$-bundle gerbes with connections is symmetric monoidal \cite{waldorf1}, whereas $\grbcon{\Gamma\!,\kappa}M$ fails to be monoidal (unless $\Gamma$ is commutative). Thus, for $A$-bundle gerbes $\mathcal{G}$ and $\mathcal{H}$, there is no natural way to relate $i_{*}(\mathcal{G} \otimes \mathcal{H})$ with $i_{*}(\mathcal{G})$ or $i_{*}(\mathcal{H})$. However, if $i_{*}(\mathcal{G})$ is trivializable, one expects $i_{*}(\mathcal{G} \otimes \mathcal{H})$ to be isomorphic to $i_{*}(\mathcal{H})$. Our current framework does not yet capture this intuition; the following lemma constructs the missing isomorphism.    

\begin{lemma}
\label{versuchs-iso}
Suppose $\mathcal{G}$ and $\mathcal{H}$ are $A$-bundle gerbes with connection over $M$. Suppose 
\begin{equation*}
\mathcal{T}: i_{*}(\mathcal{G}) \to i_{*}^{(A,\chi)}(\mathcal{I}_0)
\end{equation*}
is a trivialization for some $A\in \Omega^1(M,\mathfrak{g})$ and $\chi\in \Omega^2(M,\mathfrak{a})$. Then, there exists a canonical strict 1-isomorphism
\begin{equation*}
\mathcal{T}\otimes \id_{\mathcal{H}}: i_{*}(\mathcal{G} \otimes \mathcal{H}) \to  i_{*}^{(A,\chi)}(\mathcal{H})
\end{equation*}
in $\grbcon{\Gamma\!,\kappa}M$ with the following properties:
\begin{enumerate}[(a)]

\item 
\label{versuchs-iso:1}
$p_{*}(\mathcal{T}\otimes \id_{\mathcal{H}})=p_{*}(\mathcal{T})$.  

\item
\label{versuchs-iso:2}
It is natural in $\mathcal{H}$: for any 1-morphism $\mathcal{A}:\mathcal{H}_1 \to \mathcal{H}_2$, there exists a canonical 2-morphism such that the diagram
\begin{equation*}
\alxydim{@C=5em}{i_{*}(\mathcal{G} \otimes \mathcal{H}_1) \ar[r]^-{\mathcal{T} \otimes \id_{\mathcal{H}_1}} \ar[d]_{i_{*}(\id \otimes \mathcal{A})} & i_{*}^{(A,\chi)}(\mathcal{H}_1) \ar@{=>}[dl]|{\beta_{\mathcal{A}}} \ar[d]^{i_{*}(\mathcal{A})^{A}} \\ i_{*}(\mathcal{G} \otimes \mathcal{H}_2) \ar[r]_-{\mathcal{T} \otimes \id_{\mathcal{H}_2}} & i_{*}^{(A,\chi)}(\mathcal{H}_2)}
\end{equation*}
commutes, where $i_{*}(\mathcal{A})^{A}$ is the adjusted shift from \cref{shift-of-1-morphisms}; see \cref{adjusted-shift-of-i-of-a-morphism}. 

\item
\label{versuchs-iso:3}
If $\mathcal{H}=\mathcal{I}_{\rho}$, for a 2-form $\rho\in\Omega^2(M,\mathfrak{a})$, then $\mathcal{T} \otimes \id_{\mathcal{I}_{\rho}}=\mathcal{T}$.

\end{enumerate}
\end{lemma}

We remark that $\mathcal{T}\otimes \id_{\mathcal{H}}$ is merely a notation; there is no tensor product between $\Gamma\!$-bundle gerbes and their morphisms. Furthermore, from \cref{remark-adjusted-extension:2}, we have
\begin{equation*}
i_{*}^{(A,\chi)}(\mathcal{I}_0)=\mathcal{I}_{A,B^{(A,\chi)}}\text{,}
\end{equation*}
with $B^{(A,\chi)}:= i_{*}(\chi)+\tfrac{1}{2}\kappa_{*}(A \wedge A)$. 
All bundle gerbes in \cref{versuchs-iso} are fake-flat; in particular, $(A,B^{(A,\chi)})$ is a fake-flat $\Gamma\!$-connection (since $\mathcal{I}_{A,B^{(A,\chi)}}$ is isomorphic to $i_{*}(\mathcal{G})$, which is fake-flat by assumption).

\begin{proof}
Suppose $\mathcal{G}=(Y,\pi,b,P,\mu)$ and $\mathcal{H}=(Y',\pi',b',P',\mu')$.
Recall that 
\begin{equation*}
i_{*}^{(A,\chi)}(\mathcal{H})=(Y',\pi',\pi'^{*}A,i_{*}(b')+\pi'^{*}B^{(A,\chi)},i_{*}^{A}(P'),i_{*}(\mu'))\text{.}
\end{equation*}
Let the trivialization $\mathcal{T}: i_{*}(\mathcal{G}) \to i_{*}^{(A,\chi)}(\mathcal{I}_0)$, viewed as a strict 1-morphism, be $\mathcal{T}=(\pi,Q,\nu)$, where $Q$ is a principal $\Gamma\!$-bundle over $Y$ with an $i_{*}(b)$-$\pi^{*}B^{(A,\chi)}$-adjusted $0$-$\pi^{*}A$-connection, and $\nu$ is an isomorphism
\begin{equation*}
\nu:\pr_2^{*}Q \otimes i_{*}(P) \to \pr_1^{*}Q\text{.}
\end{equation*}
We define $\mathcal{T}\otimes \id_{\mathcal{H}}$ as a strict 1-morphism using the projection $\pr_2: Y \times_X Y' \to Y'$, the principal $\Gamma\!$-bundle $\tilde Q := \pr_1^{*}Q$ over $Y \times_X Y'$, and the bundle isomorphism $\tilde\nu$ defined fibrewise over $((y_1,y_1'),(y_2,y_2'))$ by
\begin{equation*}
\alxydim{@C=1em}{\tilde Q_{y_2,y_2'} \otimes i_{*}(P_{y_1,y_2} \otimes P'_{y_1',y_2'}) \ar@{=}[r] & Q_{y_2}\otimes i_{*}(P_{y_1,y_2} )\otimes i_{*}(P'_{y_1',y_2'} )\ar[d]^{ \nu_{y_1,y_2} \otimes \id} \\ &
 Q_{y_1}\otimes i_{*}(P'_{y_1',y_2'} )
\ar[d]^{\beta \text{ from \cref{tensor-product-with-abelian-bundle}}}\\ & i_{*}^{ A}(P'_{y_1',y_2'} ) \otimes Q_{y_1} \ar@{=}[r] & i_{*}^{ A}(P'_{y_1',y_2'}) \otimes \tilde Q_{y_1,y_1'}\text{.}}
\end{equation*}
Note that $\tilde Q$ carries a $0$-$\pi'^{*}A$-connection that is $\pr_1^{*}i_{*}(b)$-$\pi'^{*}B^{(A,\chi)}$-adjusted, and hence also $i_{*}(\pr_1^{*}b+\pr_2^{*}b')$-$(\pi'^{*}B^{(A,\chi)}+i_{*}(b'))$-adjusted by \cref{shift-of-curving}. The condition from \cref{eq:refinementdiag-strict} in \cref{strict-1-morphism} reads:
\begin{equation*}
\alxydim{@C=6em}{Q{_{y_3}}\otimes i_{*}(P_{y_2,y_3} )\otimes i_{*}(P'_{y_2',y_3'}) \otimes i_{*}(P_{y_1,y_2} )\otimes i_{*}( P'_{y_1',y_2'}) \ar[d]_{\nu_{y_2,y_3} \otimes \id \otimes \id \otimes \id} \ar[r]^-{\id \otimes i_{*}(\mu_{\mathcal{G} \otimes \mathcal{H}})} & Q{_{y_3}}\otimes i_{*}(P_{y_1,y_3} )\otimes i_{*}(P'_{y_1',y_3'})  \ar[dd]^{\nu_{y_1,y_3} \otimes \id}\\  Q_{y_2} \otimes  i_{*}(P'_{y_2',y_3'}) \otimes i_{*}(P_{y_1,y_2} )\otimes i_{*}( P'_{y_1',y_2'}) \ar[d]_{\beta \otimes \id \otimes \id}  \\   i_{*}^{A}(P'_{y_2',y_3'}) \otimes Q_{y_2}\otimes  i_{*}(P_{y_1,y_2} )\otimes i_{*}( P'_{y_1',y_2'})\ar[d]_{\id \otimes \nu_{y_1,y_2} \otimes \id} &  Q_{y_1} \otimes i_{*}(P'_{y_1',y_3'})  \ar[dd]^{\beta} \\ i_{*}^{A}(P'_{y_2',y_3'}) \otimes Q_{y_1}\otimes i_{*}( P'_{y_1',y_2'}) \ar[d]_{\id \otimes \beta } & \\ i_{*}^{A}(P'_{y_2',y_3'}) \otimes  i_{*}^{A}(P'_{y'_1,y'_2} )\otimes Q_{y_1}  \ar[r]_-{\mu'\otimes \id} & i_{*}^{A}(P'_{y_1',y_3'}) \otimes  Q{_{y_1}}}
\end{equation*}
The top of the diagram involves the bundle gerbe product $\mu_{\mathcal{G} \otimes \mathcal{H}}$ of $\mathcal{G}\otimes \mathcal{H}$, given by  
\begin{equation*}
\alxydim{}{P_{y_2,y_3} \otimes P'_{y_2',y_3'} \otimes P_{y_1,y_2} \otimes P'_{y_1',y_2'} \ar[r]^{\id \otimes \beta \otimes \id} &  P_{y_2,y_3}  \otimes P_{y_1,y_2}\otimes P'_{y_2',y_3'} \otimes P'_{y_1',y_2'} \ar[r]^-{\mu \otimes \mu'} & P_{y_1,y_3}   \otimes P'_{y_1',y_3'}\text{,}}
\end{equation*}
where $\beta$ denotes the standard braiding of principal $A$-bundles, which coincides with the half-braiding by \cref{braiding-abelian}. The diagram commutes due to the commutativity of the corresponding diagram for $\nu$, the naturality of the half-braiding $\beta$, and the hexagon axioms for $\beta$ (\cref{reduction-to-abelian}).
This completes the construction of the strict 1-morphism $\mathcal{T}\otimes \id_{\mathcal{H}}$.

\cref{versuchs-iso:1*,versuchs-iso:3*} are evident. It suffices to prove  \cref{versuchs-iso:2*} in the case where $\mathcal{A}$ is a strict 1-morphism, i.e. $\mathcal{A}=(\rho,R,\chi)$. If $\mathcal{H}_i=(\pi'_i,Y_i',b_i',P_i',\mu_i')$, then  $\rho: Y_1' \to Y_2'$, $R$ is a principal $A$-bundle over $Y_1'$ with connection of curvature $b_1'-b_2'$. Then, the clockwise composition yields the strict 1-morphism consisting of the map  $\rho \circ \pr_2: Y \times_X Y_1' \to Y_2'$, the principal $\Gamma\!$-bundle $\pr_2^{*}i_{*}(R)^{A} \otimes \pr_1^{*} Q$, and the bundle isomorphism given over a point $((y,y_1'),(\tilde y, \tilde y_1'))$ by $\chi_{y_1',\tilde y_1'} \circ \nu_{y,\tilde y}$. The counter-clockwise  composition yields the strict 1-morphism consisting of the map $\pr_2 \circ (\id \times \rho): Y \times_X Y_1' \to Y_2'$, the principal $\Gamma\!$-bundle $\pr_1^{*}Q \otimes \pr_2^{*}i_{*}(R)$, and the bundle isomorphism given by $\nu_{y,\tilde y} \circ \chi_{y_1'\tilde y_1'}$. The two maps $Y \times_X Y_1' \to Y_2'$  are identical, and the half-braiding establishes a connection-preserving isomorphism
\begin{equation*}
\pr_2^{*}i_{*}(R)^{A} \otimes \pr_1^{*}Q \to \pr_1^{*}Q \otimes \pr_2^{*}i_{*}(R) 
\end{equation*}   
by \cref{tensor-product-with-abelian-bundle}. Its naturality ensures it defines a 2-morphism. 
\end{proof}

\begin{remark}
If the trivialization $\mathcal{T}$ is a refinement, i.e., $Q=\trivlin_g^{\varphi}$ for a gauge transformation $(g,\varphi):(0,i_{*}(b)) \to \pi^{*}(A,B^{(A,\chi)})$ over $Y$, and $\nu$ is an isomorphism $\nu: \pr_2^{*}\trivlin_{g}^{\varphi} \otimes i_{*}(P) \to  \pr_1^{*}\trivlin_g^{\varphi}$ over $Y^{[2]}$, then $\mathcal{T} \otimes \id_{\mathcal{H}}$ is again a refinement. The gauge transformation is
\begin{equation*}
\pr_1^{*}(g,\varphi):(0,i_{*}(\pr_1^{*}b+\pr_2^{*}b')) \to \pr_2^{*}(\pi'^{*}A,i_{*}(b')+\pi'^{*}B^{(A,\chi)})
\end{equation*}
over $Y \times_X Y'$, and the isomorphism over $((y,y'),(\tilde y,\tilde y'))\in (Y \times_X Y') \times_X (Y \times_X Y')$ is
\begin{equation*}
\alxydim{@C4em}{(\trivlin_{g}^{\varphi})_{\tilde y} \otimes i_{*}(P_{y,\tilde y} \otimes P'_{y',\tilde y'})\ar[r]^-{\nu_{y,\tilde y}\otimes \id}& (\trivlin_{g}^{\varphi})_{y} \otimes i_{*}(P'_{y',\tilde y'}) \ar[r]^{\beta} &  i_{*}^{A}(P'_{y',\tilde y'}) \otimes (\trivlin_{g}^{\varphi})_{y}.}
\end{equation*} 
\end{remark}

\setsecnumdepth{2}

\section{The multiplicative bundle gerbe of a Lie 2-group}

\label{multiplicative-bundle-gerbe}

Let $\Gamma$ be a smoothly separable, central Lie 2-group equipped with an adjustment $\kappa$ adapted to a splitting $u$.
In \cref{the-bundle-gerbe-of-a-crossed-module,multiplicativity}, we review the construction of an associated strictly multiplicative $A$-bundle gerbe $\mathcal{G}_{\Gamma}$ with connection over $F$. Without connections, this construction is due to \cite{Nikolausa}; the incorporation of connections (using the adjustment) is due to Téllez-Domínguez \cite{Tellez2023}.
In \cref{the-trivialization}, we show that the $\Gamma\!$-bundle gerbe $i_{*}(\mathcal{G}_{\Gamma})$ is trivializable and demonstrate how this trivialization is compatible with the multiplicative structure. Without connections, this result again appears in \cite{Nikolausa}; the extension to connections is new. The material presented here is essential for the lifting theory developed in \cref{lifting-theory}. 

\subsection{The bundle gerbe of a crossed module}

\label{the-bundle-gerbe-of-a-crossed-module}
 
We construct an $A$-bundle gerbe $\mathcal{G}_{\Gamma}$ from the crossed module $\Gamma$, reviewing \cite{Nikolausa}. The surjective submersion of $\mathcal{G}_{\Gamma}$ is $p: G \to F$, and its principal $A$-bundle over $G \times_F G$ has total space $\Gamma_1 := H \times G$, projection $(h,g) \mapsto (g,t(h)g)$, and $A$-action $(h,g)\cdot a := (ha,g)$.
The bundle gerbe product is the composition in $\Gamma$, i.e.,
\begin{equation*}
\mu((h_{23},g_{23}) ,(h_{12},g_{12})) := (h_{23},g_{23})\circ (h_{12},g_{12})=(h_{23}h_{12},g_{12})\text{.}
\end{equation*}

\begin{remark}
\label{semi-direct-product}
The notation $\Gamma_1$ reflects that $\Gamma_1=H \times G$ is the Lie group of morphisms when Lie 2-groups are regarded as groupoids internal to the category of Lie groups. The group structure is given by the semi-direct product
\begin{equation*}
(h_1,g_1) \cdot (h_2,g_2) := (h_1\alpha(g_1,h_2),g_1g_2)\text{.}
\end{equation*}
\end{remark}

In order to define a connection on $\mathcal{G}_{\Gamma}$, we use the splitting $u:\mathfrak{g} \to \mathfrak{h}$, and its associated retract $j:=j_u$; see \cref{Lie-2-groups}.
Following \cite[Prop. 3.9]{Tellez2023}, a connection $\omega$ on the principal $A$-bundle $\Gamma_1$ over $G \times_F G$ is defined by setting
\begin{equation}
\label{definition-of-omega}
\omega := j((\alpha_{\pr_2}^{-1})_{*}(\pr_1^{*}\theta)) \in \Omega^1(\Gamma_1,\mathfrak{a})\text{,}
\end{equation}
where $\theta \in \Omega^1(H,\mathfrak{h})$ is the Maurer-Cartan form on $H$.
One can easily check (using that $\Gamma$ is central) that this is indeed a connection,
and that the isomorphism $\mu$ is connection-preserving.

Next we describe how an adjustment adds a curving, i.e., a 2-form $B_{\Gamma}\in \Omega^2(G)$ such that \cref{relation-between-B-and-omega} holds, i.e. $\mathrm{d}\omega=\pr_1^{*}B_{\Gamma}-\pr_2^{*}B_{\Gamma}$. Curvings always exist \cite{murray}, but they are only  unique up to the addition of pullbacks of 2-forms from $F$. 
We recall that, for any curving, $\mathrm{d}B_{\Gamma}$ descends to a unique 3-form $H_{\Gamma}\in \Omega^3(F)$, the curvature of $\mathcal{G}_{\Gamma}$. 

\begin{lemma}
\label{curvature-of-curving}
The 2-form
\begin{equation*}
B_{\Gamma} := \tfrac{1}{2}j(\kappa_{*}(\theta \wedge \theta)) \in \Omega^2(G,\mathfrak{a})
\end{equation*}
is a curving for the connection $\omega$ on $\mathcal{G}_{\Gamma}$, and its curvature is
\begin{equation*}
H_{\Gamma}:=- \tfrac{1}{6}b_{\kappa}(\theta\wedge [\theta\wedge\theta])\in \Omega^3(F,\mathfrak{a})\text{,}
\end{equation*}
where $b_{\kappa}$ is the symmetric invariant bilinear form associated with $\kappa$.
\end{lemma}

\begin{proof}
To verify \cref{relation-between-B-and-omega}, one can check that both sides, $\mathrm{d}\omega$ and $\pr_1^{*}B_{\Gamma}-\pr_2^{*}B_{\Gamma}$, are equal to  
\begin{equation}
\label{curvature-of-the-connection}
\mathrm{d}\omega = -\tfrac{1}{2}j(\alpha_{\pr_2}^{-1})_{*}(\pr_1^{*}[\theta\wedge \theta])-j\alpha_{*}(\pr_2^{*}\theta \wedge (\alpha_{\pr_2}^{-1})_{*}(\pr_1^{*}\theta))=\pr_1^{*}B_{\Gamma}-\pr_2^{*}B_{\Gamma}\text{.} 
\end{equation}
One can then easily derive using \cref{properties-of-kappa-stern:5} that $\mathrm{d}B_{\Gamma}=-\tfrac{1}{4}j(\kappa_{*}(\theta\wedge[\theta\wedge\theta]))$. 
It is then a straightforward matter of using the definition of $b_{\kappa}$ as the symmetrization of $\kappa_{*}$ to show that this is indeed equal to the claimed expression.
\end{proof}

\subsection{Multiplicativity}

\label{multiplicativity}

A multiplicative structure on an  $A$-bundle gerbe  $\mathcal{G}$ with connection over $F$ consists of a 2-form $\rho\in\Omega^2(F \times F,\mathfrak{a})$, a 1-isomorphism
\begin{equation*}
\mathcal{M}:\pr_1^{*}\mathcal{G} \otimes \pr_2^{*}\mathcal{G} \to m^{*}\mathcal{G} \otimes \mathcal{I}_{\rho}
\end{equation*}
over $F \times F$, where $m: F \times F \to F$ denotes the multiplication of the group $F$, and a certain \quot{associativity} 2-isomorphism over $F \times F \times F$; see \cite[Def. 1.3]{waldorf5}. 
In certain cases, the multiplicative structure is strict (see \cite[\S 2]{Waldorf}), meaning that the 1-isomorphism $\mathcal{M}$ is induced from a refinement (\cref{refinement}), and the associativity 2-isomorphism is just a strict associativity condition. 

For the bundle gerbe $\mathcal{G}_{\Gamma}$, this is indeed the case: we set 
\begin{equation}
\label{definition-of-rho}
\rho_{\Gamma} := b_{\kappa}(\pr_2^{*}\bar\theta \wedge \pr_1^{*}\theta)\in \Omega^2(F \times F,\mathfrak{a})\text{,}
\end{equation}
where $\theta$ and $\bar\theta$ denote the left and right invariant Maurer-Cartan forms, respectively. We collect the following structure for a refinement $\mathcal{M}_{\Gamma}:=(f,(1,\tau),\nu)$ in the notation of \cref{refinement}. The first entry is the map $f := (p\times p,m)$ between the domains of the surjective submersions of the bundle gerbes $\pr_1^{*}\mathcal{G} \otimes \pr_2^{*}\mathcal{G}$ and $m^{*}\mathcal{G} \otimes \mathcal{I}_{\rho}$,
\begin{equation*}
\alxydim{@C=5em}{
G \times G \ar[d]_{p \times p} \ar[r]^-{f} & (F \times F) \ttimes mp G \ar[d]^{\pr_1}
\\ F \times F \ar@{=}[r] & F \times F\text{.}}
\end{equation*}
The third entry is the following 1-form $\tau$, defined in \cite{Tellez2023}, on the domain of the surjective submersion of the source bundle gerbe:
\begin{equation}
\label{definition-of-tau}
\tau := j \kappa(\inv \circ \pr_2,\pr_1^{*}\theta) \in \Omega^1(G \times G,\mathfrak{a})\text{.}
\end{equation}
For the last entry, we observe that the principal $A$-bundle of $\pr_1^{*}\mathcal{G} \otimes \pr_2^{*}\mathcal{G}$ lives over $G^2{\times_{F^2}}G^2$, and is given  by $\pr_{13}^{*}\Gamma_1 \otimes_A \pr_{24}^{*}\Gamma_1$.  
The last entry is the isomorphism 
\begin{equation}
\label{iso-nu-of-mult-structure}
\nu_{\Gamma}: {}^{\pr_2^{*}(1,\tau)}(\pr_{13}^{*}\Gamma_1 \otimes_A \pr_{24}^{*}\Gamma_1) \to  (m^{[2]})^* (\Gamma_1^{\pr_1^{*}(1,\tau)})  
\end{equation}
between principal $A$-bundles over $G^2{\times_{F^2}}G^2$, which is induced by the multiplication $\Gamma_1 \times \Gamma_1 \to \Gamma_1$ of the Lie group $\Gamma_1$; see \cref{semi-direct-product}. 
For this, it is important that we assumed $\Gamma$ to be central, because only then the multiplication map is $A$-equivariant and induces a bundle morphism. 

\begin{lemma}
\label{multiplication-refinement}
The quadruple $\mathcal{M}_{\Gamma}=(f,(1,\tau),\nu_\Gamma)$ is a refinement 
\begin{equation*}
\pr_1^{*}\mathcal{G}_{\Gamma} \otimes \pr_2^{*}\mathcal{G}_{\Gamma} \to m^{*}\mathcal{G}_{\Gamma} \otimes \mathcal{I}_{\rho_{\Gamma}}
\end{equation*}
and equips $\mathcal{G}_{\Gamma}$ with a strictly multiplicative structure. 
\end{lemma}

\begin{proof}
We need to verify the following three conditions. First, the pair $(1,\tau)$ must be a gauge transformation 
\begin{equation*}
(0,\pr_1^{*}B_{\Gamma} + \pr_2^{*}B_{\Gamma}) \to (0,m^{*}B_{\Gamma}+(p \times p)^{*}\rho_{\Gamma})\text{,}
\end{equation*}
which is equivalent to the condition
\begin{equation}
\label{multiplicativity-of-B}
\pr_1^{*}B_{\Gamma} + \pr_2^{*}B_{\Gamma}-(p \times p)^{*}\rho_{\Gamma} - m^{*}B_{\Gamma} =\mathrm{d}\tau\text{.} 
\end{equation}
Indeed, a first calculation reveals 
\begin{equation*}
\pr_1^{*}B_{\Gamma} + \pr_2^{*}B_{\Gamma}-m^{*}B_{\Gamma}=-\tfrac{1}{2} j\kappa(\inv \circ \pr_2,[\pr_1^{*}\theta\wedge \pr_1^{*}\theta])- j\kappa_{*}^{skew}(\Ad^{-1}_{\pr_2}(\pr_1^{*}\theta) \wedge  \pr_2^{*}\theta)
\end{equation*}
and a second calculation shows that
\begin{equation*}
\mathrm{d}\tau =-j\kappa_{*}(\pr_2^{*}\theta \wedge \Ad_{\pr_2}^{-1}(\pr_1^{*}\theta))-\tfrac{1}{2}j  \kappa(\inv \circ \pr_2,[\pr_1^{*}\theta\wedge \pr_1^{*}\theta])\text{.} 
\end{equation*}
The difference between these expressions is 
\begin{align*}
\pr_1^{*}B_{\Gamma} + \pr_2^{*}B_{\Gamma}-m^{*}B_{\Gamma}-\mathrm{d}\tau 
&=- j\kappa_{*}^{skew}(\Ad^{-1}_{\pr_2}(\pr_1^{*}\theta) \wedge  \pr_2^{*}\theta)+j\kappa_{*}(\pr_2^{*}\theta \wedge \Ad_{\pr_2}^{-1}(\pr_1^{*}\theta))
\\&=-\tfrac{1}{2} j\kappa_{*}(\Ad^{-1}_{\pr_2}(\pr_1^{*}\theta) \wedge  \pr_2^{*}\theta) +\tfrac{1}{2}j\kappa_{*}(\pr_2^{*}\theta\wedge \Ad^{-1}_{\pr_2}(\pr_1^{*}\theta))
\\&= j\kappa_{*}^{sym}(\pr_2^{*}\theta \wedge \Ad^{-1}_{\pr_2}(\pr_1^{*}\theta))
\\&= b_\kappa(p^{*}\pr_2^{*}\theta \wedge \Ad^{-1}_{p \circ \pr_2}(p^{*}\pr_1^{*}\theta))
\\&= p^{*}b_\kappa(\pr_2^{*}\bar\theta \wedge \pr_1^{*}\theta)
\\&= (p \times p)^{*}\rho_{\Gamma}\text{;} 
\end{align*}
this shows \cref{multiplicativity-of-B}.

Second, we must show that $\nu_\Gamma$ is connection-preserving. This amounts to verifying the identity 
\begin{equation}
\label{multiplicativity-of-omega}
\tau_{t(h_1)g_1,t(h_2)g_2}+\omega_{h_1,g_1} + \omega_{h_2,g_2}  = \omega_{h_1\alpha(g_1,h_2),g_1g_2}+ \tau_{g_1,g_2}
\end{equation}
at each point $((g_1,h_1),(g_2,h_2))$ of $\Gamma_1 \times \Gamma_1$. 
To prove it, we calculate
\begin{multline*}
\omega_{h_1,g_1} + \omega_{h_2,g_2} - \omega_{h_1\alpha(g_1,h_2),g_1g_2}\\=j((\alpha_{g_1})_{*}^{-1}(\theta_{h_1})-(\alpha_{g_2})_{*}^{-1}(\Ad^{-1}_{h_2}((\alpha_{g_1})^{-1}_{*}(\theta_{h_1}))+(\tilde\alpha_{h_2})_{*}(\theta_{g_1})))
\end{multline*}
and obtain the same result for $\tau_{g_1,g_2}-\tau_{t(h_1)g_1,t(h_2)g_2}$; this shows \cref{multiplicativity-of-omega}.

Third, 
strict associativity of the refinement $\mathcal{M}_{\Gamma}$ requires 
\begin{equation}
\label{associativity-refinement-connection}
\tau_{g_1,g_2g_3} + \tau_{g_2,g_3} = \tau_{g_1g_2,g_3} + \tau_{g_1,g_2}\text{,}
\end{equation}
and the associativity of the bundle isomorphism $\nu_\Gamma$. \cref{associativity-refinement-connection} is straightforward to check, and the associativity of $\nu_\Gamma$ follows from the associativity law of the group $\Gamma_1$.
\end{proof}

\begin{remark}
Going beyond the presentation here, Téllez-Domínguez has proved that equivalence classes of adjustments on $\Gamma$ are in bijection with equivalence classes of connections on $\mathcal{G}_{\Gamma}$.
\end{remark}

\begin{remark}
The quadruple $(H_{\Gamma},\rho_{\Gamma},0,0)$ is a 4-cocycle in the Bott-Shulman-Stasheff double complex $\Omega^{\bullet}(F^{\bullet},\mathfrak{a})$ that computes $\h^4(BF,\mathfrak{a})$ \cite{bott3}. It represents the image of $b_{\kappa}\in \mathrm{Sym}^2(\mathfrak{f},\mathfrak{a})^{F}$ under the Chern-Weil homomorphism \cite{dupont1}. 
\end{remark}

\subsection{Triviality}

\label{the-trivialization}

We show that the extended $\Gamma\!$-bundle gerbe $i_{*}(\mathcal{G}_{\Gamma})$ is canonically trivializable and demonstrate how this trivialization is compatible with the multiplicative structure of $\mathcal{G}_{\Gamma}$. This extends \cite[Lemma 4.2.2]{Nikolausa} to the setting with connections.

\begin{proposition}
\label{trivialization-of-g-gamma}
Let $\Gamma$ be a smoothly separable, central crossed module with a splitting $u$ and an adjustment $\kappa$ adapted to $u$. Let $q:=q_u$ be the section of $\mathfrak{G}$ determined by $u$ (see \cref{Lie-2-groups}).
Consider the differential forms
\begin{align*}
A &:= q(\theta)\in \Omega^1(F,\mathfrak{g})
\\
B &:= \tfrac{1}{2}\kappa_{*}(q(\theta)\wedge q(\theta)) \in \Omega^2(F,\mathfrak{h})\text{.} 
\end{align*}
Then $(A,B)$ is a fake-flat $\Gamma\!$-connection on $F$, and there exists a canonical isomorphism 
\begin{equation*}
 \mathcal{T}_{\Gamma}:i_{*}(\mathcal{G}_{\Gamma}) \to \mathcal{I}_{A,B}\text{.} 
\end{equation*}
Moreover, $p_{*}(\mathcal{T}_{\Gamma})$, as an isomorphism between trivial principal $F$-bundles over $F$, is given by the inversion map $\inv: F \to F$. 
\end{proposition}

\begin{remark}
We have $B=B^{(A,0)}$; hence $\mathcal{I}_{A,B}=i_{*}^{(q(\theta),0)}(\mathcal{I}_0)$ is the adjusted extension of the trivial $A$-bundle gerbe (\cref{shifting-connections}). By \cref{extension-and-adjusted-shift-lemma}, we have
\begin{equation*}
\mathrm{fcurv}(A,B)=q(\mathrm{d}\theta + \tfrac{1}{2}[\theta \wedge\theta])=0\text{.}
\end{equation*} 
\end{remark}

We decompose the proof of \cref{trivialization-of-g-gamma} into a sequence of lemmas. The following 1-form on $G \times_F G$ will be important:
\begin{equation*}
\alpha :=- (\alpha_{\pr_2})_{*}(u(\pr_1^{*}\theta - \pr_2^{*}\theta)) \in \Omega^1(G \times_F G,\mathfrak{h})\text{.}
\end{equation*}
The following lemma follows from straightforward calculations and establishes two key properties of $\alpha$.

\begin{lemma}
\label{lemma-triv-1}
The 1-form $\alpha$ satisfies the following conditions:
\begin{enumerate}[(i)]

\item 
\label{lemma-triv-1:i}
Over $G \times_FG \times_FG$, it satisfies a \quot{twisted} cocycle condition
\begin{equation*}
\pr_{23}^{*}\alpha +(\alpha_{\Delta\circ \pr_{23}})_{*} (\pr_{12}^{*}\alpha)=\pr_{13}^{*}\alpha\text{,}
\end{equation*}
where $\Delta: G \times_F G \to G$ is defined by $\Delta(g,g') := g'g^{-1}$. 

\item
\label{lemma-triv-1:ii}
The curvature of $\alpha$ satisfies the equation
\begin{equation*}
\mathrm{d}\alpha + \tfrac{1}{2}[\alpha\wedge \alpha]  = i_{*}(\pr_1^{*}B_{\Gamma}-\pr_2^{*}B_{\Gamma})
\end{equation*}
over $G \times_F G$.

\end{enumerate}
\end{lemma}

We may consider the trivial principal $\Gamma\!$-bundle $\trivlin_{\Delta}^{\alpha}$ over $G \times_FG$. \cref{lemma-triv-1:i}, together with the obvious equality $\Delta(g_1,g_3)=\Delta(g_2,g_3)\Delta(g_1,g_2)$, induces by \cref{tensor-product-of-trivial-bundles,tensor-product-of-trivial-bundles-with-connections} a  connection-preserving bundle isomorphism
\begin{equation*}
\gamma_{\pr_{23},\pr_{12}}: \pr_{23}^{*}\trivlin^{\alpha}_{\Delta} \otimes \pr_{12}^{*}\trivlin^{\alpha}_{\Delta} \to \pr_{13}^{*}\trivlin^{\alpha}_{\Delta}\text{.}
\end{equation*}

\begin{lemma}
\label{lemma-triv-3}
\label{extension-of-bundle-is-trivializable}
There is a canonical isomorphism 
\begin{equation*}
\tau: i_{*}(\Gamma_1) \to \trivlin_{\Delta}^{\alpha}
\end{equation*}
over $G \times_F G$. It has the following two properties:
\begin{enumerate}[(a)]
\item 
\label{commutative-diagram-for-tau}
It is compatible with the bundle gerbe product $\mu$ in the sense that the diagram
\begin{equation*}
\alxydim{}{\pr_{23}^{*}i_{*}(\Gamma_1) \otimes \pr_{12}^{*}i_{*}(\Gamma_1) \ar[r]^-{\tau \times \tau} \ar[d]_{i_{*}(\mu)} & \pr_{23}^{*}\trivlin^{\alpha}_{\Delta} \otimes \pr_{12}^{*}\trivlin^{\alpha}_{\Delta} \ar[d]^{\gamma_{\pr_{23},\pr_{12}}} \\ \pr_{13}^{*}i_{*}(\Gamma_1) \ar[r]_-{\tau} & \pr_{13}^{*}\trivlin^{\alpha}_{\Delta}}
\end{equation*}
of bundle isomorphisms over $G \times_FG\times_FG$ is commutative. 

\item
\label{multiplicativity-of-tau}
It is compatible with the multiplication isomorphism $\nu_{\Gamma}$ of \cref{iso-nu-of-mult-structure} in the sense that the following diagram of bundle isomorphisms over $((g_1,g_2),(g_1',g_2'))\in G^2 \times_{F^2} G^2$ is commutative:
\begin{equation*}
\begin{tikzcd}
i_{*}(\Gamma_1|_{g_1,g_1'} \otimes \Gamma_1|_{g_2,g_2'}) \arrow[r, "\mu"] \arrow[dd, "i_{*}(\nu_\Gamma)"'] & i_{*}(\Gamma_1)_{g_1,g_1'} \otimes i_{*}(\Gamma_1)_{g_2,g_2'}  \arrow[r, "{\tau \otimes \id}"]
&
\trivlin_\Delta|_{g_1,g_1'} \otimes i_{*}(\Gamma_1)_{g_2,g_2'}
\arrow[d, "\gamma ^{-1}\otimes \id"] 
\\
&&\trivlin_\id|_{g_1'} \otimes \trivlin_\inv|_{g_1} \otimes i_{*}(\Gamma_1)_{g_2,g_2'}
  \arrow[d, "\id \otimes \beta"]
\\
i_{*}(\Gamma_1)_{g_1g_2,\,g_1'g_2'}
  \arrow[d, "\tau"']
&& \trivlin_\id|_{g_1'} \otimes i_{*}(\Gamma_1)_{g_2,g_2'} \otimes \trivlin_\inv|_{g_1}
  \arrow[d, "{\id \otimes \tau \otimes \id}"]
\\
\trivlin_\Delta|_{g_1g_2,g_1'g_2'}  & \trivlin_\id|_{g_1'} \otimes\trivlin_\Delta|_{g_1g_2,g_2'} \arrow[l, "\gamma"]   &\trivlin_\id|_{g_1'} \otimes\trivlin_\Delta|_{g_2,g_2'}\otimes \trivlin_\inv|_{g_1} \arrow[l, "\id \otimes \gamma"]
\end{tikzcd}
\end{equation*}
Here, $\mu$ is the isomorphism of \cref{monoidal-covariance} expressing the monoidality of the extension along $i$, and $\beta$ is the half-braiding of \cref{reduction-to-abelian}.

\end{enumerate}
\end{lemma} 

\begin{proof}
We describe explicitly the principal $\Gamma\!$-bundle $i_{*}(\Gamma_1)$ over $G \times_F G$ according to \cref{covariance-of-Gamma-bundles}. It has total space $i_{*}(\Gamma_1):=(H \times G \times H)/A$, where $a\in A$ acts on $(h,g,h')$ by $(h a,g,a^{-1}h')$. The anchor $i_{*}(\Gamma_1) \to G$ is $(h,g,h') \mapsto t(h')^{-1}$, the action is $(h,g,h')\cdot h'' := (h,g,h'h'')$, and the projection is $(h,g,h') \mapsto (g,t(h)g)$. 
The claimed bundle isomorphism $\tau$ is defined by
\begin{equation*}
\tau(h,g,h'):=(g,t(h)g,hh')\text{.}
\end{equation*}

Next, we include the connection $\omega$ on $\Gamma_1$ defined in \cref{definition-of-omega}. Again following \cref{covariance-of-Gamma-bundles}, the extended connection on $i_{*}(\Gamma_1)$ is, at a point $(h,g,h')\in i_{*}(\Gamma_1)$,
\begin{align*}
\omega'|_{h,g,h'} &= \Ad^{-1}_{h'}(i_{*}(\omega|_{h,g}))+\theta_{h'}=\Ad^{-1}_{h'}(i(j((\alpha_g^{-1})_{*}(\theta_h))))+\theta_{h'}=ij((\alpha_g^{-1})_{*}(\theta_h))+\theta_{h'}\text{;}
\end{align*}
and this is a $0$-$0$-connection. Checking that $\tau$ is connection-preserving amounts to showing that
\begin{equation*}
\omega' = \tau^{*}\omega_{\alpha}
\end{equation*}
where $\omega_{\alpha}$ is the connection 1-form on $\trivlin_{\Delta}^{\alpha}$, see \cref{connection-on-trivial-bundle}. This follows directly from combining the defining formulas.

\Cref{commutative-diagram-for-tau} is an easy check; see also \cite[Lemma 4.2.1]{Nikolausa}.
\Cref{multiplicativity-of-tau} is a simple check using the definitions: we start with an element 
\begin{equation*}
(((h_1,g_1),(h_2,g_2)),h')\in i_{*}(\Gamma_1|_{g_1,g_1'} \otimes \Gamma_1|_{g_2,g_2'})
\end{equation*}
i.e., $g_1'=t(h_1)g_1$ and $g_2'=t(h_2)g_2$. The left vertical arrows result in
\begin{equation*}
\tau(i_{*}(\nu_\Gamma)(((h_1,g_1),(h_2,g_2)),h'))=\tau(h_1\alpha(g_1,h_2),g_1g_2,h')=(g_1g_2,g_1'g_2',h_1\alpha(g_1,h_2)h')\text{.}
\end{equation*}
Going clockwise, we obtain
\begin{align*}
(((h_1,g_1),(h_2,g_2)),h')\; & 
\xymatrix@C=4em{\ar@{|->}[r]^-{\mu} &} && \mquad((h_1,g_1,h'),(h_2,g_2,1))
\\&\xymatrix@C=4em{\ar@{|->}[r]^-{\tau \otimes \id} &}&& \mquad((g_1,g'_1,hh'),(h_2,g_2,1))
\\&\xymatrix@C=4em{\ar@{|->}[r]^-{\gamma^{-1} \otimes \id}&}&& \mquad((g_1',h'),(g_1,\alpha(g_1^{\prime-1},h_1)),(h_2,g_2,1))
\\&&&\qquad = ((g_1',h'),(g_1,\alpha(g_1^{-1},h_1)),(h_2,g_2,1))
\\&\xymatrix@C=4em{\ar@{|->}[r]^-{\id \otimes \beta} & }&&\mquad((g_1',h'),(h_2,g_2,1),(g_1,\alpha(g_1^{-1},h_1)))
\\&\xymatrix@C=4em{\ar@{|->}[r]^-{\id \otimes \tau \otimes \id}& }&& \mquad((g_1',h'),(g_2,g_2',h_2),(g_1,\alpha(g_1^{-1},h_1)))
\\&\xymatrix@C=4em{\ar@{|->}[r]^-{\id \otimes \gamma}& }&& \mquad((g_1',h'),(g_1g_2,g_2',\alpha(g_2'g_2^{-1},\alpha(g_1^{-1},h_1))h_2))
\\&&&\qquad=((g_1',h'),(g_1g_2,g_2',h_2\alpha(g_1^{-1},h_1)))
\\&\xymatrix@C=4em{\ar@{|->}[r]^-{\gamma}& }&& \mquad(g_1g_2,g_1'g_2',h_2h_1h')
\\&&&\qquad=(g_1g_2,g_1'g_2',\alpha(g_1',h_2\alpha(g_1^{-1},h_1))h')
\\&&&\qquad=(g_1g_2,g_1'g_2',h_1\alpha(g_1,h_2)h')
\end{align*}
This shows that the diagram is commutative. 
\end{proof}

\begin{lemma}
\label{lemma-triv-2}
The inversion $g := \inv: G \to G$ and the 1-form $\varphi:=u(\theta) \in \Omega^1(G,\mathfrak{h})$ form a flat gauge transformation
\begin{equation*}
(g,\varphi): (0,i_{*}(B_{\Gamma}))\to p^{*}(A,B)\text{.}
\end{equation*}
\end{lemma}

\begin{proof}
The condition for a gauge transformation is
\begin{equation*}
t_{*}(\varphi) = t_{*}(u(\theta))=\theta - q(p_{*}(\theta))=-\inv^{*}\bar\theta - p^{*}A\text{.}  
\end{equation*}
To compute the curvature of $(g,\varphi)$  
we make three separate calculations: first,
\begin{multline*}
\mathrm{d}\varphi + \tfrac{1}{2}[\varphi \wedge \varphi]=  \tfrac{1}{2}ij \kappa_{*}(\theta \wedge \theta)- ij \kappa_{*}^{skew}(\theta \wedge q(p^{*}\theta))\\+ \tfrac{1}{2}p^{*}ij \kappa_{*}(q(\theta) \wedge q(\theta))-u([\theta \wedge q(p^{*}\theta)])+\tfrac{1}{2}p^{*}u([q(\theta)\wedge q(\theta)])\text{.} 
\end{multline*}
Second,
\begin{equation*}
p^{*}B-i_{*}(B_{\Gamma})= \tfrac{1}{2}p^{*}ij \kappa_{*}(q(\theta)\wedge q(\theta))-\tfrac{1}{2}ij \kappa_{*}(\theta \wedge \theta)+\tfrac{1}{2}p^{*}u([q(\theta)\wedge q(\theta)])\text{.} 
\end{equation*} 
Third,
\begin{multline*}
\alpha_{*}(p^{*}A \wedge \varphi) \\=ij \kappa_{*}^{skew}(q(p^{*}\theta) \wedge \theta)-ip^{*}j \kappa_{*}(q(\theta) \wedge q( \theta))+u([q(p^{*}\theta) \wedge \theta])-p^{*}u([q(\theta) \wedge q(\theta)])\text{.}
\end{multline*}
The combination of these three computations yields the claim that the curvature of $(g,\varphi)$ vanishes. 
\end{proof}

\begin{proof}[Proof of \cref{trivialization-of-g-gamma}]
We define the trivialization $\mathcal{T}_{\Gamma}$ as a refinement $\mathcal{T}_{\Gamma}=(p,(g,\varphi),\nu)$; see \cref{refinement}.
The first ingredient is the projection $p:G \to F$. The second is the gauge transformation
\begin{equation*}
(g,\varphi): (0,i_{*}(B_{\Gamma}))\to p^{*}(A,B)
\end{equation*}
of \cref{lemma-triv-2}, and the third is the bundle isomorphism
\begin{equation*}
\nu: \pr_2^{*}\trivlin_g^{\varphi}\otimes  i_{*}(\Gamma_1) \to  (\rho^{[2]})^* (\trivlin_1^0 )\otimes \pr_1^{*}\trivlin_{g}^{\varphi}
\end{equation*}
over $G \times_F G$ defined as the composition
\begin{equation*}
\alxydim{@C=4em}{\pr_2^{*}\trivlin_g^{\varphi}\otimes  i_{*}(\Gamma_1) \ar[r]^{\id \otimes \tau} & \pr_2^{*}\trivlin_g^{\varphi}\otimes  \trivlin_{\Delta}^{\alpha} \ar[r]^-{\gamma_{g \circ \pr_2,\Delta}} & \pr_1^{*}\trivlin_g^{\varphi}}\text{.}
\end{equation*}
That this is connection-preserving follows, for the first part, from \cref{lemma-triv-3}, and, for the second part, from the definitions of $\alpha$ and $\varphi$, as $\pr_2^{*}\varphi + (\alpha_{g \circ \pr_2})_{*}(\pr_1^{*}\varphi)=\pr_1^{*}\varphi$, and \cref{tensor-product-of-trivial-bundles-with-connections}. That diagram \cref{eq:refinementdiag} in \cref{refinement} is commutative follows from \cref{commutative-diagram-for-tau} and the associativity of $\gamma$ in \cref{tensor-product-of-trivial-bundles}. 

The map $p_{*}(\mathcal{T}_{\Gamma})$ is derived from the anchor of the bundle $\trivlin_{g}^{\varphi}$, which is $p \circ g = \inv$. 
\end{proof}

Next, we discuss the compatibility between the trivialization $\mathcal{T}_{\Gamma}$ and the multiplicative structure of $\mathcal{G}_{\Gamma}$. We consider the following three variations of the trivialization $\mathcal{T}_{\Gamma}$. First, we apply the trivialization trick of \cref{versuchs-iso} to 
\begin{equation*}
\pr_1^{*}\mathcal{T}_{\Gamma}: \pr_1^{*}\mathcal{G}_{\Gamma} \to \pr_1^{*}\mathcal{I}_{A,B} =  i_{*}^{(\pr_1^{*}q(\theta),0)}(\mathcal{I}_0)\text{,}
\end{equation*}
getting a 1-morphism
\begin{equation}
\label{first-variation-of-T}
\alxydim{@C=5em}{i_{*}(\pr_1^{*}\mathcal{G}_{\Gamma} \otimes \pr_2^{*}\mathcal{G}_{\Gamma}) \ar[r]^-{\pr_1^{*}\mathcal{T} \otimes \id_{\pr_2^{*}\mathcal{G}_{\Gamma}}} & i_{*}^{(\pr_1^{*}q(\theta),0)}(\pr_2^{*}\mathcal{G}_{\Gamma})\text{.} }
\end{equation} 
Second, we apply the adjusted shift of \cref{shift-of-1-morphisms} to the 1-morphism
\begin{equation*}
\pr_2^{*}\mathcal{T}_{\Gamma}:i_{*}(\pr_2^{*}\mathcal{G}_{\Gamma}) \to \pr_2^{*}\mathcal{I}_{A,B}=i_{*}^{(\pr_2^{*}q(\theta),0)}(\mathcal{I}_0)
\end{equation*} 
and to the shifting 1-form $\nu=\pr_1^{*}q(\theta)\in \Omega^1(G\times G,\mathfrak{g})$. This requires some calculations. The anchor map $p_{*}(\pr_2^{*}\mathcal{T}_{\Gamma}): F\times F \to M$ is $f := \inv \circ \pr_2$, and
\begin{equation*}
\tilde q_u(\nu,f) = \nu + q(\Ad_f(p(\nu))-p(\nu))=q(\Ad_{\pr_2}^{-1}(\pr_1^{*}\theta))=q(m^{*}\theta-\pr_2^{*}\theta)\text{.}
\end{equation*}
According to \cref{shift-of-1-morphisms}, the adjusted shift in the source is then by the pair $(\nu,0)$
and the shift in the target is by the pair
\begin{equation*}
(\pr_2^{*}q(\theta)+\tilde q_u(\nu,f),b_{\kappa}(\pr_2^{*}\theta\wedge p_{*}\tilde q_u(\nu,f)))=(m^{*}q(\theta),b_{\kappa}(\pr_2^{*}\theta\wedge \Ad_{\pr_2}^{-1}(\pr_1^{*}\theta)))=(m^{*}q(\theta),\rho_{\Gamma})\text{.}
\end{equation*}
Thus, the shifted 1-isomorphism is
\begin{equation}
\label{second-variation-of-T}
\alxydim{@C=5em}{i_{*}^{(\pr_1^{*}q(\theta),0)}(\pr_2^{*}\mathcal{G}_{\Gamma})\ar[r]^-{\pr_2^{*}\mathcal{T}^{\pr_1^{*}q(\theta)}} & i_{*}^{(m^{*}q(\theta),\rho_{\Gamma})}(\mathcal{I}_0)\text{.} }
\end{equation} 
The third variation applies again the trivialization trick of \cref{versuchs-iso}, now to the trivialization $m^{*}\mathcal{T}_{\Gamma}$, getting a 1-isomorphism
\begin{equation}
\label{third-variation-of-T}
\alxydim{@C=5em}{i_{*}(m^{*}\mathcal{G}_{\Gamma} \otimes \mathcal{I}_{\rho_{\Gamma}}) \ar[r]^-{m^{*}\mathcal{T}_{\Gamma} \otimes \id_{\mathcal{I}_{\rho_{\Gamma}}}} & i_{*}^{(m^{*}q(\theta),0)}(\mathcal{I}_{\rho_{\Gamma}})=i_{*}^{(m^{*}q(\theta),\rho_{\Gamma})}(\mathcal{I}_0)\text{.}}
\end{equation}
Together with the multiplicative structure, the variations collected above form the 1-morphisms in the following diagram.

\begin{proposition}
\label{multiplicativity-of-trivialization}
The trivialization $\mathcal{T}_{\Gamma}$ is compatible with the multiplicative structure $\mathcal{M}_{\Gamma}$ on $\mathcal{G}_{\Gamma}$ in the sense that there exists a canonical 2-morphism
\begin{equation}
\label{multiplicativity-of-trivialization-diag}
\alxydim{@C=8em@R=4em}{i_{*}(\pr_1^{*}\mathcal{G}_{\Gamma} \otimes \pr_2^{*}\mathcal{G}_{\Gamma}) \ar[d]_{i_{*}(\mathcal{M}_{\Gamma})} \ar[r]^-{\pr_1^{*}\mathcal{T} \otimes \id_{\pr_2^{*}\mathcal{G}_{\Gamma}}}  &  i_{*}^{(\pr_1^{*}q(\theta),0)}(\pr_2^{*}\mathcal{G}_{\Gamma}) \ar@{=>}[dl]|*+{\sigma_{\Gamma}} \ar[d]^{\pr_2^{*}\mathcal{T}^{\pr_1^{*}q(\theta)}} \\ i_{*}(m^{*}\mathcal{G}_{\Gamma} \otimes \mathcal{I}_{\rho_{\Gamma}} ) \ar[r]_{m^{*}\mathcal{T}_{\Gamma} \otimes \id_{\mathcal{I}_{\rho_{\Gamma}}}} & i_{*}^{(m^{*}q(\theta),\rho_{\Gamma})}(\mathcal{I}_{0})}
\end{equation}
in the bicategory $\grbcon{\Gamma} {G \times G}^{\ff}$.
Moreover, the 2-morphism $\sigma_{\Gamma}$ is compatible with the associator $\alpha_{\Gamma}$, in the sense that  the  equality between  2-morphisms in the bicategory $\grbcon\Gamma {G \times G \times G}^{\ff}$, depicted in \cref{figure-3}, holds. 
\begin{figure}[h]
\tikzset{
  catlab/.style={
    midway,
    sloped,
    above,
    font=\scriptsize
  },
  catlab-r/.style={
    midway,
    right,
    font=\scriptsize
  },
  catlab-l/.style={
    midway,
    left,
    font=\scriptsize
  },
  catlab-b/.style={
    midway,
    below,
    font=\scriptsize
  },
  catlab-i/.style={
    midway,
    inner sep=1pt,
    fill=white,
    font=\scriptsize
  },
  morph/.style={
    ->,
  },
  2morph/.style={
    ->,
    double,
    shorten >=10pt,
    shorten <=10pt,
    bend right=10,
  }
}
\begin{flushleft}
\begin{tikzpicture}[scale=1.4]

\node (1-2-3) at (0,  0) {$i_{*}(\mathcal{G}_1 \otimes \mathcal{G}_2 \otimes \mathcal{G}_3)$};
\node (1-2-3a) at (5,  0) {$i_{*}^{(q(\theta_1),0)}(\mathcal{G}_2 \otimes \mathcal{G}_3)$};
\node (12-3) at (1.7,  -1.2) {$i_{*}^{(0,\rho_{1,2})}(\mathcal{G}_{12} \otimes \mathcal{G}_3)$};
\node (1-23) at (0,  -2.3) {$i_{*}^{(0,\rho_{2,3})}(\mathcal{G}_{1} \otimes \mathcal{G}_{23} )$};
\node (123a) at (0,  -4.5) {$i_{*}^{(0,\rho_{1,23}+\rho_{2,3})}(\mathcal{G}_{123} )$};
\node (123b) at (1.7,  -3) {$i_{*}^{(0,\rho_{12,3}+\rho_{1,2})}(\mathcal{G}_{123} )$};

\node (t3) at (5,  -1.5) {$i_{*}^{(q(\theta_{12}),\rho_{1,2})}(\mathcal{G}_3$)};

\node (t1-23) at (5,  -4.5) {$i_{*}^{(q(\theta_{123}),\rho_{1,23}+\rho_{2,3})}(\mathcal{I}_0)$};
\node  (t12-3) at (5,  -3) {$i_{*}^{(q(\theta_{123}),\rho_{12,3}+\rho_{1,2})}(\mathcal{I}_0)$};

\draw[morph] (1-2-3) -- (12-3) node[catlab-i] {$i_{*}(\mathcal{M}_{1,2} \otimes \id)$};
\draw[morph] (1-2-3) -- (1-23) node[catlab-l] {$i_{*}(\id \otimes \mathcal{M}_{2,3})$};
\draw[morph] (12-3) -- (123b) node[catlab-i] {$i_{*}(\mathcal{M}_{12,3})$};
\draw[morph] (1-23) -- (123a) node[catlab-l] {$i_{*}(\mathcal{M}_{1,23})$};

\draw[morph] (12-3) -- (t3) node[catlab] {$\mathcal{T}_{12} \otimes \id $};

\draw[morph] (t3) -- (t12-3) node[catlab-r] {$\mathcal{T}_3^{q(\theta_{12})} $};

\draw[morph] (123a) -- (t1-23) node[catlab-b] {$\mathcal{T}_{123}$};
\draw[morph] (123b) -- (t12-3) node[catlab-b] {$\mathcal{T}_{123}$};

\draw[morph] (1-2-3) -- (1-2-3a) node[catlab] {$\mathcal{T}_1 \otimes \id \otimes \id $};
\draw[morph] (1-2-3a) -- (t3) node[catlab-r] {$(\mathcal{T}_2 \otimes \id)^{q(\theta_1)}$};

\draw[double] (t12-3) -- (t1-23);
\draw[double] (123a) -- (123b);

\draw[2morph] (12-3) to node[catlab] {$\alpha_{\Gamma}$} (1-23);
\draw[2morph] (1-2-3a) to node[catlab] {$\sigma_{1,2} \otimes \id$} (12-3);
\draw[2morph] (t3) to node[catlab] {$\sigma_{12,3}$} (123b);

\end{tikzpicture}
\end{flushleft}

\begin{flushright}
\begin{tikzpicture}[scale=1.2]
\node  at (-2,-2.3) {$=$};

\node  (1-2-3) at (0,  0) {$i_{*}(\mathcal{G}_1 \otimes \mathcal{G}_2 \otimes \mathcal{G}_3)$};
\node  (1-2-3a) at (5,  0) {$i_{*}^{(q(\theta_1),0)}(\mathcal{G}_2 \otimes \mathcal{G}_3)$};
\node  (t3) at (5,  -1.5) {$i_{*}^{(q(\theta_{12}),\rho_{1,2})}(\mathcal{G}_3$)};
\node  (1-23) at (0,  -1.5) {$i_{*}^{(0,\rho_{2,3})}(\mathcal{G}_{1} \otimes \mathcal{G}_{23} )$};
\node  (t12-3) at (5,  -3) {$i_{*}^{(q(\theta_{123}),\rho_{12,3}+\rho_{1,2})}(\mathcal{I}_0)$};
\node  (123a) at (0,  -4.5) {$i_{*}^{(0,\rho_{1,23}+\rho_{2,3})}(\mathcal{G}_{123} )$};
\node  (t1-23) at (5,  -4.5) {$i_{*}^{(q(\theta)_{123},\rho_{1,23}+\rho_{2,3})}(\mathcal{I}_0)$};
\node  (g23) at (2.3,  -2.8) {$i_{*}^{(q(\theta_1),\rho_{2,3})}(\mathcal{G}_{23} )$};

\draw[morph] (t3) -- (t12-3) node[catlab-r] (A) {$\mathcal{T}_3^{q(\theta_{12})} $};
\draw[morph] (1-2-3) -- (1-2-3a) node[catlab] {$\mathcal{T}_1 \otimes \id \otimes \id $} ;
\draw[morph] (1-2-3a) -- (t3) node[catlab-r] {$(\mathcal{T}_2 \otimes \id)^{q(\theta_1)}$} ;
\draw[morph] (1-2-3a) -- (g23) node[catlab] {$i_{*}^{(q(\theta_1),0)}(\mathcal{M}_{2,3})$} ;
\draw[morph] (1-2-3) -- (1-23) node[catlab-l] {$i_{*}(\id \otimes \mathcal{M}_{2,3})$};
\draw[morph] (1-23) -- (123a) node[catlab-l] {$i_{*}(\mathcal{M}_{1,23})$};
\draw[morph] (123a) -- (t1-23) node[catlab-b] {$\mathcal{T}_{123}$};
\draw[morph] (1-23) -- (g23) node[catlab-i] {$\mathcal{T}_{1} \otimes \id$};
\draw[morph] (g23) -- (t1-23) node[catlab-i] {$\mathcal{T}_{23}$};

\draw[double] (t12-3) -- (t1-23);

\draw[2morph] (A) to node[catlab] {$\sigma_{2,3}$} (g23);
\draw[2morph] (g23) to  node[catlab] {$\sigma_{1,23}$} (123a);

\end{tikzpicture}
\end{flushright}
\caption{}
\label{figure-3}
\end{figure}  
\end{proposition} 

\begin{proof}
All occurring 1-morphisms in \cref{multiplicativity-of-trivialization-diag} are refinements, and we compute them separately, first, displaying them through their associated strict 1-morphisms. We use the data $\mathcal{T}_{\Gamma}=(p,(g,\varphi),\nu)$ of the trivialization $\mathcal{T}_{\Gamma}$ collected in the proof of \cref{trivialization-of-g-gamma}. To begin in the clockwise direction:
\begin{itemize}
\item 
 $\pr_1^{*}\mathcal{T} \otimes \id_{\pr_2^{*}\mathcal{G}_{\Gamma}}$ is the strict 1-morphism consisting of the map $p \times \id: G^2 \to F^2$, of the principal $\Gamma\!$-bundle $\pr_1^{*}\trivlin_g^{\varphi}$
over $G^2$, and of the following bundle isomorphism, over a point $(g_1,g_2),(g_1',g_2') \in G^2 \times_{F^2} G^2$ given by
\begin{align*}
\alxydim{@C=3em}{(\trivlin_{g}^{\varphi})_{g_1'} \otimes i_{*}(\Gamma_1)_{g_1,g_1'} \otimes i_{*}(\Gamma_1)_{g_2,g_2'} \ar[r]^-{\nu_{g_1,g_1'} \otimes \id} & (\trivlin_g^{\varphi})_{g_1} \otimes i_{*}(\Gamma_1)_{g_2,g_2'} \ar[r]^-{\beta} & i_{*}^{A}(\Gamma_1)_{g_2,g_2'} \otimes (\trivlin_g^{\varphi})_{g_1}\text{.} }
\end{align*}

\item
 $(\pr_2^{*}\mathcal{T}_{\Gamma})^{\pr_1^{*}p^{*}A}$ is the strict 1-morphism consisting of the map $\id \times p: F \times G \to F^2$, of the principal $\Gamma\!$-bundle $(\pr_2^{*}\trivlin_g^{\varphi})^{\pr_1^{*}p^{*}A}$,
and of the following bundle isomorphism, over a point $(f_1,g_2),(f_1,g_2') \in (F \times G) \times_{F^2} (F \times G)$ given by
\begin{equation*}
\nu_{g_2,g_2'}:(\trivlin_{g}^{\varphi})_{g_2'}^{A_{f_1}} \otimes i_{*}^{A_{f_1}}(\Gamma_1)_{g_2,g_2'} \to  (\trivlin_g^{\varphi})^{A_{f_1}}_{g_2} \text{.}
\end{equation*}
\end{itemize}
The clockwise composite is thus the strict 1-morphism consisting of the map $p ^2: G^2 \to F^2$, of the principal $\Gamma\!$-bundle $Q^{cw}:=(\pr_2^{*}\trivlin_g^{\varphi})^{\pr_1^{*}p^{*}A} \otimes \pr_1^{*}\trivlin_g^{\varphi}$ over $G^2$, and of the following bundle morphism $\nu^{cw}$, over a point $(g_1,g_2),(g_1',g_2') \in G^2 \times_{F^2} G^2$ given by
\begin{align*}
\alxydim{@C=4em}{(\trivlin_{g}^{\varphi})_{g_2'}^{A_{f_1}} \otimes(\trivlin_{g}^{\varphi})_{g_1'} \otimes i_{*}(\Gamma_1)_{g_1,g_1'} \otimes i_{*}(\Gamma_1)_{g_2,g_2'} \ar[d]^-{\id \otimes \nu_{g_1,g_1'} \otimes \id} \\ (\trivlin_{g}^{\varphi})_{g_2'}^{A_{f_1}} \otimes(\trivlin_g^{\varphi})_{g_1} \otimes i_{*}(\Gamma_1)_{g_2,g_2'} \ar[r]^-{\id \otimes \beta} & (\trivlin_{g}^{\varphi})_{g_2'}^{A_{f_1}} \otimes i_{*}^{A}(\Gamma_1)_{g_2,g_2'} \otimes (\trivlin_g^{\varphi})_{g_1} 
\ar[d]^{\nu_{g_2,g_2'} \otimes \id } 
\\ &(\trivlin_g^{\varphi})^{A_{f_1}}_{g_2}\otimes (\trivlin_g^{\varphi})_{g_1} \text{.} }
\end{align*}
Next we look at the counter-clockwise direction:
\begin{itemize}

\item 
 $i_{*}(\mathcal{M})$ is the strict 1-morphism consisting, following \cref{multiplication-refinement}, of the map $f: G^2 \to F^2 \ttimes mp G$, the principal $\Gamma\!$-bundle $\trivlin_1^{i_{*}\tau}$, and of  the following bundle isomorphism, over a point $((g_1,g_2),(g_1',g_2'))\in G^2 \times_{F^2} G^2$ given by
\begin{equation*}
\alxydim{@C=8em}{(\trivlin_1^{i_{*}\tau})_{g_1,g_2} \otimes i_{*}(\Gamma_1)_{g_1,g_1'} \otimes i_{*}(\Gamma_1)_{g_2,g_2'} \ar[r]^-{(\nu_\Gamma)_{(g_1,g_2),(g_1',g_2')}} &  i_{*}(\Gamma_1)_{g_1g_2,g_1'g_2'} \otimes (\trivlin_1^{i_{*}\tau})_{g_1',g_2'}\text{.}}
\end{equation*} 

\item
$m^{*}\mathcal{T}_{\Gamma} \otimes \id_{\mathcal{I}_{\rho_{\Gamma}}}$ is the strict 1-morphism consisting (\cref{versuchs-iso:3}) of the map $\pr_{12}:F^2 \ttimes mp G \to F^2$, of the principal $\Gamma\!$-bundle $\pr_3^{*}\trivlin_g^{\varphi}$ over $F^2 \ttimes mp G$, 
and of the following bundle isomorphism, over a point $(f_1,f_2,g),(f_1,f_2,g')$ given by
\begin{align*}
\alxydim{@C=4em}{(\trivlin_{g}^{\varphi})_{g'} \otimes i_{*}(\Gamma_1)_{g,g'}  \ar[r]^-{\nu_{g,g'}} & (\trivlin_g^{\varphi})_{g} \text{.} }
\end{align*}

\end{itemize}
The counter-clockwise composite is thus the strict 1-morphism consisting of the map $p^2: G^2 \to F^2$, of the principal $\Gamma\!$-bundle $Q^{cc} := m^{*}\trivlin_g^{\varphi} \otimes \trivlin_1^{i_{*}\tau}$ over $G^2$, 
and of the following bundle morphism $\nu^{cc}$, over a point $(g_1,g_2),(g_1',g_2') \in G^2 \times_{F^2} G^2$ given by
\begin{align*}
\alxydim{@C=4em}{(\trivlin_{g}^{\varphi})_{g_1'g_2'} \otimes(\trivlin_1^\tau)_{g_1,g_2} \otimes i_{*}(\Gamma_1)_{g_1,g_1'} \otimes i_{*}(\Gamma_1)_{g_2,g_2'}  \ar[d]^-{\id \otimes (\nu_{\Gamma})_{(g_1,g_2),(g_1',g_2')} } \\ 
(\trivlin_{g}^{\varphi})_{g_1'g_2'} \otimes i_{*}(\Gamma_1)_{g_1g_2,g_1'g_2'} \otimes (\trivlin_1^{\tau})_{g_1',g_2'}
\ar[r]^-{\nu_{g_1g_1',g_2g_2'} \otimes \id} & (\trivlin_g^{\varphi})_{g_1g_2} \otimes (\trivlin_1^{\tau})_{g_1',g_2'} 
 \text{.} }
\end{align*}

Now we are in position to construct the 2-isomorphism $\sigma_{\Gamma}$ in the bicategory $\grbcon\Gamma M^{strict}$.   Its data is 
 a connection-preserving bundle isomorphism $\sigma_{\Gamma}: Q^{cw} \to Q^{cc}$, which we construct using the isomorphisms $\gamma$ from \cref{tensor-product-of-trivial-bundles-with-connections} and the fact that they are connection-preserving (\cref{tensor-product-of-trivial-bundles-with-connections}):
\begin{align*}
\hspace{8em}\alxydim{@C=0.2em}{\hspace{-12em}Q^{cw}=  \pr_2^{*}(\trivlin_{g}^{\varphi})^{\pr_1^{*}p^{*}A}  \otimes \pr_1^{*}\trivlin_{g}^{\varphi} \eqcref{adjusted-shift-of-connection} \trivlin_{g \circ \pr_2}^{\pr_2^{*}\varphi+\kappa(g \circ \pr_2,\pr_1^{*}p^{*}A)}  \otimes \trivlin_{g \circ \pr_1}^{\pr_1^{*}\varphi}  
\ar[d]^-{\gamma_{g\circ \pr_1,g \circ \pr_2}} 
\\  \trivlin_{(g \circ \pr_2)\cdot (g \circ \pr_1)}^{\pr_2^{*}\varphi+\kappa(g \circ \pr_2,\pr_1^{*}p^{*}A)+(\alpha_{g \circ \pr_2})_{*}(\pr_1^{*}\varphi)} \ar@{=}[r] & \trivlin_{g \circ m}^{m^{*}\varphi-i_{*}\tau} \ar[d]^-{\gamma^{-1}_{g \circ m,1}} \\ & m^{*}\trivlin_{g}^{\varphi} \otimes \trivlin_1^{i_{*}\tau}=Q^{cc}\hspace{-3em}}
\end{align*}
The equality in the middle is -- after invoking the definition $g=\inv$ --  due to the equalities
\begin{align*}
(g_1g_2)^{-1}&=g_2^{-1}g_1^{-1}
\\
\varphi_{g_2} +(\alpha_{g_2^{-1}})_{*}(\varphi_{g_1})+\kappa(g_2^{-1},A_{p(g_1)}) &=\varphi_{g_1g_2}-i_{*}\tau_{g_1,g_2}\text{,}
\end{align*}
of which the second is a simple calculation  using the definitions $\varphi=u(\theta)$, and $A=q(\theta)$,
and the one of $\tau$ in \cref{definition-of-tau}. 
It remains to check  the commutative diagram \cref{condition-for-2-morphisms-strict} for 2-morphisms between strict 1-morphisms. Here, it takes the form depicted in \cref{figure-1}.
\begin{figure}[h]
\begin{equation*}
\footnotesize
\hspace{-5em}\alxydim{@C=-7em@R=4em}{
Q^{cw}_{g_1',g_2'} \otimes_{} i_{*}(\Gamma_1)_{g_1,g_1'} \otimes i_{*}(\Gamma_1)_{g_2,g_2'} \ar[dr]|{\id \otimes \tau_{g_1,g_1'} \otimes \id} \ar[dd]|>>>>>>>>{\nu_{g_1,g_1'} \otimes \id} \ar@/^4pc/[rrrr]^{(\sigma_{\Gamma})_{g_1',g_2'}\otimes \id} \ar[rr] \ar@/_5pc/[ddddd]|{\nu^{cw}_{(g_1,g_2),(g_1',g_2')}} && \trivlin_{g_2^{\prime-1}g_1^{\prime-1}} \otimes_{} i_{*}(\Gamma_1)_{g_1,g_1'} \otimes i_{*}(\Gamma_1)_{g_2,g_2'}\ar@{}[d]|*+[o][F-]{A} \ar[rr] 
&& Q^{cc}_{g_1',g_2'} \otimes_{} i_{*}(\Gamma_1)_{g_1,g_1'} \otimes i_{*}(\Gamma_1)_{g_2,g_2'}\ar[dl]|{\id \otimes \tau_{g_1,g_1'} \otimes \id} 
\ar[dd]|>>>>>>>{\id \otimes (\nu_{\Gamma})_{(g_1,g_2),(g_1',g_2')}} \ar@/^4pc/[ddddd]|{\nu^{cc}_{(g_1,g_2),(g_1',g_2')}} \\
&\trivlin_{g_{2}^{\prime-1}} \otimes \trivlin_{g_{1}^{\prime-1}} \otimes \trivlin_{g_1'g_1^{-1}}\otimes i_{*}(\Gamma_1)_{g_2,g_2'}\ar@{}[dr]|*+[o][F-]{B} \ar[rr] \ar[dl]&&\trivlin_{g_2^{\prime-1}g_1^{\prime-1}}\otimes \trivlin_{g_1'g_1^{-1}} \otimes i_{*}(\Gamma_1)_{g_2,g_2'} \ar[dl]  \\
\trivlin_{g_2'}^{-1} \otimes \trivlin_{g_1}^{-1} \otimes i_{*}(\Gamma_1)_{g_2,g_2'}\ar@{}[drr]|*+[o][F-]{D}\ar[rr]  \ar[d]^{\id \otimes \beta}
& 
&\trivlin_{g_2^{\prime-1}g_1^{\prime-1}}\otimes \trivlin_{g_1'} \otimes  \trivlin_{g_1^{-1}} \otimes i_{*}(\Gamma_1)_{g_2,g_2'}\hspace{4em}\ar@{}[rr]|*+[o][F-]{C} \ar[d]|{\id \otimes \id \otimes \beta}
&&\hspace{-3em}\trivlin_{g_2^{\prime-1}g_1^{\prime-1}} \otimes i_{*}(\Gamma_1)_{g_1g_2,g_1'g_2'}  \ar[dl]|{\id \otimes \tau_{g_1g_2,g_1'g_2'}} \ar[ddd]|{\nu_{g_1g_1',g_2g_2'} } \\  
\trivlin_{g_2'}^{-1} \otimes i_{*}(\Gamma_1)_{g_2,g_2'} \otimes \trivlin_{g_1}^{-1} \ar[dr]|{\id \otimes \tau_{g_2,g_2'} \otimes \id} \ar[dd]|>>>>>>>>>>>{\nu_{g_2,g_2'} \otimes \id } &&\hspace{-5em}\trivlin_{g_2^{\prime-1}g_1^{\prime-1}}\otimes \trivlin_{g_1'}  \otimes i_{*}(\Gamma_1)_{g_2,g_2'}\otimes  \trivlin_{g_1^{-1}} \ar[d]|{\id \otimes \id \otimes \tau_{g_2,g_2'} \otimes \id} &\trivlin_{g_2^{\prime-1}g_1^{\prime-1}}  \otimes \trivlin_{g_1'g_2'g_2^{-1}g_1^{-1}} \hspace{-7em}  \ar[ddr] \\
&\hspace{-7em}\trivlin_{g_2'}^{-1} \otimes \trivlin_{g_2'g_2^{-1}} \otimes \trivlin_{g_1}^{-1} \ar[r] \ar[dl]& \trivlin_{g_2^{\prime-1}g_1^{\prime-1}}\otimes \trivlin_{g_1'}  \otimes \trivlin_{g_2'g_2^{-1}}\otimes  \trivlin_{g_1^{-1}} \ar@{}[d]|*+[o][F-]{E} \hspace{-3em}\ar[ur]
\\
Q^{cw}_{g_1,g_2} \ar@/_4pc/[rrrr]_{ (\sigma_{\Gamma})_{g_1,g_2}} \ar[rr]
&& \trivlin_{g_2^{-1}g_1^{-1}} \ar[rr]
&& Q^{cc}_{g_1,g_2}
}
\end{equation*}
\caption{}
\label{figure-1}
\end{figure}
Compared to \cref{condition-for-2-morphisms-strict}, we have rotated the diagram by 90 degrees; moreover, we have stripped all connections from the notation, and every unlabelled arrow is an occurrence of the bundle isomorphism $\gamma$. We claim that all subdiagrams commute; hence, the outer shape is a commutative diagram:
\begin{itemize}

\item 
The four outer, lens-shaped diagrams represent
the given formulas for $(\sigma_{\Gamma})$, $\nu^{cw}$ and $\nu^{cc}$, respectively. 

\item
The three triangular diagrams represent the definition of the bundle isomorphism $\nu$ from \cref{trivialization-of-g-gamma}. 

\item
Diagrams B and E commute due to the associativity of the bundle isomorphism $\gamma$; see \cref{tensor-product-of-trivial-bundles}.

\item
Diagrams A and D commute due the functoriality of the tensor product.

\item
Diagram C is the commutative diagram of \cref{multiplicativity-of-tau}. 

\end{itemize} 
It remains to prove the compatibility with the associator $\alpha_{\Gamma}$. We have to verify that two 2-morphisms between two strict 1-morphisms are equal. The 2-morphisms are just $\Gamma\!$-bundle isomorphisms over $G^3$, coming from the associator $\alpha_{\Gamma}$ and the 2-morphism $\sigma_{\Gamma}$. But since  the multiplicative structure on $\mathcal{G}_{\Gamma}$ is strict, the associator $\alpha_{\Gamma}$ is actually the identity. And since $\sigma_{\Gamma}$ is defined only using the bundle morphism $\gamma$, all we have to check is the equality between different applications of $\gamma$. By \cref{tensor-product-of-trivial-bundles}, $\gamma$ is associative, hence, any two such combinations must be equal. 
\end{proof}

\section{Lifting theory}

\label{lifting-theory}

In this section, we prove a central result, \cref{lifting-theorem}, which recasts non-abelian bundle gerbes with adjusted and adapted connections in terms of a twisted abelian framework. The twist is represented by an abelian bundle 2-gerbe, called the Chern-Simons bundle 2-gerbe. This generalizes the lifting theory for non-abelian bundle gerbes developed in joint work with Nikolaus in \cite{Nikolausa} to include connections. At the same time, it extends the work of Téllez-Domínguez \cite{Tellez2023} from a cocycle description to a global one, using descent theory.

Let $\Gamma$ be a smoothly separable, central crossed module with $F:=\pi_0\Gamma$, equipped with a splitting $u$ and an adapted adjustment $\kappa$, and let $M$ be a smooth manifold.

\subsection{The Chern-Simons bundle 2-gerbe}

\label{lifting-2-gerbe}

We start with the central object of interest in this section.

\begin{definition}
\label{gamma-lifts}
If $P$ is a principal $F$-bundle over $M$ with connection, then a \emph{$(\Gamma\!,u,\kappa)$-lift} of $P$ is a $\Gamma\!$-bundle gerbe $\mathcal{G}$ over $M$ with adapted connection, together with a connection-preserving bundle isomorphism $\varphi: p_{*}(\mathcal{G}) \to P$.
\end{definition}

$(\Gamma\!,u,\kappa)$-lifts form a bicategory  $\liftcon {(\Gamma\!,u,\kappa)}P{}$, whose 1-morphisms $(\mathcal{G},\varphi) \to (\mathcal{G}',\varphi')$ are 1-morphisms $\mathcal{A}:\mathcal{G} \to \mathcal{G}'$ in in $\grbcon{\Gamma\!,u,\kappa} M$ such that $\varphi' \circ p_{*}(\mathcal{A}) = \varphi$, and whose 2-morphisms are arbitrary 2-morphisms in $\grbcon{\Gamma\!,u,\kappa} M$.
The lifting theory characterizes this bicategory, in particular whether it is non-empty (see \cref{lifting-theorem}).

\begin{remark}
By \cref{flatness-of-underlying-bundle}, $(\Gamma\!,u,\kappa)$-lifts are fake-flat if and only if   $P$ is flat.
\end{remark}

We consider a principal $F$-bundle $P$ with connection $\eta\in \Omega^1(P,\mathfrak{f})$ over  $M$.
Let $b_{\kappa}$ be the  symmetric invariant bilinear form associated to $\kappa$, and let $\mathcal{G}_{\Gamma}$ be the multiplicative $A$-bundle gerbe over $F$ with connection constructed in \cref{multiplicative-bundle-gerbe}.
The formulas for the curvature $H_{\Gamma}$ (\cref{curvature-of-curving}) and for the 2-form $\rho_{\Gamma}$ (\cref{definition-of-rho}) ensure the consistency required to construct the Chern-Simons 2-gerbe $\mathbb{CS}_{P}(\mathcal{G}_\Gamma)$ \cite[\S 3.2]{waldorf5}, see \cite{johnson1,carey4} for earlier treatments of this specific 2-gerbe, and \cite{stevenson2} for a general framework for bundle 2-gerbes.

\noindent
We recall that $\mathbb{CS}_{P}(\mathcal{G}_\Gamma)=(P,\rho,C,\mathcal{H},\mathcal{M}',\alpha')$ consists of the following ingredients:
\begin{enumerate}[(a)]

\item
The surjective submersion $\rho:P \to M$, i.e., the bundle projection, and the 3-form curving
\begin{equation*}
C:= b_{\kappa}(\eta\wedge \mathrm{d}\eta) + \tfrac{1}{3}b_{\kappa}(\eta\wedge [\eta\wedge \eta])\in \Omega^3(P)\text{,}
\end{equation*}
i.e., the classical Chern-Simons 3-form associated with the connection $\eta$ and the bilinear form $b_{\kappa}$.
We recall that $C$ satisfies the equation
\begin{equation}
\label{Chern-Simons-identity-1}
\pr_2^{*}C = \pr_1^{*}C +\mathrm{d}\omega + \delta^{*}H_{\Gamma}
\end{equation}
over $P^{[2]}$, where $\omega\in \Omega^2(P^{[2]},\mathfrak{a})$ is defined by
\begin{equation*}
\omega := b_{\kappa}(\pr_1^{*}\eta \wedge \delta^{*}\bar\theta)
\end{equation*}
and  the  map $\delta: P^{[2]} \to F$ is defined by $p_1\delta(p_1,p_2)=p_2$, for $(p_1,p_2)\in P^{[2]}$.

\item
Over $P^{[2]}$, we consider the  $A$-bundle gerbe with connection
\begin{equation*}
\mathcal{H} := \delta^{*}\mathcal{G}_{\Gamma} \otimes \mathcal{I}_{\omega}\text{.}
\end{equation*}
\Cref{Chern-Simons-identity-1,curvature-of-curving} show that \begin{align*}
\pr_2^{*}C - \pr_1^{*}C &= \mathrm{curv}(\mathcal{H})\text{.}
\end{align*}

\item
Over $P^{[3]}$,   a map $\delta_2: P^{[3]} \to F \times F$ is defined by   $(p_1,p_2)\delta_2(p_1,p_2,p_3)=(p_2,p_3)$.
Pulling back the 1-isomorphism $\mathcal{M}_{\Gamma}$ of the multiplicative structure on $\mathcal{G}_{\Gamma}$, we obtain
\begin{equation*}
\delta_2^{*}\mathcal{M}_{\Gamma} : \pr_{12}^{*}\delta^{*}\mathcal{G}_{\Gamma} \otimes \pr_{23}^{*}\delta^{*}\mathcal{G}_{\Gamma} \to \pr_{13}^{*}\delta^{*}\mathcal{G}_{\Gamma} \otimes \delta_2^{*}\mathcal{I}_{\rho_{\Gamma}}\text{,}
\end{equation*}
where $\rho_{\Gamma}$ is from \cref{definition-of-rho}.
Taking the tensor product with the trivial bundle gerbe $\pr_{12}^{*}\mathcal{I}_{\omega} \otimes \pr_{23}^{*}\mathcal{I}_{\omega}$, and using the symmetric braiding available for $A$-bundle gerbes, we obtain a 1-isomorphism
\begin{equation*}
\mathcal{M}': \pr_{12}^{*}\mathcal{H} \otimes \pr_{23}^{*}\mathcal{H} \to \pr_{13}^{*}\mathcal{H}
\end{equation*}
because
\begin{equation}
\label{delta-omega-rho}
\pr_{12}^{*}\omega +\pr_{23}^{*}\omega  =\pr_{13}^{*}\omega-\delta_2^{*}\rho_{\Gamma}\text{.}
\end{equation}

\item
Over $P^{[4]}$, a map $\delta_3: P^{[4]} \to F^3$ is defined by   $(p_1,p_2,p_3)\delta_3(p_1,p_2,p_3,p_4)=(p_2,p_3,p_4)$.
Pulling back the 2-isomorphism $\alpha_{\Gamma}$ of the multiplicative structure of $\mathcal{G}_{\Gamma}$, we obtain  a 2-isomorphism $\alpha':=\delta_3^{*}\alpha_{\Gamma}$ filling the following diagram:
\begin{equation*}
\alxydim{@=1.8cm}{\pr_{12}^{*}\mathcal{H} \otimes \pr_{23}^{*}\mathcal{H} \otimes \pr_{34}^{*}\mathcal{H} \ar[r]^-{\pr_{123}^{*}\mathcal{M}' \otimes \id} \ar[d]_{\id \otimes \pr_{234}^{*}\mathcal{M}'} & \pr_{13}^{*}\mathcal{H} \otimes \pr_{34}^{*}\mathcal{H} \ar@{=>}[dl]|*+{\alpha'} \ar[d]^{\pr_{134}^{*}\mathcal{M}'} \\ \pr_{12}^{*}\mathcal{H} \otimes \pr_{24}^{*}\mathcal{H} \ar[r]_-{\pr_{124}^{*}\mathcal{M}'} & \pr_{14}^{*}\mathcal{H}}
\end{equation*}
Finally, $\alpha'$ satisfies the pentagon axiom, inherited from the associativity constraint $\alpha_{\Gamma}$. While this coherence condition is essential for the definition of the 2-gerbe, it does not appear in the subsequent arguments.
\end{enumerate}
This completes the definition of the Chern-Simons 2-gerbe $\mathbb{CS}_{P}(\mathcal{G}_\Gamma)$.
Observe that the 4-form $\mathrm{d}C\in \Omega^4(P,\mathfrak{a})$ descends along the bundle projection $\rho: P \to M$ to the 4-form
\begin{equation*}
\Omega_{\eta} := b_{\kappa}(F_{\eta} \wedge F_{\eta}) \in \Omega^4(M,\mathfrak{a})\text{,}
\end{equation*}
the \emph{curvature} of the 2-gerbe, where $F_{\eta}=\mathrm{d}\eta + \tfrac{1}{2}[\eta\wedge \eta]$ is the curvature of $\eta$. $\Omega_{\eta}$ itself is closed and represents the characteristic class of   $\mathbb{CS}_{P}(\mathcal{G}_\Gamma)$ in de Rham cohomology.  Let us note an obvious fact:

\begin{lemma}
If the connection $\eta$ on $P$ is flat, then the Chern-Simons 2-gerbe is also flat.
\end{lemma}

 Next we review the bicategory $\trivcon{\mathbb{CS}_P(\mathcal{G}_{\Gamma})}$ of  trivializations of $\mathbb{CS}_P(\mathcal{G}_{\Gamma})$, which specializes the general definition of this bicategory for a bundle 2-gerbe with connection; see \cite[Lem. 3.2.3]{waldorf8}.
A trivialization
   $\mathbb{T}$ of $\mathbb{CS}_P(\mathcal{G}_{\Gamma})$ consists of the following structure:
\begin{enumerate}[(a)]

\item
an $A$-bundle gerbe $\mathcal{S}$ with connection over $P$,

\item
a connection-preserving isomorphism
\begin{equation*}
\mathcal{A}: \mathcal{H} \otimes \pr_2^{*}\mathcal{S} \to \pr_1^{*}\mathcal{S}
\end{equation*}
over $P^{[2]}$, and

\item
a connection-preserving 2-isomorphism
\begin{equation*}
\alxydim{@=1.3cm}{\pr_{12}^{*}\mathcal{H} \otimes \pr_{23}^{*}\mathcal{H} \otimes \pr_3^{*}\mathcal{S} \ar[r]^-{\id \otimes \pr_{23}^{*}\mathcal{A}} \ar[d]_{\mathcal{M}' \otimes \id} & \pr_{12}^{*}\mathcal{H} \otimes \pr_{2}^{*}\mathcal{S} \ar@{=>}[dl]|*+{\sigma} \ar[d]^{\pr_{12}^{*}\mathcal{A}} \\ \pr_{13}^{*}\mathcal{H} \otimes \pr_{3}^{*}\mathcal{S} \ar[r]_-{\pr_{13}^{*}\mathcal{A}} & \pr_1^{*}\mathcal{S}}
\end{equation*}
over $P^{[3]}$.

\end{enumerate}
The 2-isomorphism $\sigma$ has to satisfy the following compatibility condition with the 2-isomorphism $\alpha'$ of the bundle 2-gerbe:
\begin{multline}
\label{coherence-for-sigma}
\small
\alxydim{@C=0em@R=4.5em}{&\hspace{-5em}\mathcal{H}_{12} \otimes \mathcal{H}_{23} \otimes \mathcal{H}_{34} \otimes \mathcal{S}_{4} \ar[rr]^-{\id \otimes \id \otimes \mathcal{A}_{34}} \ar[dr]|>>>>>>>>>{\id\otimes \mathcal{M}'_{234} \otimes \id}="1" \ar[dl]|<<<<<<<<{\mathcal{M}'_{123} \otimes\id \otimes \id} && \mathcal{H}_{12} \otimes \mathcal{H}_{23} \otimes \mathcal{S}_{3} \ar@{=>}@/_1pc/[dl]^{\id \otimes \sigma_{234}} \ar[dr]^{\id \otimes \mathcal{A}_{23}} \\ \mquad\mathcal{H}_{13} \otimes \mathcal{H}_{34} \otimes \mathcal{S}_{4}\mquad \ar[dr]|{\mathcal{M}'_{134} \otimes \id}="2" &\ar@{=>}@/_0.6pc/"1";"2"^>>>>>>{\alpha' \otimes \id} \hspace{6em}& \mquad\mathcal{H}_{12} \otimes \mathcal{H}_{24} \otimes \mathcal{S}_{4}  \ar[dl]|{\mathcal{M}'_{124} \otimes \id} \ar[rr]|-{\id \otimes \mathcal{A}_{24}}  && \mathcal{H}_{12}\otimes \mathcal{S}_{2} \ar@{=>}@/_1pc/[dlll]^{\sigma_{124}} \ar[dl]^{\mathcal{A}_{12}} \\ & \mathcal{H}_{14} \otimes \mathcal{S}_{4} \ar[rr]_-{\mathcal{A}_{14}} && \mathcal{S}_{1}}
\\\small
=\alxydim{@C=0.5em@R=4.5em}{&\hspace{-3em}\mathcal{H}_{12} \otimes \mathcal{H}_{23} \otimes \mathcal{H}_{34} \otimes \mathcal{S}_{4} \ar[rr]^-{\id \otimes \id \otimes \mathcal{A}_{34}}  \ar[dl]|{\mathcal{M}'_{123} \otimes\id \otimes \id} && \mathcal{H}_{12} \otimes \mathcal{H}_{23} \otimes \mathcal{S}_{3}\mqquad \ar[dl]|{\mathcal{M}_{123}'\otimes \id} \ar[dr]|{\id \otimes \mathcal{A}_{23}}="1" \\ \mathcal{H}_{13} \otimes \mathcal{H}_{34} \otimes \mathcal{S}_{4} \ar[rr]|-{\id \otimes \mathcal{A}_{34}} \ar[dr]|{\mathcal{M}'_{134} \otimes \id} && \mathcal{H}_{13} \otimes \mathcal{S}_{3}\mquad \ar@{=>}@/_1pc/[dl]^{\sigma_{134}}  \ar[dr]|{\mathcal{A}_{13}}="2" \ar@{=>}@/_1pc/"1";"2"^{\sigma_{123}}   &\hspace{6em}& \mquad\mathcal{H}_{12}\otimes \mathcal{S}_{2}  \ar[dl]^{\mathcal{A}_{12}} \\ & \mathcal{H}_{14} \otimes \mathcal{S}_{4} \ar[rr]_-{\mathcal{A}_{14}} && \mathcal{S}_{1}}
\end{multline}
Here, we have written $\mathcal{H}_{ij}$ as an abbreviation for $\pr_{ij}^{*}\mathcal{H}$, and similarly for $\mathcal{S}_i$, $\mathcal{A}_{ij}$, and $\sigma_{ijk}$.

\begin{remark}
Each trivialization $\mathbb{T}=(\mathcal{S},\mathcal{A},\sigma)$ has a \emph{covariant derivative} $K\in \Omega^3(M,\mathfrak{a})$, which is the unique 3-form whose pullback along the bundle projection is  $\rho^{*}K=C+\mathrm{curv}(\mathcal{S}) \in \Omega^3(P,\mathfrak{a})$. 
We have $\mathrm{d}K=\Omega_{\eta}$; see \cite[Lem. 3.2.5 \& 3.2.5]{waldorf8}. 
\end{remark}

1-morphisms and 2-morphisms between trivializations are defined in the canonical way (see \cite[Lemma 2.2.4]{waldorf8}). Concretely, a 1-morphism $(\mathcal{B},\beta):\mathbb{T}_1 \to \mathbb{T}_2$, with $\mathbb{T}_i=(\mathcal{S}_i,\mathcal{A}_i,\sigma_i)$, consists of an isomorphism $\mathcal{B}:\mathcal{S}_1 \to \mathcal{S}_2$ of bundle gerbes with connection over $P$ and a 2-isomorphism
\begin{equation*}
\alxydim{}{\mathcal{H} \otimes \pr_2^{*}\mathcal{S}_1 \ar[d]_{\id \otimes \pr_2^{*}\mathcal{B}} \ar[r]^-{\mathcal{A}_1} &  \pr_1^{*}\mathcal{S}_1 \ar@{=>}[dl]|*+{\beta} \ar[d]^{\pr_1^{*}\mathcal{B}} \\ \mathcal{H} \otimes \pr_2^{*}\mathcal{S}_2 \ar[r]_-{\mathcal{A}_2} & \pr_1^{*}\mathcal{S}_2}
\end{equation*}
of bundle gerbes with connection over $P^{[2]}$, subject to a coherence condition over $P^{[3]}$ involving $\beta$, $\sigma_1$, and $\sigma_2$. A 2-morphism $(\mathcal{B},\beta) \Rightarrow (\mathcal{B}',\beta')$ is a 2-isomorphism $\gamma: \mathcal{B}\Rightarrow \mathcal{B}'$ over $P$, subject to a coherence condition over $P^{[2]}$ involving $\beta$ and $\beta'$.

Every  bundle 2-gerbe $\mathbb{G}$ with structure group $A$ has a characteristic class $c_{\mathbb{G}}$ in the  \v Cech cohomology $\check\h^3(M,\sheaf A)$, obtained by choosing iterated sections into all submersions and bundles, under which the 2-isomorphism $\alpha'$ over $P^{[4]}$ becomes a \v Cech 3-cocycle that represents $c_{\mathbb{G}}$; see \cite{stevenson2}. In the case of the Chern-Simons 2-gerbe, this class can be expressed neatly in terms of the data of the principal $F$-bundle $P$ and the Lie 2-group $\Gamma$, under the assumption that $A$ is connected: then,  we have an exact sequence $0 \to Z \to \mathfrak{a} \to A \to 0$ with $Z$ a discrete subgroup of $\mathfrak{a}$, inducing an isomorphism $\check\h^3(M,\sheaf A) \cong \h^4(M,Z)$. Likewise, the Segal--Mitchison smooth group cohomology $\h^3_{\mathrm{SM}}(F,A)$,  which classifies multiplicative bundle gerbes such as $\mathcal{G}_{\Gamma}$, is now isomorphic to the singular cohomology $\h^4(BF,Z)$ of the classifying space of $F$ \cite{pries2}.
The corresponding class $k_{\Gamma} \in \h^4(BF,Z)$ is also known as the \emph{k-invariant} of the Lie 2-group $\Gamma$.
We note the following result.


\begin{proposition}
\label{2-gerbes}
\begin{enumerate}[(i)]

\item
\label{2-gerbes-1}
Suppose $A$ is connected,  let $k_{\Gamma} \in \h^4(BF,Z)$ be the k-invariant of the Lie 2-group $\Gamma$, and let $f:M \to BF$ be a classifying map for the principal $F$-bundle $P$. Then,
\begin{equation*}
c_{\mathbb{CS}_P(\mathcal{G}_{\Gamma})}  =f^{*}k_{\Gamma} \in \h^4(M,Z)\text{.}
\end{equation*}

\item
\label{2-gerbes-2}
$\mathbb{CS}_P(\mathcal{G}_{\Gamma})$ admits trivializations if and only if $c_{\mathbb{CS}_P(\mathcal{G}_{\Gamma})}=0$. In this case, trivializations form a torsor over the monoidal bigroupoid $\grbcon AM$ of bundle gerbes with connection over $M$.

\end{enumerate}
\end{proposition}

\begin{proof}
(i) is \cite[Thm. 3.13]{waldorf5}. (ii) is a combination of \cite[Prop. 11.2]{stevenson2} and \cite[Lem. 3.2.2 and 3.2.3]{waldorf8}.
\end{proof}

In the setting without connections, it has been proved in \cite[Thm. 4.2.7]{Nikolausa} that there is an equivalence of bicategories,
\begin{equation*}
\triv{\mathbb{CS}_P(\mathcal{G}_{\Gamma})} \cong  \lift{\Gamma} P{}
\end{equation*}
between trivializations of $\mathbb{CS}_P(\mathcal{G}_{\Gamma})$ and  $\Gamma\!$-lifts of $P$. The goal of the next section is to generalize this result to the setting with connections.

\subsection{The lifting theorem}

Let $\rho:P \to M$ be a principal $F$-bundle with connection $\eta$ over $M$, and let $\mathbb{CS}_{P}(\mathcal{G}_\Gamma)$ be the Chern-Simons bundle 2-gerbe reviewed in \cref{lifting-2-gerbe}. The main task of this section is to construct a functor
\begin{equation}
\label{functor-trivializations-to-gerbes}
\mathcal{D}: \trivcon{\mathbb{CS}_{P}(\mathcal{G}_\Gamma)} \to  \des_{\grbcon{\Gamma\!,u,\kappa}-}(\rho)\text{,}
\end{equation}
into the bicategory of descent data for the sheaf $\grbcon{\Gamma\!,u,\kappa}-$ of $\Gamma\!$-bundle gerbes with adapted connection, with respect to the surjective submersion $\rho:P \to M$.

We start with a trivialization $\mathbb{T}=(\mathcal{S},\mathcal{A},\sigma)$.
We consider the $\Gamma\!$-bundle gerbe with connection
\begin{equation*}
\widetilde{\mathcal{S}} := i_{*}^{(q(\eta),0)}(\mathcal{S})
\end{equation*}
over $P$, where $q:\mathfrak{f} \to \mathfrak{g}$ is the section of $\Gamma$ determined by the splitting $u$. By \cref{extension-and-adjusted-shift-lemma}, the connection on $\widetilde{\mathcal{S}}$ is adapted to $u$ and has fake-curvature $q(F_{\eta}) \in \Omega^2(P,\mathfrak{g})$. Hence, $\widetilde{\mathcal{S}}$ is an object in $\grbcon{\Gamma\!,u,\kappa}P$. It serves as the first component of the matching family $\mathcal{D}(\mathbb{T})$.

Next, we work over $P^{[2]}$, considering the 1-isomorphism $\mathcal{A}: \mathcal{H} \otimes \pr_2^{*}\mathcal{S} \to \pr_1^{*}\mathcal{S}$, with $\mathcal{H}=\delta^{*}\mathcal{G}_{\Gamma} \otimes \mathcal{I}_{\omega}$. We recall from \cref{trivialization-of-g-gamma} that $\mathcal{G}_{\Gamma}$ admits a trivialization $\mathcal{T}_{\Gamma}:i_{*}(\mathcal{G}_{\Gamma}) \to i_{*}^{(q(\theta),0)}(\mathcal{I}_0)$. Applying the trivialization trick from \cref{versuchs-iso}, we obtain a trivialization
\begin{equation*}
\mathcal{T}_{\mathcal{H}} := \delta^{*}\mathcal{T}_{\Gamma} \otimes \id_{\mathcal{I}_{\omega}}:i_{*}(\mathcal{H}) \to i_{*}^{(\vartheta,\omega)}(\mathcal{I}_0)
\end{equation*}
of $i_{*}(\mathcal{H})$ over $P^{[2]}$, where we set $\vartheta := \delta^{*}q(\theta)\in \Omega^1(P^{[2]},\mathfrak{g})$ for brevity. Note that $p_{*}(\mathcal{T}_{\mathcal{H}})=\inv  \circ \delta$, as maps $P^{[2]} \to F$. Applying the trivialization trick once more, we obtain a 1-isomorphism
\begin{equation*}
\mathcal{T}_{\mathcal{H}} \otimes \id_{\pr_2^{*}\mathcal{S}}:i_{*}(\mathcal{H} \otimes \pr_2^{*}\mathcal{S}) \to i_{*}^{(\vartheta,\omega)}(\pr_2^{*}\mathcal{S})\text{.}  
\end{equation*} 
Thus, we may form the composite $\mathcal{A}'$ defined by
\begin{equation*}
\alxydim{@C=4em}{i_{*}^{(\vartheta,\omega)}(\pr_2^{*}\mathcal{S}) \ar[rr]^-{(\mathcal{T}_{\mathcal{H}} \otimes \id_{ \pr_2^{*}\mathcal{S}})^{-1}} && i_{*}(\mathcal{H} \otimes \pr_2^{*}\mathcal{S}) \ar[r]^-{i_{*}(\mathcal{A})} & i_{*}(\pr_1^{*}\mathcal{S})\text{,} }
\end{equation*}
for which we observe $p_{*}(\mathcal{A}')=(p_{*}\mathcal{T}_{\mathcal{H}})^{-1}=\delta:P^{[2]} \to F$. 
Using \cref{shift-of-1-morphisms}, we shift $\mathcal{A}'$ along the 1-form $\nu:=q(\Ad_{\delta}^{-1}(\pr_1^{*}\eta))\in \Omega^1(P^{[2]},\mathfrak{g})$. The relevant formulas are:
\begin{align}
\label{shift-formulas-1}
A_1'&= \vartheta+\nu=q(\pr_2^{*}\eta) 
&\chi_1' &=\omega+b_{\kappa}(p_{*}(\vartheta)\wedge p_{*}(\nu ))= 0
\\
\label{shift-formulas-2}
A_2' &=  \tilde q_u(\delta,\nu)=q(\pr_1^{*}\eta)
& \chi_2 &= 0
\end{align}
Thus, the shift yields a 1-isomorphism
\begin{equation*}
\widetilde{\mathcal{A}} := (\mathcal{A}')^{\nu}: \pr_2^{*}\widetilde{\mathcal{S}} = i_{*}^{\pr_2^{*}q(\eta)}(\pr_2^{*}\mathcal{S}) \to i_{*}^{\pr_1^{*}q(\eta)}(\pr_1^{*}\mathcal{S})=\pr_1^{*}\widetilde{\mathcal{S}}
\end{equation*}
over $P^{[2]}$; this constitutes the second component of the matching family $\mathcal{D}(\mathbb{T})$. 

We consider the diagram of 1-morphisms in the bicategory $\grbcon \Gamma{P^{[3]}}$ depicted in \cref{figure-2}, employing the shorthand notations $(..)_i := \pr_i^{*}(..)$ and $(..)_{ij} := \pr_{ij}^{*}(..)$ for bundle gerbes, morphisms, and differential forms, as well as $\omega_{123}:= \omega_{12}+\omega_{23}$. Furthermore, we suppress the index $\Gamma$ in $\mathcal{G}_{\Gamma}$, $\mathcal{T}_{\Gamma}$, and $\rho_{\Gamma}$.
\begin{figure}[h]
\begin{equation*}
\small
\hspace{-1.5em}\alxydim{@!0@C=2.4em@R=3.5em}{
  &&&&&&&&& i_{*}^{(\vartheta_{12},\omega_{12})}(\mathcal{S}_2) \ar@/^5pc/[dddddddddrrrrrrrrr]^{\pr_{12}^{*}\mathcal{A}'} \ar[drdrdr]|{(\pr_{12}^{*}\mathcal{T}_{\mathcal{H}} \otimes \id)^{-1}} &&&&&
  \\
  \\
  &&&&&&& {i_{*}^{(\vartheta_{12},\omega_{12})}(\mathcal{H}_{23} \otimes \mathcal{S}_3)}\ar@{}[drrrrr]|*+[o][F-]{B} \ar[drdrr]|{(\pr_{12}^{*}\mathcal{T}_{\mathcal{H}} \otimes \id)^{-1}} \ar@{=}[ddd] \ar[urur]|{i_{*}(\mathcal{A}_{23})^{\vartheta_{12}}}
  \\
  &&&&&&&&&&&& *+[r]{i_{*}(\mathcal{H}_{12} \otimes \mathcal{S}_2)} \ar[drrdrrddrddr]|{i_{*}(\mathcal{A}_{12})}="5"
  \\
  &&&&&&&&&&
  i_{*}(\mathcal{H}_{12} \otimes \mathcal{H}_{23} \otimes \mathcal{S}_3)
  \ar[urr]|{i_{*}(\id \otimes \mathcal{A}_{23})} \ar@/^3pc/[drrrdddd]|>>>>>>>>{i_{*}(\mathcal{M}')}="4"  
  \\
  &&&& &&& \mquad i_{*}^{(\delta_2^{*}q(\theta)_1,\omega_{123})}(\delta_2^{*}\mathcal{G}_2 \otimes \mathcal{S}_3) \ar[dldl]|>>>>>>{(\delta_2^{*}\mathcal{T}_{2} \otimes \id)^{\delta_2^{*}q(\theta_1)}} &&&&  & 
  \\
  &&& {i_{*}^{(\vartheta_{13},\omega_{123}+\delta_2^{*}\rho)}(\mathcal{S}_3)} \ar@{}[uulu]|*+[o][F-]{C} \ar[uuruurrr]|{((\pr_{23}^{*}\mathcal{T}_{\mathcal{H}} \otimes \id)^{-1})^{\vartheta_{13}-\vartheta_{23}}} \ar@{=}[rdr] &&&&&&&\hspace{1em} i_{*}^{(0,\omega_{123})}(\delta_2^{*}(\mathcal{G}_1 \otimes \mathcal{G}_2) \otimes  \mathcal{S}_3) \ar@{=}[uu] \ar[ddl]|{i_{*}(\delta_2^{*}\mathcal{M} \otimes \id)} \ar[ulll]|{\delta_2^{*}\mathcal{T}_{1} \otimes \id} &   
  \\
  &&&&& i_{*}^{(\delta_2^{*}m^{*}q(\theta),\delta_2^{*}\rho+\omega_{123})}(\mathcal{S}_3)
\ar@{}[urrrrr]|*+[o][F-]{A}  \\
  &&&&   &&&&&i_{*}^{(0,\omega_{123})}(\delta_2^{*}(m^{*}\mathcal{G} \otimes \mathcal{I}_{\rho}) \otimes \mathcal{S}_3) \ar@{=}[drrrr] \ar[ullll]|>>>>>>>{\delta_2^{*}m^{*}\mathcal{T} \otimes \id}
  \\
  i_{*}^{\vartheta_{13},\omega_{13}}(\mathcal{S}_3) \ar@/^5pc/[uuuuuuuuurrrrrrrrr]^{(\pr_{23}^{*}\mathcal{A}')^{\vartheta_{13}-\vartheta_{23}}}="1" \ar@/_4pc/[rrrrrrrrrrrrrrrrrr]_{\pr_{13}^{*}\mathcal{A}'} \ar[rrrrrrrrrrrrr]_{(\pr_{13}^{*}\mathcal{T}_{\mathcal{H}} \otimes \id)^{-1}} \ar@{=}[uruurr]
  &&&&&&&&&&&&& i_{*}(\mathcal{H}_{13} \otimes \mathcal{S}_3) \ar[rrrrr]_-{i_{*}(\mathcal{A}_{13})} &&&&&
  i_{*}(\mathcal{S}_1) \ar@{}"4";"5"|*+[o][F-]{D}
}
\end{equation*}
\caption{}
\label{figure-2}
\end{figure}
The equality signs rely solely on definitions and identities among the differential forms arising in the adjusted extension:
\begin{itemize}
\item 
$i_{*}^{\vartheta_{13},\omega_{13}}(\mathcal{S}_3)=i_{*}^{(\vartheta_{13},\omega_{123}+\delta_2^{*}\rho)}(\mathcal{S}_3)$ because $\omega_{13}= \omega_{123}+\delta_2^{*}\rho$, which follows from \cref{delta-omega-rho}, 
\item
$i_{*}^{(\vartheta_{13},\omega_{123}+\delta_2^{*}\rho)}(\mathcal{S}_3)=i_{*}^{(\delta_2^{*}m^{*}q(\theta),\delta_2^{*}\rho+\omega_{123})}(\mathcal{S}_3)$ because $\vartheta_{13}= \delta_2^{*}m^{*}q(\theta)$, as $\pr_{13}=m\circ \delta_2$,

\item
$i_{*}^{(0,\omega_{123})}(\delta_2^{*}(m^{*}\mathcal{G} \otimes \mathcal{I}_{\rho}) \otimes \mathcal{S}_3)=i_{*}(\mathcal{H}_{13} \otimes \mathcal{S}_3)$, combining the two points above, 

\item
$i_{*}^{(\vartheta_{12},\omega_{12})}(\mathcal{H}_{23} \otimes \mathcal{S}_3)=i_{*}^{(\delta_2^{*}q(\theta)_1,\omega_{123})}(\delta_2^{*}\mathcal{G}_2 \otimes \mathcal{S}_3)$ uses $\pr_{23}=\pr_2 \circ\delta_2$ and $\pr_{12}=\pr_1 \circ \delta_2$, and

\item
$i_{*}(\mathcal{H}_{12} \otimes \mathcal{H}_{23} \otimes \mathcal{S}_3)=i_{*}^{(0,\omega_{123})}(\delta_2^{*}(\mathcal{G}_1 \otimes \mathcal{G}_2) \otimes  \mathcal{S}_3)$ follows from the preceding observations. 

\end{itemize}
We claim that all subdiagrams represent the definition of one of the bounding 1-morphisms, with three exceptions: 
\begin{itemize}
\item 
Subdiagram A is $\delta_2^{*}( - \otimes \mathcal{S}_3)$ applied to the diagram \cref{multiplicativity-of-trivialization-diag} in \cref{multiplicativity-of-trivialization}, and is thus filled by the 2-isomorphism $\delta_2^{*}(\sigma_{\Gamma} \otimes \id_{\id_{\mathcal{S}_3}})$. 

\item
Subdiagram B is filled by the 2-isomorphism from \cref{versuchs-iso:2}.

\item
Subdiagram C is commutative, since the adjusted shift $(\pr_{23}^{*}\mathcal{A}')^{\vartheta_{13}-\vartheta_{23}}$ decomposes via \cref{shift-of-1-morphisms:b} into the composite
\begin{equation*}
\alxydim{@C=6em}{i_{*}^{(\vartheta,\omega)}(\pr_2^{*}\mathcal{S}) \ar[rr]^-{((\pr_{23}^{*}\mathcal{T}_{\mathcal{H}} \otimes \id)^{-1})^{\vartheta_{13}-\vartheta_{23}}} && i_{*}(\mathcal{H} \otimes \pr_2^{*}\mathcal{S}) \ar[r]^-{i_{*}(\mathcal{A})^{\vartheta_{12}}} & i_{*}(\pr_1^{*}\mathcal{S})\text{,} }
\end{equation*}
where the shift by $\vartheta_{12}$ is computed from
\begin{equation*}
p_{*}((\pr_{23}^{*}\mathcal{T}_{\mathcal{H}} \otimes \id)^{-1})=\inv  \circ p_{*}(\pr_{23}^{*}\mathcal{T}_{\mathcal{H}})=\inv  \circ \pr_{23}^{*}p_{*}(\inv \circ \delta)=\delta \circ \pr_{23}
\end{equation*}
and 
\begin{align*}
\tilde q(p_{*}((\pr_{23}^{*}\mathcal{T}_{\mathcal{H}} \otimes \id)^{-1}),\vartheta_{13}-\vartheta_{23})
&=\tilde q(\delta \circ \pr_{23},\vartheta_{13}-\vartheta_{23})
\\&= q(\Ad_{\delta \circ \pr_{23}}(\pr_{13}^{*}\delta^{*}\theta-\pr_{23}^{*}\delta^{*}\theta))
\\&=q(\pr_{12}^{*}\delta^{*}\theta)
\\&=\vartheta_{12}\text{.}
\end{align*}

\item
Subdiagram D is filled by the 2-isomorphism $i_{*}(\sigma)$.

\end{itemize}
Consequently, the diagram defines a 2-isomorphism
\begin{equation*}
\sigma': \pr_{12}^{*}\mathcal{A}' \circ (\pr_{23}^{*}\mathcal{A}')^{\vartheta_{13}-\vartheta_{23}} \Rightarrow \pr_{13}^{*}\mathcal{A}'\text{.}
\end{equation*}
As the final step, we shift $\sigma'$ by $\pr_{13}^{*}\nu$ according to \cref{shift-of-1-morphisms}. The necessary computations yield:
\begin{align*}
(\pr_{12}^{*}\mathcal{A}')^{\tilde q(\pr_{23}^{*}\delta,\pr_{13}^{*}\nu)}&= \pr_{12}^{*}\widetilde{\mathcal{A}}
\\
((\pr_{23}^{*}\mathcal{A}')^{\vartheta_{13}-\vartheta_{23}})^{\pr_{13}^{*}\nu} &=\pr_{23}^{*}\widetilde{\mathcal{A}}
\end{align*} 
Thus, we obtain a 2-isomorphism
\begin{equation*}
\tilde\sigma: \pr_{12}^{*}\widetilde{\mathcal{A}} \circ \pr_{23}^{*}\widetilde{\mathcal{A}} \Rightarrow \pr_{13}^{*}\widetilde{\mathcal{A}}\text{.}
\end{equation*} 
The coherence condition for $\tilde\sigma$,
\begin{equation*}
\alxydim{@C=6em@R=5em}{\widetilde{\mathcal{S}_3}  \ar[r]^{\widetilde{A}_{23}} & \widetilde{\mathcal{S}_2} \ar[d]^{\widetilde{\mathcal{A}_{12}}} \\ \widetilde{\mathcal{S}_4} \ar[u]^{\widetilde{\mathcal{A}_{34}}} \ar[r]_{\widetilde{\mathcal{A}_{14}}} \ar[ur]|{\widetilde{\mathcal{A}_{24}}}="1" & \widetilde{\mathcal{S}_1} \ar@{=>}[ul];"1"|{\widetilde\sigma_{234}} \ar@{<=}"1"|{\widetilde\sigma_{123}} }
=
\alxydim{@C=6em@R=5em}{\widetilde{\mathcal{S}_3} \ar[dr]|{\widetilde{\mathcal{A}_{13}}}="1"  \ar[r]^{\widetilde{A}_{23}} & \widetilde{\mathcal{S}_2} \ar[d]^{\widetilde{\mathcal{A}_{12}}} \\ \widetilde{\mathcal{S}_4} \ar[u]^{\widetilde{\mathcal{A}_{34}}} \ar[r]_{\widetilde{\mathcal{A}_{14}}}  & \widetilde{\mathcal{S}_1}\text{,}  \ar@{=>}[u];"1"|{\widetilde\sigma_{123}} \ar@{=>}"1";[l]|{\widetilde\sigma_{134}}  }
\end{equation*}
follows from the equalities of 2-morphisms in \cref{multiplicativity-of-trivialization} and  \cref{coherence-for-sigma},
 between which the occurrences of $\alpha_{\Gamma}$ cancel.  This completes the  constriction of the matching family $\mathcal{D}(\mathbb{T}):=(\widetilde{\mathcal{S}},\widetilde{\mathcal{A}},\widetilde{\sigma})$, and thereby gives the  definition of the functor $\mathcal{D}$ of \cref{functor-trivializations-to-gerbes} on the level of objects.  

We remark that the primary challenge in constructing the matching family above was to eliminate the multiplicative bundle gerbe $\mathcal{G}_{\Gamma}$ from the trivialization data. However, 1-morphisms and 2-morphisms between trivializations do not involve $\mathcal{G}_{\Gamma}$ or its data, rendering their treatment significantly simpler. Let $(\mathcal{B},\beta): \mathbb{T}_1 \to \mathbb{T}_2$ be a 1-morphism between trivializations, with $\mathbb{T}_i=(\mathcal{S}_i,\mathcal{A}_i,\sigma_i)$. We define $\widetilde{\mathcal{B}} :=i_{*}(\mathcal{B})^{q(\eta)}:\widetilde{\mathcal{S}_1} \to \widetilde{\mathcal{S}_2}$ and consider the 2-isomorphism
\begin{equation*}
\alxydim{@C=4em}{i_{*}^{(\vartheta,\omega)}(\pr_2^{*}\mathcal{S}_1) \ar@/^2pc/[rrr]^{\mathcal{A}_1'} \ar[d]_{i_{*}(\pr_2^{*}\mathcal{B})^{\vartheta}} \ar[rr]^-{(\mathcal{T}_{\mathcal{H}} \otimes \id_{ \pr_2^{*}\mathcal{S}})^{-1}} && i_{*}(\mathcal{H} \otimes \pr_2^{*}\mathcal{S}_1) \ar@{=>}[dll]|*+{\beta_{\pr_2^{*}\mathcal{B}}} \ar[d]_{i_{*}(\id \otimes \pr_2^{*}\mathcal{B})} \ar[r]^-{i_{*}(\mathcal{A}_1)} &  i_{*}(\pr_1^{*}\mathcal{S}_1) \ar@{=>}[dl]|*+{i_{*}(\beta)} \ar[d]^{i_{*}(\pr_1^{*}\mathcal{B})} \\ i_{*}^{(\vartheta,\omega)}(\pr_2^{*}\mathcal{S}_2) \ar@/_2pc/[rrr]_{\mathcal{A}_2'} \ar[rr]_-{(\mathcal{T}_{\mathcal{H}} \otimes \id_{ \pr_2^{*}\mathcal{S}})^{-1}} &&  i_{*}(\mathcal{H} \otimes \pr_2^{*}\mathcal{S}_2 )\ar[r]_-{i_{*}(\mathcal{A}_2)} & i_{*}(\pr_1^{*}\mathcal{S}_2)}
\end{equation*}
over $P^{[2]}$, where $\beta_{\pr_2^{*}B}$ is the 2-isomorphism from \cref{versuchs-iso:2}.
Shifting by the 1-form $\nu\in \Omega^1(P^{[2]},\mathfrak{g})$ yields a 2-isomorphism 
\begin{equation*}
\widetilde{\beta}: \pr_2^{*}\widetilde{\mathcal{B}} \circ \widetilde{\mathcal{A}_2} \Rightarrow \pr_1^{*}\widetilde{\mathcal{B}} \circ \widetilde{\mathcal{A}_1}\text{,}
\end{equation*}
since $i_{*}(\pr_1^{*}\mathcal{B})^{q(p_{*}\mathcal{A}_1',\nu)} =\pr_1^{*}\widetilde{\mathcal{B}}$ and $i_{*}(\pr_2^{*}\mathcal{B})^{\vartheta+\nu}=\pr_2^{*}\widetilde{\mathcal{B}}$, by \cref{shift-formulas-1,shift-formulas-2}.
The pair $(\widetilde{\mathcal{B}},\widetilde{\beta})$ constitutes a 1-morphism between the matching families $(\widetilde{\mathcal{S}_1},\widetilde{\mathcal{A}_1},\widetilde{\sigma_1})$ and $(\widetilde{\mathcal{S}_2},\widetilde{\mathcal{A}_2},\widetilde{\sigma_2})$.
Finally, any 2-morphism $\gamma: (\mathcal{B},\beta) \Rightarrow (\mathcal{B}',\beta')$ induces, without further modification, a 2-morphism $i_{*}(\gamma):(\widetilde{\mathcal{B}},\widetilde{\beta}) \Rightarrow (\widetilde{\mathcal{B}'},\widetilde{\beta'})$ of matching families.

 This completes the construction of the functor 
\begin{equation*}
\mathcal{D}: \trivcon{\mathbb{CS}_{P}(\mathcal{G}_\Gamma)} \to  \des_{\grbcon{\Gamma\!,u,\kappa}-}(\rho)\text{,}
\end{equation*}
announced in \cref{functor-trivializations-to-gerbes}.
However, $\mathcal{D}$ is not yet an equivalence, as the codomain does not fully encode the information of the principal $F$-bundle $P$ and its connection. To rectify this, we consider the diagram
\begin{equation}
\label{homotopy-pullback-X}
\alxydim{}{&& \ast \ar[d]^{(P,\eta)}\\&& \buncon FM_{\dis} \ar[d]^{r_\rho} \\\des_{\grbcon{\Gamma\!,u,\kappa}-}(\rho)\ar[r]^-{p_{*}} & \des_{(\buncon{F}-)_{\dis}}(\rho) \ar@{=}[r] &\des_{\buncon{F}-}(\rho)_{\dis}\text{,}}
\end{equation}
of functors between bicategories. The horizontal functor arises from the covariance of descent data with respect to the sheaf morphism $p_{*}: \grbcon{\Gamma\!,u,\kappa}- \to (\buncon F-)_{\dis}$. The upper vertical functor is the inclusion of the object $(P,\eta)$ of $\buncon FM$. The lower vertical functor $r_\rho$ assigns descent data to globally defined objects (see \cref{sec:plus}). Let $\mathcal{X}$ denote the (homotopy) pullback of this diagram.

\begin{lemma}
\label{lift-to-X}
The functor $\mathcal{D}$ lifts canonically along $\mathcal{X} \to \des_{\grbcon{\Gamma\!,u,\kappa}-}(\rho)$ to a functor 
\begin{equation*}
\widetilde{\mathcal{D}}: \trivcon{\mathbb{CS}_{P}(\mathcal{G}_\Gamma)} \to \mathcal{X}\text{.}
\end{equation*}
\end{lemma}

\begin{proof}
We invoke the universal property of the pullback. The vertical functor in \cref{homotopy-pullback-X}, $\ast \to \des_{\buncon{F}-}(\rho)$, assigns to $(P,\eta)$ its canonical descent data with respect to the projection $\rho$ of the bundle $P$, which is the matching family $(\rho^{*}P,\varphi)$, where $\varphi$ is the canonical isomorphism $\pr_1^{*}\rho^{*}P \cong \pr_2^{*}\rho^{*}P$ induced by the equality $\rho \circ \pr_1=\rho \circ \pr_2$.      

The composition of $\mathcal{D}$ with $p_{*}$ can be deduced from the constructions above. First, $p_{*}(\widetilde{\mathcal{S}})=p_{*}(i_{*}^{(q(\eta),0)}(\mathcal{S}))=\trivlin_1^{\eta}$ by \cref{remark-adjusted-extension:3}. Second, $p_{*}(\widetilde {\mathcal{A}})=\delta:P^{[2]} \to F$, as noted earlier, interpreted as the morphism $\varphi_\delta: \pr_1^{*}\trivlin_1^{\eta} \to \pr_2^{*}\trivlin_1^{\eta}$ of trivial $F$-bundles over $P^{[2]}$ given by left multiplication by $\delta$.   

Now, the canonical trivialization $\psi: \rho^{*}P \to \trivlin_1$ is an isomorphism of matching families $r_{\rho}(P)=(\rho^{*}P,\varphi) \cong (\trivlin_1^{\eta},\varphi_{\delta})$, which is easily verified. This establishes commutativity up to a (constant) natural isomorphism, thereby identifying $\trivcon{\mathbb{CS}_{P}(\mathcal{G}_\Gamma)}$ as a cone for the diagram \cref{homotopy-pullback-X}, and hence mapping it to the limiting cone $\mathcal{X}$.
\end{proof}

\begin{lemma}
\label{X-equivalence}
The lifted functor $\widetilde{\mathcal{D}}$ of \cref{lift-to-X} is an equivalence,
\begin{equation*}
\trivcon{\mathbb{CS}_{P}(\mathcal{G}_\Gamma)} \cong \mathcal{X}\text{.}
\end{equation*}
\end{lemma}

\begin{proof}
The construction of the functor $\mathcal{D}$ employed four tools:
\begin{enumerate}

\item 
adjusted extensions of bundle gerbes,

\item
extensions of 1-morphisms and 2-morphisms along $i: BA \to \Gamma$,
\item
various compositions with 1-isomorphisms and 2-isomorphisms, and

\item
adjusted shifts of 1-morphisms.

\end{enumerate}
Tools 3 and 4 are reversible operations (see \cref{shift-of-1-morphisms:a}). Thus, the main point of the claimed equivalence is to show that all extensions can be reversed. This will be done using \cref{gerbe-morphisms-under-i-non-flat}.

We begin by proving essential surjectivity. Suppose $(\mathcal{R},\mathcal{Q},\chi)$ is descent data for $\grbcon{\Gamma\!,u,\kappa}-$ with respect to $\rho$, together with an isomorphism 
\begin{equation*}
\varphi: p_{*}(\mathcal{R},\mathcal{Q},\chi) \to (\trivlin_1^{\eta}, \varphi_\delta)
\end{equation*}
in $\des_{\buncon F-}(\rho)$ promoting it to an object in $\mathcal{X}$. This means that $\varphi$ is a connection-preserving bundle isomorphism $\varphi: p_{*}(\mathcal{R}) \to \trivlin_1^{\eta}$ over $P$ such that 
\begin{equation}
\label{condition-for-anchors}
\pr_2^{*}\varphi \circ p_{*}(\mathcal{Q}) = \varphi_\delta \circ \pr_1^{*}\varphi\text{.}
\end{equation}
Since $\trivlin_1^{\eta}$ has a canonical section with covariant derivative $\eta$, via the isomorphism $\varphi$, $p_{*}(\mathcal{R})$ also has such a section, enhancing $\mathcal{R}$ to an object in the bicategory $\grbcon{\Gamma\!,u,\kappa}M^{\pcurved\eta}$ defined before \cref{gerbe-morphisms-under-i-non-flat}. Thus, it determines an $A$-bundle gerbe $\mathcal{S}$ with connection together with an isomorphism $\mathcal{B}: \mathcal{R} \to i_{*}^{(q(\eta),0)}(\mathcal{S})$ in $\grbcon{\Gamma\!,u,\kappa}M^{\pcurved\eta}$, i.e., $p_{*}(\mathcal{B})=\varphi$. 

We may induce along $\mathcal{B}$ a modified matching family $(\widetilde{\mathcal{S}},\mathcal{Q}',\chi')$, such that $\mathcal{B}$ gives an isomorphism 
 $(\mathcal{R},\mathcal{Q},\chi) \cong (\widetilde{\mathcal{S}},\mathcal{Q}',\chi')$ in $\des_{\grbcon{\Gamma\!,u,\kappa}-}(\rho)$.
When we equip $(\widetilde{\mathcal{S}},\mathcal{Q}',\chi')$ with the identity isomorphism 
\begin{equation*}
\id: p_{*}(\widetilde{\mathcal{S}},\mathcal{Q}',\chi') \to (\trivlin_1^{\eta}, \varphi_\delta)\text{,}
\end{equation*}
it becomes an isomorphism in $\mathcal{X}$. \cref{condition-for-anchors} now means that $p_{*}(\mathcal{Q}')=\delta$. Reversing the effects of the functor $\mathcal{D}$, we shift $\mathcal{Q}'$ by $-\nu$ and pre-compose with $\mathcal{T}_{\mathcal{H}} \otimes \id_{ \pr_2^{*}\mathcal{S}}$, obtaining an isomorphism $\mathcal{A}' :i_{*}(\mathcal{H} \otimes \pr_2^{*}\mathcal{S}) \to i_{*}(\pr_1^{*}\mathcal{S})$ with $p_{*}(\mathcal{A}')=\id$. Hence, again by \cref{gerbe-morphisms-under-i-non-flat}, there exists an isomorphism $\mathcal{A}: \mathcal{H} \otimes \pr_2^{*}\mathcal{S} \to \pr_1^{*}\mathcal{S}$ in $\grbcon A {P^{[2]}}$ together with a 2-isomorphism $\mathcal{A}'\cong i_{*}(\mathcal{A})$. 

We induce again a modified matching family $ (\widetilde{\mathcal{S}},\widetilde{\mathcal{A}},\chi'')$, with an isomorphism $(\widetilde{\mathcal{S}},\mathcal{Q}',\chi') \cong (\widetilde{\mathcal{S}},\widetilde{\mathcal{A}},\chi'')$ in $\mathcal{X}$. Here, $\chi''$ is a 2-isomorphism $\chi'': \pr_{12}^{*}\widetilde{\mathcal{A}} \circ \pr_{23}^{*}\widetilde{\mathcal{A}} \Rightarrow \pr_{13}^{*}\widetilde{\mathcal{A}}$. Again undoing the effects of the functor $\mathcal{D}$, we shift by $-\pr_{13}^{*}\nu$ and reverse all compositions performed in the large diagram above, ending with the 2-isomorphism in subdiagram D; namely, 
\begin{equation*}
\chi''':i_{*}(\pr_{12}^{*}\mathcal{A}) \circ i_{*}(\id \otimes \pr_{23}^{*}\mathcal{A}) \Rightarrow i_{*}(\pr_{13}^{*}\mathcal{A}) \circ i_{*}(\mathcal{M}' \otimes \id)\text{.}
\end{equation*}
By the local fully faithfulness of the extension $i_{*}$, there exists a unique 2-isomorphism $\sigma$ such that $i_{*}(\sigma)=\chi'''$. This shows that $\widetilde{\sigma} = \chi''$, which finally implies that $(\widetilde{\mathcal{S}},\widetilde{\mathcal{A}},\chi'')=(\widetilde{\mathcal{S}},\widetilde{\mathcal{A}},\widetilde{\sigma})$ and reveals $(\mathcal{R},\mathcal{Q},\chi)$ as having the essential preimage $\mathbb{T}:=(\mathcal{S},\mathcal{A},\sigma)$. 

On the level of 1-morphisms and 2-morphisms, everything is again somewhat simpler and works exactly as above.
\end{proof}

\begin{theorem}
\label{lifting-theorem}
Let $\Gamma$ be a smoothly separable, central crossed module equipped with a splitting $u$ and an adjustment $\kappa$ that is adapted to $u$. Let $P$ be a principal $F$-bundle with connection. Then, there is an equivalence of bicategories, 
\begin{equation*}
\trivcon{\mathbb{CS}_{P}(\mathcal{G}_\Gamma)} \cong\liftcon{{(\Gamma\!,u,\kappa)}} P{}\text{,} 
\end{equation*}
between the trivializations of the Chern-Simons bundle 2-gerbe, and the lifts of $P$ to a $\Gamma$-bundle gerbe with adjusted and adapted connection. 
The equivalence is functorial w.r.t. connection-preserving isomorphisms $P \to P'$.
Moreover, under the equivalence, the covariant derivative of a trivialization coincides with the adjusted curvature of the corresponding $(\Gamma\!,u,\kappa)$-lift.
\end{theorem}

\begin{proof}
We note that $\liftcon{(\Gamma\!,u,\kappa)} P{}$, as defined in \cref{gamma-lifts}, is the (homotopy) pullback of the diagram
\begin{equation}
\label{homotopy-pullback-lifts}
\alxydim{}{& \ast \ar[d]^{(P,\eta)} \\  \grbcon{\Gamma\!,u,\kappa}M \ar[r]_{p_{*}} & (\buncon FM)_{\dis}\text{.}}
\end{equation}
We consider the commutative diagram
\begin{equation*}
\alxydim{}{\grbcon{\Gamma\!,u,\kappa}M \ar[d]_{r_\rho} \ar[r]^-{p_{*}} & (\buncon FM)_{\dis} \ar[d]^{r_\rho}
\\ \des_{\grbcon{\Gamma\!,u,\kappa}-}(\rho) \ar[r]_-{p_{*}} &  (\des_{\buncon F-}(\rho))_{\dis}}
\end{equation*}
whose vertical arrows are equivalences since $\grbcon{\Gamma\!,u,\kappa}M$ and $\buncon F-$ are sheaves (\cref{th:equivcongrb}). This shows that the diagrams \cref{homotopy-pullback-X,homotopy-pullback-lifts} are naturally isomorphic, providing an equivalence between their homotopy limits, i.e., $\mathcal{X}\cong \liftcon{(\Gamma\!,u,\kappa)} P{}$. \cref{X-equivalence} then completes the proof.

The functoriality in connection-preserving bundle isomorphisms is clear. Concerning the claim about the curvatures, we have 
\begin{align*}
\mathrm{curv}_{\kappa}(\widetilde{\mathcal{S}})
=\mathrm{curv}_{\kappa}(i_{*}^{(q(\eta),0)}(\mathcal{S}))
\eqcref{extension-and-adjusted-shift-lemma} i_{*}(\mathrm{curv}(\mathcal{S}) +C)= i_{*}(\rho^{*}K )\text{.}
\end{align*}
\end{proof}

\begin{remark}
\label{explain-span}	
The equivalence of \cref{lifting-theorem} is established as a cospan of bicategories,
\begin{equation}
	\label{span}
\alxydim{}{\trivcon{\mathbb{CS}_{P}(\mathcal{G}_\Gamma)} \ar[r]^-{\widetilde{\mathcal{D}}} & \mathcal{X}  & \liftcon{{(\Gamma\!,u,\kappa)}} P{}\text{.} \ar[l]_-{r_{\rho}}}
\end{equation}
Concretely, a trivialization $\mathbb{T}=(\mathcal{S},\mathcal{A},\sigma)$ of $\mathbb{CS}_P(\mathcal{G}_{\Gamma})$ corresponds to a $(\Gamma\!,u,\kappa)$-lift $(\mathcal{G},\varphi)$ of $P$ if and only if their images in $\mathcal{X}$ are isomorphic: i.e., there exists an isomorphism $\mathcal{B}: (\widetilde{\mathcal{S}},\widetilde{\mathcal{A}},\widetilde{\sigma})\to r_\rho(\mathcal{G})$ in $\des_{\grbcon{\Gamma\!,u,\kappa}-}(\rho)$ such that $\id_{(\trivlin^\eta_1,\varphi_{\delta})} = \psi \circ r_{\rho}(\varphi)\circ p_{*}(\mathcal{B})$ in $\des_{\buncon{F}-}(\rho)$, where $\psi$ is the morphism $r_{\rho}(P)\to (\trivlin_1^{\eta},\varphi_{\delta})$ induced by the canonical trivialization $\psi: \rho^{*}P \to \trivlin_1$; see the proof of \cref{lift-to-X}. Choosing a functorial descent (i.e., an inverse functor to $r_\rho$) inverts the second arrow in \cref{span} and provides a functor that establishes the equivalence of \cref{lifting-theorem}.
\end{remark}

Regarding the application to string structures, there is nothing further to add beyond what was discussed in the introduction. We now describe further consequences of \cref{lifting-theorem}.
First, in conjunction with \cref{2-gerbes-2}, we obtain:

\begin{corollary}
$(\Gamma\!,u,\kappa)$-lifts of $P$ exist if and only if the characteristic class of the Chern-Simons 2-gerbe $\mathbb{CS}_P(\mathcal{G}_{\Gamma})$ vanishes in $\check \h^3(M,\sheaf{A})$, and they form a torsor over the monoidal bigroupoid $\grbcon AM$.
\end{corollary}

Together with \cref{exact-sequence-1}, we have the following.

\begin{corollary}
The sequence
\begin{equation*}
\hat\h^2(M,A) \to \hat \h^1(M,(\Gamma\!,u,\kappa))^{\aptadj} \to \hat\h^1(M,F)\to \check\h^3(M,\sheaf A)
\end{equation*}
in differential and \v Cech cohomology, which is induced by the functors $i_{*}$, $p_{*}$, and $P \mapsto c_{\mathbb{CS}_P(\mathcal{G}_{\Gamma})}$, respectively, is an exact sequence of pointed sets.
\end{corollary}

Second, by passing from $\Gamma\!$-bundle gerbes to abelian 2-gerbes and back, one can use existence results for connections available in the abelian setting.

\begin{corollary}
\label{existence-of-connections}
Every $\Gamma\!$-bundle gerbe admits an adjusted and adapted connection. 
\end{corollary}

\begin{proof}
Let $\mathcal{G}$ be a $\Gamma\!$-bundle gerbe without connection, and let $P:=p_{*}(\mathcal{G})$. Then, $\mathcal{G}$ is a $\Gamma\!$-lift of $P$ (without connection), and hence corresponds, under \cite[Thm. A]{Nikolausa}, to a trivialization $\mathbb{T}$ of the Chern-Simons 2-gerbe (without connection data). We choose a connection on $P$; then, as described in \cref{lifting-2-gerbe}, the bundle 2-gerbe acquires a connection. By \cite[Prop. 3.3.1]{waldorf8}, $\mathbb{T}$ admits a compatible connection. By \cref{lifting-theorem}, it then corresponds to a $(\Gamma\!,u,\kappa)$-lift $\mathcal{G}'$ of $P$ with connection. Because the lifting result is an equivalence also in the setting without connections \cite[Thm. A]{Nikolausa}, there exists an isomorphism $\mathcal{G} \cong \mathcal{G}'$ of bundle gerbes (without connection). By \cref{connection-transport-along-1-morphism}, $\mathcal{G}$ admits a connection, and by \cref{fake-curvature-under-isomorphisms} it must be adjusted. 
\end{proof}

Third, the fact that the equivalence of \cref{lifting-theorem} is compatible with connection-preserving isomorphisms between principal $F$-bundles means that it is an isomorphism 
\begin{equation*}
\trivcon{\mathbb{CS}_{-}(\mathcal{G}_\Gamma)} \cong\liftcon{(\Gamma\!,u,\kappa)}-{}
\end{equation*}
between presheaves of bicategories on the category $\buncon FM$. We may then perform the Grothendieck construction (\quot{unstraightening}) with both presheaves \cite{Bakovic,Johnson2021}. We observe that the canonical functor
\begin{equation*}
\grbcon\Gamma M \to \int_{P\in \buncon FM} \liftcon{(\Gamma\!,u,\kappa)}P{}:\mathcal{G} \mapsto (p_{*}\mathcal{G},\mathcal{G}) 
\end{equation*}
is an equivalence, with inverse $(P,\mathcal{G}) \mapsto \mathcal{G}$. It expresses the simple fact that every $\Gamma\!$-bundle gerbe lifts its own underlying principal $F$-bundle. This shows the following:

\begin{corollary}
\label{equivalence-with-abelian}
There is an equivalence of bicategories
\begin{equation*}
\grbcon\Gamma M \cong \int_{P\in \buncon FM} \trivcon{\mathbb{CS}_P(\mathcal{G}_{\Gamma})}\text{.}
\end{equation*}
\end{corollary}

The remarkable consequence of this result is that (adjusted and adapted) non-abelian categorical gauge theory can be expressed completely by \emph{abelian} 2-categorical gauge theory. This may be used, in the future, to study the parallel transport of connections on non-abelian bundle gerbes beyond the fake-flat regime.

\subsection{Transgression to loop spaces}

\label{transgression-to-loop-spaces}

In this complimentary section we use \cref{lifting-theorem} together with the fact that transgression to loop space is well understood for abelian bundle gerbes, in order to generalize it to non-abelian ones. 
 The basic machinery we need here is provided by the transgression functor
 \begin{equation}
 	\label{basic-transgression}
\mathscr{T}: \hc 1 \grbcon AX \to \buncon A{LX}\text{,}
\end{equation}
which takes bundle $A$-gerbes with connection over a smooth manifold $X$ to principal $A$-bundles with connections on the loop space $LX = C^\infty(S^1,M)$; see  \cite{brylinski4,waldorf5,waldorf10}. 
 As before, $\hc 1$ denotes the homotopy 1-category, i.e., the category with the same objects and 2-isomorphism classes of 1-morphisms. We remark that a corresponding functor defined on bundle gerbes \emph{without connections} does not exist -- in that setting it exists only on the level of isomorphism classes of objects ($\hc 0$). 
 
 The following discussion is again parallel to \cite{Nikolausa}, but now in fact \emph{easier} than there: over there, we had to choose connections as auxiliary data, followed by proofs that the results did not depend on these choices; here, we understand the connections as given data.

We introduce two constructions centered around the crossed module $\Gamma$ using the transgression functor $\mathscr{T}$ of \cref{basic-transgression}. They rely on the multiplicative bundle gerbe $\mathcal{G}_\Gamma$ and its connection defined in \cref{multiplicative-bundle-gerbe}, a prinicpal $F$-bundle $P$ with connection $\eta$ over $M$, where $F=\pi_0(\Gamma)$, and the Chern-Simons bundle 2-gerbe of \cref{lifting-2-gerbe}.

The first construction concerns the transgression of the Chern-Simons bundle 2-gerbe $\mathbb{CS}_P(\mathcal{G}_{\Gamma})$, which one can define following \cite[\S 5.2]{Waldorfa} and \cite{Nikolausa} in the case that $F=\pi_0(\Gamma)$ is connected, what we assume in the following. Then, its transgression is the following $A$-bundle gerbe $\mathscr{T}_{\mathbb{CS}_P(\mathcal{G}_{\Gamma})}$ with connection over $LM$: 
\begin{enumerate}
\item 
Its surjective submersion is $L\rho: LP \to LM$; surjectivity requires the connectedness of $F$.

\item
Its curving is the 2-form $B:=- \int_{S^1} \ev^{*}C\in \Omega^2(LP)$; this is the usual fibre integration of differential forms, with $\ev: S^1 \times LM \to M$ the evaluation map. 

\item 
Its principal $A$-bundle over $LP^{[2]}$ is $\mathscr{T}(\mathcal{H})$; this uses that -- on the level of curvature forms -- $\mathscr{T}$ corresponds exactly to minus the fibre integration.

\item
Finally, its bundle gerbe product over $LP^{[3]}$ is $\mathscr{T}(\mathcal{M})$; this is associative because $\hc 1$ of the 2-isomorphism of (d) in \cref{lifting-2-gerbe} becomes a commutative diagram. 

\end{enumerate}
In the same analogous way, any trivialization $\mathbb{T}$ of $\mathbb{CS}_P(\mathcal{G})$ transgresses to a trivialization of $\mathscr{T}_{\mathbb{CS}_P(\mathcal{G}_{\Gamma})}$. This establishes our first construction, a functor
 \begin{equation}
 	\label{transgression-functor-2}
 	\hc 1\trivcon{\mathbb{CS}_{P}(\mathcal{G}_\Gamma)} \to \trivcon{\mathscr{T}_{\mathbb{CS}_{P}(\mathcal{G}_\Gamma)}};
 \end{equation}
 see \cite[\S 5.2]{Waldorfa}. Here, on the left, is the bicategory of trivializations of a bundle 2-gerbe with connection, as defined in \cref{lifting-2-gerbe}, whereas on the right is the analogous category of trivializations of a bundle gerbe with connection. 
 
Our second construction is the application of the functor $\mathscr{T}$ to the multiplicative bundle gerbe $\mathcal{G}_\Gamma$ and its connection. As described in detail in \cite[\S 3.1]{waldorf5}, the multiplicative structure induces a (Fr\'echet) Lie group structure on the total space of the transgressed principal $A$-bundle $\mathscr{T}(\mathcal{G}_\Gamma)$. Viewed as a Lie group, it will be denoted by $L\Gamma_{u,\kappa}$, and we call it the \emph{loop group} of $\Gamma$. This is motivated by the fact that for $\Gamma=G_{dis}$ (equipped with the zero splitting and zero adjustment), we have $L\Gamma_{u,\kappa}$=$LG$, the classical loop group.  The loop group of $\Gamma$ depends only up to isomorphism on the splitting $u$ and the adjustment $\kappa$, and is in fact a central extension
\begin{equation}
\label{transgressed-central-extension}
1 \to A \to L\Gamma_{u,\kappa} \to LF \to 1
\end{equation}
of the loop group $LF$, see \cite[Thm. 3.7]{waldorf5}.

As a third variant of transgression -- however independent of the functor $\mathscr{T}$ -- we recall a method to \quot{transgress} the bundle $P$ to the loop space $LM$. Namely, if $F$ is connected (we assume this in the following) then $LP=C^\infty(S^1,P)$ is a principal $LF$-bundle over $LM$. Moreover, the connection $\eta$ on P induces a connection $\bar\eta \in \Omega^1(LP,L\mathfrak{f})$ on $LP$, defined by
$\bar\eta|_\tau(X)(t) := \eta|_{\tau(t)}(X(z))$,
where $\tau\in LP$ and $X\in T_\tau LP$.

Now we have a principal $LF$-bundle with connection, $LP$, and the central extension $L\Gamma_{u,\kappa}$ of $LF$. Associated to such situation is the \emph{lifting bundle gerbe} $\mathcal{L}_{LP}$ \cite{murray}:
\begin{enumerate}
	
\item 
Its surjective submersion is $L\rho: LP \to LM$.

\item 
Its principal $A$-bundle over $LP^{[2]}$ is $L\delta^{*}L\Gamma_{u,\kappa}$, the pullback of the underlying principal $A$-bundle of the central extension \cref{transgressed-central-extension} along the looping of the map $\delta: P^{[2]} \to F$.

\item 
Its bundle gerbe product is the multiplication of the group $L\Gamma_{u,\kappa}$. 

\end{enumerate}
There is a standard theory by Gomi to equip lifting gerbes with a connection \cite{gomi3}. Required for this is a splitting  $s$ of the Lie algebra extension associated to \cref{transgressed-central-extension} and a so-called reduction $r: LP \times L\mathfrak{f} \to \mathfrak{a}$. In the present setting, both are canonically present. The splitting $s$ is defined using the connection $\nu$ on the principal $A$-bundle $\mathscr{T}(\mathcal{G}_\Gamma)$ whose total space is the manifold underlying the Lie group $L\Gamma_{u,\kappa}$, by setting $s(X) := \tilde X - \nu(\tilde X)$, where $X\in L\mathfrak{f}$ and $\tilde X$ is any lift to the Lie algebra of $L\Gamma_{u,\kappa}$; see \cite[\S 2.2]{Waldorfb}. The reduction is defined by 
\begin{equation*}
	r(\tau,X) := -2 \int_{S^1} b_\kappa (\tau^*\eta, X );
\end{equation*}
see \cite[Prop. 6.2]{gomi3}. Then, lifting theory has the following result \cite{gomi3}, \cite[Thm. 2.2]{waldorf13}:

\begin{lemma}
	\label{transgression-lemma-1}
	There is a canonical equivalence of categories,
	\begin{equation*}
		\trivcon{\mathcal{L}_{LP}}
		\cong
		\liftcon {L\Gamma_{u,\kappa}}{LP}{},
	\end{equation*}
	between the category of trivializations of $\mathcal{L}_{LP}$, and the category of lifts of $LP$ (with connections) along the central extension $L\Gamma_{u,\kappa}$. 
\end{lemma}

This whole section is based on the observation that the two bundle gerbes $\mathcal{L}_{LP}$ and $\mathscr{T}_{\mathbb{CS}_P(\mathcal{G}_{\Gamma})}$ over $LM$ are, more or less by construction, isomorphic. However, the connections do not directly fit. We recall that connections on bundle gerbes can be shifted by 1-forms $\zeta \in \Omega^1(LP,\mathfrak{a})$: the curving is shifted by $\mathrm{d}\zeta$ and the connection on its principal $A$-bundle is shifted by $\pr_1^*\zeta-\pr_2^*\zeta$. The relevant 1-form here is defined by $\zeta_\tau(X):=r(\tau,\bar\eta_\tau(X))$, with $\tau\in LP$ and $X\in T_\tau LP$.  We denote by $\mathcal{L}_{LP}^\zeta$ the lifting gerbe with the shifted connection. 
The following result is proved in \cite[Thm. 7.1]{Waldorfb} in the case of the string 2-group, but the proof given there applies literally to any Lie 2-group. 

\begin{lemma}
	\label{transgression-lemma-2}
	There is a canonical isomorphism of $A$-bundle gerbes with connection over $LM$,
	\begin{equation*}
		\mathcal{L}_{LP}^\zeta \cong \mathscr{T}_{\mathbb{CS}_P(\mathcal{G}_\Gamma)}.
	\end{equation*}
\end{lemma}

Connections on trivializations can be shifted by the same 1-form. This shift yields an isomorphism of categories,
\begin{equation}
	\label{transgression-lemma-3}
	\trivcon{\mathcal{L}_{LP}} \cong \trivcon{\mathcal{L}_{LP}^\zeta}
\end{equation}
Now we are in position to combine the equivalences and functors collected above. 

\begin{theorem}
	\label{transgression-theorem}
Let $\Gamma$ be a Lie 2-group with $F=\pi_0(\Gamma)$ connected. 
Above equivalences and functors yield a functor
\begin{equation*}
\mathscr{T}:\hc 1\grbcon {\Gamma,u,\kappa} M \to \buncon{L\Gamma_{u,\kappa}}{LM}\text{.} 
\end{equation*} 
establishing the transgression of $\Gamma$-bundle gerbes with connections to the loop space.   
On the level of adjusted
non-abelian differential cohomology, it realizes a map
\begin{equation*}
\hat \h^1(M,(\Gamma\!,u,\kappa))^{\aptadj} \to \hat \h^1(LM,L\Gamma_{u,\kappa})
\end{equation*}
that covers, under  projection to ordinary (i.e., non-differential) non-abelian cohomology, the transgression map of \cite[Thm. 6.3.4]{Nikolausa}.
\end{theorem}

\begin{proof}
We summarize the construction: the first step is to go via 	\cref{lifting-theorem} into the setting of trivializations of abelian bundle 2-gerbes (after choosing a functorial descend; see \cref{explain-span}):
\begin{equation*}
		\hc 1 \grbcon {\Gamma,u,\kappa} M \to \hc 1 \trivcon{\mathbb{CS}_P(\mathcal{G}_\Gamma)} \text{.}
\end{equation*}
The second step is concerned with only the abelian bundle gerbe theory: transgression \cref{transgression-functor-2}, identification of the lifting bundle gerbe (\cref{transgression-lemma-2}), and shift of its connection \cref{transgression-lemma-3},	
	\begin{equation*}
		\alxydim{}{
			\hc 1 \trivcon{\mathbb{CS}_P(\mathcal{G}_\Gamma)} \ar[r] & \trivcon{\mathscr{T}_{\mathbb{CS}_P(\mathcal{G}_\Gamma)}} \ar[r] & 
			\trivcon{\mathcal{L}_{LP}^\zeta} \ar[r] &
			\trivcon{\mathcal{L}_{LP}} \text{.}
	    }
	\end{equation*}
The final step is the translation back from trivializations of abelian bundle gerbes to principal bundles, using \cref{transgression-lemma-1}, together with the operation that forgets the lifted bundle:
	\begin{equation*}
	\alxydim{}{
		\trivcon{\mathcal{L}_{LP}} \ar[r] &
		\liftcon {L\Gamma_{u,\kappa}}{LP}{} \ar[r] &
		\buncon{L\Gamma_{u,\kappa}}{LM}\text{.}
	}
\end{equation*}
The claim that a map on the level of differential cohomology is induced follows from \cref{classification}. That this map recovers the one of \cite{Nikolausa} follows because -- upon forgetting all connection data -- above sequence of constructions is exactly the one performed there. 
\end{proof}

\begin{appendix}

\setsecnumdepth{1}

\section{The plus construction for sheaves of bigroupoids}

\label{sec:plus}

The classical plus construction is an endofunctor $\mathcal{F} \mapsto \mathcal{F}^{+}$ on the category of set-valued presheaves over a site $\mathcal{S}$, with the property that it takes presheaves to separated presheaves and separated presheaves to sheaves. 
The plus construction has analogues for presheaves with values in (higher) categories; see \cite[Section 6.5.3]{Lurie2009} for a general discussion. Informally, sheafifying a presheaf of $n$-categories requires $n+2$ iterations, each increasing the level of separateness by one. In this article, we are interested in the stackification of presheaves of bicategories (in fact: bigroupoids), which requires 4 iterations. As far as we know, there is no written account of a general plus construction for presheaves of bicategories that carries a general presheaf through these four steps. Instead, we follow \cite{nikolaus2}, where the following procedure is proposed.

For concreteness, we restrict our discussion to the site $\Man$
of smooth manifolds, equipped with the Grothendieck pretopology of surjective submersions. Let $\mathcal{F}$ be a presheaf of bigroupoids on $\Man$.  
For every surjective submersion $\pi:Y \to M$ we let $\des_{\mathcal{F}}(\pi)$ be the bigroupoid of descent data, a.k.a. matching families, of $\mathcal{F}$ with respect to $\pi$, 
\begin{equation*}
\des_{\mathcal{F}}(\pi) := \lim \left ( \alxydim{}{\mathcal{F}(Y) \arr[r] & \mathcal{F}(Y^{[2]}) \arrr[r] & \mathcal{F}(Y^{[3]}) \arrrr[r] & \mathcal{F}(Y^{[4]})}  \right )\text{.}
\end{equation*}
Here, the arrows are induced under $\mathcal{F}$ from the face maps of the simplicial manifold $Y^{[\bullet]}$ of fibre products. 

We denote the \quot{restriction} 2-functor that equips globally defined objects with their canonical descent data by $r_{\pi}: \mathcal{F}(M) \to \des_{\mathcal{F}}(\pi)$. 
For any two objects $A,B\in \mathcal{F}(M)$, it induces a functor 
\begin{equation*} 
r_{\pi}(A,B): \hom_{\mathcal{F}(M)}(A,B) \to \hom_{\des_{\mathcal{F}}(\pi)}(r_{\pi}(A),r_{\pi}(B))
\end{equation*} 
on Hom-groupoids. We say that the presheaf $\mathcal{F}$ is:
\begin{enumerate}

\item
\emph{separated}, if $r_{\pi}(A,B)$ is full and faithful for all objects $A,B \in \mathcal{F}(M)$ and all $\pi$.

\item 
a \emph{pre-2-stack}, if $r_{\pi}(A,B)$ is an equivalence for all objects $A,B \in \mathcal{F}(M)$ and all $\pi$.

\item
a \emph{2-stack} or \emph{sheaf of bigroupoids}, if $r_{\pi}$ is an equivalence of bigroupoids for all $\pi$.

\end{enumerate}
Following \cite{nikolaus2}, we then apply two constructions:
\begin{enumerate}[(a)]

\item 
\emph{Hom-set closure}: it turns a separated presheaf of bigroupoids into a pre-2-stack.

\item
\emph{Plus construction}: it turns a pre-2-stack into a 2-stack.

\end{enumerate}

\label{sec:homsetclosure}

In the following we recall these two constructions. We start with a characterization of when a presheaf of bigroupoids is a pre-2-stack.
Let $\mathcal{F}$ be a presheaf of bigroupoids over a smooth manifold $M$. Consider two objects $A,B\in \mathcal{F}(M)$.
The assignment 
\begin{equation*}
\mathcal{F}_{A,B} :\pi \mapsto \hom_{\mathcal{F}(Y)}(\pi^{*}A,\pi^{*}B) 
\end{equation*}
of a groupoid to a smooth map $\pi:Y \to M$ is a presheaf of groupoids on the over-category $\Man/M$.  
We denote by $\des_{\mathcal{F}_{A,B}}(\pi)$ the groupoid of descent data with respect to a surjective submersion $\pi:Y \to M$, and by $r^{A,B}_{\pi}: \mathcal{F}_{A,B}(\id_M) \to \des_{\mathcal{F}_{A,B}}(\pi)$ the restriction functor. Following standard terminology, we call the presheaf $\mathcal{F}_{A,B}$ \emph{prestack} if $r^{A,B}_{\pi}$ is fully faithful, and \emph{stack} if $r_{\pi}^{A,B}$ is an equivalence of categories. 
Resolving the definition of $\mathcal{F}_{A,B}$, we have
\begin{equation*}
\des_{\mathcal{F}_{A,B}}(\pi)=\hom_{\des_{\mathcal{F}}(\pi)}(r_{\pi}(A),r_{\pi}(B))\text{,}
\end{equation*}
and $r_{\pi}^{A,B}$ is precisely the functor $r_{\pi}(A,B)$. 
Hence, we have the following:

\begin{lemma}
\label{pre-2-stack}
Let $\mathcal{F}$ be a presheaf of bigroupoids. Then, $\mathcal{F}$ is separated if and only if $\mathcal{F}_{A,B}$ is a prestack on $\Man/M$ for all $A,B$. Moreover, $\mathcal{F}$ is a pre-2-stack if and only if $\mathcal{F}_{A,B}$ is a stack on $\Man/M$ for all $A,B$.
\end{lemma}

Now let $\mathcal{F}$ be a separated presheaf of bigroupoids on $\Man$. 
Let again $M$ be a smooth manifold, and $A,B\in \mathcal{F}(M)$ be objects. We let $\mathcal{F}_{A,B}^{+}$ be the stack on $\Man/M$ associated to the prestack $\mathcal{F}_{A,B}$ via the (classical, 1-categorical) plus construction.
Define a presheaf of bigroupoids $\prestack{\mathcal{F}}$ by letting $\prestack{\mathcal{F}}(M)$ be the bigroupoid with objects $\mathcal{F}(M)$ and
\begin{equation*}
\hom_{\prestack{\mathcal{F}}(M)}(A,B) := \mathcal{F}^{+}_{A,B}(\id_M)\text{.}
\end{equation*}
For the composition, we find for each triple $A,B,C \in \mathcal{F}(M)$ a morphism
\begin{equation*}
\mathcal{F}_{B,C} \times \mathcal{F}_{A,B} \to \mathcal{F}_{A,C}
\end{equation*}
between presheaves of groupoids on $\Man/M$ induced by the composition functor of the bigroupoid $\mathcal{F}(Y)$.
As the plus construction is functorial, it induces a composition functor 
\begin{equation*}
\mathcal{F}^{+}_{B,C}(\id_M) \times \mathcal{F}^{+}_{A,B}(\id_M) \to \mathcal{F}^{+}_{A,C}(\id_M)
\end{equation*}
for the bigroupoid $\prestack{\mathcal{F}}(M)$.

\begin{proposition}
\label{lem:pre2stackification}
If $\mathcal{F}$ is separated, then
$\prestack{\mathcal{F}}$ is a pre-2-stack.
\end{proposition}

\begin{proof}
Let $A,B \in \prestack{\mathcal{F}}(M)$ be objects, i.e., objects in $\mathcal{F}(M)$. We must show that the relevant functor
\begin{equation*}
r_{\pi}(A,B):\hom_{\prestack{\mathcal{F}}(M)}(A,B) \to \hom_{\prestack{\mathcal{F}}(\checknerve\pi)}(r_{\pi}(A),r_{\pi}(B))
\end{equation*} 
is an equivalence of groupoids.
The left-hand side is by definition $\mathcal{F}^{+}_{A,B}(\id_M)$.  
On the right-hand side, an object is a pair $(g,\psi)$ consisting of a 1-morphism $g\maps \pi^{*}A \to \pi^{*}B$ in $\prestack{\mathcal{F}}(Y)$ and a 2-morphism $\psi: \pr_1^{*}g \Rightarrow \pr_2^{*}g$ in $\prestack{\mathcal{F}}(Y^{[2]})$ satisfying the cocycle condition $\pr_{23}^{*}\psi \circ \pr_{12}^{*}\psi = \pr_{13}^{*}\psi$. Thus, $g$ is an object in $\mathcal{F}^{+}_{A,B}(\pi)$, and $\psi$ is a morphism in $\mathcal{F}^{+}_{A,B}(\pi^{[2]})$ satisfying the cocycle condition. In other words, $(g,\psi)$ is an object in $\mathcal{F}_{A,B}^{+}(C(\pi))$. 
A morphism between $(g_1,\psi_1)$ and $(g_2,\psi_2)$ is a 2-morphism $\rho: g_1 \Rightarrow g_2$ in $\prestack{\mathcal{F}}(Y)$ such that $\pr_2^{*}\rho \circ \psi_1 = \psi_2 \circ \pr_1^{*}\rho$. Thus, $\rho$ is a morphism in $\mathcal{F}^{+}_{A,B}(\pi)$; in other words, $\rho$ is a morphism in $\mathcal{F}^{+}_{A,B}(C(\pi))$. This shows that 
\begin{equation*}
\hom_{\prestack{\mathcal{F}}(\checknerve\pi)}(r_{\pi}(A),r_{\pi}(B)) = \mathcal{F}^{+}_{A,B}(C(\pi))\text{.}
\end{equation*}
Since $\mathcal{F}^{+}_{A,B}$ is a stack, the two groupoids are equivalent.
\end{proof}

Now we recall the plus construction \cite[Section 3]{nikolaus2}. An object in $\mathcal{F}^{+}(M)$ is a quintuple $(Y,\pi,X,f,\alpha)$ consisting of a surjective submersion $\pi:Y \to M$, an object $X$ in $\mathcal{F}(Y)$, a 1-morphism $f:\pr_1^{*}X \to \pr_2^{*}X$ in $\mathcal{F}(Y^{[2]})$, and a 2-morphism $\alpha: \pr_{23}^{*}f \circ \pr_{12}^{*}f \Rightarrow \pr_{13}^{*}f$ in $\mathcal{F}(Y^{[3]})$ such that the diagram
\begin{equation*}
\alxydim{@C=5em@R=3em}{\pr_{34}^{*}f \circ \pr_{23}^{*}f \circ \pr_{12}^{*}f \ar@{=>}[r]^-{\pr_{234}^{*}\alpha \circ \id} \ar@{=>}[d]_{\id \circ \pr_{123}^{*}f} &\pr_{24}^{*}f \circ \pr_{12}^{*}f \ar@{=>}[d]^{\pr_{124}^{*}\alpha} \\ \pr_{34}^{*}f \circ \pr_{13}^{*}f \ar@{=>}[r]_-{\pr_{134}^{*}\alpha} & \pr_{14}^{*}f}
\end{equation*}
of morphisms in $\mathcal{F}(Y^{[4]})$ is commutative. 

\begin{remark}
\label{incompleteness-ns}
The description of 1-morphisms and 2-morphisms given in \cite{nikolaus2} is incomplete and does not work as intended. Specifically, a certain fibre product needed to construct the horizontal composition of 2-morphisms does not exist in general. We describe below a variation based on the theory of abelian bundle gerbes set up in \cite{waldorf1}.
\end{remark}

A 1-morphism $(Y_1,\pi_1,X_1,f_1,\alpha_1)\to (Y_2,\pi_2,X_2,f_2,\alpha_2)$ in $\mathcal{F}^{+}(M)$ is a quadruple $(Z,\zeta,g,\phi)$ consisting of a surjective submersion $\zeta: Z \to Y_1 \times_M Y_2$, a 1-morphism 
\begin{equation*}
g: \zeta^{*}\pr_1^{*}X_1 \to \zeta^{*}\pr_2^{*}X_2
\end{equation*}
in $\mathcal{F}(Z)$, and a 2-isomorphism
\begin{equation*}
\phi: f_2 \circ \pr_1^{*}g \Rightarrow \pr_2^{*}g\circ f_1
\end{equation*}
in $\mathcal{F}(Z\times_M Z)$, such that the diagram
\begin{align*}
\xymatrix@C=4em{\pr_{23}^{*}f_2 \circ \pr_{12}^{*}f_2 \circ \pr_1^{*}g \ar@{=>}[r]^-{\alpha_2\circ \id} \ar@{=>}[d]_{\id \circ \pr_{12}^{*}\phi} & \pr_{13}^{*}f_2 \circ \pr_1^{*}g \ar@{=>}[dd]^{\pr_{13}^{*}\phi} \\ \pr_{23}^{*}f_2 \circ \pr_2^{*}g \circ \pr_{12}^{*}f_1 \ar@{=>}[d]_{\pr_{23}^{*}\phi \circ \id} & \\ \pr_3^{*}g \circ \pr_{23}^{*}f_1 \circ \pr_{12}^{*}f_1 \ar@{=>}[r]_-{\id \circ \alpha_1} & \pr_3^{*}g \circ \pr_{13}^{*}f_1}
\end{align*}
of 2-morphisms in $\mathcal{F}(Z^{[3]})$ commutes.

Despite the incompleteness noted in \cref{incompleteness-ns}, the main theorem regarding the stackification property holds. Specifically, Nikolaus and Schweigert state in \cite[Theorem 3.3]{nikolaus2}: 
\begin{theorem}
\label{th:2stackification}
If $\mathcal{F}$ is a pre-2-stack, then $\mathcal{F}^{+}$ is a 2-stack. In particular, if $\mathcal{F}$ is a separated presheaf of bicategories, then $\prestack{\mathcal{F}}^{+}$ is a 2-stack. 
\end{theorem}

\cref{th:2stackification} produces the usual definitions of bundle gerbes: we denote by $\sheaf\ueins$ the sheaf of smooth $\ueins$-valued functions, and by $BB\sheaf\ueins$ its 2-fold delooping, which is a separated presheaf of bicategories. We have $\prestack{BB\sheaf\ueins}=B\bun{\ueins}-$, and $(B\bun{\ueins}-)^{+}$ is the usual 2-stack of $\ueins$-bundle gerbes.
For non-abelian bundle gerbes this procedure has been performed in \cite{Nikolaus}, and for non-abelian bundle gerbes with connections, we have applied it in \cref{sec:bundlegerbes}. In \cite{Kristel2020} it has been used to define 2-vector bundles. 

The following result may seem counterintuitive, as it states that taking the Hom-set closure is in fact unnecessary (over the site of smooth manifolds). It is yet often performed, as the resulting 2-stack has a more geometric interpretation.   

\begin{proposition}
\label{sheafification-shortcut}
If $\mathcal{F}$ is a separated presheaf of bicategories, then $\mathcal{F}^{+} \incl \overline{\mathcal{F}}^{+}$ is an isomorphism. In particular, $\mathcal{F}^{+}$ is a 2-stack.
\end{proposition}

\begin{proof}
Let $\mathcal{F}$ be a separated presheaf of bicategories. We must show that the 2-functor $\mathcal{F}^{+}(M) \to \prestack{\mathcal{F}}^{+}(M)$ is an equivalence. An object of $\prestack{\mathcal{F}}^{+}$ consists of a surjective submersion $\pi:Y \to M$, an object $X$ in $\mathcal{F}(Y)$, a 1-morphism $f:\pr_1^{*}X \to \pr_2^{*}X$ in $\prestack{\mathcal{F}}(Y^{[2]})$ and a 2-morphism $\alpha:\pr_{23}^{*}f \circ \pr_{12}^{*}f \Rightarrow \pr_{13}^{*}f$ in $\prestack{\mathcal{F}}(Y^{[3]})$. The 1-morphism $f$, in turn, is an object in $\mathcal{F}^{+}_{\pr_1^{*}X,\pr_2^{*}X}(\id_{Y^{[2]}})$, i.e., a triple $(\zeta,h,\varphi)$ consisting of a surjective submersion $\zeta:Z \to Y^{[2]}$, a 1-morphism $h: \zeta^{*}\pr_1^{*}X \to \zeta^{*}\pr_2^{*}X$ in $\mathcal{F}(Z)$, and a 2-morphism $\varphi: \pr_1^{*}h \Rightarrow \pr_2^{*}{h}$ in $\mathcal{F}(Z \times_{Y^{[2]}} Z)$ satisfying the cocycle condition in $\mathcal{F}(Z \times_{Y^{[2]}} Z \times_{Y^{[2]}} Z)$.

Because smooth manifolds are paracompact, by \cite[Lemma 7.2.3.5]{Lurie2009}, there exists a surjective submersion $\pi':Y' \to M$ with refinement maps $r:Y' \to Y$ and a section $s:Y'^{[2]} \to Z$ of $\zeta$ along $r^{[2]}:Y'^{[2]} \to Y^{[2]}$. Refining our object $(\pi,X,f,\alpha)$ along $r$, it becomes isomorphic in $\overline{\mathcal{F}}^{+}(M)$ to an object $(\pi',X',(r^{[2]})^{*}f,(r^{[3]})^{*}\alpha)$, with $X':= r^{*}X$. Next we consider the 1-morphism $f':=s^{*}h:\pr_1^{*}X' \to \pr_2^{*}X'$ in $\mathcal{F}(Y^{[2]})$, which can also be regarded as an object in $\mathcal{F}_{\pr_1^{*}X',\pr_2^{*}X'}(\id_{Y'^{[2]}})$. Then, it is clear that $f'$ and $(r^{[2]})^{*}f$ are isomorphic in $\mathcal{F}_{\pr_1^{*}X',\pr_2^{*}X'}^{+}(\id_{Y'^{[2]}})$ via $\varphi$. Hence, our object is isomorphic in $\overline{\mathcal{F}}^{+}(M)$ to $(\pi',X',f',\alpha')$, where the 2-morphism $\alpha'$ in $\overline{\mathcal{F}}(Y'^{[3]})$ corresponds to $(r^{[3]})^{*}\alpha$ under the isomorphism $f'\cong (r^{[2]})^{*}f$. Since $\mathcal{F}$ is separated, it follows that $\alpha'$ is already in $\mathcal{F}(Y'^{[3]})$. Thus, our object is one in $\mathcal{F}^{+}(M)$. 

Next suppose we have two objects $(\pi_1,X_1,f_1,\alpha_1)$ and $(\pi_2,X_2,f_2,\alpha_2)$ in $\mathcal{F}^{+}(M)$, and suppose that we have a 1-morphism between their images in $\overline{\mathcal{F}}^{+}(M)$. It consists of a surjective submersion $\zeta:Z \to Y_1\times_M Y_2$, a 1-morphism $g:\zeta^{*}\pr_1^{*}X_1 \to \zeta^{*}\pr_2^{*}X_2$ in $\overline{\mathcal{F}}(Z)$, and a 2-isomorphism $\phi$ in $\overline{\mathcal{F}}(Z^{[2]})$. On the other hand, $g$ is an object in $\mathcal{F}^{+}_{\zeta^{*}\pr_1^{*}X_1,\zeta^{*}\pr_2^{*}X_2}(\id_{Z})$, and thus a triple $g=(\rho,k,\varphi)$ consisting of a surjective submersion $\rho:W \to Z$, a 1-morphism $k: \rho^{*}\zeta^{*}\pr_1^{*}X_1 \to \rho^{*}\zeta^{*}\pr_2^{*}X_2$ in $\mathcal{F}(W)$, and a 2-isomorphism $\varphi: \pr_1^{*}k \Rightarrow \pr_2^{*}k$ in $\mathcal{F}(W \times_Z W)$. 
Now we refine the given 1-morphism $(\zeta,g,\phi)$ along $\rho$, obtaining a 2-isomorphic 1-morphism with surjective submersion $\zeta \circ \rho$, 1-morphism $\rho^{*}g$ in $\overline{\mathcal{F}}(W)$, and 2-isomorphism $(\rho^{[2]})^{*}\phi$ in $\overline{\mathcal{F}}(W^{[2]})$. Now, $\rho^{*}g$ is isomorphic to $k$ in $\mathcal{F}^{+}_{\rho^{*}\zeta^{*}\pr_1^{*}X_1 , \rho^{*}\zeta^{*}\pr_2^{*}X_2}(\id_W)$ via $\varphi$. This shows that our 1-morphism is 2-isomorphic to $(\zeta\circ \rho,k,\phi')$, where $\phi'$ corresponds to $\phi$ under the isomorphism $\rho^{*}g \cong k$. This is now a 2-isomorphism in $\overline{\mathcal{F}}(W^{[2]})$ between objects that are already contained in $\mathcal{F}(W^{[2]})$. Since $\mathcal{F}$ is separated, we have that $\phi'$ is a 2-isomorphism in $\mathcal{F}(W^{[2]})$. This shows that our given 1-morphism in $\overline{\mathcal{F}}^{+}(M)$ is 2-isomorphic to one in $\mathcal{F}^{+}(M)$.  

Finally, since $\mathcal{F}$ is separated, every 2-morphism in $\overline{\mathcal{F}}^{+}(M)$ has a representative in $\mathcal{F}^{+}(M)$. This completes the proof.
\end{proof}

\end{appendix}


\bibliographystyle{kobib}
\bibliography{kobib}

\end{document}